\documentclass{amsart}
\usepackage{breqn}
\usepackage{multirow}
\usepackage{schock}
\usepackage{placeins}
\let\Oldsection\section
\renewcommand{\section}{\FloatBarrier\Oldsection}

\let\Oldsubsection\subsection
\renewcommand{\subsection}{\FloatBarrier\Oldsubsection}

\let\Oldsubsubsection\subsubsection
\renewcommand{\subsubsection}{\FloatBarrier\Oldsubsubsection}
\title{Moduli of weighted stable marked cubic surfaces}

\author{Nolan Schock}
\email{schocknol@gmail.com}
\keywords{moduli space, stable pair, wall crossing, del Pezzo surface, cubic surface, root system}
\subjclass{Primary 14J10; Secondary 14J26, 14J45 14D22}

\begin{document}

\begin{abstract}
    Let $Y(E_n)$ denote the moduli space of pairs $(S,B)$ where $S$ is a del Pezzo surface of degree $9-n$ and $B$ is
    the labeled (marked) sum of its finitely many lines. When $n=6$, $Y(E_6)$ is the classical moduli space of marked
    cubic surfaces dating back to the nineteenth century. We describe the compactifications of $Y(E_5)$ and $Y(E_6)$ by
    Koll\'ar--Shepherd-Barron--Alexeev (KSBA) weighted stable pairs $(S,cB)$. There is a finite wall-and-chamber
    decomposition of the weight domain $\left(\frac{9-n}{N},1\right]$, and we explicitly identify this decomposition, as
    well as describe in detail the weighted stable pairs parameterized by the moduli spaces in each chamber. This
    generalizes the work of Hacking, Keel, and Tevelev \cite{hackingStablePairTropical2009} constructing the moduli
    space and its universal family in the weight 1 case, and in particular yields a complete description of the fibers
    of this family.
\end{abstract}
\maketitle
\tableofcontents

\section{Introduction}

A marking of a del Pezzo surface $S$ of degree $9-n$ induces a labeling $B_1,\ldots,B_N$ of its (finitely many) lines,
leading one to consider the moduli space $Y(E_n)$ of marked del Pezzo surfaces of degree $9-n$ as the moduli space of
pairs $(S,B = B_1 + \cdots + B_N)$. The space $Y(E_n)$ is one of the most classical moduli spaces in algebraic,
essentially originating in the case of cubic surfaces (i.e., del Pezzo surfaces of degree 3) in the nineteenth century
work of Cayley and Salmon \cite{cayleyTripleTangentPlanes2009}. There is a natural action of the Weyl group $W(E_n)$ of
the root system $E_n$ on $Y(E_n)$, given by permuting the marking, and the geometry of del Pezzo surfaces of degree $9-n$
and of the moduli space $Y(E_n)$ and its compactifications is intimately connected with the combinatorics of the $E_n$
root system, see \cite{demazureSeminaireSingularitesSurfaces1980, maninCubicFormsAlgebra1986}. Compactifications of
$Y(E_n)$ and related spaces have seen significant interest from numerous perspectives, see
\cite{allcockComplexHyperbolicGeometry2002, casalaina-martinNonisomorphicSmoothCompactifications2022,
    colomboPezzoModuliRoot2009, gallardoGeometricInterpretationToroidal2021, hassettModuliDegreePezzo2014,
    hackingStablePairTropical2009, heckmanModuliSpaceRational2000, narukiCrossRatioVariety1982,
    narukiModificationCayleyFamily1980, odakaCompactModuliSpaces2016, renTropicalizationPezzoSurfaces2016,
renTropicalizationClassicalModuli2014, schockE_6InvariantBirational2024, sekiguchiCrossRatioVarieties1994,
sekiguchiCrossRatioVarieties2000, zhaoCompactificationsModuliPezzo2023} for a (non-exhaustive) list of examples.

In this paper we study the most natural compactifications of $Y(E_n)$ from the perspective of modern moduli theory,
namely, compactifications via moduli of Koll\'ar--Shepherd-Barron--Alexeev (KSBA) (weighted) stable pairs
\cite{kollarFamiliesVarietiesGeneral2023}, generalizing the moduli space of (weighted) stable curves
\cite{deligneIrreducibilitySpaceCurves1969, hassettModuliSpacesWeighted2003}. In particular, recent work of Ascher,
Bejleri, Inchiostro, and Patakfalvi \cite{ascherWallCrossingModuli2023}, as well as Meng and Zhuang
\cite{mengMMPLocallyStable2023} shows that, in suitably nice situations, moduli of weighted stable pairs admit a finite
rational polyhedral wall-and-chamber decomposition of their weight domains, such that the moduli spaces for weights in
the same chamber are isomorphic, and there are birational wall-crossing maps between the moduli spaces in different
chambers. This generalizes Hassett's analogous results in the case of curves \cite{hassettModuliSpacesWeighted2003}.
Some examples of wall and chamber decompositions for KSBA stable pairs include
\cite{alexeevModuliWeightedHyperplane2015, ascherCompactModuliDegree2022}. In this paper we contribute two additional
examples, by completely describing the wall and chamber decompositions for del Pezzo surfaces of degrees 3 and 4 with
constant weight vector.

More precisely, assume $5 \leq n \leq 8$, so that $Y(E_n)$ is a nontrivial moduli space, and let $c \in
\left(\frac{9-n}{N},1\right]$. Then $K_S+cB$ is ample for a smooth marked del Pezzo surface $(S,B)$, and the moduli
space $Y^n_c \subset Y(E_n)$ parameterizing smooth marked del Pezzo surfaces $(S,cB)$ with log canonical singularities
admits a natural compactification $\oY^n_c$ by taking the (normalization of the) closure of $Y^n_c$ in the
aforementioned moduli space of KSBA weighted stable pairs. There are finitely many walls $t_i \in
\left(\frac{9-n}{N},1\right]$, dividing this interval into locally closed chambers $(t_{i-1},t_i]$ such that the moduli
space and its universal family is the same for any weight in a given chamber (cf. \cite{ascherWallCrossingModuli2023,
mengMMPLocallyStable2023}). We therefore write $\oY^n_{(t_{i-1},t_i]}$ for the moduli space of $c$-weighted stable
marked del Pezzo surfaces of degree $9-n$, for any $c \in (t_{i-1},t_i]$. The main result of this paper is the following
theorem, completely describing the wall and chamber decomposition of the weight domain $\left(\frac{9-n}{N},1\right]$,
in the case of cubic surfaces ($n=6$).

\begin{theorem}[\cref{thm:cubics_main}] \label{thm:cubics_intro}
    The weight domain $\left(\frac{1}{9},1\right]$ for the moduli space of weighted stable marked cubic surfaces
    $(S,cB)$ admits precisely five walls, at $c=1/6$, $1/4$, $1/3$, $1/2$, and $2/3$, inducing birational morphisms of
    moduli spaces
    \[
        \oY^6_{(1/9,1/6]} \xleftarrow{\sim} \oY^6_{(1/6,1/4]} \leftarrow  \oY^6_{(1/4,1/3]} \xleftarrow{\sim}
        \oY^6_{(1/3,1/2]} \xleftarrow{\sim} \oY^6_{(1/2,2/3]} \xleftarrow{\sim} \oY^6_{(2/3,1]},
    \]
    where each morphism is an isomorphism except for $\oY^6_{(1/6,1/4]} \leftarrow \oY^6_{(1/4,1/3]}$, which is an
    explicit birational morphism described in \cref{prop:flattening}. The weighted stable marked cubic surfaces
    parameterized by each moduli space are described in detail in \cref{sec:weighted_cubics}.
\end{theorem}

\begin{remark}
    In \cite{hassettModuliSpacesWeighted2003}, walls where the universal family changes are called \emph{fine} walls,
    and walls where the moduli space itself changes are called \emph{coarse} walls. In this terminology,
    \cref{thm:cubics_intro} says that we have fine walls at $c=1/6$, $1/4$, $1/3$, $1/2$, and $2/3$, and a coarse wall
    at $c=1/4$.
\end{remark}

In an appendix (\cref{sec:deg4}), we describe the analogous result for the case of weighted stable marked del Pezzo
surfaces of degree 4 ($n=5$).

\begin{theorem}[\cref{thm:deg4_main}] \label{thm:deg4_intro}
    The weight domain $\left(\frac{1}{4},1\right]$ for the moduli space weighted stable marked del Pezzo surfaces of
    degree 4, $(S,cB)$, admits precisely one wall, at $c=1/2$, inducing an isomorphism of moduli spaces
    \[
        \oY_{(1/4,1/2]}^5 \xleftarrow{\sim} \oY_{(1/2,1]}^5.
    \]
    The weighted stable marked del Pezzo surfaces described by each moduli space are described in detail in
    \cref{sec:deg4}.
\end{theorem}

As we show, the weighted stable surfaces for higher weights can be extremely complicated---one of the simplest singular
surfaces parameterized by the moduli space of weight 1 stable marked cubic surfaces consists of 8 irreducible components
(see \cref{fig:a_12-23,tab:a}); the most complicated surface parameterized by this space has a total 78 irreducible
components (see \cref{fig:a3b_degens_12-23,tab:aa2a3b}). The most significant contribution of this paper is an explicit
description of these weighted stable surfaces for any admissible weights, which in particular explains how these very
complicated surfaces for higher weights arise as stable replacements of much more reasonable surfaces for smaller
weights.

In particular, we note that the moduli spaces $\oY^n_c$, $n=5,6$, have previously been considered in the maximal weight
case $c=1$, as well as (in the case of cubic surfaces) the minimal weight case $c=1/9+\epsilon$
\cite{hackingStablePairTropical2009, gallardoGeometricInterpretationToroidal2021}.
\begin{itemize}
    \item In \cite{hackingStablePairTropical2009}, Hacking, Keel, and Tevelev explicitly construct the stable pair
        compactification $\oY^n_1$ of $Y^n_1$ for $n=5,6$, as well as its universal family, using tropical geometry and
        combinatorics of the $E_n$ root systems; however, they do not go so far as to actually describe the fibers of
        this universal family. Actually computing these fibers in the case of cubic surfaces ($n=6$) is the most
        technically challenging part of the present article, see \cref{sec:fibers,sec:eckardts,sec:root_combs}. The
        fibers for the case of del Pezzo surfaces of degree 4 ($n=5$) were previously computed by Hassett, Kresch, and
        Tschinkel \cite[Section 7]{hassettModuliDegreePezzo2014}; we reproduce their results in \cref{sec:fibers_deg4}.
    \item In \cite[Theorem 1.5]{gallardoGeometricInterpretationToroidal2021}, Gallardo, Kerr, and Schaffler describe the
        stable pair compactification $\oY^6_{1/9+\epsilon}$ of $Y^n_{1/9+\epsilon}$, including an explicit description
        of the surfaces it parameterizes. In contrast to the case of weight 1, these surfaces are simple to
        describe---they are either cubic surfaces with $A_1$ singularities, or transversal unions of three planes (with
        certain configurations of 27 not necessarily distinct lines) \cite[Table
        1]{gallardoGeometricInterpretationToroidal2021}. Our explicit description of the wall crossings as one varies $c
        \in \left(\frac{1}{9},1\right]$ make it clear how one obtains the more complicated surfaces appearing for higher
        weights as stable replacements of these surfaces for the minimal weight.
\end{itemize}
In fact, the minimal weight moduli space and its universal family were already constructed very explicitly by Naruki in
the 1980s \cite{narukiCrossRatioVariety1982}, building on ideas dating back to Cayley
\cite{cayleyTripleTangentPlanes2009} and Coble \cite{cobleAlgebraicGeometryTheta1968}. Naruki's work predates the notion
of KSBA weighted stable pairs. Naruki's compactification and its generalizations to del Pezzo surfaces of other degrees
(cf. \cite{sekiguchiCrossRatioVarieties2000}) also play a fundamental role in the work of Hacking, Keel, and Tevelev
studying the weight 1 case \cite{hackingStablePairTropical2009}.

\begin{remark}
    Naruki's compactification (i.e., $\oY^6_{1/9+\epsilon}$) is closely related to the familar Geometric Invariant
    Theory (GIT) compactification of the moduli space of cubic surfaces. Namely, let $\oM(E_6)$ denote the natural
    $W(E_6)$-cover of this GIT moduli space obtained by fixing a marking on the cubic surfaces; then $\oM(E_6)$ has 40
    singular points parameterizing different markings of the unique GIT polystable cubic surface $\{w^3=xyz\} \subset
    \bP^3$ with three $A_2$ singularities, and Naruki's compactification is obtained by blowing up these singular
    points, and replacing these GIT polystable marked cubic surfaces with the transversal unions of three planes as
    mentioned above \cite{narukiCrossRatioVariety1982, gallardoGeometricInterpretationToroidal2021}. (See also
    \cite{allcockComplexHyperbolicGeometry2002, casalaina-martinNonisomorphicSmoothCompactifications2022}, where
    Hodge-theoretic interpretations of these compactifications are considered.)
\end{remark}

\begin{remark}
    Wall crossings for moduli of weighted marked del Pezzo surfaces have also been considered for the analogous
    \emph{K-moduli space}, obtained when one takes $c \in \left(0,\frac{9-n}{N}\right)$, so that the pairs $(S,cB)$ are
    log Fano \cite{zhaoCompactificationsModuliPezzo2023}. Namely, in \cite{zhaoCompactificationsModuliPezzo2023}, Zhao
    shows that there are no walls on the K-moduli side for $n=5,6,7$; furthermore, the moduli spaces on this side are
    explicitly described in \cite{odakaCompactModuliSpaces2016}. In particular, for the case of cubic surfaces, the
    K-moduli space agrees with the GIT moduli space described in the previous remark.
\end{remark}

\subsection{Outline}

Since the current paper is quite lengthy, we outline its structure as well as the proofs of the main theorems. Broadly
speaking, the proofs of the main theorems proceed by the following three steps.
\begin{enumerate}
    \item Explicitly construct the weight 1 stable pair compactification $\oY^n_1$ of $Y^n_1$ ($n=5,6$), as well as its
        universal family. (This step was done by Hacking, Keel, and Tevelev \cite{hackingStablePairTropical2009}.)
    \item Explicitly compute the fibers of the universal family of stable pairs over $\oY^n_1$ ($n=5,6$).
    \item Run the log minimal model program on the fibers $(S,cB)$ of the universal family over $\oY^n_1$ as one
        decreases $c$ from 1 towards the minimal weight $\frac{9-n}{N}+\epsilon$, in order to obtain the wall and
        chamber decomposition of the weight domain as well as the corresponding descriptions of the stable models
        parameterized by the moduli spaces in each chamber.
\end{enumerate}

We briefly explain the details of the above three steps.

\subsubsection{Explicit construction of the weight 1 moduli space and its universal family}

We carry out step (1) in \cref{sec:comps,sec:morphisms,sec:eckardts}. Briefly, the construction of $\oY^n_1$ and its
universal family as described in \cite{hackingStablePairTropical2009} proceeds as follows.

For $n=5,6,7$, there is a natural $W(E_n)$-equivariant smooth projective compactification $\oY(E_n)$ of $Y(E_n)$, the
\emph{log canonical compactification}. For $n=5$, $\oY(E_5) \cong \oM_{0,5}$ is the moduli space of stable $5$-pointed
genus zero curves. For $n=6$, $\oY(E_6)$ is Naruki's compactification mentioned above, and for $n=7$, $\oY(E_7)$ is
Sekiguchi's compactification constructed in \cite{sekiguchiCrossRatioVarieties2000} as a generalization of Naruki's
results. We describe all of these compactifications in \cref{sec:comps}. By realizing these compactifications as
\emph{tropical compactifications}, Hacking, Keel, and Tevelev show that the natural morphism $Y(E_{n+1}) \to Y(E_n)$
obtained by contracting a $(-1)$-curve on a del Pezzo surface of degree $9-(n+1)$ (also induced by the inclusion of root
systems $E_n \subset E_{n+1}$) extends to a morphism of compactifications $\pi : \oY(E_{n+1}) \to \oY(E_n)$
\cite{hackingStablePairTropical2009}, which is explicitly described by combinatorics of the root systems $E_n$ and
$E_{n+1}$.

When $n=5$, the morphism $\pi : \oY(E_6) \to \oY(E_5)$ is flat, and realizes $\oY(E_5)$ as the moduli space $\oY^5_1$ of
weight 1 stable marked del Pezzo surfaces of degree 4, with universal family $\oY(E_6)$ \cite[Theorem
10.19]{hackingStablePairTropical2009}, see \cref{sec:deg4}.

When $n=6$, the morphism $\pi : \oY(E_7) \to \oY(E_6)$ is not flat, but there is a canonical flattening procedure
yielding a flat morphism $\wt\pi : \wY(E_7) \to \wY(E_6)$, where $\wY(E_6)$ and $\wY(E_7)$ are explicit birational
modifications of $\oY(E_6)$ and $\oY(E_7)$ \cite[Definition 10.21]{hackingStablePairTropical2009}, see
\cref{prop:flattening}. This morphism is still not the universal family of weight 1 stable marked cubic surfaces, due to
the presence of Eckardt points on the fibers. (Our results show that, because of these Eckardt points, $\wt\pi$ is
instead the universal family of weight $c=2/3$ stable marked cubic surfaces.) The blowup $\ddot{Y}(E_7) \to \wY(E_7)$ of
the locus cutting out these Eckardt points induces by the composition with $\wt\pi$ a flat morphism $\ddot{\pi} :
\ddot{Y}(E_7) \to \wY(E_6)$, which finally realizes $\wY(E_6) = \oY^6_1$ as the moduli space of weight 1 stable marked
cubic surfaces, with universal family $\ddot{\pi} : \ddot{Y}(E_7) \to \wY(E_6)$ \cite[Theorem
10.31]{hackingStablePairTropical2009}.  We describe the morphism $\pi : \oY(E_7) \to \oY(E_6)$ and its flattening in
\cref{sec:morphisms}, and we describe the morphism $\ddot{\pi} : \ddot{Y}(E_7) \to \wY(E_6)$ in \cref{sec:eckardts}.
Our results in \cref{sec:eckardts} include an explicit description of the possible configurations of Eckardt points on
fibers of $\wt\pi : \wY(E_7) \to \wY(E_7)$, thus giving a more complete and explicit description of $\ddot{Y}(E_7)$ than
in \cite{hackingStablePairTropical2009}. The results of \cref{sec:eckardts} are also of independent interest, as they
describe the possible configurations of Eckardt points on singular cubic surfaces.

\subsubsection{Computation of the fibers of the weight 1 universal family}

We carry out step (2), computing the fibers of the weight 1 universal family, in
\cref{sec:fibers,sec:eckardts,sec:root_combs} for the case of cubic surfaces. We carry out step (2) for the case of del
Pezzo surfaces of degree 4 in \cref{sec:deg4}.

In the case of del Pezzo surfaces of degree 4, the fibers of the universal family $\pi : \oY(E_6) \to \oY(E_5)$ are
computed as follows. The spaces $\oY(E_6)$ and $\oY(E_5)$ are tropical compactifications of their interiors $Y(E_6)$ and
$Y(E_5)$, obtained by taking their closures in respective toric varieties $X(\cF(E_6))$ and $X(\cF(E_5))$. The morphism
$\pi$ is induced by a morphism of explicitly constructed toric varieties $X(\cF(E_6)) \to X(\cF(E_5))$ (cf.
\cite{hackingStablePairTropical2009} and \cref{sec:flattening}), and the fibers of $\pi$ are the pullbacks of the fibers
of this morphism of toric varieties. There is a combinatorial procedure for computing the fibers of a morphism of toric
varieties in terms of the corresponding map of fans \cite{huToricMorphismsFibrations2004}, and pulling this back to our
morphism $\pi : \oY(E_6) \to \oY(E_5)$ amounts to computing the fibers as follows. The fibers are pairs $(S,B)$
consisting of a reducible surface $S$ and a sum $B$ of finitely many curves ($16$, in this case, since there are 16
lines on a del Pezzo surface of degree 4).  Given a codimension $k$ boundary stratum $Z$ of $\oY(E_5)$, the fiber
$(S,B)$ of $\pi$ over a general point of $Z$ is obtained by gluing together the general fibers of the restrictions
$\pi\vert_W : W \to Z$ for each codimension $k$ boundary stratum $W$ of $\oY(E_6)$ whose image in $\oY(E_5)$ is $Z$,
where the gluing is determined by the intersections of the corresponding boundary strata of $W$. This gluing gives the
surface $S$. The curves $B$ on $S$ are determined by the intersections of the strata $W$ with the 16 \emph{horizontal}
divisors, surjecting onto the space $\oY(E_5)$. All of this is explicitly described---the combinatorics of root
subsystems of $E_6$ and $E_5$ governs the irreducible components of $(S,B)$, and the restrictions of $\pi$ to the
boundary strata of $\oY(E_6)$ are given explicitly, allowing one to compute their general fibers. All of these
computations are carried out in detail in \cref{sec:deg4}.

The case of cubic surfaces follows the exact same strategy as in the previous paragraph, but with the following
modifications. Analogously to the previous paragraph, the varieties $\ddot{Y}(E_7)$ and $\wY(E_6)$ are tropical
compactifications, in respective toric varieties $X(\ddot{\cF}(E_7))$ and $X(\cwF(E_6))$, and the morphism $\ddot{\pi} :
\ddot{Y}(E_7) \to \wY(E_6)$ is pulled back from a morphism of toric varieties $X(\ddot{\cF}(E_7)) \to X(\cwF(E_6))$.
However the fan $\cwF(E_6)$ is not fine enough to distinguish cubic surfaces with Eckardt points. Thus we instead first
compute the fibers of the morphism $\wt\pi : \wY(E_7) \to \wY(E_6)$, pulled back from a morphism of toric varieties
$X(\cwF(E_7)) \to X(\cwF(E_6))$, ignoring Eckardt points (see \cref{sec:flattening}). After computing the fibers of
$\wt\pi$, we then explicitly describe the posssible configurations of Eckardt points on these fibers in order to
describe the corresponding fibers for $\ddot{\pi}$ (see \cref{sec:eckardts}).

In practice, it is also not so straightforward to compute the fibers of $\wt\pi : \wY(E_7) \to \wY(E_6)$, because the
space $\wY(E_7)$ is not smooth, nor is it easy to explicitly describe its strata. Instead, we modify our approach
further, as follows. The morphism $\wt\pi$ is a canonical flattening of $\pi : \oY(E_7) \to \oY(E_6)$, and away from the
locus where $\pi$ is not flat, $\wt\pi$ and $\pi$ have the same fibers. Thus we actually carry out the strategy
explained for degree 4 del Pezzo surfaces above, applied to the morphism $\pi : \oY(E_7) \to \oY(E_6)$, taking extra
care that when a stratum $W$ of $\oY(E_7)$ intersects the non-flat locus of $\pi$, we must instead consider the
flattening of $\wt\pi$ on the strict transform of $W$. This whole procedure is explained in detail in \cref{sec:fibers},
with the details near the non-flat locus explained in \cref{sec:nonflat_A32}. Analogous to the degree 4 case, the fibers
of $\wt\pi$ and $\pi$ are described in terms of combinatorics of root systems of $E_7$ and $E_6$. The combinatorics in
this case are extremely involved, and are described in \cref{sec:root_combs}.

\begin{remark}
    The maps of fans $\cF(E_6) \to \cF(E_5)$ and $\ddot{\cF}(E_7) \to \cwF(E_6)$ describe the universal families of
    \emph{tropical} del Pezzo surfaces of degrees 4 and 3, respectively, see \cite{renTropicalizationPezzoSurfaces2016,
    cuetoAnticanonicalTropicalCubic2019}. Tropical del Pezzo surfaces are the dual complexes to the (weight 1) stable
    marked del Pezzo surfaces constructed in this paper. Away from the Eckardt locus, these tropical surfaces are
    studied in \cite{renTropicalizationPezzoSurfaces2016, cuetoAnticanonicalTropicalCubic2019} by methods different
    from those of the present article, and our results agree with the results of \textit{loc. cit}.
\end{remark}

\subsubsection{Wall crossings and the log minimal model program on fibers}

Once the fibers of the universal family for weight 1 stable marked del Pezzo surfaces are described, it is
straightforward to run the log minimal model program to describe the weighted stable models of the fibers as one varies
$c$ in the weight domain $\left(\frac{9-n}{N},1\right]$. Namely, suppose $(S,cB)$ is a weighted stable marked del Pezzo
surface for some weight $c \in \left(\frac{9-n}{N},1\right]$. Fix $c' \in \left(\frac{9-n}{N},1\right]$ with $c' \leq
c$. The restriction of $K_S+c'B$ to an irreducible component $S'$ of $S$ is given by
\[
    K_{S'} + \Delta + c'B',
\]
where $\Delta$ is the gluing locus of $S'$ to the other irreducible components of $S$, and $B'$ is the restriction of
$B$ to $S'$. With this, it is easy to compute the log minimal model of $(S',\Delta+c'B')$, and as we will see, the log
minimal models for the different irreducible components glue to give the stable model of $(S,cB)$ for the smaller weight
$c'$. We carry this out in detail in \cref{sec:proof_main}.

\begin{remark}
    It is also possible to work in the opposite direction, beginning with the minimal weight $c=\frac{9-n}{N}+\epsilon$
    and computing the stable models as one increases $c$. This has the benefit of better clarifying the weighted stable
    surfaces appearing for larger weights, but the drawback that it is more difficult to compute stable models for
    increasing weights. We discuss some examples of this ``bottom-up'' approach in \cref{sec:bottom_up}.
\end{remark}

\subsection*{Acknowledgements}

I would like to thank Patricio Gallardo and Luca Schaffler for encouraging me that this project was a worthwhile
endeavor. I also thank Sebastian Casalaina-Martin, Izzet Coskun, Maria Angelica Cueto, Donggun Lee, and Jon Kim for very
helpful conversations and feedback.

\section{Compactifications of the moduli space of marked del Pezzo surfaces}
\label{sec:comps}

In this section we summarize the results of Hacking, Keel, and Tevelev constructing explicit compactifications of the
moduli spaces $Y(E_n)$ of marked del Pezzo surfaces of degree $9-n$, $n=5,6,7$, and using these compactifications to
construct the universal families of (weight 1) stable marked del Pezzo surfaces of degrees $9-n$, $n=5,6$
\cite{hackingStablePairTropical2009}.

\subsection{The log canonical compactification of the moduli space of marked del Pezzo surfaces}

\begin{definition}
    For $n=4,5,6,7,8$, let $\Lambda_{1,n}$ be the lattice $\bZ^{n+1}$ with basis $h,e_1,\ldots,e_n$ and pairing $h^2 =
    1$, $e_i^2=-1$, $h \cdot e_i = 0$, $e_i \cdot e_j = 0$ for $i \neq j$, and let $k = -3h+e_1 + \cdots + e_n$. The
    $E_n$ lattice is the lattice
    \[
        \Lambda(E_n) = k^{\perp} = \{\alpha \in \Lambda_{1,n} \mid \alpha \cdot k = 0\}.
    \]
    The $E_n$ root system is the set
    \[
        E_n = \{\alpha \in \Lambda(E_n) \mid \alpha^2 = -2\} = \{\alpha \in \Lambda_{1,n} \mid \alpha \cdot k = 0,
        \alpha^2 = -2\}.
    \]
    The Weyl group $W(E_n)$ of $E_n$ is the group generated by reflections through the hyperplanes orthogonal to the
    roots. Equivalently, it is the group of automorphisms of $\Lambda_{1,n}$ preserving $k$ and the pairing
    \cite[Theorem 23.9]{maninCubicFormsAlgebra1986}.
\end{definition}

\begin{definition}
    A \emph{marking} of a smooth del Pezzo surface $S$ of degree $9-n$ is an isometry $\Pic S \to \Lambda_{1,n}$ sending
    $K_S$ to $k$. Let $Y(E_n)$ denote the moduli space of marked del Pezzo surfaces of degree $9-n$.
\end{definition}

It is well-known that a marking of a del Pezzo surface $S$ of degree $9-n$ is equivalent to a construction of $S$ as the
blowup of $n$ ordered points in general position on $\bP^2$ \cite{demazureSeminaireSingularitesSurfaces1980}. This
induces a labeling of the finitely many lines on the del Pezzo surface $S$, and vice-versa, thus we usually refer to a
marked del Pezzo surface as a pair $(S,B=B_1 + \cdots + B_n)$, where $B$ is the labeled sum of the finitely many lines
on $S$. For more details on markings of del Pezzo surfaces and their relationship to the $E_n$ lattice, we refer to the
classical sources \cite{demazureSeminaireSingularitesSurfaces1980, maninCubicFormsAlgebra1986}.

By choosing appropriate coordinates on $\bP^2$ so that a smooth del Pezzo surface $S$ of degree $9-n$ is given by the
blowup of $\bP^2$ at the $n$ points $(1:0:0)$, $(0:1:0)$, $(0:0:1)$, $(1:1:1)$, $(1:x_1:y_1)$,$\ldots$,
$(1:x_{n-4}:y_{n-4})$, one sees that $Y(E_n)$ is a smooth, rational variety of dimension $2n-8$, namely, the complement
in $\bA^{n-4}_{x_i} \times \bA^{n-4}_{y_i}$ of finitely many hypersurfaces describing the loci where the $n$ points on
$\bP^2$ specified above are not in general position (cf. \cite[Section 6.3]{hackingStablePairTropical2009},
\cite[Section 7.2]{schockQuasilinearTropicalCompactifications2023}).

Note that the Weyl group $W(E_n)$ acts on $Y(E_n)$ by permuting the markings of a del Pezzo surface of degree $9-n$. We
summarize the descriptions of some geometrically interesting $W(E_n)$-equivariant compactifications of $Y(E_n)$ for
$n=5,6,7$ previously constructed by Naruki \cite{narukiCrossRatioVariety1982}, Sekiguchi
\cite{sekiguchiCrossRatioVarieties2000}, and Hacking-Keel-Tevelev \cite{hackingStablePairTropical2009}.

\begin{definition}[{\cite[1.14]{hackingStablePairTropical2009}}] \label{def:bdry_complex}
    For $n=5,6,7$, let $\cR(E_n)$ be the abstract simplicial complex defined as follows. Vertices of $\cR(E_n)$ are in
    bijection with root subsystems of $E_n$ (where $E_5=D_5$) of the form $D_2 = A_1 \times A_1$ ($n=5$), $A_1$ and $A_2
    \times A_2 \times A_2$ ($n=6$), or $A_1$, $A_2$, $A_3 \times A_3$, and $A_7$ ($n=7$). A collection of vertices forms
    a simplex if the corresponding root subsystems are pairwise either orthogonal or contained in one another. In
    addition, for $E_7$, remove the 7-simplices formed by 7-tuples of pairwise orthogonal $A_1$ root subsystems.
\end{definition}

\begin{theorem}[{\cite{hackingStablePairTropical2009}}] \label{thm:comps}
    Assume $n \leq 7$. Then there is a $W(E_n)$-equivariant smooth projective compactification $\oY(E_n)$ of $Y(E_n)$
    with simple normal crossings boundary $B$ and such that $K_{\oY(E_n)}+B$ is ample. The boundary complex of
    $\oY(E_n)$ is the simplicial complex $\cR(E_n)$. One has the following descriptions.
    \begin{enumerate}
        \item $\oY(E_5) \cong \oM_{0,5}$. There are 10 boundary divisors all isomorphic to $\oM_{0,4}$, corresponding to
            $D_2=A_1 \times A_1$ root subsystems of $E_5$.
        \item $\oY(E_6)$ is Naruki's cross-ratio variety \cite{narukiCrossRatioVariety1982}. It has 36 boundary divisors
            $\cong \oM_{0,6}$ corresponding to $A_1$ subsystems, and $40$ boundary divisors $\cong (\oM_{0,4})^3$
            corresponding to $A_2^3$ subsystems.
        \item $\oY(E_7)$ is Sekiguchi's cross-ratio variety \cite{sekiguchiCrossRatioVarieties2000}. It has a total of
            1065 boundary divisors.
            \begin{enumerate}
                \item 63 boundary divisors corresponding to $A_1$ subsystems, described in \cref{sec:A1divs} below.
                \item 336 boundary divisors $\cong \oM_{0,4} \times \oM_{0,7}$, corresponding to $A_2$ subsystems.
                \item 630 boundary divisors $\cong \oM_{0,5} \times \oM_{0,5} \times \oM_{0,4}$, corresponding to $A_3
                    \times A_3$ subsystems.
                \item 36 boundary divisors $\cong \oM_{0,8}$, corresponding to $A_7$ subsystems.
            \end{enumerate}
    \end{enumerate}
\end{theorem}

In the above theorem (and everywhere else in this article), $\oM_{0,n}$ refers to the moduli space of stable $n$-pointed
curves of genus zero. Recall that the boundary divisors of $\oM_{0,n}$ are of the form $D_{I,J} \cong \oM_{0,\lvert I
\rvert + 1} \times \oM_{0, \lvert J \rvert + 1}$, where $I,J \subset [n] = \{1,\ldots,n\}$ form a partition of $[n]$,
with $2 \leq \lvert I \rvert, \lvert J \rvert \geq n-2$. We write $D_{\lvert I \rvert, \lvert J \rvert}$ to refer to a
divisor of the form $D_{I,J}$ without specifying the particular sets $I$ and $J$.

\begin{remark}
    The statement of the above theorem implies that $\oY(E_n)$ is the log canonical compactification of $Y(E_n)$, i.e.,
    $(\oY(E_n),B)$ has log canonical singularities and $K_{\oY(E_n)} + B$ is ample. This is one of the major results of
    \cite{hackingStablePairTropical2009}. The explicit descriptions of the boundary divisors and the count of boundary
    divisors of each type are given in \cite{narukiCrossRatioVariety1982, sekiguchiCrossRatioVarieties2000} (see also
    \cite[Sections 7, 9]{hackingStablePairTropical2009}).
\end{remark}

\begin{notation} \label{not:roots}
    For a vertex $\Theta$ of $\cR(E_n)$, we write $D(\Theta)$ for the corresponding boundary divisor of $\oY(E_n)$.

    We use the realization of $E_n$ induced by the marking on a del Pezzo surface of degree $9-n$, i.e., $E_n = \{\alpha
    \in \Pic S \mid \alpha \cdot K_S = 0, \alpha^2=-2\}$. For $n=4,5,6,7$, we write the positive roots of $E_n$ as $ij =
    e_i-e_j$, $ijk = h-e_i-e_j-e_k$ for $i,j,k \in [n]=\{1,\ldots,n\}$, and, for $n=6,7$, $i=2h-\sum_{j \neq i} e_j$,
    where $i=7$ for $n=6$ and $i \in [7]$ for $n=7$ (cf. \cite{hackingStablePairTropical2009}). We refer to root
    subsystems of $E_n$ by their (unordered) tuples of positive roots---for instance,
    \[
        (123,456,7) \times (12,13,23) \times (45,46,56)
    \]
    is the $A_2^3$ root subsystem of $E_6$ whose three $A_2$ components have the positive roots indicated.
\end{notation}

The following proposition describes the intersections of boundary divisors on $\oY(E_n)$.

\begin{proposition}[{\cite{narukiCrossRatioVariety1982}, \cite{sekiguchiCrossRatioVarieties2000}, \cite[Proposition
    9.17]{hackingStablePairTropical2009}}] \label{prop:bdry_strata}
    \begin{enumerate}
        \item On $\oY(E_5)$, two boundary divisors intersect if and only if the corresponding $D_2$ subsystems of $E_5$
            are orthogonal. The intersection in this case is a point.
        \item On $\oY(E_6)$, one has the following types of intersections among boundary divisors.
            \begin{enumerate}
                \item $D(A_1)$ and $D(A_1')$ intersect if and only if $A_1 \perp A_1'$. In this case, the intersection
                    is isomorphic to $\oM_{0,5} \cong D_{2,4}$ on $D(A_1) \cong \oM_{0,6}$.
                \item $D(A_1)$ and $D(A_2^3)$ intersect if and only if $A_1 \subset A_2^3$. In this case, the
                    intersection is isomorphic to $\oM_{0,4} \times \oM_{0,4} \cong D_{3,3}$ on $D(A_1) \cong
                    \oM_{0,6}$. This is a divisor of the form $p \times \bP^1 \times \bP^1$, $\bP^1 \times p \times \bP^1$,
                    or $\bP^1 \times \bP^1 \times p$, $p=0$, $1$, or $\infty$, on $D(A_2^3) \cong (\bP^1)^3$.
                \item No two $D(A_2^3)$ divisors intersect.
            \end{enumerate}
        \item On $\oY(E_7)$, one has the following types of intersections among boundary divisors.
            \begin{enumerate}
                \item $D(A_1)$ and $D(A_1')$ intersect if and only if $A_1 \perp A_1'$. In this case, the intersection
                    is isomorphic to a smooth projective fourfold called $\oZ(D_4)$, described in \cref{sec:A1divs}.
                \item $D(A_1)$ and $D(A_2)$ intersect if and only if either $A_1 \subset A_2$, or $A_1 \perp A_2$. In
                    the former case, the intersection is isomorphic to $p \times \oM_{0,7}$ on $D(A_2) \cong \oM_{0,4}
                    \times \oM_{0,7}$ (where $p=0,1,\infty$ is a boundary divisor of $\oM_{0,4} \cong \bP^1$). In the
                    latter case, the intersection is isomorphic to $\oM_{0,4} \times \oM_{0,6} \cong \oM_{0,4} \times
                    D_{2,5} \subset D(A_2)$.
                \item $D(A_1)$ and $D(A_3^2)$ intersect if and only if $A_1 \subset A_3^2$ or $A_1 \perp A_3^2$. In the
                    former case, the intersection is isomorphic to $\oM_{0,4} \times \oM_{0,5} \times \oM_{0,4} \cong
                    D_{2,3} \times \oM_{0,5} \times \oM_{0,4}$ or $\oM_{0,5} \times \oM_{0,4} \times \oM_{0,4} \cong
                    \oM_{0,5} \times D_{2,3} \times \oM_{0,4}$ as a divisor on $D(A_3^2) \cong \oM_{0,5} \times
                    \oM_{0,5} \times \oM_{0,4}$. In the latter case, the intersection is isomorphic to $\oM_{0,5} \times
                    \oM_{0,5} \times p$ on $D(A_3^2)$, where $p=0,1,\infty$ is a boundary divisor of $\oM_{0,4} \cong
                    \bP^1$.
                \item $D(A_1)$ and $D(A_7)$ intersect if and only if $A_1 \subset A_7$. In this case, the intersection
                    is isomorphic to $\oM_{0,7} \cong D_{2,6}$ on $D(A_7) \cong \oM_{0,8}$.
                \item $D(A_2)$ and $D(A_2')$ intersect if and only if $A_2 \perp A_2'$. In this case, the intersection
                    is isomorphic to $\oM_{0,4} \times \oM_{0,4} \times \oM_{0,5} \cong \oM_{0,4} \times D_{3,4}$ on
                    $D(A_2) \cong \oM_{0,4} \times \oM_{0,7}$.
                \item $D(A_2)$ and $D(A_3^2)$ intersect if and only if $A_2 \subset A_3^2$. In this case, the
                    intersection is isomorphic to $\oM_{0,4} \times \oM_{0,4} \times \oM_{0,5} \cong \oM_{0,4} \times
                    D_{3,4}$ on $D(A_2) \cong \oM_{0,4} \times \oM_{0,7}$, and to $\oM_{0,4} \times \oM_{0,5} \times
                    \oM_{0,4} \cong D_{2,3} \times \oM_{0,5} \times \oM_{0,4}$ or $\oM_{0,5} \times \oM_{0,4} \times
                    \oM_{0,4} \cong \oM_{0,5} \times D_{2,3} \times \oM_{0,4}$ on $D(A_3^2) \cong \oM_{0,5} \times
                    \oM_{0,5} \times \oM_{0,4}$.
                \item $D(A_2)$ and $D(A_7)$ intersect if and only if $A_2 \subset A_7$. In this case, the intersection
                    is isomorphic to $\oM_{0,4} \times \oM_{0,6} \cong \oM_{0,4} \times D_{2,5}$ on $D(A_2) \cong
                    \oM_{0,4} \times \oM_{0,7}$, and $D_{3,5}$ on $D(A_7) \cong \oM_{0,8}$.
                \item No two $D(A_3^2)$ divisors intersect.
                \item $D(A_3^2)$ and $D(A_7)$ intersect if and only if $A_3^2 \subset A_7$. In this case, the
                    intersection is isomorphic to $\oM_{0,5} \times \oM_{0,5} \times p$ as a divisor on $D(A_3^2) =
                    \oM_{0,5} \times \oM_{0,5} \times \oM_{0,4}$ (where $p=0,1,\infty$ is a boundary divisor of
                    $\oM_{0,4}$), and to $D_{4,4}$ on $D(A_7) \cong \oM_{0,8}$.
                \item No two $D(A_7)$ divisors intersect.
            \end{enumerate}
    \end{enumerate}
\end{proposition}

\begin{remark} \label{rmk:bdry_strata}
    Note that any divisor $D$ on $\oY(E_n)$ not of type $A_1$ on $\oY(E_7)$ is isomorphic to a product of copies of
    $\oM_{0,n_i}$'s, for appropriate $n_i$'s. The intersections of $D$ with the other boundary divisors on $\oY(E_n)$
    are in bijection with the natural boundary divisors of $D$ as a product of $\oM_{0,n_i}$'s. In particular, the
    intersections of any number of boundary divisors of $\oY(E_n)$ recover the natural boundary strata of the products
    of $\oM_{0,n_i}$'s. For the analogous results concerning $A_1$ divisors on $\oY(E_7)$, see
    \cref{prop:A1divs,prop:A1strata}.
\end{remark}

\subsection{$A_1$ divisors on $\oY(E_7)$} \label{sec:A1divs}

In order to more precisely describe the $A_1$ divisors on $\oY(E_7)$, we recall Naruki's explicit construction of
$\oY(E_6)$ as a sequence of blowups and blowdowns of a toric variety \cite{narukiCrossRatioVariety1982}.

Let $Z(D_4)$ be the complement in $(\bC^*)^4 \cong \Hom(\Lambda(D_4),\bC^*)$ of the hypertori $\{e^{\alpha}=1\}$ defined
by the positive roots $\alpha$ of the $D_4$ root system. The $k$-dimensional strata of this arrangement of hypertori
(i.e., the connected components of the intersections of the hypertori) are in bijection with root subsystems of $D_4$
obtained by deleting $k+1$ vertices from the extended Dynkin diagram for $D_4$ (see, for instance \cite[Theorem
7.17]{alexeevADESurfacesTheir2020}). Let $\oZ(D_4)$ be the minimal wonderful compactification of $Z(D_4)$, obtained as
follows. First let $X(\Sigma)$ be the complete toric variety associated to the Coxeter fan $\Sigma$ whose cones are the
$D_4$ Weyl chambers. Then $\oZ(D_4)$ is the blowup of $X(\Sigma)$ at the identity $e \in (\bC^*)^4$ (corresponding to
the whole $D_4$ root system), the closures of the 12 curves in $(\bC^*)^4$ corresponding to $A_3$ root subsystems, and
the closures of the 16 surfaces in $(\bC^*)^4$ corresponding to $A_2$ root subsystems. The divisors over the $A_3$
curves are isomorphic to $\oM_{0,4} \times \oM_{0,5}$, and $\oY(E_6)$ is obtained from $\oZ(D_4)$ by contracting each
such divisor via the projection $\oM_{0,4} \times \oM_{0,5} \to \oM_{0,5}$.

Note that the strict transform of the exceptional divisor over the identity $e \in (\bC^*)^4$ is isomorphic to the
blowup of $\bP^3$ at 12 points corresponding to $A_3$ root subsystems of $D_4$, and 16 lines corresponding to $A_2$ root
subsystems of $D_4$. This is the minimal wonderful compactification $\oX(D_4)$ of the complement $X(D_4)$ of the
hyperplanes in $\bP^3$ defined by the positive roots of $D_4$. This is mapped isomorphically to a divisor in $\oY(E_6)$
parameterizing marked cubic surfaces such that three specified lines (determined by the choice of $D_4$ subsystem) meet
in an Eckardt point. The $W(E_6)$-orbit of this divisor in $\oY(E_6)$ consists of 45 irreducible components, each
isomorphic to $\oX(D_4)$, making up the Eckardt locus of $\oY(E_6)$. We refer to \cite{narukiCrossRatioVariety1982} for
more details. (These divisors are called tritangent divisors in \cite{narukiCrossRatioVariety1982}.)

\begin{remark}
    The space $Z(D_4)$ is isomorphic to the moduli space of smooth marked cubic surfaces together with an anticanonical
    cycle consisting of three lines meeting in a triangle, and the morphism $\oZ(D_4) \to \oY(E_6)$ described above is
    induced by the morphism $Z(D_4) \to Y(E_6)$ forgetting the anticanonical cycle \cite[Appendix by E.
    Looijenga]{narukiCrossRatioVariety1982}. Hacking, Keel, and Tevelev similarly construct $\oY(E_n)$ for $5 \leq n
    \leq 7$ using the moduli space of smooth marked del Pezzo surfaces of degree $9-n$ together with an anticanonical
    cuspidal curve---this space is identified with the complement in $\bP^{n-1}$ of the $E_n$ hyperplane arrangement
    \cite[Section 6]{hackingStablePairTropical2009}.
\end{remark}

The spaces $\oZ(D_4)$ and $\oX(D_4)$ described above appear in the boundary of $\oY(E_7)$. The following
results are contained in \cite{sekiguchiCrossRatioVarieties2000}.

\begin{proposition}[{\cite{sekiguchiCrossRatioVarieties2000}}] \label{prop:A1divs}
    For $A_1 \subset E_7$, let $D^{\circ}(A_1)$ be the moduli space of arrangements of 7 points $p_1,\ldots,p_7$ in
    $\bP^2$ which are in general position except for the following.
    \begin{enumerate}
        \item $A_1=ij$: points $p_i$ and $p_j$ coincide.
        \item $A_1=ijk$: points $p_i,p_j,p_k$ lie on a line.
        \item $A_1=i$: points $p_j$ for $j \neq i$ lie on a conic.
    \end{enumerate}
    Then $D(A_1) \subset \oY(E_7)$ is a smooth, projective, simple normal crossings compactification of
    $D^{\circ}(A_1)$, whose boundary consists of one irreducible component for each intersection of $D(A_1)$ with
    another boundary divisor of $\oY(E_7)$. These intersections are described in \cref{prop:bdry_strata}, 3(a)--(d).
\end{proposition}

\begin{proposition}[{\cite{sekiguchiCrossRatioVarieties2000}}] \label{prop:A1strata}
    \begin{enumerate}
        \item Any nonempty intersection of two $D(A_1)$ divisors in $\oY(E_7)$ is isomorphic to $\oZ(D_4)$.
        \item There are two types of nonempty intersections of three $D(A_1)$ divisors in $\oY(E_7)$.
            \begin{enumerate}
                \item One type isomorphic to the variety $\oX(D_4)$ described above, appearing in $D(A_1) \cap D(A_1)
                    \cong \oZ(D_4)$ of the strict transform of the exceptional divisor over $e \in (\bC^*)^4$. This type
                    occurs when the orthogonal complement of the three $A_1$'s is a $D_4 \subset E_7$.
                \item One type appearing in $D(A_1) \cap D(A_1) \cong \oZ(D_4)$ as the strict transform of the closure of
                    a $D_4$ root hypertorus. This type occurs when the orthogonal complement of the three $A_1$'s is a $3A_1
                    \subset E_7$.
            \end{enumerate}
        \item Any nonempty intersection of four $D(A_1)$ divisors in $\oY(E_7)$ is isomorphic to the blowup of $\bP^2$
            at six points such that any three points lie on a line. This is the minimal resolution of a cubic surface
            with four $A_1$ singularities.
        \item Any nonempty intersection of five $D(A_1)$ divisors in $\oY(E_7)$ is a curve $\cong \bP^1$.
        \item Any nonempty intersection of six $D(A_1)$ divisors in $\oY(E_7)$ is a point.
    \end{enumerate}
\end{proposition}

\begin{remark}
    Note in particular that the intersection $Z$ of several $A_1$ divisors is a smooth, projective, simple normal
    crossings compactification of the moduli space $Z^{\circ}$ of arrangements of 7 points $p_1,\ldots,p_7$ in $\bP^2$
    which are in general position except for the degeneracies described by the $A_1$'s. For instance, for $D(7) \cap
    D(56)$, the points $p_1,\ldots,p_6$ all lie on a conic, with $p_6$ infinitely near to $p_5$.
\end{remark}

\begin{remark}
    As we will see in \cref{sec:eckardts}, the first type of intersection of three $A_1$ divisors in $\oY(E_7)$
    (occuring when the orthogonal complement of the three $A_1$'s is a $D_4 \subset E_7$) corresponds geometrically to
    Eckardt points in the fibers of the universal family of weighted stable marked cubic surfaces.
\end{remark}

\section{The morphism $\pi : \oY(E_7) \to \oY(E_6)$ and its flattening} \label{sec:morphisms}

In this section, we summarize the results of Hacking, Keel, and Tevelev, showing that the natural morphism $Y(E_7) \to
Y(E_6)$ obtained by contracting a $(-1)$-curve on a del Pezzo surface of degree 2 extends to an explicitly described
morphism of compactifications $\pi : \oY(E_7) \to \oY(E_6)$ \cite{hackingStablePairTropical2009} (\cref{thm:pi}).
Furthermore, this morphism is flat except along a certain union of boundary divisors of $\oY(E_7)$, and there is a
canonical flattening procedure yielding a flat morphism $\wt\pi : \wY(E_7) \to \wY(E_6)$, where $\wY(E_6)$ and
$\wY(E_7)$ are explicit birational modifications of $\oY(E_6)$ and $\oY(E_7)$ \cite[Proposition
10.23]{hackingStablePairTropical2009} (\cref{sec:flattening}). As we will see in
\cref{sec:fibers,sec:eckardts,sec:weighted_cubics}, $\wt\pi$ gives the universal family of \emph{weighted} stable marked
cubic surfaces, for weights $1/2 < c \leq 2/3$, and a suitable blowup $\ddot{Y}(E_7)$ of $\wY(E_7)$ induces the
universal family of weight 1 stable marked cubic surfaces \cite[Theorem 10.31]{hackingStablePairTropical2009}
(\cref{prop:resolve_eckardt}). The failure of $\wt\pi : \wY(E_7) \to \wY(E_6)$ itself to be the weight 1 universal family
is due to the presence of Eckardt points on the fibers, see \cref{sec:eckardts}.

\begin{remark} \label{rmk:deg4_map}
    Analogously to the results described in this section, Hacking, Keel, and Tevelev also show that the natural morphism
    $Y(E_6) \to Y(E_5)$ extends to a morphism of compactifications $\oY(E_6) \to \oY(E_5)$, and in this case the
    morphism is flat and gives the universal family of stable marked del Pezzo surfaces of degree 4 \cite[Theorem
    10.19]{hackingStablePairTropical2009}, see \cref{sec:deg4}.
\end{remark}

\begin{theorem}[{\cite[Proposition 10.9]{hackingStablePairTropical2009}}] \label{thm:pi}
    The morphism $Y(E_7) \to Y(E_6)$ obtained by contracting a $(-1)$-curve on a del Pezzo surface of degree $2$ extends
    to a morphism of compactifications $\pi : \oY(E_7) \to \oY(E_6)$. The restriction of $\pi$ to each boundary divisor
    of $\oY(E_7)$ is described in \cref{tab:pi}, where each morphism $\pi_I : \oM_{0,n} \to \oM_{0,m}$ denotes the
    canonical fibration dropping points not labeled by $I$, and the morphism $\oM_{0,7} \to \oM_{0,4} \times \oM_{0,4}$
    appearing in the restriction of $\pi$ to $D(A_2) \to D(A_2^3)$ is (up to relabeling) of the form $\pi_{1237} \times
    \pi_{4567}$.
    \begin{table}[htpb]
        \centering
        \caption{The morphism $\pi : \oY(E_{7}) \to \oY(E_6)$ (see \cite[Proposition
        10.9]{hackingStablePairTropical2009} for more details).}
        \label{tab:pi}
        \begin{tabular}{| c | c | c | c |}
            \hline
            $D(\Theta) \subset \oY(E_7)$ & $\pi(D(\Theta)) \subset \oY(E_6)$ & Condition & $\pi\vert_{D(\Theta)}$ \\
            \hline
            \hline
            $D(A_1)$ & $D(A_1)$ & $A_1 \subset E_6$ & \cref{lem:A1_fiber} \\
            $D(A_1)$ & $\oY(E_6)$ & $A_1 \not\subset E_6$ & \\
            \hline
            $D(A_2)$ & $D(A_1)$ & $A_2 \cap E_6 = A_1$ & $\oM_{0,4} \times \oM_{0,7} \to pt \times \oM_{0,6}$ \\
            \multirow{2}{*}{$D(A_2)$} & \multirow{2}{*}{$D(A_2^3)$} & $A_2 \subset E_6$,  & \multirow{2}{*}{$\oM_{0,4} \times
            \oM_{0,7} \to \oM_{0,4} \times (\oM_{0,4} \times \oM_{0,4})$} \\
                     & & $A_2^3 = A_2 \times A_2^{\perp}$ & \\
            \hline
            \multirow{2}{*}{$D(A_3^2)$} & \multirow{2}{*}{$D(A_2^3)$} & $A_3^2 \cap E_6 = A_2^2$,  & \multirow{2}{*}{$\oM_{0,5}
            \times \oM_{0,5} \times \oM_{0,4} \to \oM_{0,4} \times \oM_{0,4} \times \oM_{0,4}$} \\
                       & & $A_2^3 = A_2^2 \times (A_2^2)^{\perp}$ & \\
            $D(A_3^2)$ & $D(A_1) \cap D(A_1)$ & $A_3^2 \cap E_6 = A_3 \times A_1 \times A_1$ & $\oM_{0,5} \times
            \oM_{0,5} \times \oM_{0,4} \to \oM_{0,5}$ \\
            \hline
            $D(A_7)$ & $D(A_1)$ & $A_7 \cap E_6 = A_1 \times A_5$ & $\oM_{0,8} \to \oM_{0,6}$ \\
            \hline
        \end{tabular}
    \end{table}
\end{theorem}

\subsection{Flattening of $\pi : \oY(E_7) \to \oY(E_6)$} \label{sec:flattening}

Note from \cref{thm:pi} that the morphism $\pi : \oY(E_7) \to \oY(E_6)$ is not flat along the union of the $D(A_3^2)
\mapsto D(A_1) \cap D(A_1)$ divisors (henceforth \emph{non-flat} $A_3^2$ divisors \cite[Definition
10.10]{hackingStablePairTropical2009}). There is a canonical flattening of the morphism $\pi$, described as follows.

The variety $\oY(E_n)$ is a \emph{tropical compactification}, obtained as the closure of $Y(E_n)$ in a toric variety
$X(\cF(E_n))$ associated to a fan $\cF(E_n)$ whose underlying simplicial complex is the complex $\cR(E_n)$ of
\cref{def:bdry_complex} (\cite[Theorem 1.16]{hackingStablePairTropical2009}). The dominant morphism $Y(E_{n+1}) \to
Y(E_n)$ contracting a $(-1)$-curve induces a surjection of fans $\cF(E_{n+1}) \to \cF(E_n)$ by \cite[Theorem
1.17]{hackingStablePairTropical2009}, and the corresponding morphism of toric varieties pulls back to the morphism $\pi
: \oY(E_{n+1}) \to \oY(E_n)$ \cite[Theorem 1.18]{hackingStablePairTropical2009}. In particular, the morphism $\pi :
\oY(E_{n+1}) \to \oY(E_n)$ is entirely determined by the morphism of toric varieties $X(\cF(E_{n+1})) \to X(\cF(E_n))$.
There is a canonical combinatorial procedure for flattening a morphism of toric varieties (or more generally, toroidal
embeddings) \cite{abramovichWeakSemistableReduction2000, hackingStablePairTropical2009}, which in the present situation
is described as follows.

\begin{proposition}[{\cite[Section 10]{hackingStablePairTropical2009}}] \label{prop:flattening}
    Let $p : \cF(E_7) \to \cF(E_6)$ be the morphism of fans induced by $\pi : \oY(E_7) \to \oY(E_6)$. Let $\cwF(E_6)$ be
    the refinement of $\cF(E_6)$ obtained by taking the barycentric subdivision of cones formed by 4-tuples of pairwise
    orthogonal $A_1$ subsystems, as well as the corresponding minimal subdivisions of the cones formed by 3 $A_1$
    subsystems and an $A_2^3$ subsystem. Let $\cwF(E_7)$ be the corresponding refinement of $\cF(E_7)$,
    \[
        \cwF(E_7) = \{p^{-1}(\gamma) \cap \sigma \mid \gamma \in \cwF(E_6), \sigma \in \cF(E_7)\}.
    \]
    Let $\wY(E_n)$ be the toroidal modification of $\oY(E_n)$ corresponding to the refinement $\cwF(E_n)$ of $\cF(E_n)$.
    Then the induced morphism $\wt\pi : \wY(E_7) \to \wY(E_6)$ is flat with reduced fibers.
\end{proposition}

\begin{remark}
    Note that $\wY(E_6)$ is the blowup of $\oY(E_6)$, in increasing order of dimension, of the boundary strata formed by
    intersections of $D(A_1)$ divisors (cf. \cref{thm:comps}). Thus $\wY(E_6)$ is a smooth projective variety. On the
    other hand, the fan $\cwF(E_7)$ is no longer even simplicial, and correspondingly $\wY(E_7)$ is no longer smooth.
\end{remark}

\section{Fibers of $\wt\pi : \wY(E_7) \to \wY(E_6)$} \label{sec:fibers}

The goal of this section is to compute the fibers of the flat family $\wt\pi : \wY(E_7) \to \wY(E_6)$ constructed in
\cref{prop:flattening}. This is the most technically challenging part of the present article. For a simpler example of
these computations, we refer the reader to \cref{sec:fibers_deg4}, where we analogously compute the fibers of the
flat family $\oY(E_6) \to \oY(E_5)$.

Recall that $\wY(E_6)$ is the blowup of $\oY(E_6)$ along all intersections of $A_1$ divisors, in increasing order of
dimension. We introduce the following notation, cf. \cite{renTropicalizationPezzoSurfaces2016}.

\begin{notation} \label{not:strata}
    The boundary divisors of type $a_i$, $i=1,2,3,4$ on $\wY(E_6)$ are the strict transforms of the exceptional divisors
    over the intersections of $i$ $A_1$ divisors in $\oY(E_6)$. For $i=1$, we write $a=a_1$. The boundary divisors of
    type $b$ are the strict transforms of the $A_2^3$ divisors on $\oY(E_6)$. We write $D(\Sigma)$ for the boundary
    divisor whose corresponding root subsystems of $E_6$ are labeled by $\Sigma$; for instance, $D(A_1 \perp A_1')$
    indicates the type $a_2$ boundary divisor corresponding to two orthogonal $A_1$ root subsystems of $E_6$.

    Higher codimension boundary strata are labeled by juxtaposition---thus for instance the curves of type $aa_2a_3$ on
    $\wY(E_6)$ are the curves obtained as intersections of a boundary divisor of type $a$, a boundary divisor of type
    $a_2$, and a boundary divisor of type $a_3$.

    Let $T$ be a boundary stratum of $\wY(E_6)$, labeled as above. We say a surface $(S,B)$ obtained as the fiber of
    $\wt\pi : \wY(E_7) \to \wY(E_6)$ over a general point of $T$ is of type $T$.
\end{notation}

\begin{remark}
    It is straightforward to count the number of strata of $\wY(E_6)$ of each type---see \cite[Table
    1]{renTropicalizationPezzoSurfaces2016}. However, we note a mistake in the aforementioned table, which states that
    there are 1620 strata of type $a_4$, when in fact there are only 135.
\end{remark}

\begin{notation}[{\cite[Definition 10.18]{hackingStablePairTropical2009}}] \label{not:hor_divs}
    The \emph{horizontal} $A_1$ divisors on $\wY(E_7)$ are the strict transforms of the $A_1$ divisors of $\oY(E_7)$
    which surject onto $\oY(E_6)$ under the morphism $\pi : \oY(E_7) \to \oY(E_6)$. We denote by $B_{\wt\pi}$ the sum of
    the horizontal $A_1$ divisors on $\wY(E_7)$. The \emph{lines} on a fiber of $\wt\pi : \wY(E_7) \to \wY(E_6)$ are
    given by the intersections of this fiber with the horizontal $A_1$ divisors.
\end{notation}

Using the natural embedding of $E_6$ in $E_7$ induced by the realizations of these root systems described in
\cref{not:roots}, the horizontal $A_1$ divisors on $\wY(E_7)$ correspond to the 27 $A_1$ root subsystems of $E_7$ given
by $A_1 = (ij7)$, $ij \subset [6]$, $(i)$, $i=1,\ldots,6$, and $(i7)$, $i=1,\ldots,6$.

In the below proposition, we write $\wt\pi : (\wY(E_7),B_{\wt\pi}) \to \wY(E_6)$ in order to indicate the fact that we
keep track of the lines on the fibers of $\wt\pi$ cut out by the horizontal $A_1$ divisors.

\begin{theorem} \label{thm:fibers}
    The fiber of the morphism $\wt\pi : (\wY(E_7), B_{\wt\pi}) \to \wY(E_6)$ over a general point of $\wY(E_6)$ (i.e., a
    point of $Y(E_6)$) is a smooth marked cubic surface. For a given boundary stratum $Z$ of $\wY(E_6)$, the fiber of
    $\wt\pi$ over a general point of $Z$ is a reducible surface as described in
    \cref{fig:a_12-23,fig:a2_12-23,fig:aa2_12-23,fig:a3_12-23,fig:a3_12-23_degens,fig:a4_12-23,fig:a4_12-23_degens,fig:b_13-23,fig:ab_12-23,fig:a2b_12-23,fig:aa2b_12-23,fig:a3b_12-23,fig:a3b_degens_12-23}
    and \cref{sec:root_combs}.
\end{theorem}

\begin{proof}
    The general strategy behind our computation of the fibers of $\wt\pi : \wY(E_7) \to \wY(E_6)$ is the following. Let
    $Z$ be a codimension $k$ stratum of $\wY(E_6)$, and let $(S,B)$ be the fiber of $\wt\pi$ over a general point of
    $Z$. For each codimension $k$ stratum $W$ of $\wY(E_7)$ dominating $Z$, the general fiber of $\wt\pi\vert_{W} : W
    \to Z$ is a surface $S'$ yielding an irreducible component of the fiber $(S,B)$. The lines $B'$ on $S'$ are obtained
    by the intersections of $S'$ with the horizontal $A_1$ divisors on $\wY(E_7)$ (equivalently, the general fibers of
    the intersections $W \cap D(A_1)$ for $D(A_1)$ a horizontal $A_1$ divisor). The fiber $(S,B)$ is the union of the
    irreducible components $(S',B')$ for each codimension $k$ stratum $W$ of $\wY(E_7)$ dominating $Z$, where the
    intersections of these irreducible components are described by the general fibers of the restrictions of $\wt\pi$ to
    the intersections of the corresponding strata of $\wY(E_7)$. In order to compute the general fibers of the
    restrictions of $\wt\pi$, we use \cref{tab:pi} together with the lemmas of \cref{sec:restricted_fibers} below.

    In practice, the strategy of the previous paragraph requires a minor modification. Namely, recall that $\wt\pi :
    \wY(E_7) \to \wY(E_6)$ is the canonical flattening of the morphism $\pi : \oY(E_7) \to \oY(E_6)$, and $\pi$ is flat
    outside of the union of the non-flat $A_3^2$ divisors. Therefore, outside of the non-flat $A_3^2$ divisors, $\wt\pi$ is
    the pullback of $\pi$, and on this locus we may compute the fibers of $\pi$ instead (see \cite[Proof of Proposition
    10.23]{hackingStablePairTropical2009}.) To handle the fibers of $\wt\pi$ near the non-flat $A_3^2$ divisors, we describe
    the restriction of $\wt\pi$ to the strict transforms of these divisors, see \cref{lem:nonflat_A32}, \cite[Proof of
    Proposition 10.27]{hackingStablePairTropical2009}.

    We apply the strategy summarized above to a chosen $W(E_6)$-representative of each type boundary stratum of
    $\wY(E_6)$; this is sufficient as the boundary strata of a given type are transitively permuted by $W(E_6)$. The
    results are described by combinatorics of root subsystems $E_6$ and $E_7$ described in detail in
    \cref{sec:root_combs}. The tables in \textit{loc. cit.} list, for each chosen boundary stratum $Z$ of $\wY(E_6)$,
    the collections of root subsystems of $E_7$ yielding the corresponding boundary strata of $\wY(E_7)$ dominating $Z$.
    From these root subsystems it is easy to describe the lines on the corresponding irreducible components by
    intersecting with the horizontal $D(A_1)$ divisors for $A_1$ contained in or orthogonal to the given collection of
    root subsystems. Likewise we describe the intersections of these irreducible components by examining when the
    corresponding collections of root subsystems are compatible and so form a higher codimension boundary stratum of
    $\wY(E_7)$. As a result of these computations, we obtain the surfaces described in the relevant figures of the
    statement of the proposition.
\end{proof}

\begin{remark}
    The validity of the above strategy is justified by the fact that the morphism $\wt\pi : \wY(E_7) \to \wY(E_6)$ is
    obtained as the pullback of a morphism of toric varieties $X(\cwF(E_7)) \to X(\cwF(E_6))$ (\cref{prop:flattening}).
    The strategy outlined above is just the pullback of the usual strategy for computing the fibers of a morphism of
    toric varieties, see for instance \cite[Section 2]{huToricMorphismsFibrations2004}. Partial results concerning the
    fibers of $X(\cwF(E_7)) \to X(\cwF(E_6))$ (or rather, of the corresponding map of fans $\cwF(E_7) \to \cwF(E_6)$)
    have been described in \cite[Table 1]{renTropicalizationPezzoSurfaces2016}, and our results agree with \textit{loc.
    cit}.
\end{remark}

\subsection{Fibers of restrictions of $\wt\pi$} \label{sec:restricted_fibers}

In this subsection we collect some lemmas that describe the general fibers of the restrictions of $\wt\pi : \wY(E_7) \to
\wY(E_6)$. Recall from the proof \cref{thm:fibers} that these yield irreducible components of the fibers of $\wt\pi :
\wY(E_7) \to \wY(E_6)$. As discussed in the proof of \textit{loc. cit.}, outside of the union of the non-flat $A_3^2$
divisors, we can instead compute the fibers of $\pi : \oY(E_7) \to \oY(E_6)$.

\begin{lemma} \label{lem:fiber_m0n}
    The general fiber of the natural forgetful map $\oM_{0,n+2} \to \oM_{0,n}$, forgetting the last two marked points,
    is the blowup $Bl_n\bF_0$ of $\bF_0 \cong \bP^1 \times \bP^1$ at $n$ points on the diagonal.
\end{lemma}

\begin{proof}
    Factor the map as $\oM_{0,n+2} \xrightarrow{f_{n+2}} \oM_{0,n+1} \xrightarrow{f_{n+1}} \oM_{0,n}$. Then the general
    fiber of $f_{n+1} : \oM_{0,n+1} \to \oM_{0,n}$ is a curve $C \cong \bP^1$ with $n$ marked points cut out by the
    intersections with the $n$ divisors $D_{i,n+1}$ for $i=1,\ldots,n$. The fiber of $f_{n+2} : \oM_{0,n+2} \to
    \oM_{0,n+1}$ over a general point of $C$ is a curve $C' \cong \bP^1$ with $n+1$ marked points cut out by the
    intersections with the $n+1$ divisors $D_{i,n+2}$ for $i=1,\ldots,n+1$. The fiber over a marked point of $C$ is a
    stable $(n+1)$-pointed curve with 2 components, with points $i,n+1$ on 1 component, and the remaining points on the
    other component. Thus $\pi^{-1}(C) \to C$ is a fibration over $\bP^1$ with general fiber $\bP^1$, and $n$ special
    fibers the described stable curves, marked points being cut out by intersections with the divisors $D_{i,n+2}$ for
    $i=1,\ldots,n+1$. Doing the same thing but swapping the order of the forgetful maps, one obtains an analogous
    description but now with marked points being cut out by intersections with the divisors $D_{i,n+1}$ for
    $i=1,\ldots,n,n+2$. It follows that the general fiber is the blowup of $\bF_0$ at $n$ points---the section
    $D_{n+1,n+2}$ gives the diagonal and the remaining sections $D_{i,n+2}$ and $D_{i,n+1}$ give horizontal and vertical
    rulings, respectively.
\end{proof}

\begin{lemma} \label{lem:A2_fiber}
    Consider the morphism $\pi : \oY(E_7) \to \oY(E_6)$.
    \begin{enumerate}
        \item The general fiber of $\pi\vert_{D(A_2)} : D(A_2) \to D(A_2^3)$ is the blowup of $\bP^2$ at 6 points, with
            3 on one line and 3 on another line. This is the minimal resolution of a cubic surface with an $A_2$
            singularity.
        \item The general fiber of $\pi\vert_{D(A_1) \cap D(A_2)} : D(A_1) \cap D(A_2) \to D(A_1) \cap D(A_2^3)$ for
            $A_1 \perp A_2$ is the blowup of $\bP^2$ at 5 points, with 2 on one line and 3 on another line.
    \end{enumerate}
\end{lemma}

\begin{proof}
    \begin{enumerate}
        \item The map $\pi\vert_{D(A_2)}$ is given by $\oM_{0,4} \times \oM_{0,7} \xrightarrow{id \times f} \oM_{0,4}
            \times (\oM_{0,4} \times \oM_{0,4})$, where up to reordering $f = \pi_{1237} \times \pi_{4567}$
            (\cref{tab:pi} and \cite[Proposition 10.9]{hackingStablePairTropical2009}). Thus it suffices to consider the
            general fiber of $f : \oM_{0,7} \to \oM_{0,4} \times \oM_{0,4}$. We can obtain coordinates on $M_{0,n}$ as
            the complement in $\bP^{n-3}$ of the hyperplanes $x_i=0$ and $x_i=x_j$, by fixing the last two points at $0$
            and $\infty$, and scaling the remaining points, so that the points are represented by the columns of the
            matrix
            \[
                \begin{bmatrix}
                    1 & 1 & 1 & 1 & 0 \\
                    x_1 & \ldots & x_{n-2} & 0 & 1
                \end{bmatrix},
            \]
            with the condition that the points are in general position being equivalent to the matrix having full rank.
            Under these coordinates, the map $f : \oM_{0,7} \to \oM_{0,4} \times \oM_{0,4}$ is given on the interiors by
            \[
                \begin{bmatrix}
                    1 & 1 & 1 & 1 & 1 & 1 & 0 \\
                    x_1 & x_2 & x_3 & x_4 & x_5 & 0 & 1
                \end{bmatrix}
                \to
                \left(\begin{bmatrix}
                    1 & 1 & 1 & 0 \\
                    x_1 & x_2 & x_3 & 1
                \end{bmatrix},
                \begin{bmatrix}
                    1 & 1 & 1 & 0 \\
                    x_4 & x_5 & 0 & 1
                \end{bmatrix}
                \right).
            \]
            In the first matrix on the right, we can subtract $x_3$ times the first row from the second row to obtain
            the matrix
            \[
                \begin{bmatrix}
                    1 & 1 & 1 & 0 \\
                    x_1-x_3 & x_2-x_3 & 0 & 1
                \end{bmatrix},
            \]
            giving coordinates on $\bP^1$ so that the restriction of $f$ to $M_{0,7} \to M_{0,4} \times M_{0,4}$ is
            described by
            \[
                (x_1:x_2:x_3:x_4:x_5) \mapsto (x_1-x_3 : x_2-x_3) \times (x_4:x_5).
            \]
            Fixing a general point $(a : b) \times (c : d)$ in $M_{0,4} \times M_{0,4}$, we see that the fiber of this map
            is given by the points $(x_1:x_2:x_3:x_4:x_5) \in M_{0,7}$ such that
            \[
                (x_1-x_3)b = (x_2-x_3)a, \;\;\; x_4d=x_5c.
            \]
            These equations cut out a plane $\bP^2$ in $\bP^4$ that intersects the hyperplanes $x_i=0$ and $x_i=x_j$ in
            a collection of 11 lines with 6 non-normal crossings points, 3 on one line and 3 on another.
            Since $\oM_{0,n}$ is obtained from $\bP^{n-3}$ by blowing up the loci where the hyperplanes $x_i=0$,
            $x_i=x_j$ are not normal crossings \cite{kapranovChowQuotientsGrassmannians1993}, it follows that the
            general fiber of $f : \oM_{0,7} \to \oM_{0,4} \times \oM_{0,4}$ is the blowup of $\bP^2$ at these 6 points.
        \item The condition that $A_1 \perp A_2$ implies that $D(A_1) \cap D(A_2) = \oM_{0,4} \times D_{2,5}$, where
            $D_{2,5}$ is a divisor on $\oM_{0,7}$ parameterizing a stable curve with 2 marked points on one component,
            and 5 marked points on the other component (see \cref{prop:bdry_strata}). We have $D_{2,5} \cong \oM_{0,6}$,
            and the induced map is given by $\oM_{0,6} \to \text{pt} \times \oM_{0,4}$. The general fiber is the blowup
            of $\bF_0$ at 4 points along the diagonal (\cref{lem:fiber_m0n}). This is the same as blowing up $\bP^2$ at
            4 points in general position, and then a 5th point on the line between 2 of the points.
    \end{enumerate}
\end{proof}

\begin{remark}
    Part (1) of the lemma is used when computing fibers of type $b$ (cf. the components of type 1 in
    \cref{fig:b_13-23}). Part (2) is used when computing fibers of type $ab$ (cf. the components of type 2 in
    \cref{fig:ab_12-23}). Note in part (2) of the lemma that if instead $A_1 \subset A_2$, then $D(A_1) \cap D(A_2) =
    \text{pt} \times \oM_{0,7}$ (\cref{prop:bdry_strata}), and the general fiber is the same as the general fiber of
    $\pi\vert_{D(A_2)}$. This says that the type $ab$ fiber also has a component as described in part (1) of the lemma;
    this is the component of type 1 in \cref{fig:ab_12-23}.
\end{remark}

\begin{remark}
    Note the map $\pi\vert_{D(A_2)} : D(A_2) \to D(A_2^3)$ is not everywhere flat, but the locus where it is not flat is
    the intersection with the non-flat $A_3^2$ locus \cite[Proof of Proposition 10.22]{hackingStablePairTropical2009},
    and therefore is dealt with by \cref{sec:nonflat_A32} below, cf. \cref{ex:nonflat_A2}.
\end{remark}

\begin{lemma} \label{lem:A1_fiber}
    Let $k=1,\ldots,4$ and consider a collection of $k$ pairwise orthogonal $A_1 \subset E_6$'s. Let $Z$ be the
    intersection of the corresponding divisors in $\oY(E_7)$, and $W$ the intersection of the corresponding divisors in
    $\oY(E_6)$. Then $\pi$ restricts to a morphism $Z \to W$ whose general fiber is the minimal resolution of a singular
    cubic surface with $k$ $A_1$ singularities.
\end{lemma}

\begin{proof}
    It follows from \cite[Lemma 10.8]{hackingStablePairTropical2009} and Sekiguchi's explicit descriptions of the strata
    $Z \subset \oY(E_7)$ \cite{sekiguchiCrossRatioVarieties2000} (cf. \cref{sec:A1divs}) that $\pi\vert_Z : Z \to W$ is
    determined by the pullback of boundary divisors of $W$. The interior $Z^{\circ}$ of $Z$ (i.e., the complement in $Z$
    of the intersections with the boundary divisors) is a moduli space of 7 points in $\bP^2$ lying in a certain
    degenerate configuration, and the general fiber of $\pi\vert_Z$ is the blowup of $\bP^2$ at the first 6 points. For
    instance:
    \begin{enumerate}
        \item $D(7)$ gives the blowup of 6 points on a conic.
        \item $D(7) \cap D(56)$ gives the blowup of 6 points $p_1,\ldots,p_6$ such that $p_5$ and $p_6$ are infinitely
            near.
        \item $D(123) \cap D(345) \cap D(156)$ gives the blowup of 6 points $p_1,\ldots,p_6$ such that $p_1,p_2,p_3$ lie
            on a line, $p_3,p_4,p_5$ lie on a line, and $p_1,p_5,p_6$ lie on a line.
        \item $D(123) \cap D(345) \cap D(146) \cap D(256)$ gives the blowup of 6 points $p_1,\ldots,p_6$ such that any 3
            points lie on a line (see also \cref{prop:A1strata}).
    \end{enumerate}
    In any case, the general fiber is the minimal resolution of a cubic surface with $k$ $A_1$ singularities.
\end{proof}

\begin{remark}
    Recall from \cref{prop:A1strata} that there are 2 types of intersections of 3 $D(A_1)$ divisors on $\oY(E_7)$,
    depending on whether the orthogonal complement of the 3 $A_1$'s in $E_7$ is a $D_4$ or a $3A_1$. The condition that
    all 3 $A_1$'s lie in $E_6$ ensures that only the second type occurs in the above lemma. Intersections of the first
    type occur, for instance, as the intersection of 3 horizontal $A_1$ divisors (and in particular, correspond to
    Eckardt points on the fibers of $\wt\pi : \wY(E_7) \to \wY(E_6)$, see \cref{sec:eckardts}).
\end{remark}

\subsubsection{Fibers near non-flat $A_3^2$ divisors} \label{sec:nonflat_A32}

Recall from \cref{tab:pi} and \cite[Proposition 10.9]{hackingStablePairTropical2009} that the restriction of $\pi :
\oY(E_7) \to \oY(E_6)$ to a non-flat $A_3^2$ divisor is the first projection
\[
    \oM_{0,5} \times \oM_{0,5} \times \oM_{0,4} \to \oM_{0,5}.
\]
In the below lemma we identify $\oM_{0,4}$ with $\bP^1$ and $\oM_{0,5}$ with the blowup of $\bP^2$ at 4 points in
general position \cite{kapranovChowQuotientsGrassmannians1993}.

\begin{lemma}[{\cite[Lemma 10.24, Proof of Lemma 10.27]{hackingStablePairTropical2009}}] \label{lem:nonflat_A32}
    Let $p_1,\ldots,p_4$ be 4 points in general position in $\bP^2$, and let $\ell_{ij}$ be the line between points
    $p_i$ and $p_j$. Then the restriction of $\wt\pi : \wY(E_7) \to \wY(E_6)$ to the strict transform of a non-flat
    $A_3^2$ divisor is the morphism
    \[
        \oM_{0,5} \times \oM_{0,5} \times \oM_{0,4} \xrightarrow{id \times f} \oM_{0,5} \times \oM_{0,4},
    \]
    where $id : \oM_{0,5} \to \oM_{0,5}$ is the identity map on the first component, and $f : \oM_{0,5} \times \oM_{0,4}
    \to \oM_{0,4}$ is the strict transform of the linear system on $\bP^2 \times \bP^1$ spanned by the two divisors
     $(\ell_{12} \times \bP^1) \cup (\bP^2 \times \{0\})$ and $(\ell_{34} \times \bP^1) \cup (\bP^2 \times \{1\})$.

    In particular, the general fiber of $\wt\pi\vert_{D(A_3^2)}$ is the blowup of $\bP^2$ at 5 points---the 4 points
    $p_1,\ldots,p_4$ in general position, and the point $p_5 = \ell_{12} \cap \ell_{34}$.
\end{lemma}

Denote by $X$ the general fiber of $\wt\pi\vert_{D(A_3^2)}$ as described in the above lemma. Any surface obtained as a
fiber of $\wt\pi : \wY(E_7) \to \wY(E_6)$ whose type involves $a_i$ for $i=2$, $3$, or $4$ contains at least one
component that is isomorphic to $X$, or to a special fiber of $\wt\pi\vert_{D(A_3^2)}$ yielding a degeneration of $X$
into four irreducible components as described for type $aa_2$ in \cref{ex:aa2_fiber}. (For other surfaces having a
special fiber of $\wt\pi\vert_{D(A_3^2)}$ as a component, the degeneration of $X$ is isomorphic to that described in the
$aa_2$ case, see
\cref{tab:aa2,tab:aa3,tab:aa4,tab:a2a3,tab:a2a4,tab:a3a4,tab:aa2a4,tab:aa3a4,tab:a2a3a4,tab:aa2a3a4,tab:aa2b,tab:aa3b,tab:a2a3b,tab:aa2a3b}
for the analogous computations.) For detailed pictures of surfaces involving $X$ or a degeneration of $X$, see
\cref{fig:a2_12-23,fig:aa2_12-23,fig:a3_12-23,fig:a3_12-23_degens,fig:a4_12-23,fig:a4_12-23_degens,fig:a2b_12-23,fig:aa2b_12-23,fig:a3b_12-23,fig:a3b_degens_12-23}.

\subsection{Examples}

In order to clarify the structure of the proof of \cref{thm:fibers}, we describe several examples in detail.

\begin{example} \label{ex:a_fiber}
    In this example we compute the fiber of $\wt\pi : \wY(E_7) \to \wY(E_6)$ over a general point of the type $a$
    boundary divisor $D(7)$. Since a general point of $D(7)$ does not lie in the image of the non-flat locus of
    $\wY(E_7)$, it is enough to compute the fiber of $\pi : \oY(E_7) \to \oY(E_6)$ instead. There are 8 boundary
    divisors of $\oY(E_7)$ surjecting onto $D(7)$.
    \begin{enumerate}
        \item One boundary divisor $D(A_1) = D(7)$. By \cref{lem:A1_fiber}, the general fiber of $\pi\vert_{D(A_1)}$ is
            the minimal resolution of a cubic surface with one $A_1$ singularity. It intersects the boundary divisor
            $D(A_7)=D(X_0)$ below in the exceptional fiber over the singular point, and it intersects the six boundary
            divisors $D(A_2)=D(i,7,i7)$ below in the strict transforms of the six lines passing through the singular
            point. In addition, $D(A_1)$ intersects the 15 horizontal $A_1$ divisors $D(ij7)$, for $ij \subset [6]$, in
            15 lines which do not pass through the singular point.
        \item One boundary divisor $D(A_7) = D(X_0)$, where $X_0$ is the $A_7$ subsystem of $E_7$ whose positive roots
            are listed in \cref{tab:A7s}. We have $D(A_7) \cong \oM_{0,8}$, and the restriction of $\pi$ to $D(A_7)$ is
            the morphism $\oM_{0,8} \to \oM_{0,6}$ dropping two marked points. Thus by \cref{lem:fiber_m0n}, the general
            fiber of $\pi\vert_{D(A_7}$ is isomorphic to the blowup $Bl_6\bF_0$ of $\bF_0 \cong \bP^1 \times \bP^1$ at
            six points on its diagonal.  Analyzing the proof of \cref{lem:fiber_m0n} in comparison with the description
            of the intersections of boundary divisors in \cref{prop:bdry_strata}, we see that the six exceptional
            divisors on $Bl_6\bF_0$ are cut out by the intersections with the six boundary divisors $D(A_2)=D(i,7,i7)$
            described below, and the strict transform of the diagonal is cut out by the intersection with the boundary
            divisor $D(A_1)=D(7)$. Furthermore, there are 12 horizontal $A_1$ divisors which intersect $D(X_0)$, these
            are given by $A_1=(i)$ and $A_1 = (i7)$, for $i=1,\ldots,6$. The divisors $D(i)$ cut out the strict
            transforms of lines in one ruling of $\bP^1 \times \bP^1$, and the divisors $D(i7)$ cut out the strict
            transforms of lines in another ruling of $\bP^1 \times \bP^1$.
        \item Six boundary divisors $D(A_2) = D(i,7,i7)$, for $i=1,\ldots,6$. We have $D(A_2) \cong \oM_{0,4} \times
            \oM_{0,7}$, and the restriction of $\pi$ to $D(A_2)$ is the morphism $\oM_{0,4} \times \oM_{0,7} \to
            \oM_{0,3} \times \oM_{0,6}$ given by dropping a marked point in each coordinate. The general fiber of
            $\pi\vert_{D(A_2)}$ is therefore isomorphic to $\bP^1 \times \bP^1$, with three distinguished rulings $pt
            \times \bP^1$ (coming from the three marked points on $\oM_{0,3}$), and six distinguished rulings $\bP^1
            \times pt$ (coming from the six marked points on $\oM_{0,6}$).

            Since $(7) \subset (i,7,i7)$, we see by \cref{prop:bdry_strata} that $D(7) \cap D(i,7,i7)$ is equal to $p
            \times \oM_{0,7} \subset D(i,7,i7)$, where $p$ is a boundary divisor of $\oM_{0,4}$. It follows that $D(7)$
            cuts out one of the three distinguished rulings $pt \times \bP^1$ on the general fiber of
            $\pi\vert_{D(A_2)}$. Similarly, the two other distinguished rulings $pt \times \bP^1$ come from the
            restrictions to $D(i,7,i7)$ of the horizontal $A_1$ divisors $D(i)$ and $D(i7)$.

            Since $(i,7,i7) \subset X_0$, we see by \cref{prop:bdry_strata} that $D(i,7,i7) \cap D(X_0)$ is equal to
            $\oM_{0,4} \times D_{2,5} \subset D(i,7,i7)$, where $D_{2,5} \cong \oM_{0,6}$ is one of the horizontal
            divisors (sections) of $\oM_{0,7} \to \oM_{0,6}$. It follows that $D(X_0)$ cuts out one of the six
            distinguished rulings $\bP^1 \times pt$ on the general fiber of $\pi\vert_{D(A_2)}$. Similarly, the five
            other distinguished rulings $\bP^1 \times pt$ come from restrictions to $D(i,7,i7)$ of horizontal $A_1$
            divisors for $A_1 \perp (i,7,i7)$. Explicitly, if for instance $i=1$, then the five relevant horizontal
            $A_1$s are given by $A_1=(127), (137), (147), (157), (167)$.
    \end{enumerate}
    Gluing the above descriptions of the irreducible components together, we obtain the surface of type $a$ with eight
    irreducible components pictured in \cref{fig:a_12-23}.
\end{example}

\begin{example} \label{ex:b_fiber}
    In this example we compute the fiber of $\wt\pi : \wY(E_7) \to \wY(E_6)$ over a general point of the type $b$
    boundary divisor $D(A_2^3) = D((123,456,7) \times (12,13,23) \times (45,46,56))$. Since a general point of this
    divisor does not lie in the image of the non-flat locus of $\wY(E_7)$, it is enough to compute the fiber of $\pi :
    \oY(E_7) \to \oY(E_6)$ instead. There are 12 boundary divisors of $\oY(E_7)$ surjecting onto $D(A_2^3)$.
    \begin{enumerate}
        \item Three divisors $D(A_2)$ for $A_2=(123,456,7)$, $(12,13,23)$, and $(45,46,56)$. We have $D(A_2) \cong
            \oM_{0,4} \times \oM_{0,7}$, and $\pi\vert_{D(A_2)}$ is given by $\oM_{0,4} \times \oM_{0,7} \to \oM_{0,4}
            \times (\oM_{0,4} \times \oM_{0,4})$, where the map is the identity on the first component, and a product of
            forgetful morphisms on the second component, as in \cref{tab:pi,lem:A2_fiber}. By \cref{lem:A2_fiber}, the
            general fiber of $\pi\vert_{D(A_2)}$ is isomorphic to the blowup of $\bP^2$ at six points, with three on one
            line and three on another line. The intersections of $D(A_2)$ with the other two $D(A_2)$s give the strict
            transforms of these two lines. A given $D(A_2)$ intersects six of the nine $D(A_3^2)$ divisors below, in the
            six exceptional divisors. Finally, a given $D(A_2)$ intersects 9 horizontal $A_1$ divisors, in the strict
            transforms of lines connecting one of the blown up points on one line, with one of the blown up points on
            the other line.
        \item Nine divisors $D(A_3^2)$, for $A_3^2 \cap E_6 = A_2^2$, where $A_2^2 \subset A_2^3$ is the subsystem
            consisting of two of the three $A_2$'s making up $A_2^3$. For explicit descriptions of these nine $A_3^2$
            subsystems of $E_7$, see \cref{tab:b}. We have $D(A_3^2) \cong \oM_{0,5} \times \oM_{0,5} \times \oM_{0,4}$,
            and $\pi\vert_{D(A_3^2)}$ is a morphism $\oM_{0,5} \times \oM_{0,5} \times \oM_{0,4} \to \oM_{0,4} \times
            \oM_{0,4} \times \oM_{0,4}$, given by a forgetful map in the first two coordinates, and the identity in the
            last coordinate. It follows that the general fiber of $\pi\vert_{D(A_3^2)}$ is isomorphic to $\bP^1 \times
            \bP^1$, with four distinguished rulings $pt \times \bP^1$ (coming from the four marked points on the first
            copy of $\oM_{0,4}$), and four distinguished rulings $\bP^1 \times pt$ (coming from the four marked points
            on th second copy of $\oM_{0,4}$). With $A_3^2 \cap E_6 = A_2^2$, we see that $D(A_2) \cap D(A_3^2)$ for the two
            $A_2$s making up $A_2^2$ cut out one line in each respective ruling of $\bP^1 \times \bP^1$. The remaining
            three lines in each ruling are given by intersections with horizontal $A_1$ divisors.
    \end{enumerate}
    Gluing the above descriptions of the irreducible components together, we obtain the surface of type $b$ with 12
    irreducible components pictured in \cref{fig:b_13-23}.
\end{example}

\begin{example} \label{ex:a2_fiber}
    In this example we compute the fiber of $\wt\pi : \wY(E_7) \to \wY(E_6)$ over a general point of the type $a_2$
    boundary divisor $D(A_1 \perp A_1') = D(7 \perp 56)$. A general point of this divisor lies in the image of the
    non-flat $A_3^2$ divisor $D(A_3^2(7,56))$, where $A_3^2(7,56)$ is the $A_3^2$ root subsystem of $E_7$ described in
    \cref{tab:nonflat_A32}. Thus in order to compute the fiber of $\wt\pi$ over a general point of $D(A_1 \perp A_1')$,
    we are required to actually work with $\wt\pi$ instead of $\pi$. On the other hand, away from the non-flat $A_3^2$
    divisors, $\wt\pi$ is the pullback of $\pi$, so we can still begin our investigation by first studying $\pi :
    \oY(E_7) \to \oY(E_6)$. The image of $D(7 \perp 56)$ in $\oY(E_6)$ is the codimension two boundary stratum $D(7)
    \cap D(56)$. There is a total of 20 boundary strata of $\oY(E_7)$ surjecting onto $D(7) \cap D(56)$: 19 codimension
    two boundary strata, plus the codimension one non-flat $A_3^2$ divisor mentioned above. Explicitly, we have the
    following description of the strata and the corresponding general fibers.
    \begin{enumerate}
        \item One codimension two stratum $D(7) \cap D(56) \subset \oY(E_7)$. By \cref{lem:A1_fiber}, the general fiber
            of $\pi\vert_{D(7) \cap D(56)} : D(7) \cap D(56) \to D(7) \cap D(56)$ is isomorphic to the minimal
            resolution of a cubic surface with two $A_1$ singularities.
        \item Two codimension two strata $D(A_1) \cap D(A_7)$, where $A_1 \subset A_7$ (see \cref{tab:a2} for explicit
            descriptions of these root systems). By \cref{prop:bdry_strata} intersection $D(A_1) \cap D(A_7)$ is
            isomorphic to $\oM_{0,7} \cong D_{2,6}$ on $D(A_7) \cong \oM_{0,8}$, and the restriction of $\pi$ to $D(A_1)
            \cap D(A_7)$ is a forgetful morphism $\oM_{0,7} \to \oM_{0,5}$. It follows by \cref{lem:fiber_m0n} that the
            general fiber of $\pi\vert_{D(A_1) \cap D(A_7)}$ is isomorphic to the blowup of $\bF_0 \cong \bP^1 \times
            \bP^1$ at 5 points on its diagonal.
        \item One boundary divisor $D(A_3^2(7,56))$, where
            \[
                A_3^2(7,56) = (5,6,7,56,57,67) \times (12,13,14,23,24,34).
            \]
            The morphism $\pi$ is not flat over $D(A_3^2(7,56))$; its flattening has general fiber isomorphic to the
            blowup of $\bP^2$ at five points, with two points on one line, two points on another line, and the fifth
            point the intersection of the two lines, as described in \cref{lem:nonflat_A32}.
        \item Four codimension two strata $D(A_2) \cap D(A_7)$, where $A_2 \subset A_7$ (see \cref{tab:a2} for explicit
            descriptions of these root systems). By \cref{prop:bdry_strata}, the intersection $D(A_2) \cap D(A_7)$ is
            isomorphic to $\oM_{0,4} \times \oM_{0,6}$. The restriction of $\pi$ to $D(A_2) \cap D(A_7)$ is a morphism
            $\oM_{0,4} \times \oM_{0,6} \to \oM_{0,3} \times \oM_{0,5}$, given as the product of a forgetful map in each
            coordinate. Thus the general fiber of $\pi\vert_{D(A_2) \cap D(A_7)}$ is isomorphic to $\bF_0 \cong \bP^1
            \times \bP^1$, with three distinguished rulings $p \times \bP^1$ and five distinguished rulings $\bP^1
            \times p$. One of the rulings $p \times \bP^1$ is given by the intersection with the component described in
            (3) above, and one of the rulings $\bP^1 \times p$ is given by the intersection with a component described
            in (2) above. The remaining rulings are cut out by horizontal $A_1$ divisors.
        \item
            \begin{enumerate}
                \item Eight codimension two strata $D(A_1) \cap D(A_2)$, where $A_1 \perp A_2$ (see \cref{tab:a2} for
                    explicit descriptions of these root systems). By \cref{prop:bdry_strata}, the intersection $D(A_1)
                    \cap D(A_2)$ is isomorphic to $\oM_{0,4} \times \oM_{0,6}$. As in (4) above, $\pi\vert_{D(A_1) \cap
                    D(A_2)}$ is given by a product of forgetful maps $\oM_{0,4} \times \oM_{0,6} \to \oM_{0,3} \times
                    \oM_{0,5}$, and its general fiber is isomorphic to $\bF_0 \cong \bP^1 \times \bP^1$, with three
                    distinguished rulings $p \times \bP^1$ and five distinguished rulings $\bP^1 \times p$. This time,
                    one of the rulings $p \times \bP^1$ is given by the intersection with the component described in (1)
                    above, and the remaining two rulings $p \times \bP^1$ are cut out by horizontal $A_1$ divisors.
                    Three of the rulings $\bP^1 \times p$ are cut out by horizontal $A_1$ divisors. The remaining two
                    rulings $\bP^1 \times p$ are given by the intersections with a component described in (2) above, and
                    a component described in 5(b) below.
                \item Four codimension two strata $D(A_2) \cap D(A_2')$, where $A_2 \perp A_2'$ (see \cref{tab:a2} for
                    explicit descriptions of these root systems). By \cref{prop:bdry_strata}, the intersection $D(A_2)
                    \cap D(A_2')$ is isomorphic to $\oM_{0,4} \times \oM_{0,4} \times \oM_{0,5}$. The restriction of
                    $\pi$ to $D(A_2) \cap D(A_2')$ is given by $\oM_{0,4} \times \oM_{0,4} \times \oM_{0,5} \to
                    \oM_{0,3} \times \oM_{0,3} \times \oM_{0,5}$, where the map is a forgetful morphism in the first and
                    second coordinates, and the identity in the third component. It follows that a general fiber of
                    $\pi\vert_{D(A_2) \cap D(A_2')}$ is isomorphic to $\bF_0 \cong \bP^1 \times \bP^1$, with three
                    distinguished rulings $p \times \bP^1$, and three distinguished rulings $\bP^1 \times p$. One of the
                    rulings $p \times \bP^1$ is given by the intersection with a component described in 5(a) above,
                    likewise for one of the rulings $\bP^1 \times p$, and the other four rulings are cut out by
                    horizontal $A_1$ divisors.
            \end{enumerate}
    \end{enumerate}
    Gluing the above descriptions of irreducible components together, we obtain the surface of type $a_2$ with 20
    irreducible components pictured in \cref{fig:a2_12-23}.
\end{example}

\begin{example} \label{ex:aa2_fiber}
    In this example we consider the fiber of $\wt\pi : \wY(E_7) \to \wY(E_6)$ over a general point of the type $aa_2$
    stratum $D(7) \cap D(7 \perp 56)$. The strata of $\wY(E_7)$ that surject onto $D(7) \cap D(7 \perp 56)$ are given by
    intersections of strata that surject onto $D(7)$ with strata that surject onto $D(7 \perp 56)$. We directly verify
    that if neither such stratum is (the exceptional divisor over the) non-flat $A_3^2$ divisor surjecting onto $D(7
    \perp 56)$, then the general fiber is isomorphic to the corresponding general fiber for the type $a_2$ case as
    described in \cref{ex:a2_fiber}. Thus the only place where the $aa_2$ fiber differs from the $a_2$ fiber is near the
    non-flat $A_3^2$ divisor. Near the non-flat $A_3^2$ divisor, there are four strata surjecting onto $D(7) \cap D(7
    \perp 56)$. We describe these strata on $D(A_3^2) \subset \oY(E_7)$; the corresponding general fibers for the
    restrictions of $\wt\pi$ to these strata are then explicitly computed using \cref{lem:nonflat_A32}.
    \begin{enumerate}
        \item One stratum $D(A_7) \cap D(A_3^2)$, where $A_7=X_0$ is the root subsystem appearing in
            \cref{ex:a_fiber,tab:A7s}. This stratum has the form $\oM_{0,5} \times \oM_{0,5} \times 0$ on $D(A_3^2) \cong
            \oM_{0,5} \times \oM_{0,5} \times \oM_{0,4}$, following the notation of \cref{lem:nonflat_A32}. The general
            fiber of $\wt\pi\vert_{D(A_7) \cap D(A_3^2)}$ is $\oM_{0,5} \cong Bl_4\bP^2$.
        \item One stratum $D(7) \cap D(A_3^2)$. This stratum has the form $\oM_{0,5} \times \ell_{12}
            \times \oM_{0,4}$ on $D(A_3^2) \cong \oM_{0,5} \times \oM_{0,5} \times \oM_{0,4}$, following the notation of
            \cref{lem:nonflat_A32}. The general fiber of $\wt\pi$ restricted to this stratum is $\bF_0$.
        \item Two strata $D(A_2) \cap D(A_3^2)$ for $A_2=(i,7,i7)$, $i=5,6$. These strata have the form $\oM_{0,5} \times
            e_j \times \oM_{0,4}$ on $D(A_3^2) \cong \oM_{0,5} \times \oM_{0,5} \times \oM_{0,4}$, following the
            notation of \cref{lem:nonflat_A32}, where $e_j$ is an exceptional divisor over one of the four points in the
            blowup $\oM_{0,5} \cong Bl_4\bP^2$. The general fiber of $\wt\pi\vert_{D(A_2) \cap D(A_3^2)}$ is $\bF_0$.
    \end{enumerate}
    The above four components glue to each other and to the remainder of the fiber over $\wt\pi$ over a general point of
    $D(7) \cap D(7 \perp 56)$ to give the surface of type $aa_2$ with 23 irreducible components pictured in
    \cref{fig:aa2_12-23}.
\end{example}

\begin{example} \label{ex:nonflat_A2}
    For brevity we avoid going into detail for further examples. We mention however that computations analogous to those
    in \cref{ex:a2_fiber} and \cref{ex:aa2_fiber} show that any surface whose type involves $a_i$ for $i=2,3,4$ has at
    least one component isomorphic to either the blowup $X$ of $\bP^2$ at five points in special position as described
    in \cref{lem:nonflat_A32}, or to the degeneration of this surface into four irreducible components as described in
    \cref{ex:aa2_fiber}. For instance, the intersection of a non-flat $D(A_3^2)$ and a $D(A_2) \mapsto D(A_2^3)$ is of
    the form $\oM_{0,5} \times \ell \times \oM_{0,4} \subset D(A_3^2)$, where $\ell$ is a boundary divisor of
    $\oM_{0,5}'$, so the general fiber of $\wt\pi\vert_{D(A_3^2) \cap D(A_2)}$ in this case is isomorphic to $X$. This
    gives a component of a type $a_2b$ surface.
\end{example}

To summarize, we prove \cref{thm:fibers} and describe all fibers of $\wt\pi : \wY(E_7) \to \wY(E_6)$ by carrying out
computations analogous to the above examples for $W(E_6)$-representatives of all strata of $\wY(E_6)$. The combinatorics
of the root subsystems describing irreducible components of the fibers are given in \cref{sec:root_combs}, and we glue
the corresponding irreducible components as in the above examples to obtain the pictures of the fibers in
\cref{sec:weighted_cubics}.

\section{Eckardt points} \label{sec:eckardts}

Recall that an Eckardt point of a smooth cubic surface is a point at which three of the lines on the surface intersect.
In this section we describe the possible configurations of Eckardt points on the fibers of the morphism $\wt\pi :
\wY(E_7) \to \wY(E_6)$ computed in the previous section, and explain how to blowup these Eckardt points to obtain the
universal family of weight 1 stable marked cubic surfaces, cf. \cite[Theorem 10.31]{hackingStablePairTropical2009}.

We first must be more precise about what we mean by an Eckardt point of a fiber of $\wt\pi$.

\begin{definition}[{\cite[Definition 10.21]{hackingStablePairTropical2009}}]
    The \emph{Eckardt locus} of $\wY(E_7)$ is the union of the intersections of three horizontal $A_1$ divisors of
    $\wt\pi : \wY(E_7) \to \wY(E_6)$. The \emph{Eckardt points} on a fiber $(S,B)$ of $\wt\pi : (\wY(E_7),B_{\wt\pi})$
    are the intersection points of $(S,B)$ with the Eckardt locus of $\wY(E_7)$.
\end{definition}

\begin{remark} \label{rmk:eckardt_comps}
    From the combinatorics of the $A_1$ subsystems of $E_7$ giving horizontal $A_1$ divisors of $\wt\pi$, one sees that
    the Eckardt locus of $\wY(E_7)$ is the disjoint union of 45 irreducible components. The images of these components
    in $\wY(E_6)$ give the Eckardt divisors of $\wY(E_6)$, also obtained as the strict transforms under $\wY(E_6) \to
    \oY(E_6)$ of Naruki's tritangent divisors on $\oY(E_6)$ (cf. \cref{sec:A1divs} and
    \cite{narukiCrossRatioVariety1982}). The intersections among Eckardt and boundary divisors of $\wY(E_6)$
    are described in detail in \cite[Section 4]{schockE_6InvariantBirational2024}.
\end{remark}

\begin{lemma}[{\cite[Proposition 10.23]{hackingStablePairTropical2009}}]
    Any Eckardt points on a fiber $(S,B)$ of $\wt\pi : (\wY(E_7),B_{\wt\pi}) \to \wY(E_6)$ lie in the smooth locus of
    $S$.
\end{lemma}

\begin{proposition}[{\cite[Theorem 10.31]{hackingStablePairTropical2009}}] \label{prop:resolve_eckardt}
    Let $\ddot{Y}(E_7)$ denote the blowup of $\wY(E_7)$ along the Eckardt locus, and let $\ddot{\pi} :
    (\ddot{Y}(E_7),B_{\ddot{\pi}}) \to \wY(E_6)$ be the induced morphism to $\wY(E_6)$, where $B_{\ddot{\pi}}$ is the
    strict transform of $B_{\wt\pi}$. Then $\ddot{\pi}$ is a flat family of stable pairs, giving the universal family of
    (unweighted) stable marked cubic surfaces over the moduli space $\wY(E_6)$. Each fiber $(\ddot{S},\ddot{B})$ of
    $\ddot{\pi}$ is obtained from a fiber $(S,B)$ of $\wt\pi$ by blowing up each Eckardt point on $(S,B)$, and attaching
    to the exceptional divisor a copy of $\bP^2$ glued along a general line. The pullbacks of the three lines passing
    through the Eckardt point on $(S,B)$ intersect the $\bP^2$ in three lines in general position, as pictured in
    \cref{fig:eckardt}.
    \begin{figure}[htpb]
        \centering
        \includegraphics[width=0.8\linewidth]{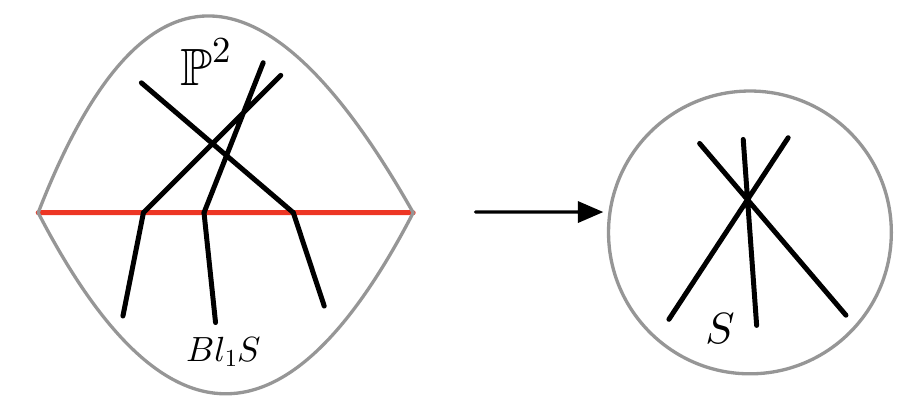}
        \caption{The resolution of an Eckardt point. On the right, three lines intersect in a point, and on the left
        this point is blown up and a copy of $\bP^2$ is attached to the exceptional line.}%
        \label{fig:eckardt}
    \end{figure}
\end{proposition}

\begin{remark}
    As the above lemma and proposition are proven in \cite{hackingStablePairTropical2009}, we do not include their
    proofs. However, we comment briefly on a concrete approach to proving these results in keeping with the philosophy
    of the current paper.

    The lemma can be proven by a direct computation of the fibers of $\wt\pi$ over points in the intersection of Eckardt
    and boundary strata of $\wY(E_6)$; as mentioned in \cref{rmk:eckardt_comps}, these intersections are described in
    \cite[Section 4]{schockE_6InvariantBirational2024}. More carefully, including the Eckardt divisors in the boundary
    of $\wY(E_6)$ induces a finer stratification of the boundary, and one may obtain the lemma by computing, for each
    stratum in this finer stratification, the fiber of $\wt\pi$ over a general point of the given stratum. One finds
    that the only Eckardt points on a fiber of $\wt\pi$ lie in the smooth locus of the fiber.

    In order to prove the proposition, one can again follow the same general strategy as for computing the fibers of
    $\wt\pi$. Namely, let $(S,B)$ be a fiber of $\wt\pi$ and suppose that it contains an Eckardt point. Assume that the
    Eckardt point does not lie on a component of $(S,B)$ obtained from a non-flat $A_3^2$ divisor of $\oY(E_7)$. Then
    the fiber $(S,B)$ near the Eckardt point is computed as the pullback of the corresponding fiber of $\pi : \oY(E_7)
    \to \oY(E_6)$. Observe that if three horizontal $A_1$ divisors intersect, then the orthogonal complement of the
    three $A_1$ root subsystems of $E_7$ is a $D_4 \subset E_7$, and the restriction of $\pi$ to the Eckardt component
    cutting out the Eckardt point of interest on $(S,B)$ is given by the identity morphism $\oX(D_4) \to \oX(D_4)$ (cf.
    \cref{sec:A1divs,prop:A1strata}). Blowing up the Eckardt component induces the projection map $\oX(D_4) \times \bP^2 \to
    \oX(D_4)$, whose general fiber is $\bP^2$, and the three Eckardt divisors give three lines in general position on
    this $\bP^2$. If instead the Eckardt point of $(S,B)$ lies on a component obtained from non-flat $A_3^2$ divisor of
    $\oY(E_7)$, then a direct computation using the description of the flattening of $\pi$ as in \cref{lem:nonflat_A32}
    shows the same description for the blowup of this Eckardt point.
\end{remark}

We now turn our attention to describing the potential configurations of Eckardt points on fibers of $\wt\pi$. For this,
first observe from our explicit descriptions of the fibers (\cref{thm:fibers}) that the only irreducible components of
a fiber which potentially contain Eckardt points are those isomorphic to $\wS_{kA_1}$, $k=0,\ldots,4$, $\wS_{A_2}$, or
$Bl_5\bP^2$, where $\wS_{kA_1}$ is the minimal resolution of a cubic surface with $k$ $A_1$ singularities, $\wS_{A_2}$
is the minimal resolution of a cubic surface with an $A_2$ singularity (obtained as the blowup of $\bP^2$ at three
points on one line and three points on another line, appearing for instance in fibers of type $b$), and $Bl_5\bP^2$ is
the blowup of $\bP^2$ at the 5 points not in general position described in \cref{lem:nonflat_A32}. It is also
straightforward to observe that any Eckardt point on a component isomorphic to $\wS_{kA_1}$ or $\wS_{A_2}$ arises as the
strict transform of an Eckardt point on the corresponding singular surface $S_{kA_1}$ or $S_{A_2}$. Thus we reduce to
describing the potential configurations of Eckardt points on cubic surfaces with at worst $A_1$ or one $A_2$
singularities, as well as on the special surface $Bl_5\bP^2$ of \cref{lem:nonflat_A32}.

\begin{remark}
    Surprisingly, Eckardt points on singular cubic surfaces do not appear to be well-studied. Some partial results
    appear in \cite{tuSemistableSingularCubic2005}, where Eckardt points are called star points. We note however that
    our notion of Eckardt points on singular cubic surfaces is more restrictive than in
    \cite{tuSemistableSingularCubic2005}, as we only consider points which occur as the intersection of three distinct
    lines of multiplicity one.
\end{remark}

We require the following result, which in the case of a smooth cubic surface is classical.

\begin{lemma} \label{lem:eckardt_constraints}
    Let $S$ be a cubic surface that is either smooth, has only $A_1$ singularities, or has exactly one singularity, of
    type $A_2$.
    \begin{enumerate}
        \item A given multiplicity one line of $S$ can contain at most two Eckardt points.
        \item Suppose $p$ and $p'$ are two Eckardt points of $S$ having no line in common. Then there is a unique third
            Eckardt point $p''$ also having no lines in common with $p$ or $p'$.
    \end{enumerate}
\end{lemma}

\begin{proof}
    \begin{enumerate}
        \item First suppose $S$ is smooth. One may choose a marking of $S$ as the blowup of 6 points $p_1,\ldots,p_6$ in
            $\bP^2$ so that the Eckardt points of interest lie on the line $c_6$ obtained as the strict transform of the
            conic through the first 5 points. The Eckardt points lying on $c_6$ must be obtained as triples
            $(e_i,c_6,\ell_{i6})$, where $e_i$ is the exceptional divisor over a point $p_i$, $i=1,\ldots,5$ lying on the
            conic, and $\ell_{i6}$ is the strict transform of the line through $p_i$ and $p_6$. For this triple to give
            an Eckardt point, the line through $p_i$ and $p_6$ must be tangent to the conic at $p_i$. But there are only
            two lines tangent to a conic passing through a point not on the conic.

            In the singular case, the minimal resolution of $S$ is obtained as the blowup of a degeneration of the 6
            points $p_1,\ldots,p_6$ so that they are no longer in general position. As in the smooth case, an Eckardt
            point on the line $c_6$ must have the form $(e_i,c_6,\ell_{i6})$. It follows that to have an Eckardt point
            with the chosen marking, $p_6$ cannot lie on the conic through the points $p_1,\ldots,p_5$, since otherwise
            the line through $p_i$ and $p_6$ will not be tangent to the conic. Therefore, as in the smooth case there
            are at most two Eckardt points on $\ell_{i6}$, unless some of the points $p_1,\ldots,p_5$ on the conic are
            infinitely near. But then the line between the infinitely near point and $p_6$ has multiplicity greater than
            one, hence the tangent point cannot be an Eckardt point in our definition.
        \item In the smooth case, let $\ell$ be the line in $\bP^3$ passing through the points $p$ and $p'$. Then $\ell$
            must intersect $S$ in a third point $p''$, which one can show is an Eckardt point, see for instance
            \cite[Proposition 9.1.26]{dolgachevClassicalAlgebraicGeometry}. The same argument applies in the singular
            case, since in our definition of an Eckardt point on a singular cubic surface any Eckardt point lies in the
            smooth locus, and all the lines in question have multiplicity one.
    \end{enumerate}
\end{proof}

The below lemma is a special case of \cite[Proposition 5.3]{tuSemistableSingularCubic2005} and \cite[Lemma
4.8]{coolsSingularHypersurfacesPossessing2011}.

\begin{lemma}[{\cite[Proposition 5.3]{tuSemistableSingularCubic2005}, \cite[Lemma
    4.8]{coolsSingularHypersurfacesPossessing2011}}] \label{lem:eckardt_specialization}
    Let $S$ be a cubic surface that has only $A_1$ singularities, or has exactly one singularity, of type $A_2$. Then
    there exists a one-parameter smoothing of $S$ such that any Eckardt point of $S$ is in the closure of the locus of
    Eckardt points of the generic fiber.
\end{lemma}

\begin{proposition} \label{prop:eckardt_configs}
    Let $S$ be a cubic surface that is either smooth, has only $A_1$ singularities, or has exactly one singularity, of
    type $A_2$.
    \begin{enumerate}
        \item If $S$ is smooth, then $S$ has either 0, 1, 2, 3, 4, 6, 9, 10, or 18 Eckardt points.
        \item If $S$ has one $A_1$ singularity, then $S$ has either 0, 1, 2, 3, 4, or 6 Eckardt points.
        \item If $S$ has two $A_1$ singularities, then $S$ has either 0 or 1 Eckardt points.
        \item If $S$ has three $A_1$ singularities, then $S$ has either 0 or 1 Eckardt points.
        \item If $S$ has four $A_1$ singularities, then $S$ has no Eckardt points.
        \item If $S$ has one $A_2$ singularity, then $S$ has either 0, 1, 2, or 3 Eckardt points.
    \end{enumerate}
\end{proposition}

\begin{proof}
    \begin{enumerate}
        \item The result for smooth cubic surfaces is classical and well-known, see for instance
            \cite{eckardtUeberDiejenigenFlaechen1876, bettenEckardtPointConfiguration2022}. The possible configurations
            of Eckardt points are constrained by \cref{lem:eckardt_constraints}. We note for later use in this proof
            that there is a unique smooth cubic surface with 18 Eckardt points, namely, the Fermat cubic
            \[
                \{x_0^3+x_1^3+x_2^3+x_3^3=0\} \subset \bP^3.
            \]
            Likewise, there is a unique smooth cubic surface with 10 Eckardt points, namely, the Clebsch diagonal cubic
            \[
                \{x_0+x_1+x_2+x_3+x_4=0, x_0^3+x_1^3+x_2^3+x_3^3+x_4^3=0\} \subset \bP^4.
            \]
            Finally, any smooth cubic surface with 9 Eckardt points is a \emph{cyclic} cubic surface, defined by an
            equation of the form
            \[
                \{x_0^3=G(x_1,x_2,x_3)\} \subset \bP^3,
            \]
            where $G(x_1,x_2,x_3)$ is the equation of a smooth plane cubic curve (cf. \cite[Proposition
            10.6]{dolgachevAutomorphismsCubicSurfaces2019}). The 9 Eckardt points on a cyclic cubic surface arise from the
            9 inflection points of the smooth plane cubic curve $\{G(x_1,x_2,x_3)=0\} \subset \bP^2$.
        \item Choose a marking of $S$ so that the minimal resolution of $S$ is the blowup of $\bP^2$ at 6 points on a
            conic. Then the lines of multiplicity one on $S$ are the strict transforms $\ell_{ij}$, $ij \subset [6]$, of
            the 15 lines through 2 of the 6 points. From this, we see that $S$ can only have 0, 1, 2, 3, 4, or 6 Eckardt
            points. For instance, we have the following possible choices for configurations of triples of lines giving
            Eckardt points.
            \begin{align*}
                1 &: (\ell_{12},\ell_{34},\ell_{56}), \\
                2 &: (\ell_{12},\ell_{34},\ell_{56}), (\ell_{13},\ell_{24},\ell_{56}), \\
                3 &: (\ell_{12},\ell_{34},\ell_{56}), (\ell_{15},\ell_{24},\ell_{36}), (\ell_{13},\ell_{45},\ell_{26}),
                \\
                4 &: (\ell_{12},\ell_{34},\ell_{56}), (\ell_{13}, \ell_{26}, \ell_{45}),
                (\ell_{15},\ell_{24},\ell_{36}), (\ell_{12}, \ell_{36}, \ell_{45}), \\
                6 &: (\ell_{12},\ell_{34},\ell_{56}), (\ell_{15},\ell_{24},\ell_{36}), (\ell_{13},\ell_{45},\ell_{26}),
                     (\ell_{12},\ell_{45},\ell_{36}), (\ell_{15},\ell_{34},\ell_{26}), (\ell_{13},\ell_{24},\ell_{56}).
            \end{align*}
            By applying \cref{lem:eckardt_constraints}, we see that there are no possible configurations with a
            different number of Eckardt points. (See \cite{tuConfigurationSpacesNonsingular2009} for similar arguments.)

            Furthermore, it is straightforward to construct cubic surfaces with one $A_1$ singularity and the prescribed
            numbers of Eckardt points. For instance, the cubic surface defined by
            \[
                \{x_0+x_1+x_2+x_3+x_4=0, x_0^3 + x_1^3 + x_2^3 + x_3^3 + \frac{1}{16}x_4^3 = 0\} \subset \bP^4
            \]
            has one $A_1$ singularity and 6 Eckardt points, cf. \cite[Proposition
            4.2(1)]{dardanelliHessiansModuliSpace2004}. (The tritangent planes to the Eckardt points are defined by the
            intersections with the hyperplanes $x_i + x_j = 0$ for $i,j \in \{0,1,2,3\}$ distinct; from this one
            directly verifies that each Eckardt point is indeed an Eckardt point in our sense, i.e., is the intersection
            point of three lines of multiplicity one.)

            We remark that one may also show that there are at most 6 Eckardt points on $S$ using
            \cref{lem:eckardt_specialization}. Namely, any Eckardt point on $S$ is the specialization of an Eckardt
            point on a smooth cubic surface, but the smooth cubic surfaces with greater than 6 Eckardt points were
            described in the proof of part (1). None of these smooth cubic surfaces specialize to a cubic surface with
            an $A_1$ singularity, thus we cannot have more than 6 points on a cubic surface with an $A_1$ singularity.

        \item Choose a marking of $S$ so that the minimal resolution of $S$ is the blowup of $\bP^2$ at six points
            $p_1,\ldots,p_6$ on a conic, such that $p_5$ and $p_6$ are infinitely near points. Then the 7 lines of
            multiplicity one on $S$ are labeled by
            \[
                \ell_{12}, \ell_{13}, \ell_{14}, \ell_{23}, \ell_{24}, \ell_{34}, \ell_{56},
            \]
            where $\ell_{ij}$ is the strict transform of the line through $p_i$ and $p_j$, cf. \cref{fig:eckardt_2A1}.
            We therefore see that the potential triples of lines giving Eckardt points are the 3 triples
            \[
                (\ell_{12},\ell_{34},\ell_{56}), \;\; (\ell_{13},\ell_{24},\ell_{56}), \;\;
                (\ell_{14},\ell_{23},\ell_{56}).
            \]
            It is not possible for all 3 of these triples to give an Eckardt point, since then the line $\ell_{56}$
            would contain 3 Eckardt points, contradicting \cref{lem:eckardt_constraints}. In fact, it is also not
            possible to have 2 Eckardt points on $S$. Indeed, following \cref{lem:eckardt_specialization}, we can
            consider a smooth cubic surface $\wS$ specializing to $S$, such that the Eckardt points on $S$ are limits of
            Eckardt points on $\wS$. Because the Eckardt points on $S$ all lie on the line $\ell_{56}$, the Eckardt
            points on $\wS$ that limit to Eckardt points on $S$ must necessarily lie on the line of $\wS$ limiting to
            $\ell_{56}$. But $\ell_{56}$ intersects the line (of multiplicity 4) connecting the two $A_1$ singularities
            of $S$, and by \cite[Lemma 5.1]{tuSemistableSingularCubic2005}, this intersection point is the limit of an
            Eckardt point on $\wS$. Since there's at most 2 Eckardt points on a line, it follows that $S$ can have at
            most 1 Eckardt point. It is indeed possible to achieve this. For instance, the cubic surface
            \[
                \{x_0+x_1+x_2+x_3+x_4=0, x_0^3+4x_1^3+4x_2^3 + 9x_3^3 + 9x_4^3=0\} \subset \bP^4
            \]
            has two $A_1$ singularities and one Eckardt point.

            \begin{figure}[htpb]
                \centering
                \includegraphics[width=0.4\linewidth]{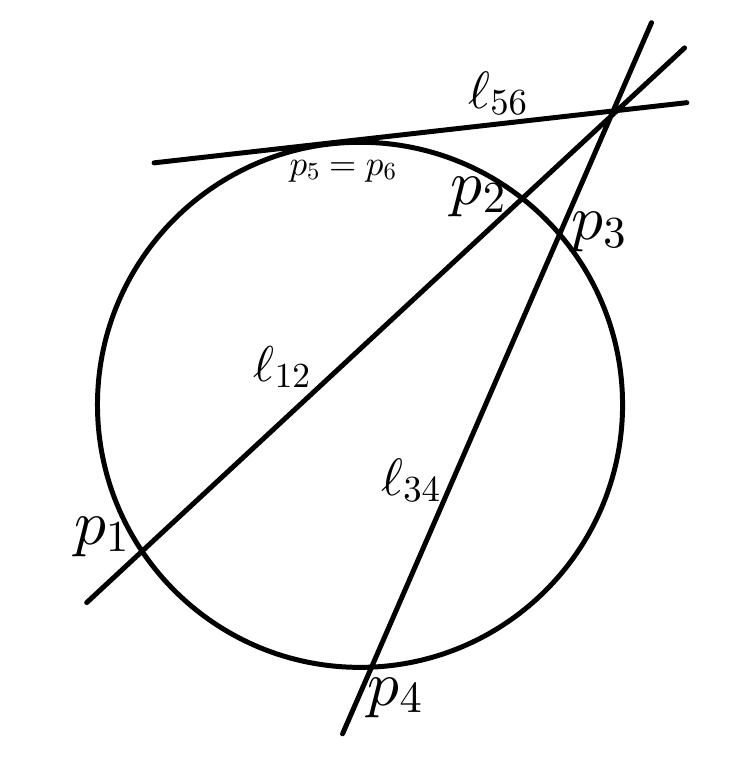}
                \caption{The configuration of lines and points in $\bP^2$ yielding an Eckardt point on a $2A_1$ cubic
                surface. In the resulting $2A_1$ cubic surface, the strict transform of the line $\ell_{56}$ is a line
            passing through the connecting line of the two singularities.}%
                \label{fig:eckardt_2A1}
            \end{figure}
        \item Choose a marking of $S$ so that the minimal resolution of $S$ is the blowup of $\bP^2$ at six points
            $p_1,\ldots,p_6$ on a conic, such that $p_5$ and $p_6$ are infinitely near points, and $p_3$ and $p_4$ are
            infinitely near points. Then the 3 lines of multiplicity one on $S$ are labeled by $\ell_{12}, \ell_{34},
            \ell_{56}$. Thus the only potential Eckardt point is given by the triple $(\ell_{12},\ell_{34},\ell_{56})$,
            cf. \cref{fig:eckardt_3A1}. It is indeed possible to achieve this. For instance, the cubic surface
            \[
                \{x_0+x_1+x_2+x_3+x_4=0, x_0^3 + x_1^3 + x_2^3 + 4x_3^3 + 4x_4^3 = 0\} \subset \bP^4
            \]
            has three $A_1$ singularities and one Eckardt point, cf. \cite[Proposition
            4.2(3)]{dardanelliHessiansModuliSpace2004}.  (It is stated in \textit{loc. cit.} that this cubic surface has
            4 Eckardt points, but we verify directly that only one of these occurs as the intersection of 3 lines of
            multiplicity one.)

            \begin{figure}[htpb]
                \centering
                \includegraphics[width=0.4\linewidth]{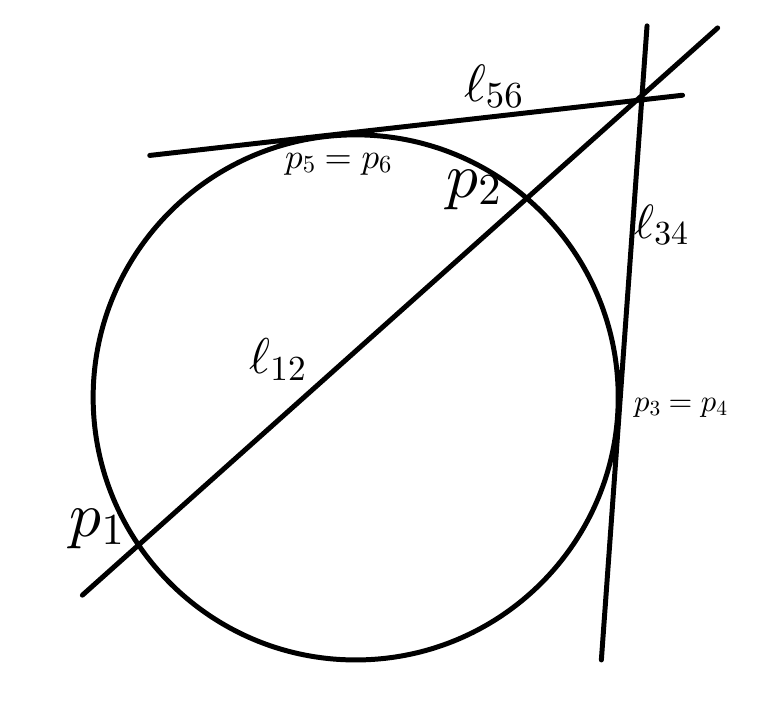}
                \caption{The configuration of lines and points in $\bP^2$ yielding an Eckardt point on a $3A_1$ cubic
                surface.}%
                \label{fig:eckardt_3A1}
            \end{figure}
        \item Choose a marking of $S$ so that the minimal resolution of $S$ is the blowup of $\bP^2$ at the six points
            of intersection of four lines in general position in $\bP^2$. Then the 3 lines of multiplicity one on $S$
            are the strict transforms of the 3 lines through a point on one of the first four lines and a point on
            another of the first four lines. If these 3 lines intersected in a point, then we would obtain a
            configuration in $\bP^2$ of 7 lines with 7 triple intersection points. This is the Fano configuration, which
            can only occur in characteristic 2. Thus there are no Eckardt points on $S$.
        \item Choose a marking of $S$ so that the minimal resolution of $S$ is the blowup of $\bP^2$ at six points
            $p_1,\ldots,p_6$, such that $p_1,p_2,p_3$ lie on one line $\ell_{123}$, and $p_4,p_5,p_6$ lie on another
            line $\ell_{456}$. The lines of multiplicity one on $S$ are the strict transforms of the 9 lines $\ell_{ij}$
            passing through $p_i$ and $p_j$, for $i\in\{1,2,3\}$ and $j\in\{4,5,6\}$. Choose coordinates on $\bP^2$ so
            that the 6 points are given by
            \begin{align*}
                p_1=(1:0:0), \;\; p_2=(1:a:0), \;\; p_3=(0:1:0),\\
                p_4=(0:0:1), \;\; p_5=(1:1:b), \;\; p_6=(1:1:1),
            \end{align*}
            where $a,b \not\in \{0,1\}$. From these points one can write explicit equations for each of the 9 lines,
            \begin{align*}
                \ell_{14}=(y=0), \;\;\; \ell_{15}=(z=by), \;\;\; \ell_{16}=(y=z),\\
                \ell_{24}=(y=ax), \;\;\; \ell_{25}=\left(-ax+y+\frac{a-1}{b}z=0\right), \;\;\; \ell_{26}=(-ax+y+(a-1)z=0),\\
                \ell_{34}=(x=0), \;\;\; \ell_{35}=(z=bx), \;\;\; \ell_{36}=(x=z).
            \end{align*}
            Three lines in $\bP^2$ meet in a point if the determinant of the matrix whose rows are the coefficients of
            the three lines is zero. Thus we can explicitly compute conditions on $a,b$ for three lines to form an
            Eckardt point. We list the conditions for some particular triples of lines to form an Eckardt point:
            \begin{align*}
                (\ell_{14},\ell_{25},\ell_{36}) : a = \frac{1}{1-b}, \;\;
                (\ell_{15},\ell_{26},\ell_{34}) : b = \frac{1}{1-a}, \\
                (\ell_{16},\ell_{24},\ell_{35}) : a=b, \;\;
                (\ell_{14},\ell_{26},\ell_{35}) : a = \frac{b}{b-1}.
            \end{align*}
            From this we see that if any two of the first three triples give Eckardt points, then the third one does as
            well (as expected from \cref{lem:eckardt_constraints}), namely, these three triples give three Eckardt
            points precisely when $a=b$ is a root of the polynomial $x^2-x+1$. On the other hand, the last two triples
            give two Eckardt points precisely when $a=b=2$. Note this is incompatible with the condition that the first
            three triples give three Eckardt points. We conclude that the number of Eckardt points on $S$ can only be 0,
            1, 2, or 3.
    \end{enumerate}
\end{proof}

\begin{remark}
    There is some interesting classical geometry related to line arrangements appearing in the case of an $A_2$
    singularity above. Namely, following the notation of the above proof, consider the three intersection points
    $\ell_{15} \cap \ell_{24}$, $\ell_{16} \cap \ell_{34}$, and $\ell_{26} \cap \ell_{35}$. The classical Pappus'
    Hexagon Theorem implies that there is a unique line in $\bP^2$ passing through these three intersection points. If
    we suppose that $\ell_{36}$, $\ell_{25}$, and $\ell_{14}$ pass through these respective intersection points as well,
    then we obtain in total a configuration of 12 lines and 9 points, such that any point is contained in 4 lines, and
    any line passes through 3 points. This is the classical Hesse configuration, which can also be constructed as the 9
    inflection points of a smooth plane cubic curve, and the 12 lines passing through 3 of the inflection points each.
\end{remark}

\begin{proposition} \label{prop:nonflat_eckardt}
    Let $S$ be the blowup of $\bP^2$ at five points $p_1,\ldots,p_5$, such that $p_1,\ldots,p_4$ are in general position
    and $p_5$ is the intersection point of the line through $p_1$ and $p_2$ with the line through $p_3$ and $p_4$.
    Consider the collection of lines on $S$ consisting of the four lines $\ell_{13}$, $\ell_{14}$, $\ell_{23}$,
    $\ell_{24}$, where $\ell_{ij}$ is the strict transform of the line through $p_i$ and $p_j$, and a line $\ell$, the
    strict transform of a line passing through $p_5$ but not through any of $p_1,\ldots,p_4$. Then $S$ has two potential
    Eckardt points---the triple intersection $\ell \cap \ell_{13} \cap \ell_{24}$, and the triple intersection $\ell
    \cap \ell_{23} \cap \ell_{14}$. It is not possible to have both of these Eckardt points. Thus, $S$ has either 0 or 1
    Eckardt points.
\end{proposition}

\begin{proof}
    It is clear that by moving the position of $\ell$ appropriately, we can obtain either Eckardt point. Thus it remains
    to show that we cannot have both. For this, suppose that $\ell$ passes through both points $p = \ell_{13} \cap
    \ell_{24}$ and $p' = \ell_{23} \cap \ell_{24}$. Consider the blowdown $S \to \bP^2$. Then we obtain a configuration of 7
    lines in $\bP^2$ with 7 triple points, namely, the 7 lines $\ell, \ell_{12}, \ell_{34}, \ell_{13}, \ell_{23},
    \ell_{14}, \ell_{24}$ and the 7 points $p_1,\ldots,p_5,p=p'$. (Here $\ell_{12}$ and $\ell_{34}$ denote respectively
    the line through $p_1$ and $p_2$, and the line through $p_3$ and $p_4$, and by abuse of notation $\ell,\ell_{ij}$
    denote the images of the corresponding lines on $S$.) This is the Fano configuration, which can only occur in
    characteristic 2.
\end{proof}

\begin{remark}
    The above propositions can also be proven using the geometry of the moduli space $\wY(E_6)$. Namely, in
    \cite[Section 4]{schockE_6InvariantBirational2024}, we explicitly describe each boundary divisor of $\wY(E_6)$ and
    its intersections with Eckardt divisors. From this, one can determine the potential configurations of Eckardt
    points on irreducible components of fibers of $\wt\pi : \wY(E_7) \to \wY(E_6)$.

    For example, a type $a_2$ divisor on $\wY(E_6)$ intersects a total of five Eckardt divisors. Two of these
    intersections (labeled 5(b) in \cite[Proposition 4.7]{schockE_6InvariantBirational2024}) give Eckardt points on the
    component described in \cref{prop:nonflat_eckardt}, and one sees from the description there that the restrictions of
    these Eckardt divisors to the $a_2$ divisor are disjoint, thus there can be at most one Eckardt point on this
    component. The remaining three intersections (labeled 5(a) in \cite[Proposition
    4.7]{schockE_6InvariantBirational2024}) give Eckardt points on the component isomorphic to the resolution of a cubic
    surface with two $A_1$ singularities. Since these intersections are nontrivial, we see that it is possible to have
    at least one Eckardt point. On the other hand, the restrictions of these three Eckardt divisors to this $a_2$
    divisor are disjoint, as follows from the explicit description of these restrictions in \cite[Proposition
    4.7]{schockE_6InvariantBirational2024}. Thus there cannot be two or more Eckardt points on this component.
\end{remark}

\begin{theorem} \label{thm:eckardts}
    The possible configurations of Eckardt points on irreducible components of fibers of $\wt\pi : \wY(E_7) \to
    \wY(E_6)$ of types $a$, $a_2$, $a_3$, $a_4$, $b$, $ab$, $a_2b$, and $a_3b$ are listed in
    \cref{tab:a_23-1,tab:a2_23-1,tab:a3_23-1,tab:a4_23-1,tab:b_23-1,tab:ab_23-1,tab:a2b_23-1,tab:a3b_23-1}. The
    remaining fibers of $\wt\pi : \wY(E_7) \to \wY(E_6)$ are obtained as degenerations of the listed fibers, where some
    of the components of the form $Bl_5\bP^2$ as described in \cref{lem:nonflat_A32} degenerate into four irreducible
    components, as described in \cref{ex:aa2_fiber}. No Eckardt points can occur on these degenerate components.
\end{theorem}

\begin{proof}
    This follows directly from \cref{prop:eckardt_configs,prop:nonflat_eckardt}, and our explicit descriptions of the
    fibers of $\wt\pi$ as in \cref{thm:fibers}.
\end{proof}

\section{Moduli of weighted stable marked cubic surfaces} \label{sec:weighted_cubics}

In this section we state our main theorem (\cref{thm:cubics_main}), describing in detail the wall crossings for the
moduli of weighted stable marked cubic surfaces, and the weighted stable pairs appearing in each moduli space.

We begin with some generalities on moduli of weighted stable pairs in the case of marked del Pezzo surfaces.

\begin{definition}
    Let $N$ be the number of lines on a smooth del Pezzo surface $S$ of degree $9-n$ (so $N=16,27,56,240$ for
    $n=5,6,7,8$). Fix a marking $B=B_1 + \cdots + B_N$ of the $(-1)$-curves on $S$.  We define the \emph{weight domain}
    for moduli of marked del Pezzo surfaces of degree $9-n$ to be
    \[
        \cP_n = \left\{\vec{c}=(c_1,\ldots,c_N) \in (0,1]^N \mid K_S + \sum_{i=1}^N c_iB_i \text{ is ample } \right\},
    \]
    where $K_S$ is the canonical divisor of $S$.  For $\vec{c},\vec{d} \in \cP_n$, we write $\vec{d} \leq \vec{c}$ if
    $d_i \leq c_i$ for all $i=1,\ldots,N$.
\end{definition}

\begin{definition}
    For $\vec{c} \in \cP_n$, let $Y^n_{\vec{c}} \subset Y(E_n)$ be the open subvariety parameterizing smooth marked del
    Pezzo surfaces $(S,B)$ such that $(S,\vec{c}B)$ is log canonical.
\end{definition}

Concretely, a smooth weighted marked del Pezzo surface $(S,\vec{c}B)$ is log canonical if whenever lines
$B_{i_1},\ldots,B_{i_k}$ intersect at a point, we have $\sum_{j=1}^k c_{i_j} \leq 2$, cf. \cite[Lemma
5.10]{gallardoGeometricInterpretationToroidal2021}. The condition that $\vec{c} \in \cP_n$ and $(S,\vec{c}B)$ is log
canonical exactly says (since $S$ is smooth) that $(S,\vec{c}B)$ is a stable pair, so $Y^n_{\vec{c}}$ is the moduli
space of smooth weighted stable marked del Pezzo surfaces of degree $9-n$.

The following is a special case of the main results of \cite{ascherWallCrossingModuli2023, mengMMPLocallyStable2023}.

\begin{theorem} \label{thm:gen_walls}
    Assume $n = 5$ or $6$. Then for any $\vec{c} \in \cP_n$, there is a normal projective variety
    $\oY_{\vec{c}}=\oY_{\vec{c}}^n$ and a flat family $(\cS_{\vec{c}},\vec{c}\cB)$, compactifying the moduli space
    $Y_{\vec{c}}^n$ of smooth weighted stable marked del Pezzo surfaces of degree $9-n$.  Furthermore, there is a finite
    rational polyhedral wall-and-chamber decomposition of $\cP_n$ such that the following hold.
    \begin{enumerate}
        \item For $\vec{c},\vec{c'}$ contained in the same chamber, there are canonical isomorphisms
            \[
                \begin{tikzcd}
                    \cS_{\vec{c}} \ar[r, "\sim"] \ar[d] & \cS_{\vec{c'}} \ar[d] \\
                    \oY_{\vec{c}} \ar[r, "\sim"] & \oY_{\vec{c'}}.
                \end{tikzcd}
            \]
        \item For $\vec{c},\vec{d}$ in different chambers with $\vec{d} \leq \vec{c}$, there is a canonical functorial
            birational wall-crossing morphism $\rho_{\vec{d},\vec{c}} : \oY_{\vec{c}} \to \oY_{\vec{d}}$ induced by a
            canonical birational map $\cS_{\vec{c}} \dashrightarrow \cS_{\vec{d}}$ sending a stable pair $(S,\vec{c}B)$
            to the stable model of the pair $(S,\vec{d}B)$.
    \end{enumerate}
\end{theorem}

\begin{proof}
    If $\vec{c}=\vec{1}=(1,\ldots,1)$, then $\oY^n=\oY_{\vec{1}}^n$ and its universal family $(\cS,\cB)$ are constructed
    in \cite{hackingStablePairTropical2009}. Namely, for $n=5$, $\oY^5 = \oY(E_5)$, with universal family $(\cS,\cB)$
    given by $\oY(E_6)$ and the sum of the horizontal $A_1$ divisors in $\oY(E_6)$ (cf. \cref{thm:pi,sec:deg4}). For
    $n=6$, $\oY^6 = \wY(E_6)$, with universal family $(\ddot{Y}(E_7),B_{\ddot{\pi}})$ as described in
    \cref{sec:morphisms,sec:fibers,sec:eckardts}. It follows immediately from our explicit computation of the fibers of
    the universal family (see \cref{sec:fibers,sec:eckardts,sec:deg4} and below) that if $c_i=1-\epsilon$, $0 < \epsilon
    \ll 1$, then $(\cS_{\vec{c}},\vec{c}\cB) \to \oY^n$ still defines a flat family of stable pairs. Thus we can assume
    all $c_i < 1$. Let $Y_{\times}^n \subset Y(E_n)$ be the open subvariety parameterizing smooth marked del Pezzo
    surfaces $(S,B)$ such that $B$ has only normal crossings. (So $Y_{\times}^5 = Y(E_5)$ and $Y_{\times}^6$ is the
    locus of cubic surfaces with no Eckardt points.) Define $\oY_{\vec{c}}^n$ to be the normalization of the closure of
    $Y_{\times}^n$ in the moduli space of $\vec{c}$-weighted stable pairs. Then the result follows from \cite[Theorem
    1.1, Theorem 5.1]{ascherWallCrossingModuli2023}.
\end{proof}

\begin{remark}
    By \cite{mengMMPLocallyStable2023}, it is not necessary to reduce from weight $1$ to weight $1-\epsilon$ in the
    above proof. Our results show that there is not a wall at $\vec{c}=\vec{1}$.
\end{remark}

In what follows we will only be interested in the case where the weight vector is constant, i.e., $\vec{c} =
(c,\ldots,c)$ for some $c \in \left(\frac{9-n}{N},1\right]$. In this case we simply write $c$ instead of $\vec{c}$. Note
that the lines on a del Pezzo surface of degree 4 are always normal crossings, so $Y^5_c=Y^5$ for all $c$. Likewise, any
point on a cubic surface where the lines are not normal crossings is an Eckardt point where 3 lines intersect. It
follows by \cite[Lemma 5.10]{gallardoGeometricInterpretationToroidal2021} that $Y^6_c=Y^6$ for $c \leq 2/3$, and
$Y^6_c=Y^6_{\times} \subset Y^6$ is the locus of cubic surfaces with no Eckardt points for $2/3 < c \leq 1$. If
$t_1,\ldots,t_m \in \left(\frac{9-n}{N},1\right]$ are the walls for $\oY^n_c$, then $\oY^n_{t_i-\epsilon} \cong
\oY^n_{t_i}$ by \cite[Theorem 1.9]{ascherWallCrossingModuli2023} (also from our explicit descriptions of the stable
surfaces below), and we write $\oY^n_{(t_{i-1},t_i]} \cong \oY^n_c$ for any $c \in (t_{i-1},t_i]$.

The following is the main theorem of this article. The analogous result for degree 4 del Pezzo surfaces is described in
\cref{sec:deg4}; we recommend the reader consult the degree 4 case before moving on to the much more challenging case
degree 3 case below.

\begin{theorem} \label{thm:cubics_main}
    The weight domain $\left(\frac{1}{9},1\right]$ for the moduli space of weighted stable marked cubic surfaces
    $(S,cB)$ admits precisely five walls, at $c=1/6$, $1/4$, $1/3$, $1/2$, and $2/3$, inducing birational morphisms of
    moduli spaces
    \[
        \oY(E_6) \cong \oY^6_{(1/9,1/6]} \xleftarrow{\sim} \oY^6_{(1/6,1/4]} \leftarrow  \oY^6_{(1/4,1/3]} \xleftarrow{\sim}
        \oY^6_{(1/3,1/2]} \xleftarrow{\sim} \oY^6_{(1/2,2/3]} \xleftarrow{\sim} \oY^6_{(2/3,1]} \cong \wY(E_6),
    \]
    where each morphism is an isomorphism except for $\oY^6_{(1/6,1/4]} \leftarrow \oY^6_{(1/4,1/3]}$, which is the
    blowdown $\wY(E_6) \to \oY(E_6)$ constructed in \cref{prop:flattening}. The weighted stable marked cubic surfaces in
    each chamber are described in the remainder of this section.
\end{theorem}

The most significant aspect of the above theorem is the explicit description of the weighted stable marked cubic
surfaces appearing in each chamber, which takes up the remainder of this section. We present these results in a bottom
up fashion, starting with the minimal weight $c=1/9+\epsilon$ and describing the stable replacements as one increases
$c$ to $1$. This is the opposite of the top down approach we use to actually prove the theorem (see
\cref{sec:proof_main}), where we start from the maximal weight $c=1$ and decrease $c$ to $1/9+\epsilon$. (However, see
also \cref{sec:bottom_up} where we discuss how one can also prove the theorem with the bottom up approach.)

In order to classify the weighted stable marked cubic surfaces, we mildly extend the notation of \cref{not:strata} for
strata of $\wY(E_6)$ and fibers of $\wt\pi : (\wY(E_7),B_{\wt\pi}) \to \wY(E_6)$. Namely, as our explicit descriptions
below reveal, the morphism $\wt\pi : (\wY(E_7),B_{\wt\pi}) \to \wY(E_6)$ gives the universal family of weighted stable
marked cubic surfaces, for weights $1/2 < c \leq 2/3$. We will say a weighted stable marked cubic surface for \emph{any}
weight $1/9 < c \leq 1$ is of type $T$ if its stable replacement for weights $1/2 < c \leq 2/3$ is of type $T$ as
described in \cref{not:strata}.

In particular, we roughly split our discussion into surfaces labeled by types $a$, $a_2$, $a_3$, $a_4$, $b$, $ab$,
$a_2b$, and $a_3b$. As we will see, any other type of weighted stable marked cubic surface is easily described as a
particular degeneration of one of these types. We also label the irreducible components of weighted stable marked cubic
surfaces by different numbered types, roughly numbered in the order that the components appear as one increases the
weights.

Some of the pictures below (e.g., for the weight $1/9+\epsilon$ surfaces) are inspired by \cite[Table
1]{gallardoGeometricInterpretationToroidal2021}. Following the conventions there, we denote lines of multiplicities 1,
2, and 4 by solid, dashed, and dotted black lines, respectively. We denote the gluing curves of irreducible components
by red lines, which are typically solid, but we occasionally draw them as dashed red lines in order to avoid confusion
about the number of irreducible components.

\subsection{Smooth weighted stable marked cubic surfaces}

\begin{proposition} \label{prop:smooth}
    For smooth weighted stable marked cubic surfaces $(S,cB)$, there is one walls.
    \begin{enumerate}
        \item For weights $1/9 < c \leq 2/3$, any weighted marked cubic surface $(S,cB)$ with $S$ smooth is stable.
        \item For weights $2/3 < c \leq 1$, smooth weighted stable marked cubic surfaces are obtained from the weight
            $1/9 < c \leq 2/3$ smooth weighted stable marked cubic surfaces by resolving Eckardt points as described in
            \cref{prop:resolve_eckardt}.
    \end{enumerate}
\end{proposition}

\begin{remark}
    As discussed in \cref{sec:eckardts}, it is a classical fact that a smooth cubic surface has $s=0$, $1$, $2$, $3$,
    $4$, $6$, $9$, $10$, or $18$ Eckardt points. Thus weighted stable marked cubic surfaces $(S,cB)$ for weights $2/3 <
    c \leq 1$ arising from smooth cubic surfaces have $s+1$ irreducible components---one isomorphic to the blowup of the
    given smooth cubic surface at its $s$ Eckardt points, and $s$ copies of $\bP^2$ attached to the $s$ exceptional
    divisors. For instance, the Fermat cubic surface $x_0^3 + \cdots + x_3^3 = 0$ in $\bP^3$ is the unique cubic surface
    with 18 Eckardt points; its stable replacement has 19 irreducible components.
\end{remark}

\subsection{Type $a$}

\begin{proposition} \label{prop:a}
    For type $a$ weighted stable marked cubic surfaces $(S,cB)$, there are three walls.
    \begin{enumerate}
        \item For weights $1/9 < c \leq 1/6$, the type $a$ surfaces are described in \cref{fig:a_19-16}.
        \item For weights $1/6 < c \leq 1/2$, the type $a$ surfaces are described in \cref{fig:a_16-12}.
        \item For weights $1/2 < c \leq 2/3$, the type $a$ surfaces are described in \cref{fig:a_12-23}.
        \item For weights $2/3 < c \leq 1$, the type $a$ surfaces are obtained from the weight $1/2 < c \leq 2/3$ type
            $a$ surfaces by resolving Eckardt points as described in \cref{prop:resolve_eckardt}. The possible
            configurations of Eckardt points are summarized in \cref{tab:a_23-1}.
    \end{enumerate}
\end{proposition}

%\subsubsection{Weights $1/9 < c \leq 1/6$}

\begin{figure}[!htpb]
    \centering
    \includegraphics[width=0.15\linewidth]{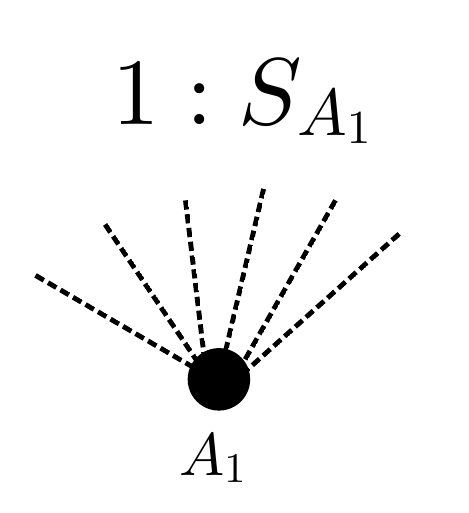}
    \caption{A type $a$ surface $(S_0,cB_0)$ for weights $1/9 < c \leq 1/6$. This is a cubic surface with an
    $A_1$ singularity. Recall that dashed lines have multiplicity 2. In addition to the six multiplicity 2 lines shown,
    there are 15 lines of multiplicity 1 on the surface, not shown.}%
    \label{fig:a_19-16}
\end{figure}

%\subsubsection{Weights $1/6 < c \leq 1/2$}

\begin{figure}[!htpb]
    \centering
    \includegraphics[width=0.5\linewidth]{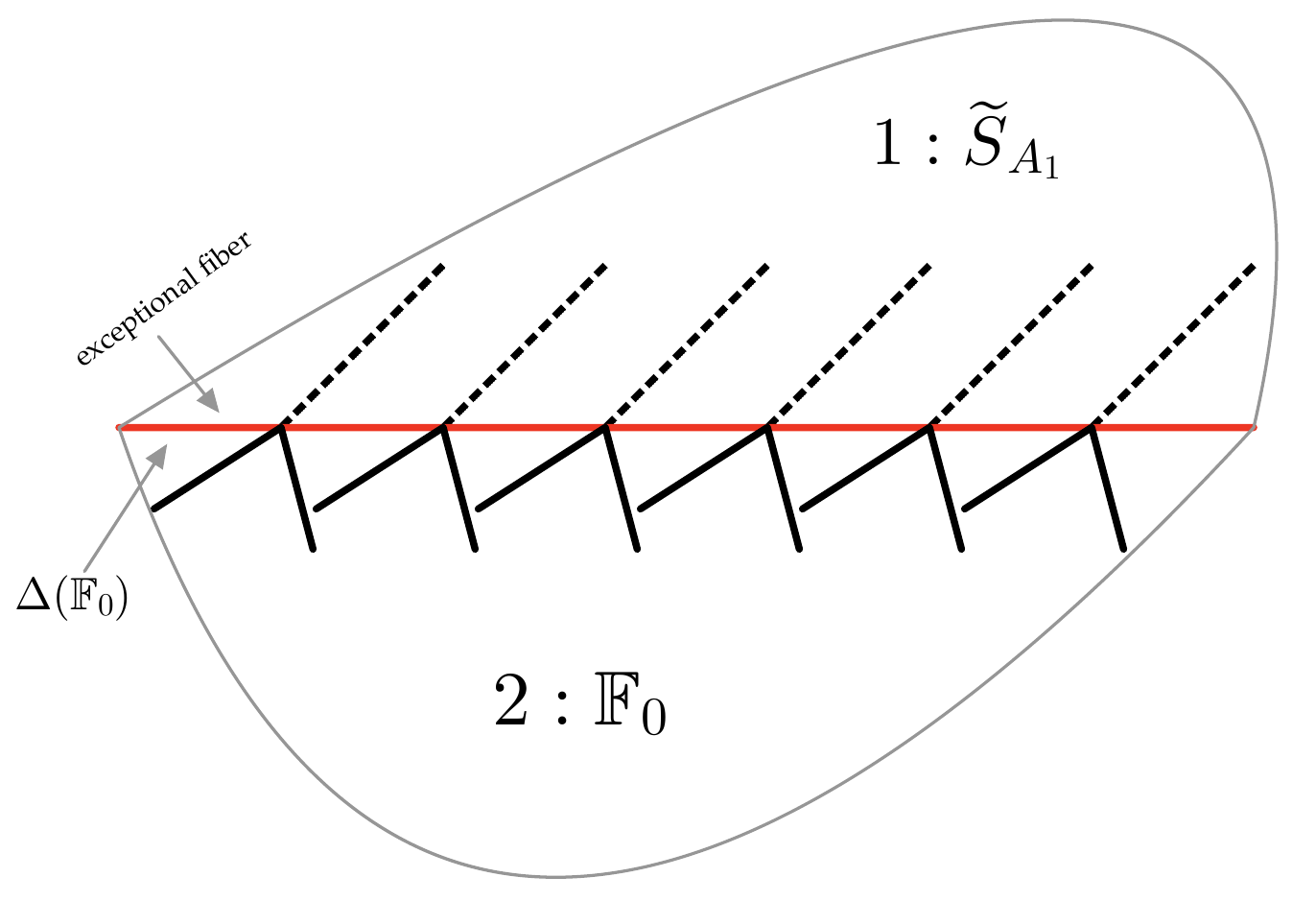}
    \caption{A type $a$ surface $(S_1,cB_1)$ for weights $1/6 < c \leq 1/2$, obtained as the stable replacement for
        these weights of the surface $(S_0,cB_0)$ of \cref{fig:a_19-16}. This is obtained by resolving the $A_1$
        singularity of $(S_0,cB_0)$, and attaching to the exceptional fiber a type 2 component isomorphic to $\bF_0
        \cong \bP^1 \times \bP^1$, glued along its diagonal. Each line of multiplicity 2 passing through the $A_1$
        singularity splits into two distinct lines on the $\bF_0$ component, one in each ruling.}%
    \label{fig:a_16-12}
\end{figure}

%\subsubsection{Weights $1/2 < c \leq 2/3$}

\begin{figure}[!htpb]
    \centering
    \includegraphics[width=0.5\linewidth]{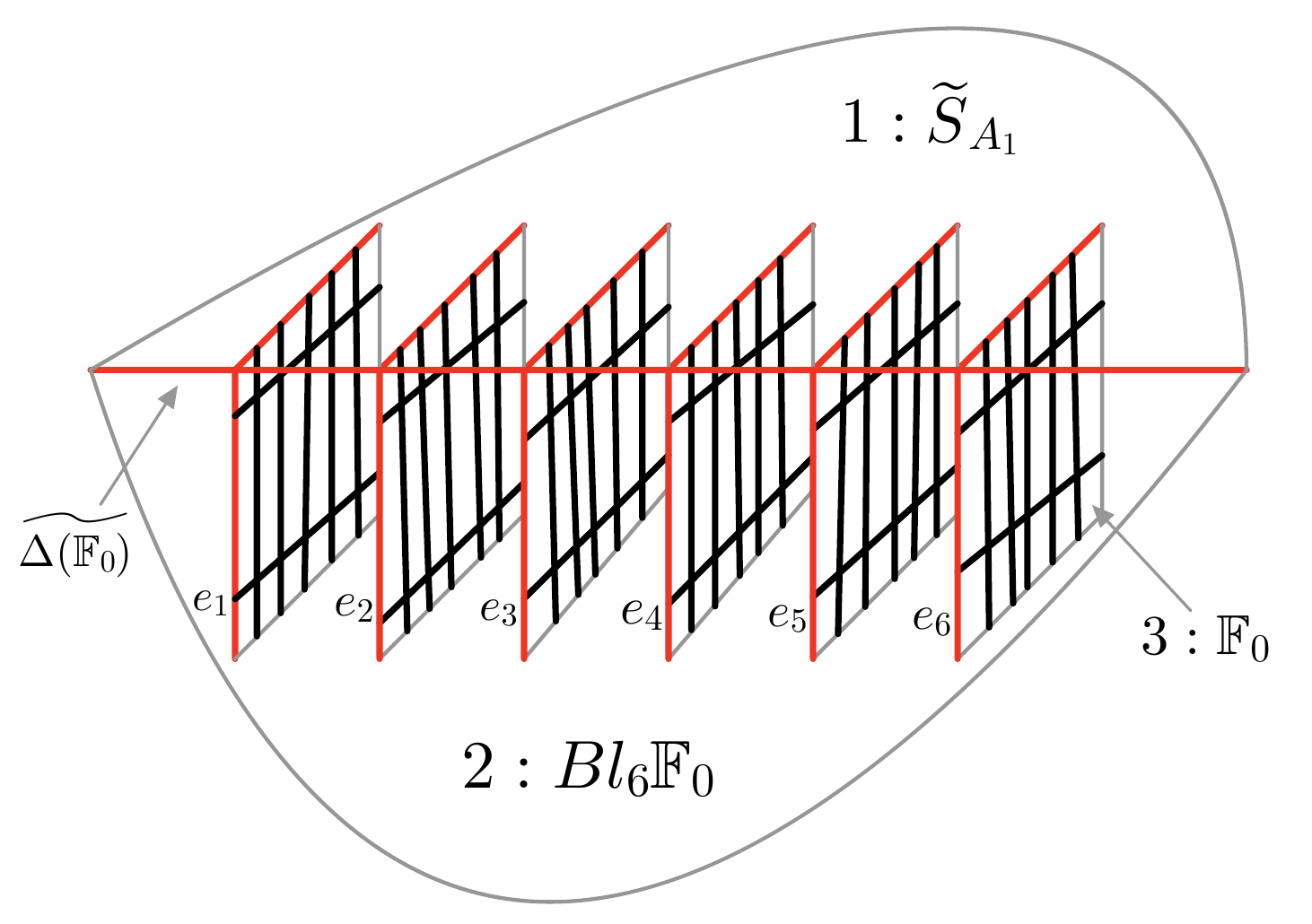}
    \caption{A type $a$ surface $(S_2,cB_2)$ for weights $1/2 < c \leq 2/3$, obtained as the stable replacement for
        these weights of the surface $(S_1,cB_1)$ of \cref{fig:a_16-12}. Over each multiplicity 2 line of $(S_1,cB_1)$,
        a type 3 component isomorphic $\bF_0$ is attached, where the two lines are distributed as distinct lines in the
        same ruling. Each of the 15 lines of multiplicity 1 on the type 1 component $\wS_{A_1}$ intersects exactly two
        type 3 components, making up one of the five lines in the other ruling on $\bF_0$.}%
    \label{fig:a_12-23}
\end{figure}

\begin{table}[!htpb]
    \centering
    \caption{The types of irreducible components of the weight $1/2 < c \leq 2/3$ surface of type $a$ pictured in
        \cref{fig:a_12-23}, and the possible numbers of Eckardt points on each component. For the component of type 1,
        $\wS_{A_1}$ refers to the minimal resolution of a cubic surface with one $A_1$ singularity. For the component of
        type 2, $Bl_6\bF_0$ refers to the blowup of $\bF_0 \cong \bP^1 \times \bP^1$ at 6 points on the diagonal.}
    \label{tab:a_23-1}
    \begin{tabular}{| c | c | c | c |}
        \hline
        Label & Surface & \# & Eckardt points \\
        \hline
        \hline
        1 & $\wS_{A_1}$ & 1 & 0, 1, 2, 3, 4, or 6 \\
        2 & $Bl_6\bF_0$ & 1 & 0 \\
        3 & $\bF_0$ & 6 & 0 \\
        \hline
    \end{tabular}
\end{table}

\subsection{Type $a_2$}

\begin{proposition} \label{prop:a2}
    For type $a_2$ weighted stable marked cubic surfaces $(S,cB)$, there are four walls.
    \begin{enumerate}
        \item For weights $1/9 < c \leq 1/6$, the type $a_2$ surfaces are described in \cref{fig:a2_19-16}.
        \item For weights $1/6 < c \leq 1/4$, the type $a_2$ surfaces are described in \cref{fig:a2_16-14}.
        \item For weights $1/4 < c \leq 1/2$, the type $a_2$ surfaces are described in \cref{fig:a2_14-12}.
            Additionally, crossing the wall $c=1/4$ introduces type $aa_2$ surfaces as degenerations of type $a_2$
            surfaces, described in \cref{fig:aa2_14-12}.
        \item For weights $1/2 < c \leq 2/3$, the type $a_2$ surfaces are described in \cref{fig:a2_12-23}. The type
            $aa_2$ surfaces are described in \cref{fig:aa2_12-23}.
        \item For weights $2/3 < c \leq 1$, the type $a_2$ surfaces are obtained from the weight $1/2 < c \leq 2/3$ type
            $a_2$ surfaces by resolving Eckardt points as described in \cref{prop:resolve_eckardt}. The possible
            configurations of Eckardt points are summarized in \cref{tab:a2_23-1}. The type $aa_2$ surfaces are obtained
            similarly.
    \end{enumerate}
\end{proposition}

%\subsubsection{Weights $1/9 < c \leq 1/6$}

\begin{figure}[!htpb]
    \centering
    \includegraphics[width=0.3\linewidth]{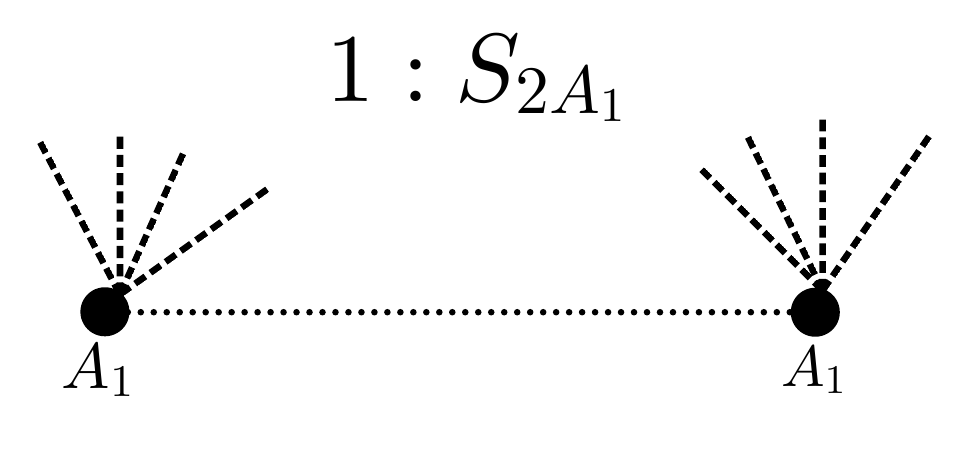}
    \caption{A type $a_2$ surface $(S_0,cB_0)$ for weights $1/9 < c \leq 1/6$. This is a cubic surface with two $A_1$
        singularities. Recall that dashed lines have multiplicity 2 and dotted lines have multiplicity 4. Each line of
        multiplicity 2 passing through one $A_1$ singularity intersects precisely one line of multiplicity 2 passing
        through the other $A_1$ singularity; these intersections are not drawn. In addition to the lines shown, there
        are 7 lines of multiplicity 1, not shown.}%
    \label{fig:a2_19-16}
\end{figure}

%\subsubsection{Weights $1/6 < c \leq 1/4$}

\begin{figure}[!htpb]
    \centering
    \includegraphics[width=0.6\linewidth]{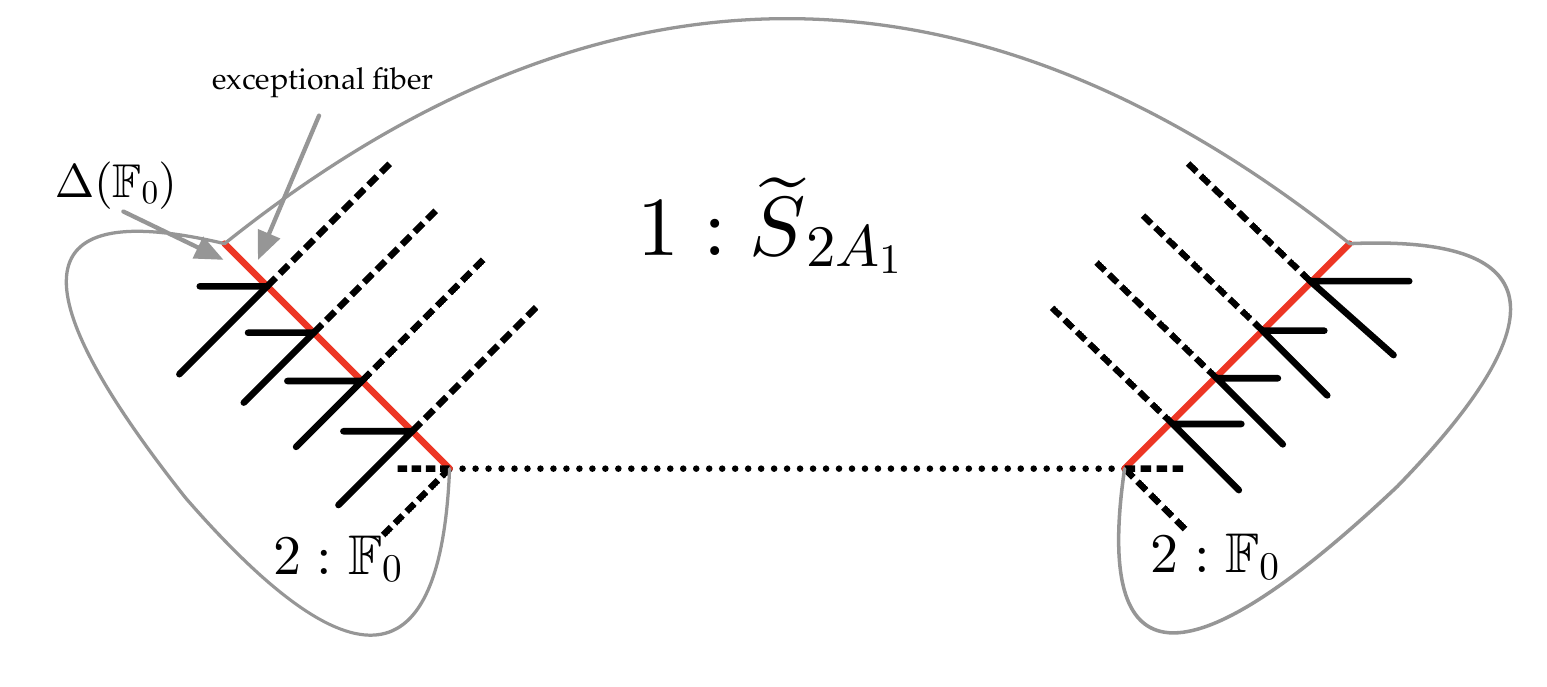}
    \caption{A type $a_2$ surface $(S_1,cB_1)$ for weights $1/6 < c \leq 1/4$, obtained as the stable replacement for
        these weights of the surface $(S_0,cB_0)$ of \cref{fig:a2_19-16}. This is obtained by resolving the two $A_1$
        singularities of $(S_0,cB_0)$, and attaching to each exceptional fiber a type 2 component isomorphic to $\bF_0
        \cong \bP^1 \times \bP^1$, glued along its diagonal. Each line of multiplicity 2 passing through an $A_1$
        singularity splits into two distinct lines on the corresponding $\bF_0$ component, one in each ruling. The line
        of multiplicity 4 splits into two distinct lines of multiplicity 2 on each $\bF_0$ component, one in each
        ruling. Additionally, as in \cref{fig:a2_19-16} each line of multiplicity 2 passing through one $A_1$
        singularity intersects precisely one line of multiplicity 2 passing through the other $A_1$ singularity; these
    intersections are not drawn.}%
    \label{fig:a2_16-14}
\end{figure}

%\subsubsection{Weights $1/4 < c \leq 1/2$}

\begin{figure}[!htpb]
    \centering
    \includegraphics[width=0.6\linewidth]{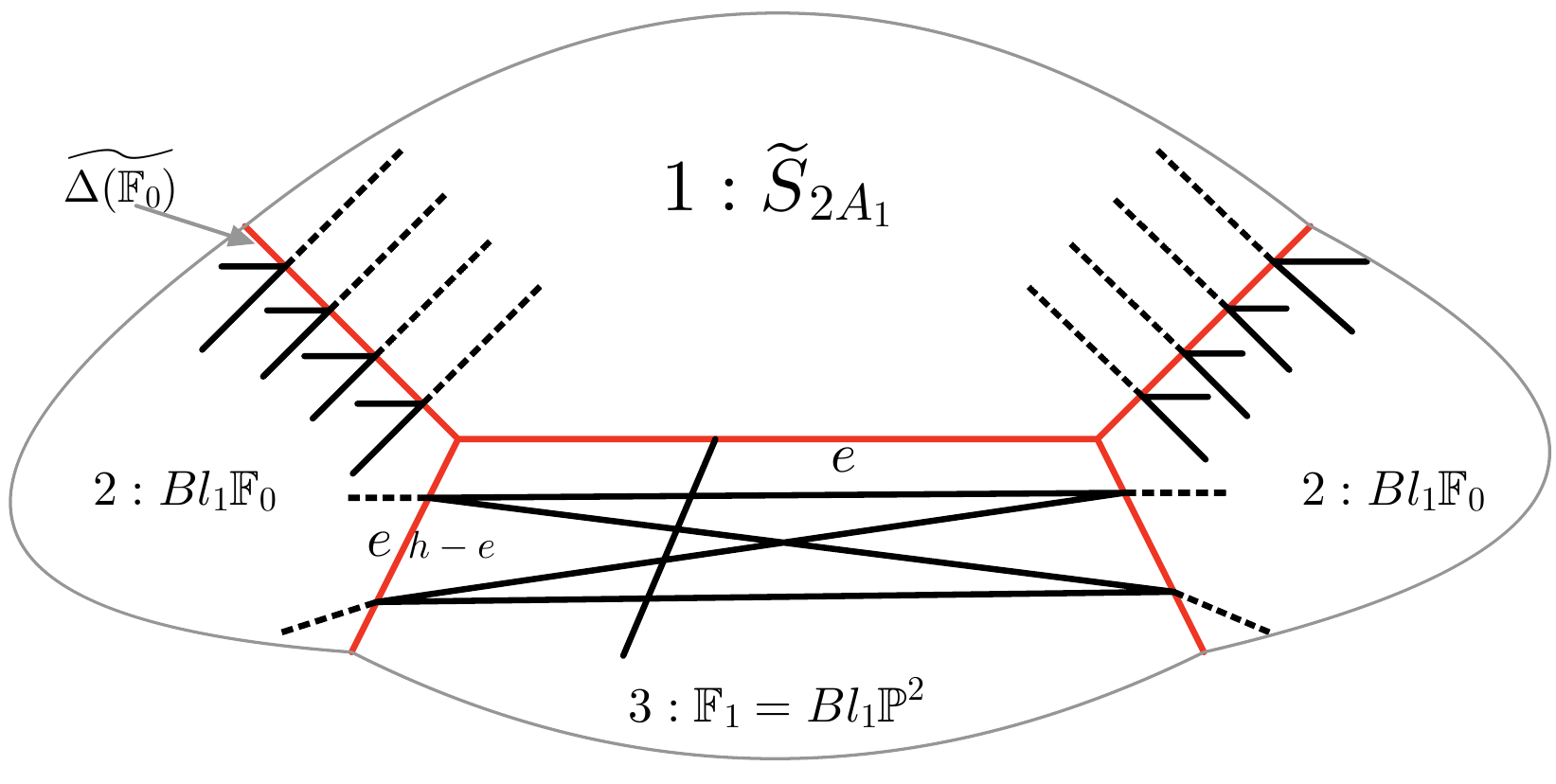}
    \caption{A type $a_2$ surface $(S_2,cB_2)$ for weights $1/4 < c \leq 1/2$, obtained as the stable replacement for
        these weights of the surface $(S_1,cB_1)$ of \cref{fig:a2_16-14}. This is obtained by blowing up the
        multiplicity 4 line in $(S_1,cB_1)$, and attaching to the exceptional divisor a type 3 component isomorphic to
        $\bF_1 \cong Bl_1\bP^2$, as shown. Four of the lines on the type 3 component come from the line of multiplicity
        4, the remaining line on this component comes from a line of multiplicity 1 on the type 1 component.
       Additionally, as in \cref{fig:a2_19-16,fig:a2_16-14} each line of multiplicity 2 passing through one $A_1$ singularity intersects
      precisely one line of multiplicity 2 passing through the other $A_1$ singularity; these intersections are not drawn.}%
    \label{fig:a2_14-12}
\end{figure}

\begin{figure}[!htpb]
    \centering
    \includegraphics[width=0.6\linewidth]{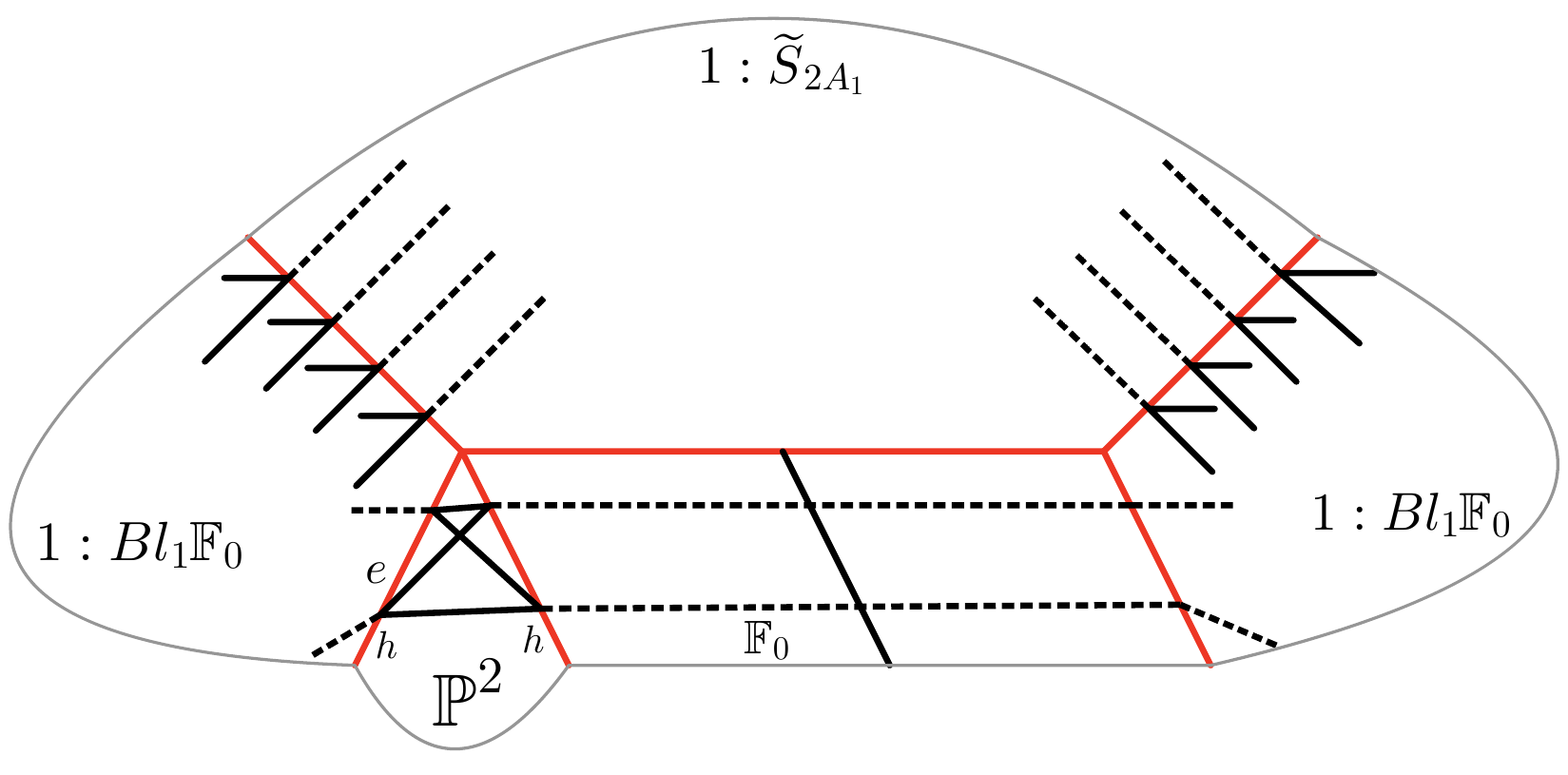}
    \caption{A type $aa_2$ surface for weights $1/4 < c \leq 1/2$, obtained as a degeneration of the surface of
    \cref{fig:a2_14-12}, where the type 3 component $\cong Bl_1\bP^2$ degenerates into two components, one isomorphic to
    $\bP^2$, and the other isomorphic to $\bF_0$, as shown. Here the horizontal lines on $\bF_0$ are rulings in one
    direction, and the slanted lines are rulings in the other direction.}%
    \label{fig:aa2_14-12}
\end{figure}

%\subsubsection{Weights $1/2 < c \leq 2/3$}

\begin{figure}[!htpb]
    \centering
    \begin{subfigure}[t]{0.7\textwidth}
        \centering
        \includegraphics[width=0.9\linewidth]{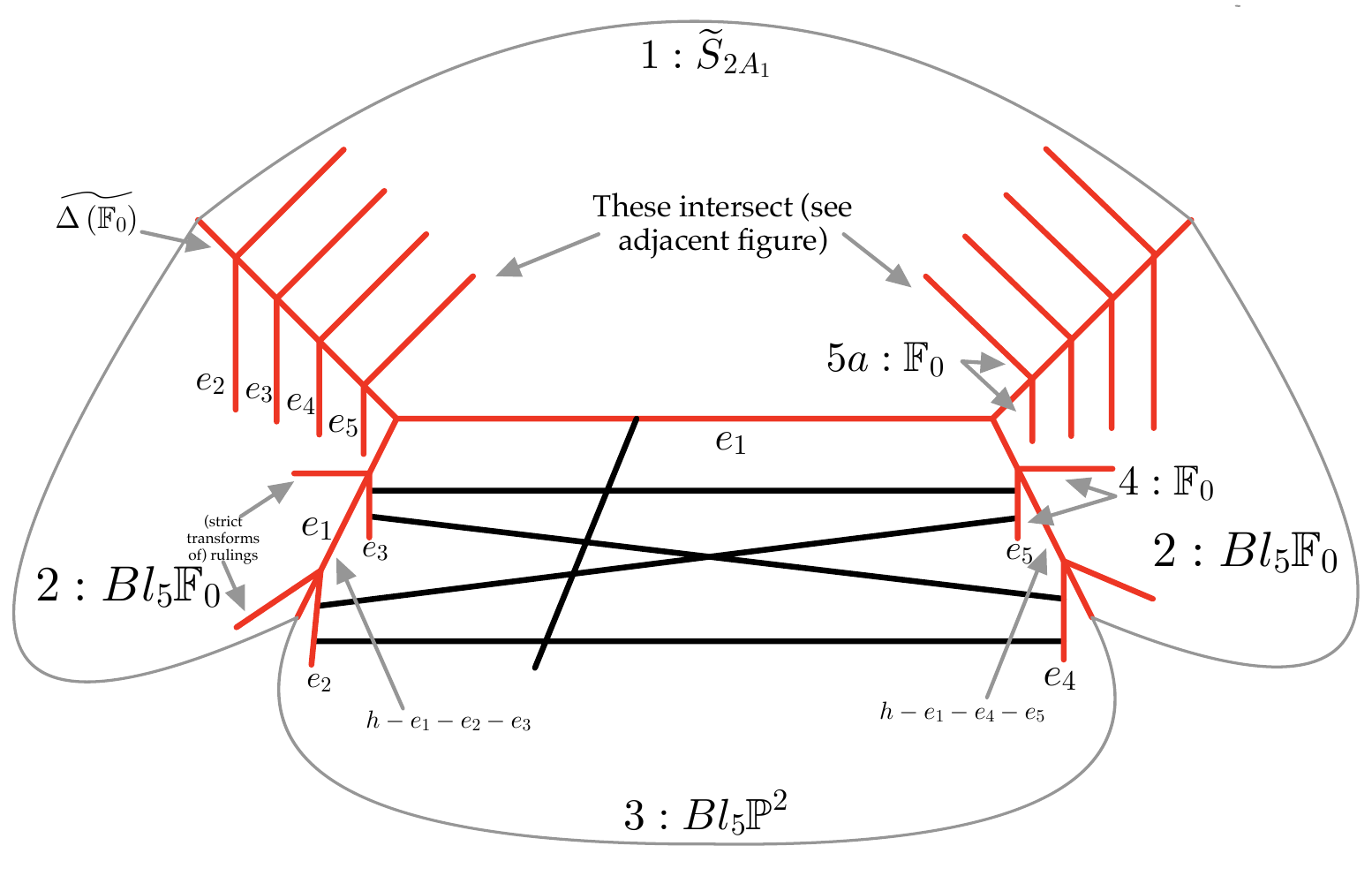}
        \caption{The type $a_2$ surface $(S_3,cB_3)$ for weights $1/2 < c \leq 2/3$.}
        \label{fig:a2_12-23_1}
    \end{subfigure}
    \begin{subfigure}[t]{0.5\textwidth}
        \centering
        \includegraphics[width=0.9\linewidth]{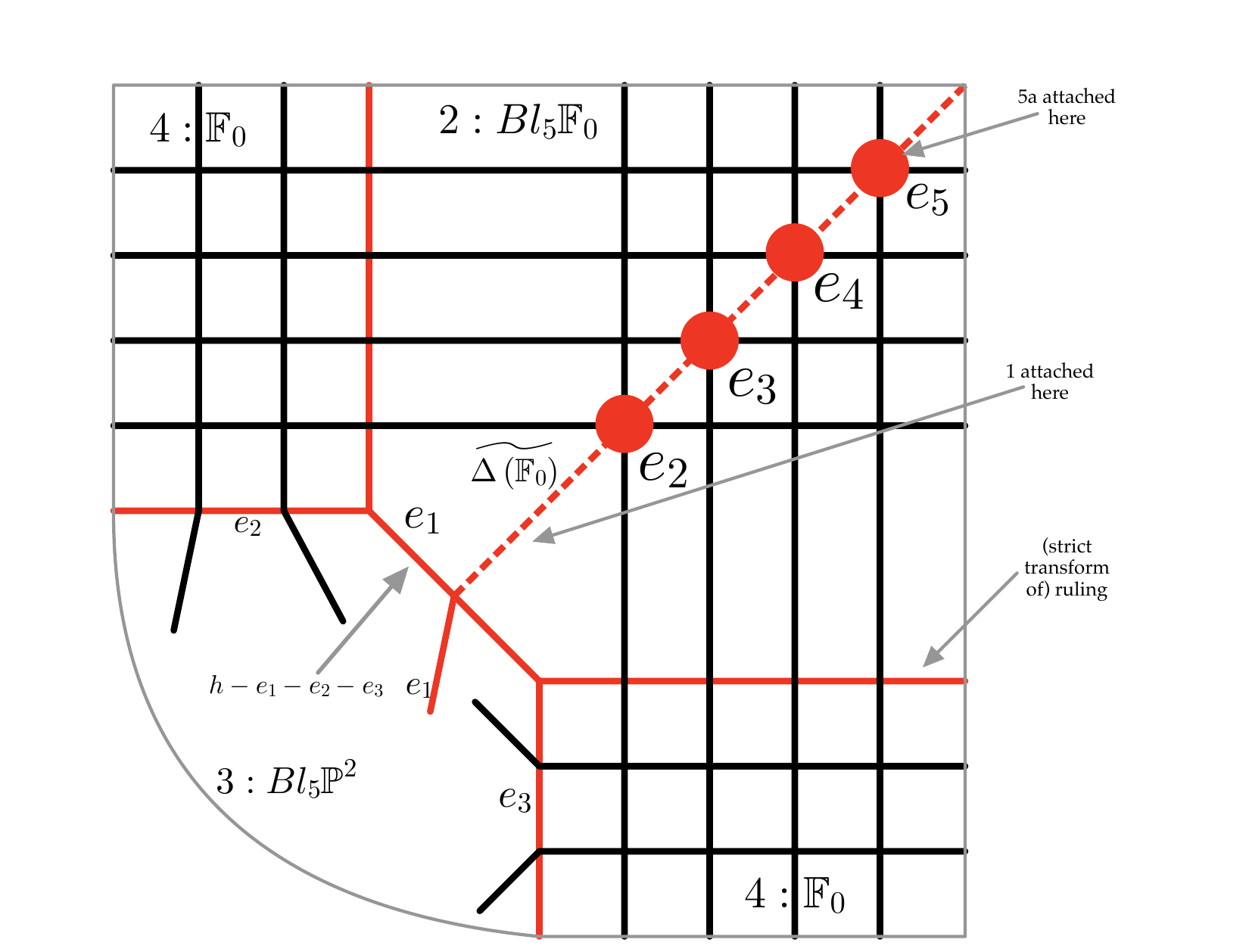}
        \caption{A view of the gluings of some type 4 components to type 2 and 3 components. The red dashed line
        indicates a gluing line, not a line of multiplicity 2. (This is done to avoid confusion about the number of
        irreducible components.)}
        \label{fig:a2_12-23_2}
    \end{subfigure}
    \begin{subfigure}[t]{0.4\textwidth}
        \centering
        \includegraphics[width=0.9\linewidth]{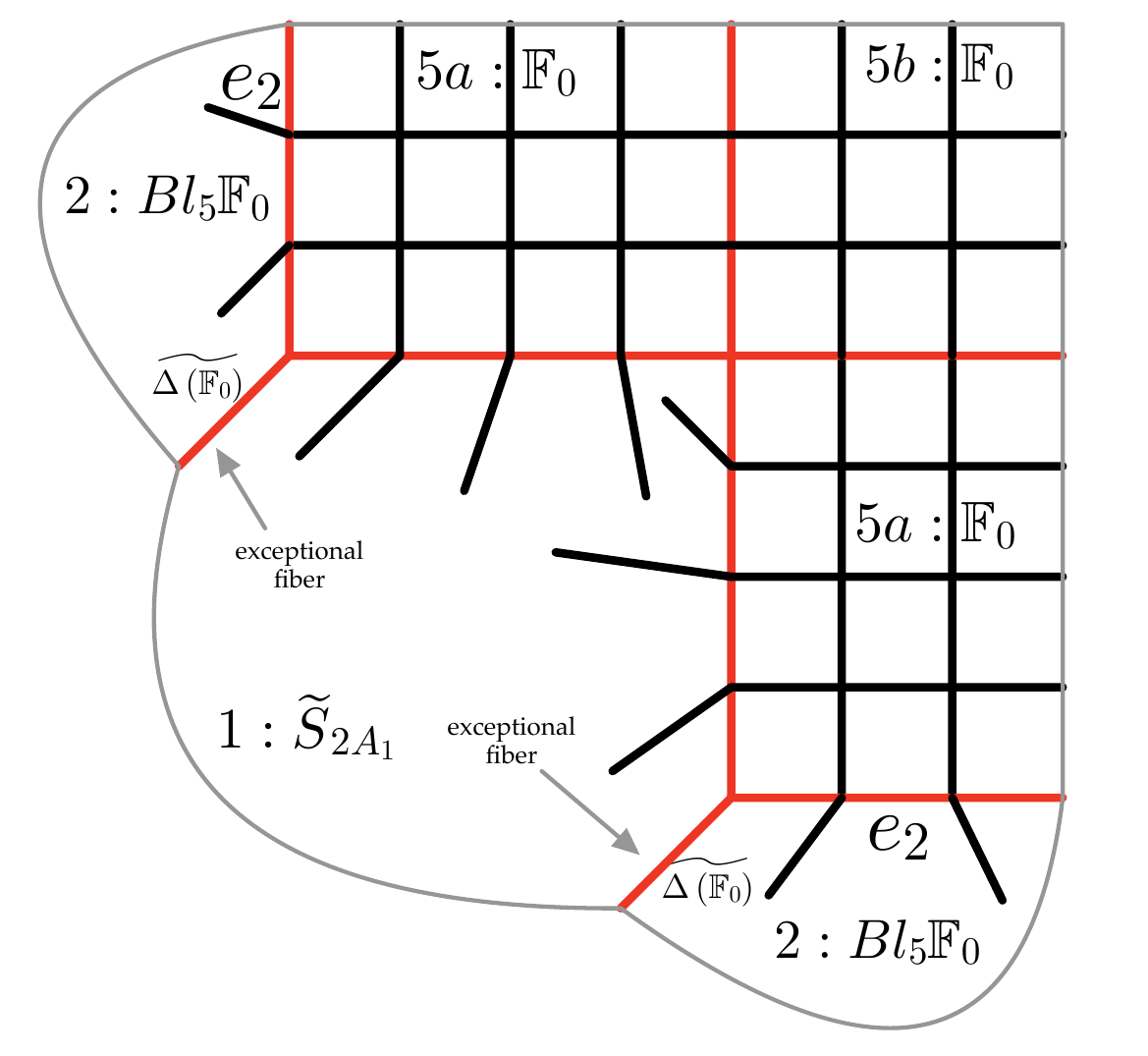}
        \caption{Type 5a components intersect in pairs, and at their intersection point another copy of $\bF_0$, labeled
        by type 5b, is attached.}
        \label{fig:a2_12-23_3}
    \end{subfigure}
    \caption{A type $a_2$ surface $(S_3,cB_3)$ for weights $1/2 < c \leq 2/3$, obtained as the stable replacement for
        these weights of the surface $(S_2,cB_2)$ of \cref{fig:a2_14-12}. This is roughly obtained by blowing up the
        lines of multiplicity 2 on $(S_2,cB_2)$, and attaching to the exceptional divisors copies of $\bF_0$. Some of
        these lines intersect (as mentioned \cref{fig:a2_19-16,fig:a2_16-14,fig:a2_14-12}), making the exact procedure
        to obtain this stable replacement more involved (see \cref{ex:bottom_up_a2}), and leading to additional
    components being attached, as pictured in \cref{fig:a2_12-23_3}.}%
    \label{fig:a2_12-23}
\end{figure}

\begin{figure}[!htpb]
    \centering
    \begin{subfigure}[t]{0.8\textwidth}
        \centering
        \includegraphics[width=0.9\linewidth]{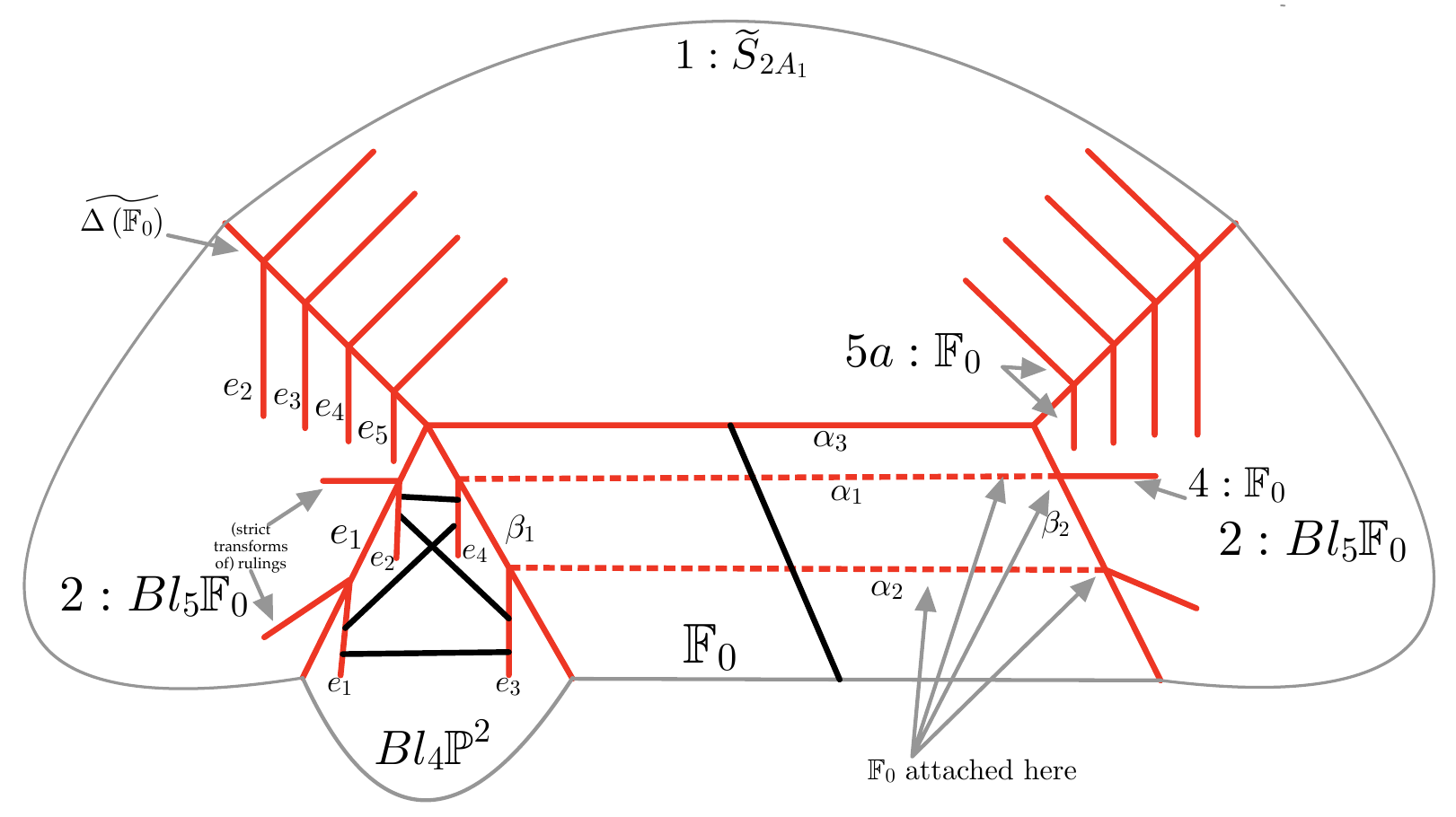}
        \caption{A type $aa_2$ surface for weights $1/2 < c \leq 2/3$. The red dashed lines indicate gluing lines, not
        lines of multiplicity 2. (This is done to avoid confusion about the number of irreducible components). On the
        new $\bF_0$ component, the lines labeled $\alpha_i$ and $\beta_i$ are lines in the two respective rulings on
        $\bF_0$.}
        \label{fig:aa2_12-23_1}
    \end{subfigure}
    \begin{subfigure}[t]{0.8\textwidth}
        \centering
        \includegraphics[width=0.9\linewidth]{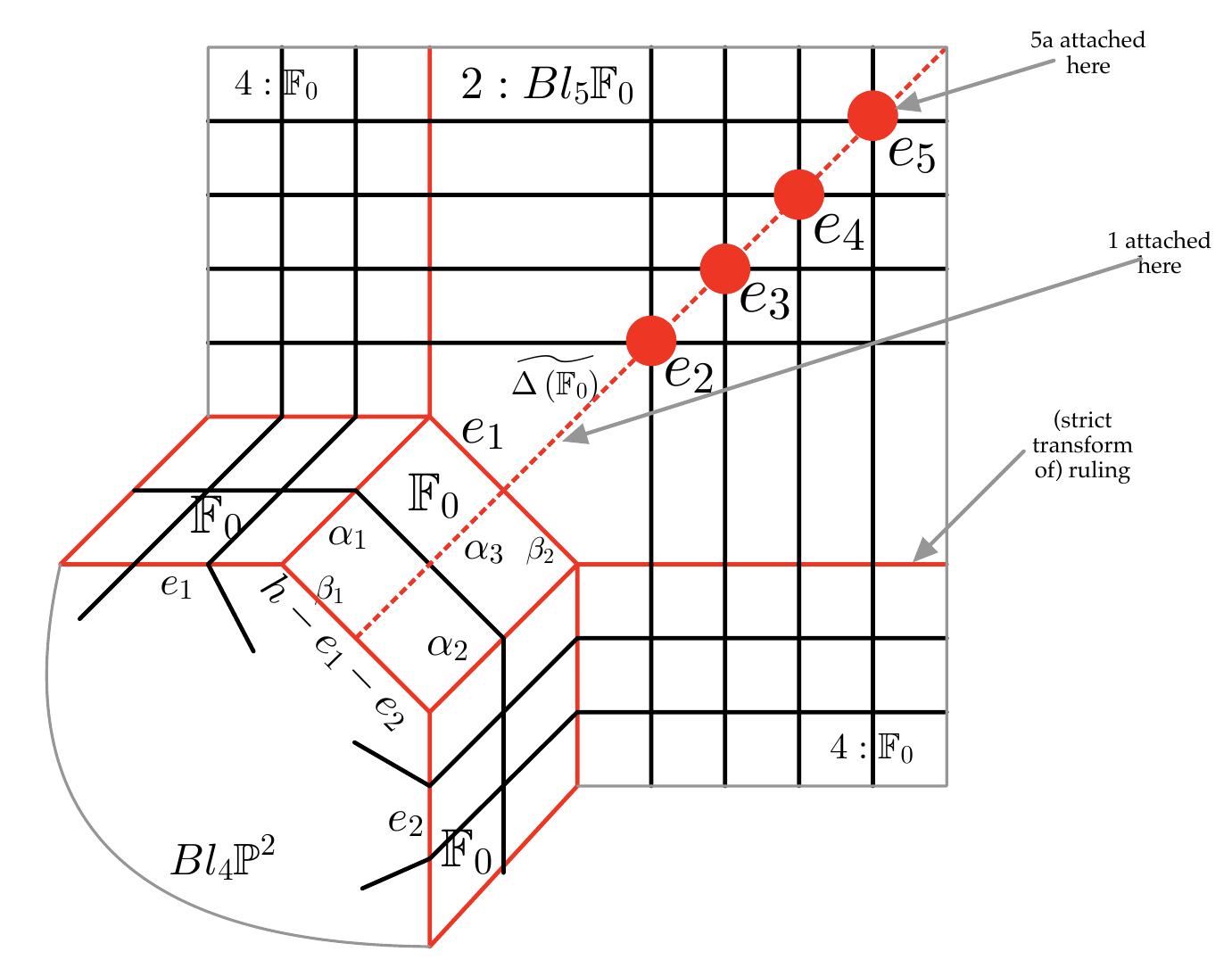}
        \caption{Another view of the four components on the type $aa_2$ surface obtained from the type 3 component on
        the type $a_2$ surface. This view allows one to see the two new $\bF_0$ components glued to the red dashed lines
        in \cref{fig:aa2_12-23_1}.}
        \label{fig:aa2_12-23_2}
    \end{subfigure}
    \caption{A type $aa_2$ surface for weights $1/2 < c \leq 2/3$, obtained as a degeneration of the surface of
    \cref{fig:a2_12-23}, and as the stable replacement for these weights of the surface of \cref{fig:aa2_14-12}. This
    type $aa_2$ surface is isomorphic to the surface $(S_3,cB_3)$ of \cref{fig:a2_12-23}, except for the type 3 component
    $\cong Bl_5\bP^2$, which degenerates into 4 irreducible components as shown. Red dashed lines indicate gluing lines
    (not lines of multiplicity 2).}%
    \label{fig:aa2_12-23}
\end{figure}

\begin{table}[htpb]
    \centering
    \caption{The types of irreducible components of the weight $1/2 < c \leq 2/3$ surface of type $a_2$ pictured in
        \cref{fig:a2_12-23}, and the possible numbers of Eckardt points on each component.  For the component of type 1,
        $\wS_{2A_1}$ refers to the minimal resolution of a cubic surface with two $A_1$ singularities. For the
        components of type 2, $Bl_5\bF_0$ refers to the blowup of $\bF_0 \cong \bP^1 \times \bP^1$ at 5 points on the
    diagonal. For the component of type 3, $Bl_5\bP^2$ refers to the special blowup of $\bP^2$ at 5 points as in
    \cref{lem:nonflat_A32}}
    \label{tab:a2_23-1}
    \begin{tabular}{| c | c | c | c |}
        \hline
        Label & Surface & \# & Eckardt points \\
        \hline
        \hline
        1 & $\wS_{2A_1}$ & 1 & 0 or 1 \\
        2 & $Bl_5\bF_0$ & 2 & 0 \\
        3 & $Bl_5\bP^2$ & 1 & 0 or 1 \\
        4 & $\bF_0$ & 4 & 0 \\
        5a & $\bF_0$ & 8 & 0 \\
        5b & $\bF_0$ & 4 & 0 \\
        \hline
    \end{tabular}
\end{table}

\subsection{Type $a_3$}

\begin{proposition} \label{prop:a3}
    For type $a_3$ weighted stable marked cubic surfaces $(S,cB)$, there are four walls.
    \begin{enumerate}
        \item For weights $1/9 < c \leq 1/6$, the type $a_3$ surfaces are described in \cref{fig:a3_19-16}.
        \item For weights $1/6 < c \leq 1/4$, the type $a_3$ surfaces are described in \cref{fig:a3_16-14}.
        \item For weights $1/4 < c \leq 1/2$, the type $a_3$ surfaces are described in \cref{fig:a3_14-12}.
            Additionally, crossing the wall $c=1/4$ introduces type $aa_3$, $a_2a_3$, and $aa_2a_3$ surfaces as degenerations of type $a_3$
            surfaces, described in \cref{fig:a3_14-12_degens}.
        \item For weights $1/2 < c \leq 2/3$, the type $a_3$ surfaces are described in \cref{fig:a3_12-23}. The type
            $aa_3$, $a_2a_3$, and $aa_2a_3$ surfaces are described in \cref{fig:a3_12-23_degens}.
        \item For weights $2/3 < c \leq 1$, the type $a_3$ surfaces are obtained from the weight $1/2 < c \leq 2/3$ type
            $a_3$ surfaces by resolving Eckardt points as described in \cref{prop:resolve_eckardt}. The possible
            configurations of Eckardt points are summarized in \cref{tab:a3_23-1}. The type $aa_3$, $a_2a_3$, and
            $aa_2a_3$ surfaces are obtained similarly.
    \end{enumerate}
\end{proposition}

%\subsubsection{Weights $1/9 < c \leq 1/6$}

\begin{figure}[!htpb]
    \centering
    \includegraphics[width=0.3\linewidth]{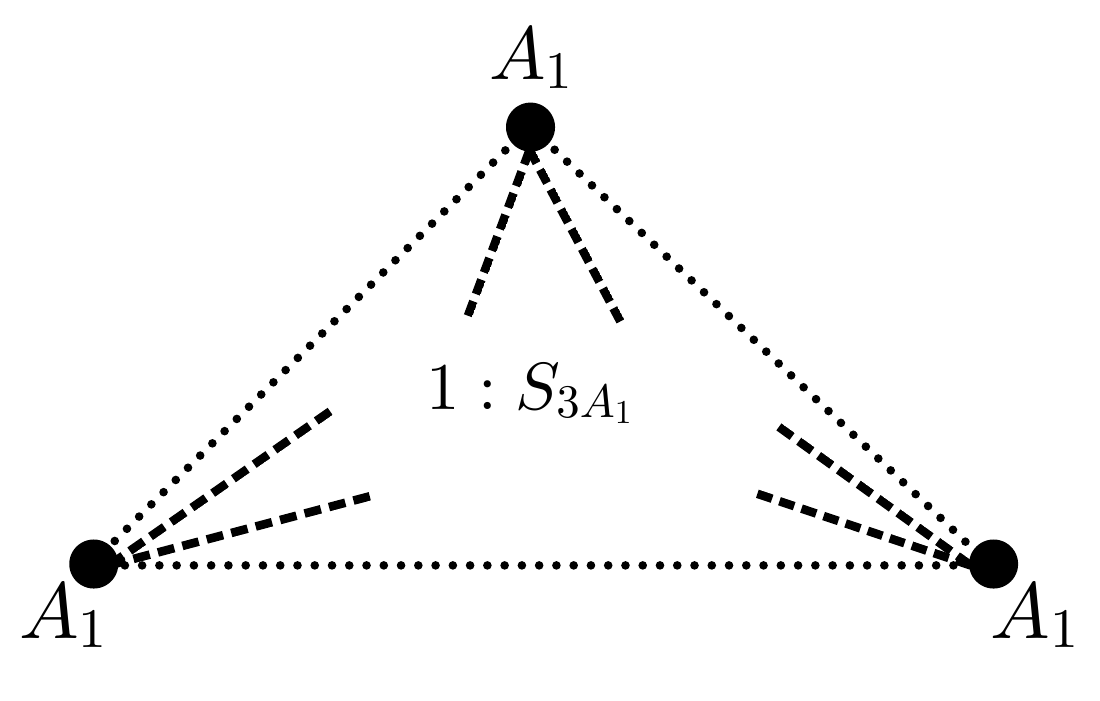}
    \caption{A type $a_3$ surface $(S_0,cB_0)$ for weights $1/9 < c \leq 1/6$. This is a cubic surface with three
    $A_1$ singularities. Recall that dashed lines have multiplicity 2 and dotted lines have multiplicity 4. In addition
    to the lines shown, there are 3 lines of multiplicity 1, not shown.}%
    \label{fig:a3_19-16}
\end{figure}

%\subsubsection{Weights $1/6 < c \leq 1/4$}

\begin{figure}[!htpb]
    \centering
    \includegraphics[width=0.5\linewidth]{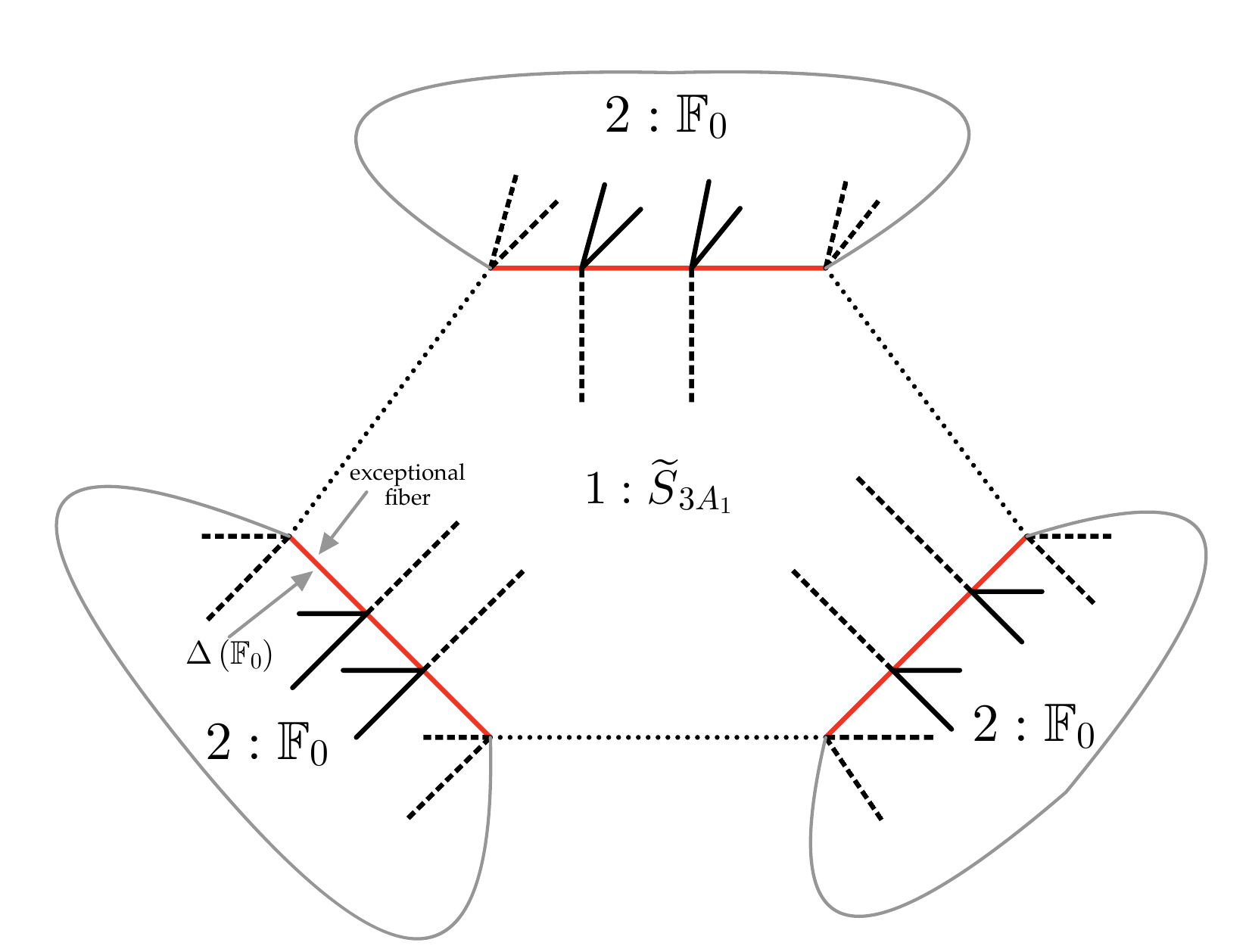}
    \caption{A type $a_3$ surface $(S_1,cB_1)$ for weights $1/6 < c \leq 1/4$, obtained as the stable replacement for
        these weights of the surface $(S_0,cB_0)$ of \cref{fig:a3_19-16}. This is obtained by resolving the three $A_1$
        singularities of $(S_0,cB_0)$, and attaching to each exceptional fiber a type 2 component isomorphic to $\bF_0
        \cong \bP^1 \times \bP^1$, glued along its diagonal. Each line of multiplicity 2 passing through an $A_1$
        singularity splits into two distinct lines on the corresponding $\bF_0$ component, one in each ruling. The line of
        multiplicity 4 splits into two distinct lines of multiplicity 2 on each corresponding $\bF_0$ component, one in each
        ruling.}%
    \label{fig:a3_16-14}
\end{figure}

%\subsubsection{Weights $1/4 < c \leq 1/2$}

\begin{figure}[!htpb]
    \centering
    \begin{subfigure}[t]{0.55\textwidth}
        \centering
        \includegraphics[width=0.9\linewidth]{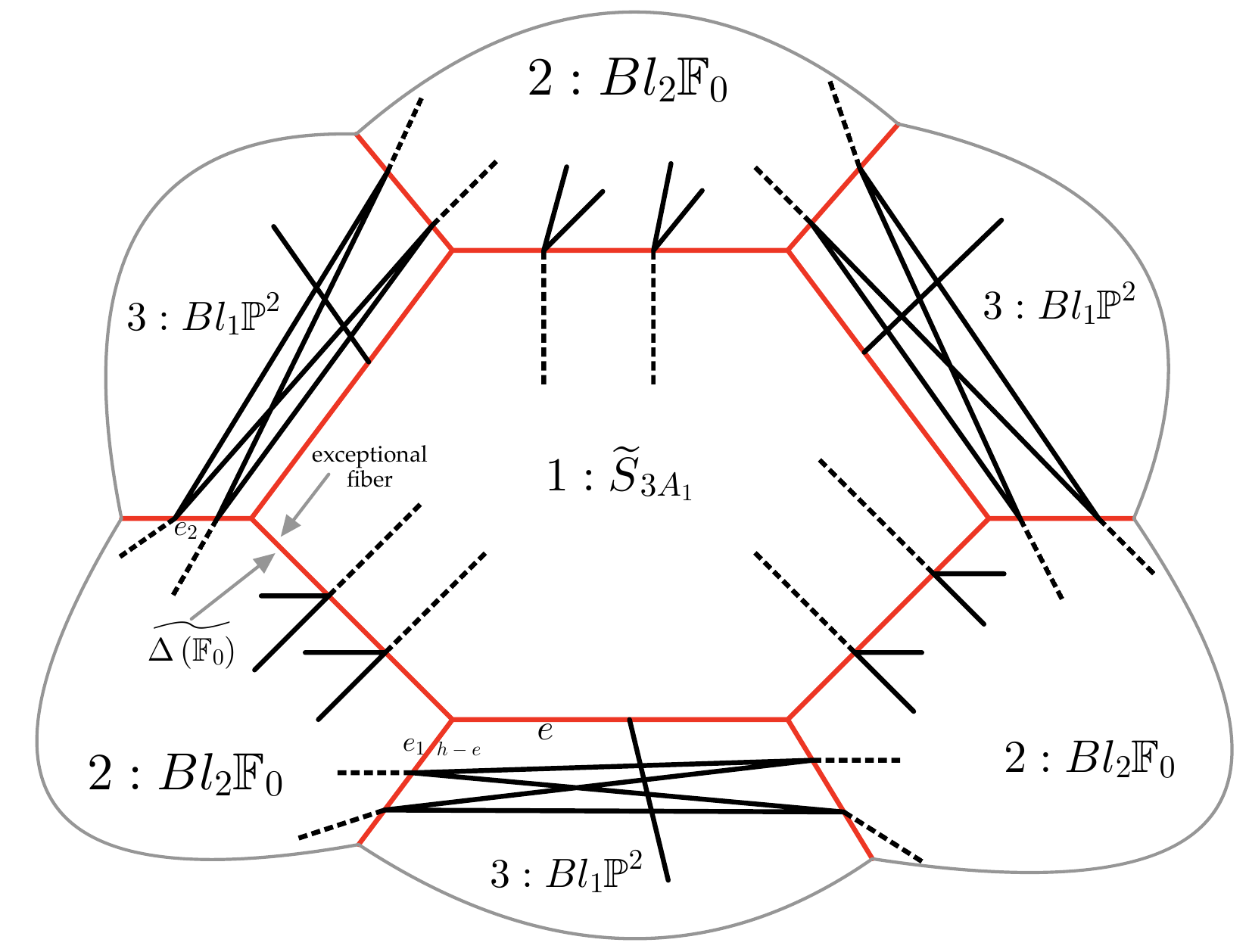}
        \caption{The type $a_3$ surface $(S_2,cB_2)$ for weights $1/4 < c \leq 1/2$.}
        \label{fig:a3_14-12_1}
    \end{subfigure}
    \begin{subfigure}[t]{0.4\textwidth}
        \centering
        \includegraphics[width=0.9\linewidth]{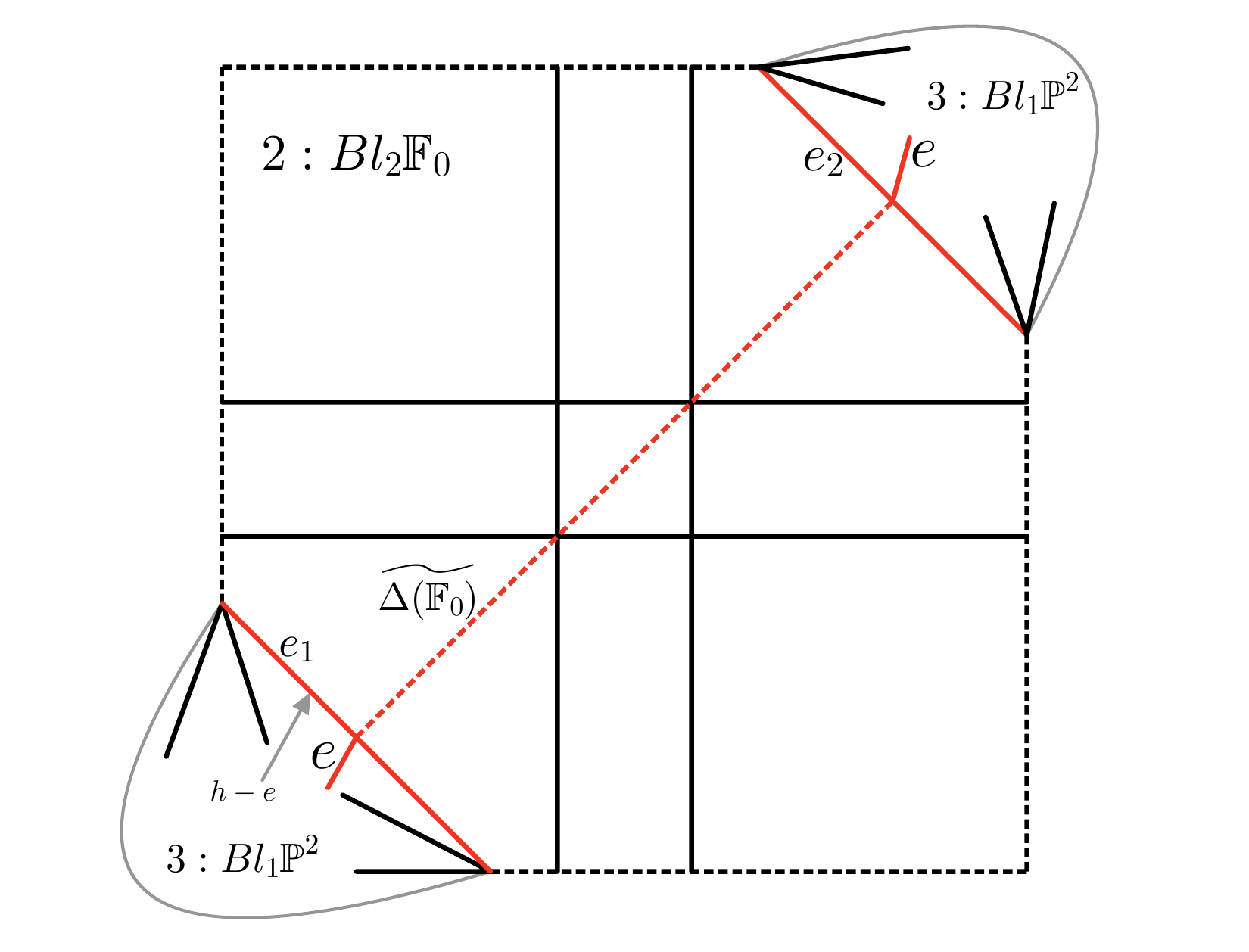}
        \caption{Another view of a type 2 component isomorphic to $Bl_2\bF_0$, showing more clearly the lines appearing
        on this component. The red dashed line indicates a gluing line, not a line of multiplicity 2.}
        \label{fig:a3_14-12_2}
    \end{subfigure}
    \caption{A type $a_3$ surface $(S_2,cB_2)$ for weights $1/4 < c \leq 1/2$, obtained as the stable replacement for
        these weights of the surface $(S_1,cB_1)$ of \cref{fig:a3_16-14}. This is obtained by blowing up the lines of
    multiplicity 4 on $(S_1,cB_1)$, and attaching to the exceptional divisors type 3 components isomorphic to $\bF_1
    \cong Bl_1\bP^2$ (cf. \cref{fig:a2_14-12}). As in \cref{fig:a2_14-12}, the five lines on a given type 3 component
    come from the blown up line of multiplicity 4, plus a line of multiplicity one intersecting the line of multiplicity
    4. Note there are three lines of multiplicity one on the surface $\wS_{3A_1}$, each such line intersects precisely one
    of the lines of multiplicity 4.}%
    \label{fig:a3_14-12}
\end{figure}

\begin{figure}[!htpb]
    \centering
    \begin{subfigure}[t]{0.6\textwidth}
        \centering
        \includegraphics[width=0.9\linewidth]{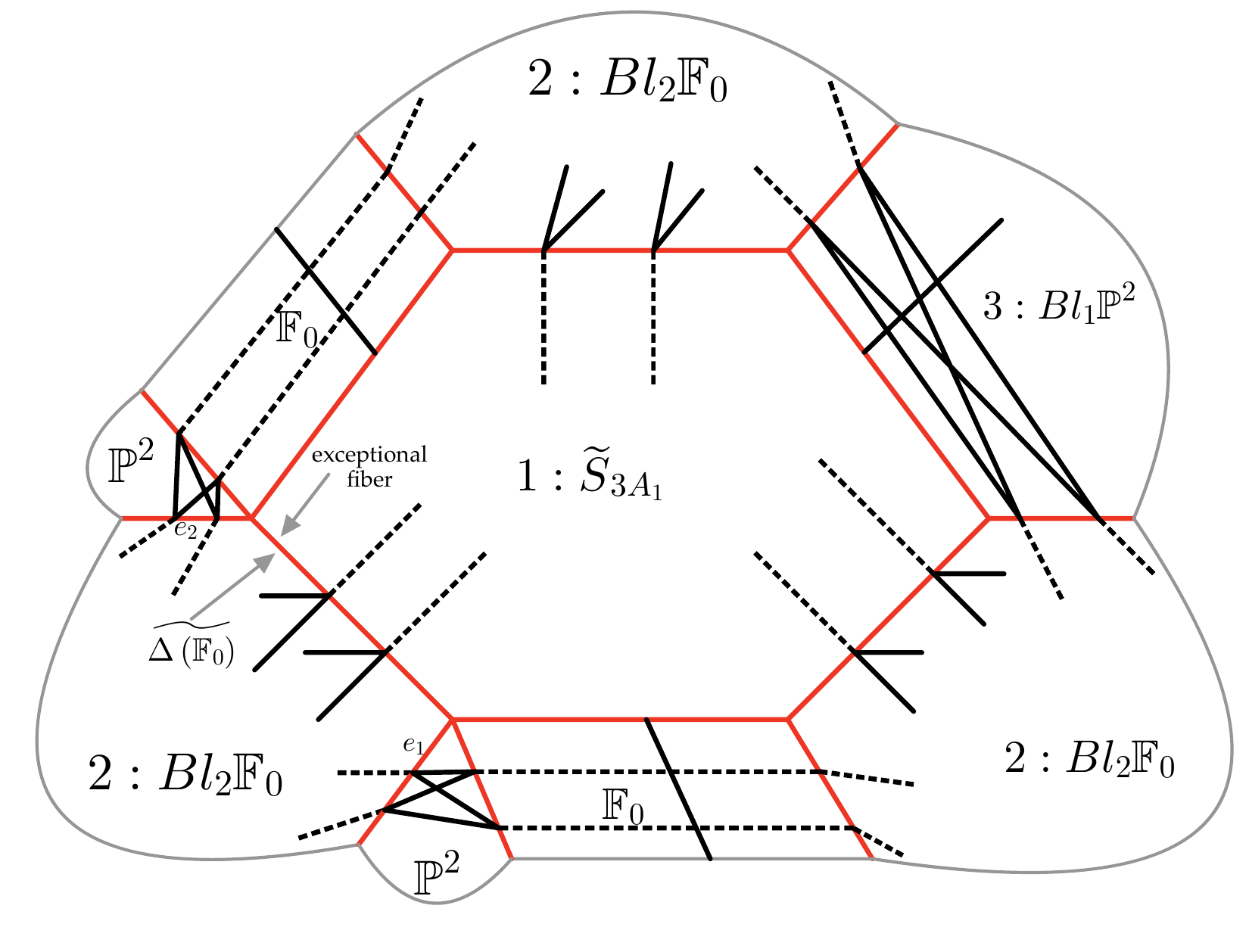}
        \caption{A type $aa_3$ surface for weights $1/4 < c \leq 1/2$.}
        \label{fig:aa3_14-12}
    \end{subfigure}
    \begin{subfigure}[t]{0.6\textwidth}
        \centering
        \includegraphics[width=0.9\linewidth]{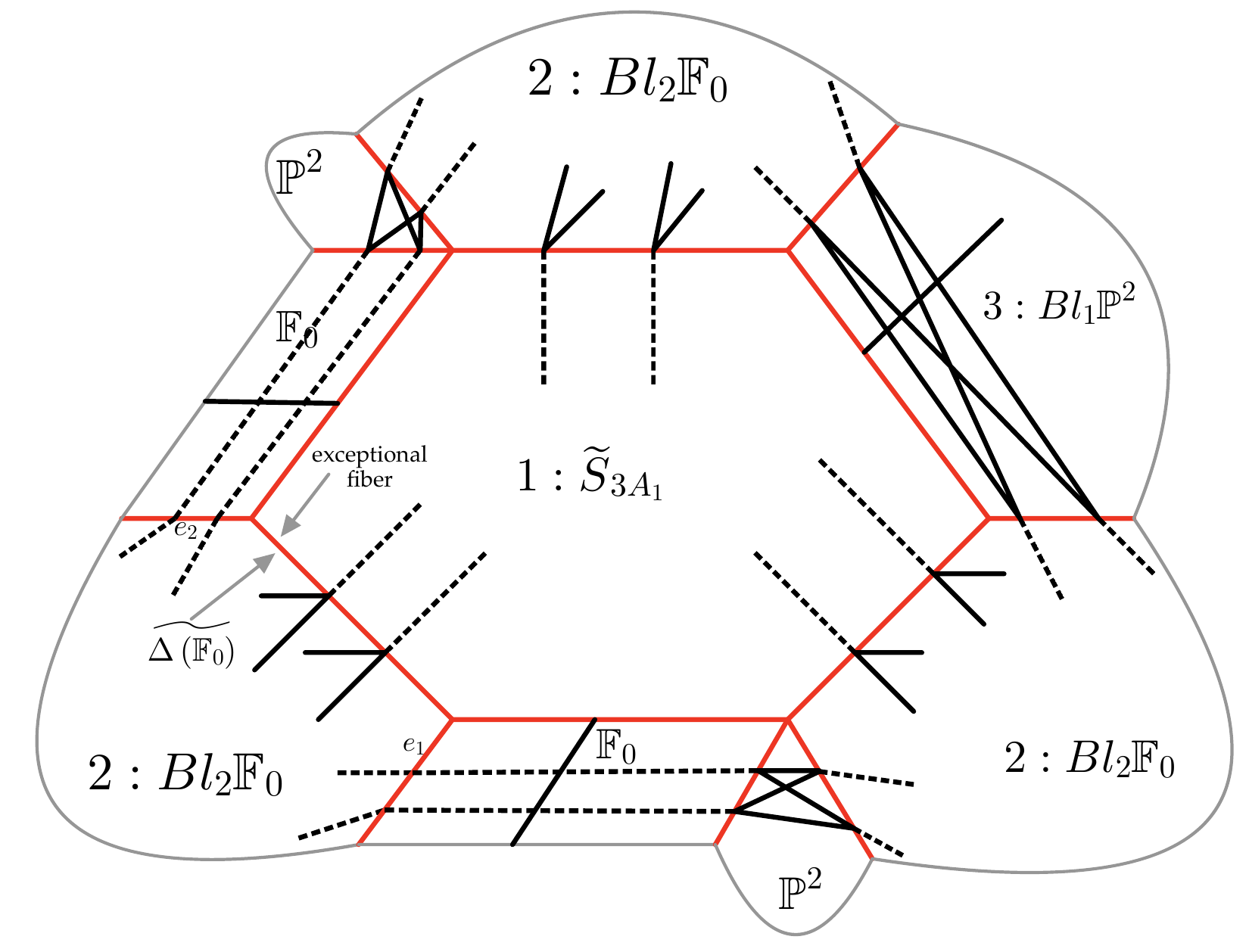}
        \caption{A type $a_2a_3$ surface for weights $1/4 < c \leq 1/2$.}
        \label{fig:a2a3_14-12}
    \end{subfigure}
    \begin{subfigure}[t]{0.6\textwidth}
        \centering
        \includegraphics[width=0.9\linewidth]{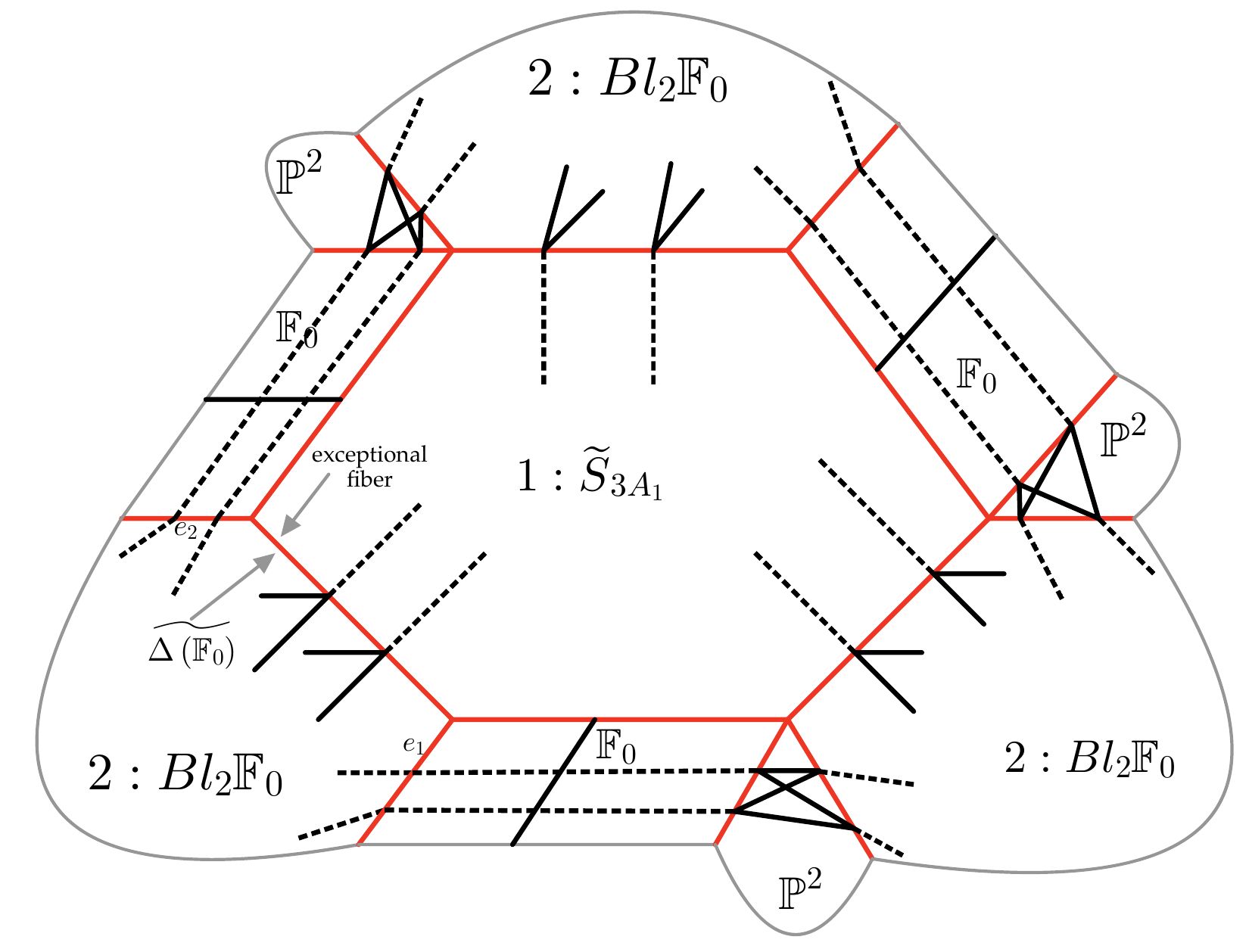}
        \caption{A type $aa_2a_3$ surface for weights $1/4 < c \leq 1/2$.}
        \label{fig:aa2a3_14-12}
    \end{subfigure}
    \caption{Degenerations of the type $a_3$ surface of \cref{fig:a3_14-12} into types $aa_3$, $a_2a_3$, and $aa_2a_3$
    surfaces, for weights $1/4 < c \leq 1/2$. Each degeneration is isomorphic to the original $a_3$ surface of
    \cref{fig:a3_14-12} away from the type 3 components shown, which degenerate into $\bP^2$ and $\bF_0$ components as
    pictured.}%
    \label{fig:a3_14-12_degens}
\end{figure}

%\subsubsection{Weights $1/2 < c \leq 2/3$}

\begin{figure}[!htpb]
    \centering
    \begin{subfigure}[t]{0.75\textwidth}
        \centering
        \includegraphics[width=0.9\linewidth]{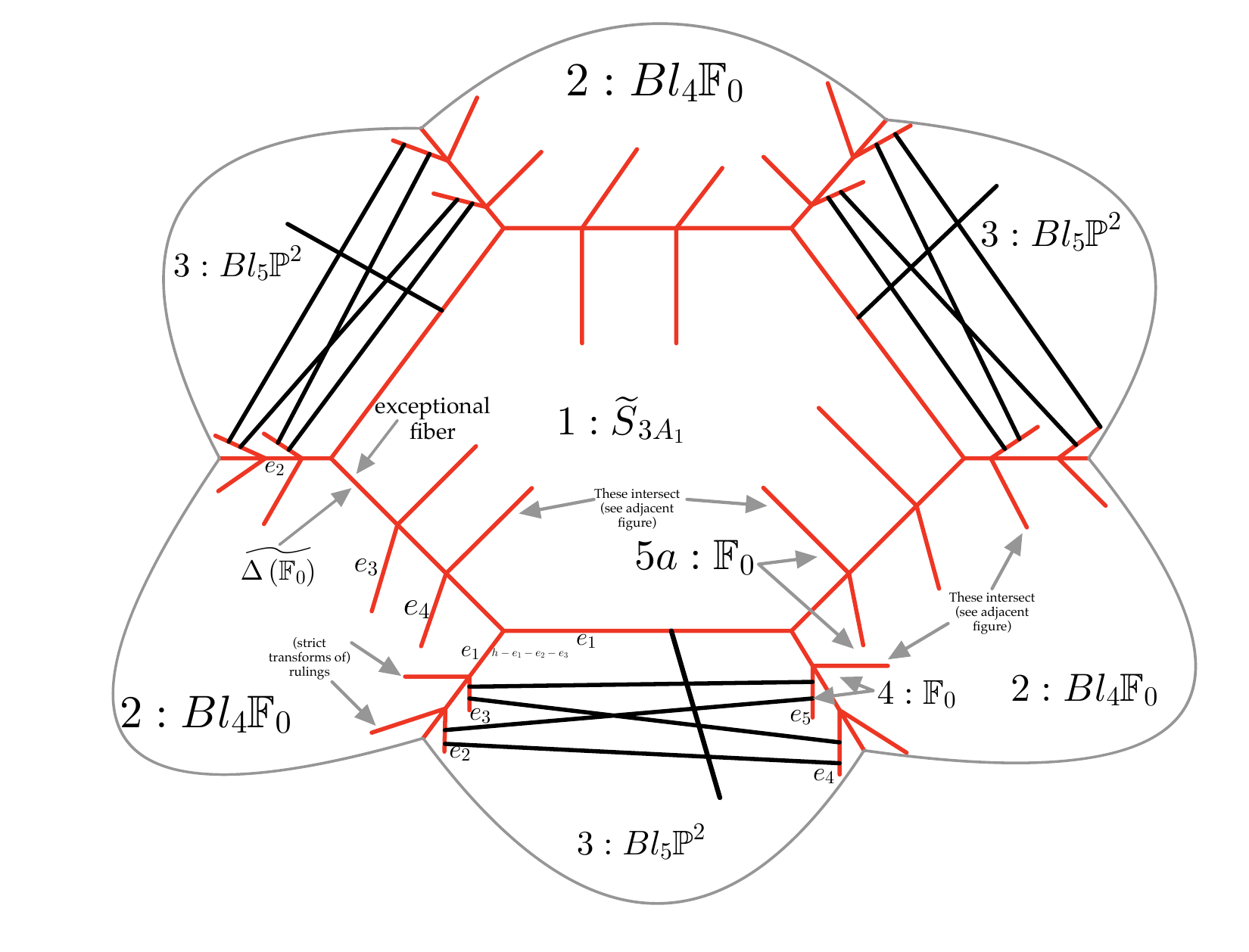}
        \caption{The type $a_3$ surface $(S_3,cB_3)$ for weights $1/2 < c \leq 2/3$.}
        \label{fig:a3_12-23_1}
    \end{subfigure}
    \begin{subfigure}[t]{0.45\textwidth}
        \centering
        \includegraphics[width=0.9\linewidth]{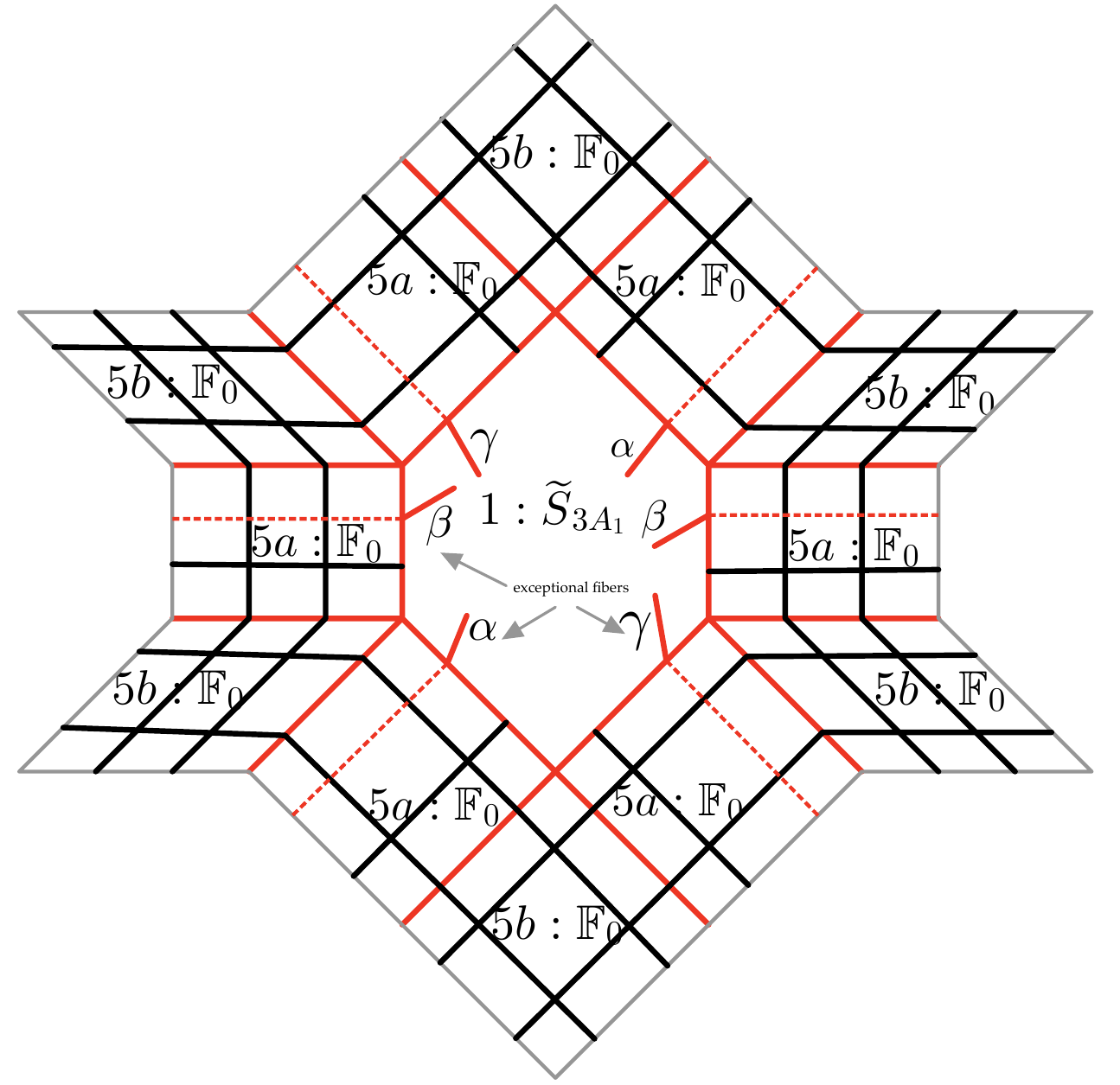}
        \caption{A view of the type 5a components and their intersections, at which type 5b components are attached.
        Here $\alpha$, $\beta$, and $\gamma$ indicate the exceptional fibers on the type 1 component where the type 2
        components are attached (matching pairs identify the same exceptional fibers). The red dashed lines indicate the
        gluings of the type 5a components to these same type 2 components.}
        \label{fig:a3_12-23_2}
    \end{subfigure}
    \begin{subfigure}[t]{0.5\textwidth}
        \centering
        \includegraphics[width=0.9\linewidth]{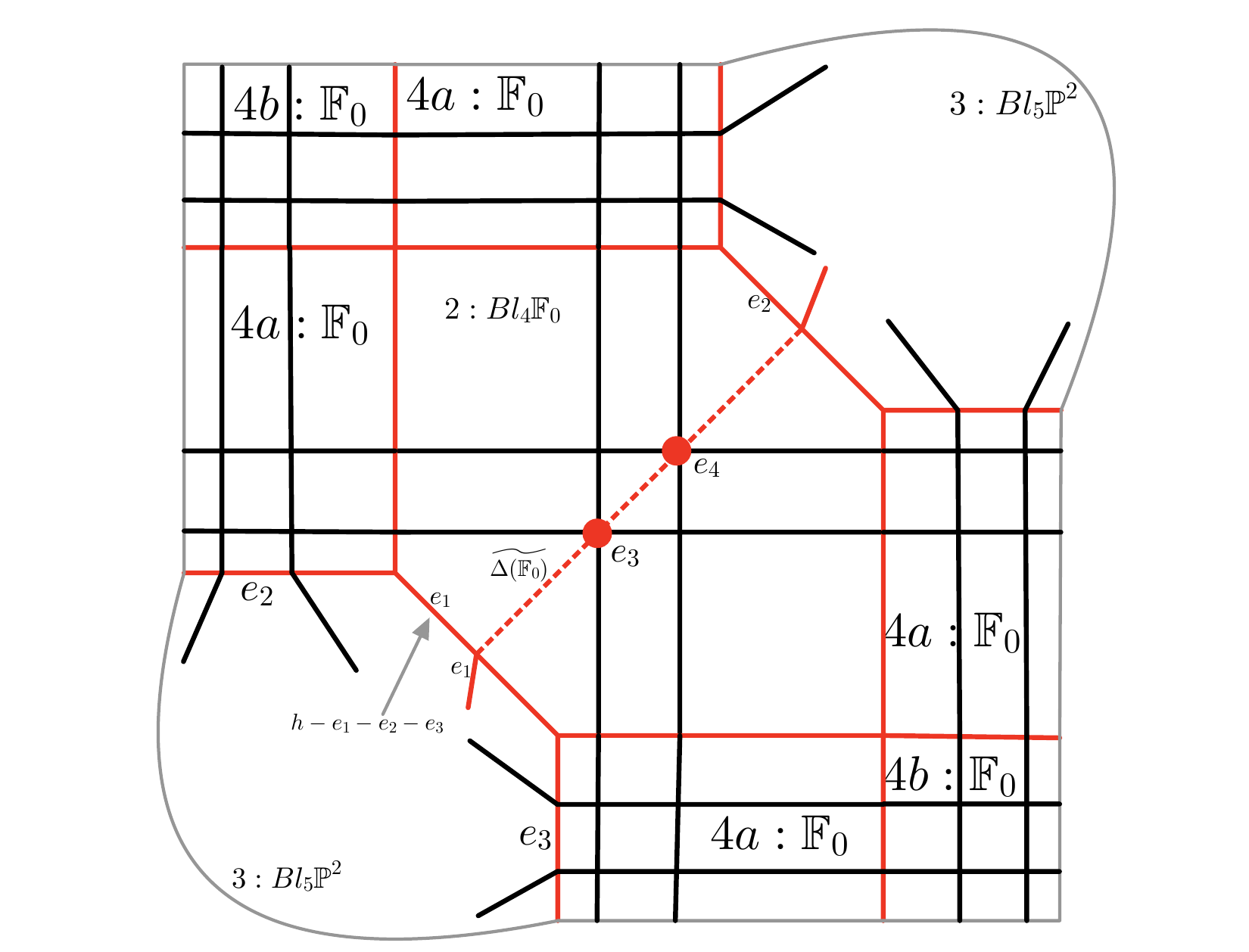}
        \caption{A view of some type 4a and 4b components and their intersections and gluings to type 2 and 3 components
        (cf. \cref{fig:a3_14-12_2}). Here the red dashed line again indicates a gluing line, not a line of multiplicity
        2.}
        \label{fig:a3_12-23_3}
    \end{subfigure}
    \caption{A type $a_3$ surface $(S_3,cB_3)$ for weights $1/2 < c \leq 2/3$, obtained as the stable replacement for
    these weights of the surface $(S_2,cB_2)$ of \cref{fig:a3_14-12}. This is obtained by resolving all lines of
    multiplicity 2 on $(S_2,cB_2)$ and their intersections, analogous to the $a_2$ case in \cref{fig:a2_12-23}. Some of the
    lines of multiplicity 2 intersect, leading to additional components being attached, as in \cref{fig:a3_12-23_2}.}%
    \label{fig:a3_12-23}
\end{figure}

\begin{table}[htpb]
    \centering
    \caption{The types of irreducible components of the weight $1/2 < c \leq 2/3$ surface of type $a_3$ pictured in
    \cref{fig:a3_12-23}, and the possible numbers of Eckardt points on each component.
    For the component of type 1, $\wS_{3A_1}$ refers to the minimal resolution of a cubic surface with three $A_1$
    singularities. For the components of type 2, $Bl_4\bF_0$ refers to the blowup of $\bF_0 \cong \bP^1 \times \bP^1$ at 4 points on
    the diagonal. For the components of type 3, $Bl_5\bP^2$ refers to the special blowup of $\bP^2$ at 5 points as in
\cref{lem:nonflat_A32}}
    \label{tab:a3_23-1}
    \begin{tabular}{| c | c | c | c |}
        \hline
        Label & Surface & \# & Eckardt points \\
        \hline
        \hline
        1 & $\wS_{3A_1}$ & 1 & 0 or 1 \\
        2 & $Bl_4\bF_0$ & 3 & 0 \\
        3 & $Bl_5\bP^2$ & 3 & 0 or 1 \\
        4a & $\bF_0$ & 12 & 0 \\
        4b & $\bF_0$ & 12 & 0 \\
        5a & $\bF_0$ & 6 & 0 \\
        5b & $\bF_0$ & 6 & 0 \\
        \hline
    \end{tabular}
\end{table}

\begin{figure}[!htpb]
    \centering
    \begin{subfigure}[t]{0.6\textwidth}
        \centering
        \includegraphics[width=0.9\linewidth]{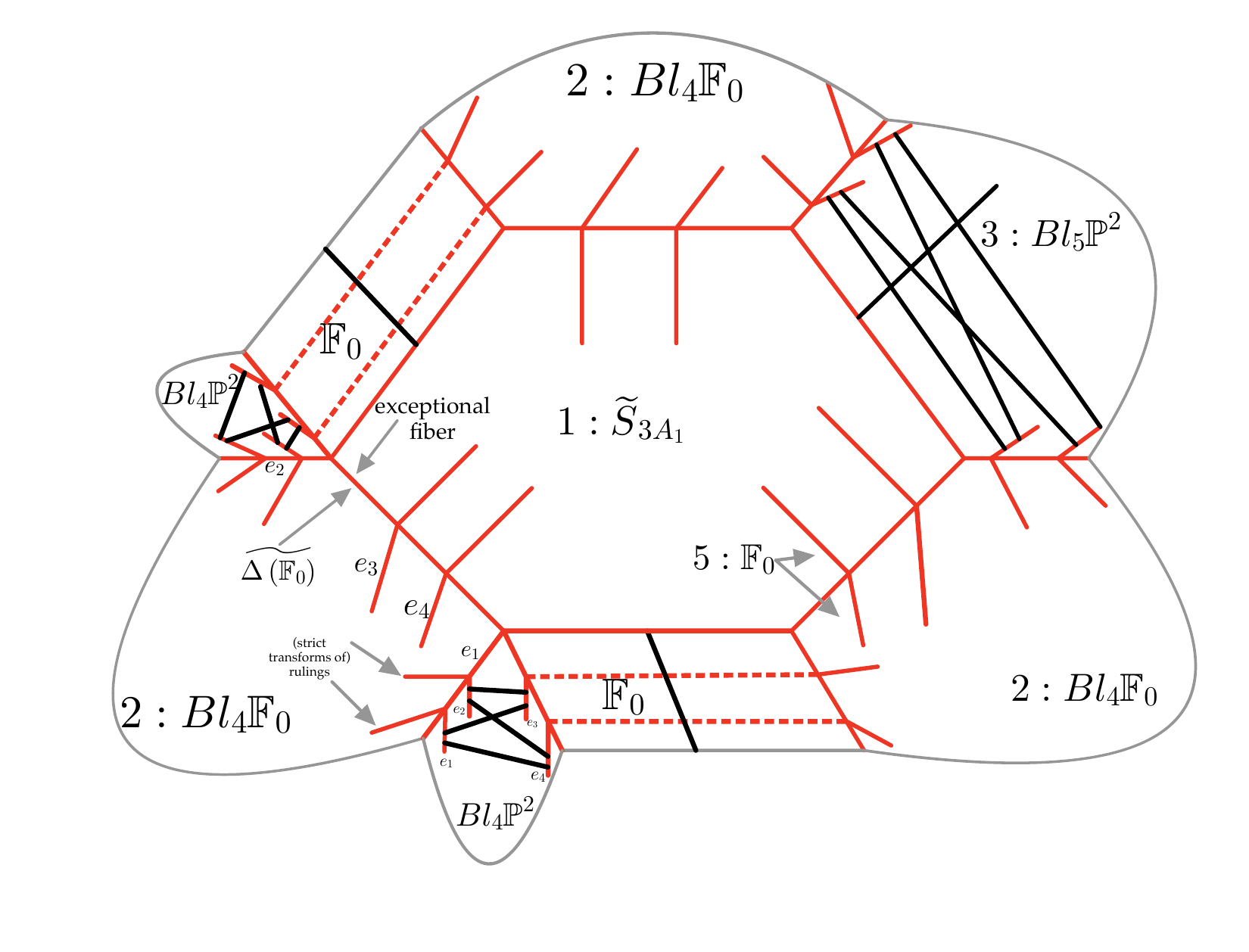}
        \caption{A type $aa_3$ surface for weights $1/2 < c \leq 2/3$.}
        \label{fig:aa3_12-23}
    \end{subfigure}
    \begin{subfigure}[t]{0.6\textwidth}
        \centering
        \includegraphics[width=0.9\linewidth]{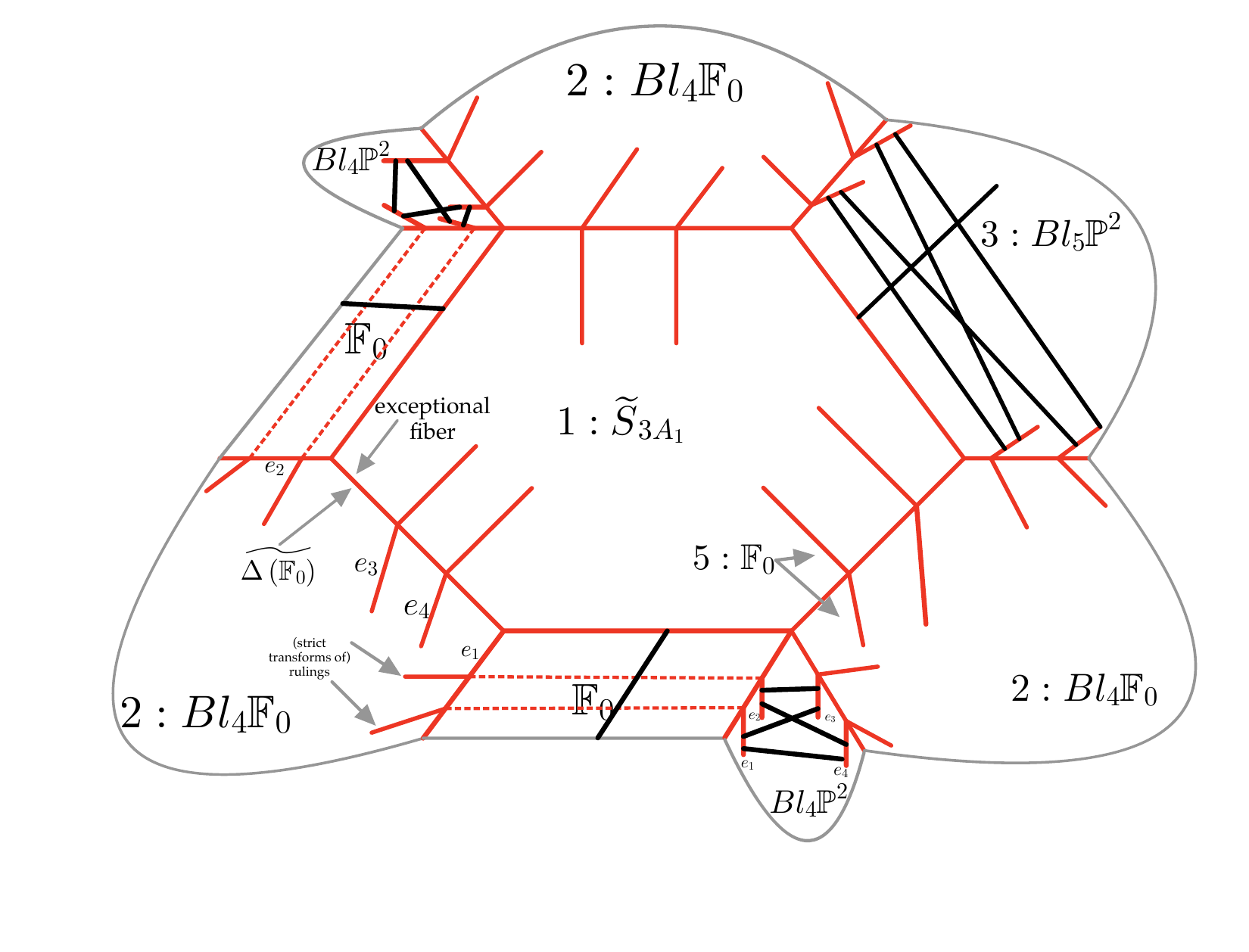}
        \caption{A type $a_2a_3$ surface for weights $1/2 < c \leq 2/3$.}
        \label{fig:a2a3_12-23}
    \end{subfigure}
    \begin{subfigure}[t]{0.6\textwidth}
        \centering
        \includegraphics[width=0.9\linewidth]{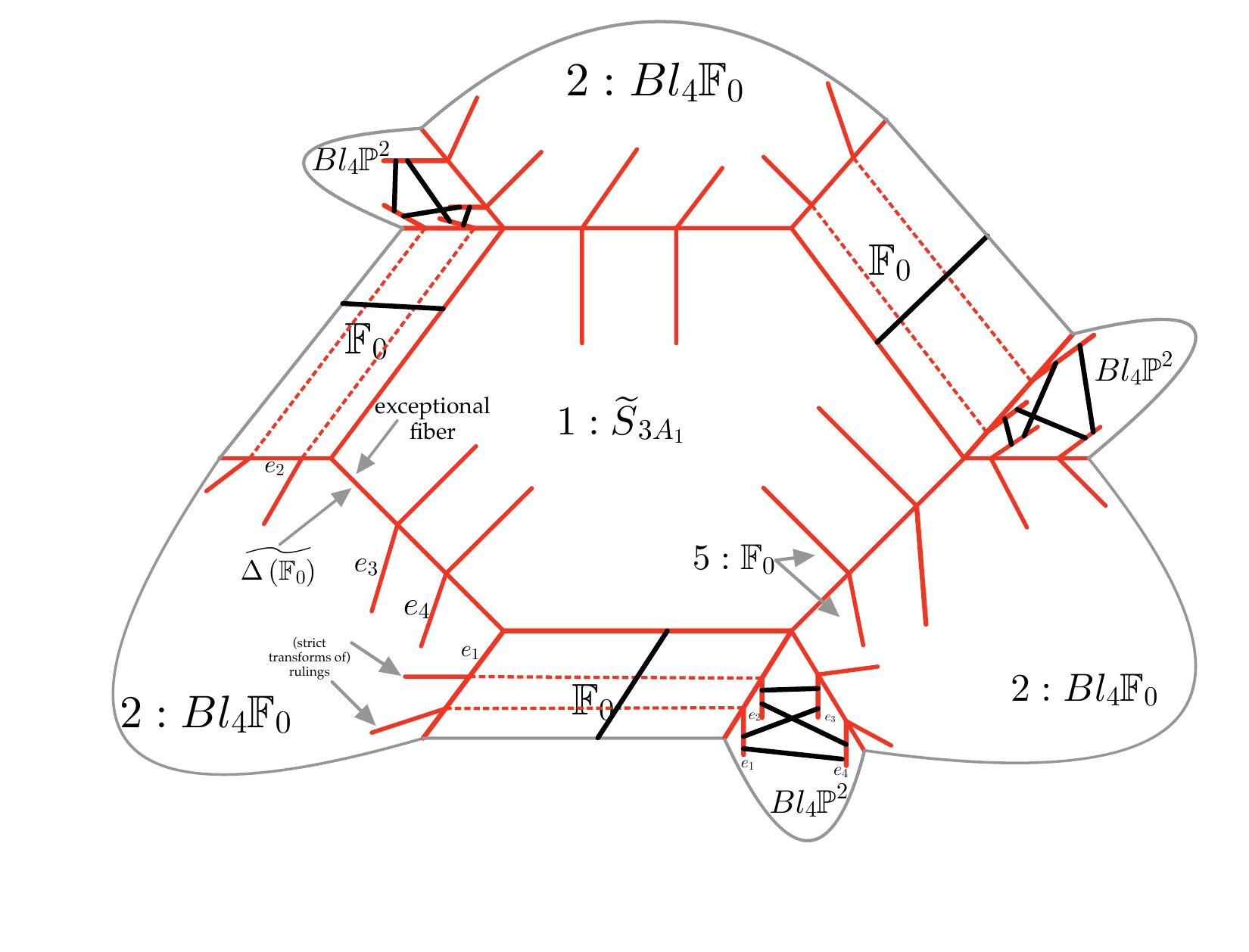}
        \caption{A type $aa_2a_3$ surface for weights $1/2 < c \leq 2/3$.}
        \label{fig:aa2a3_12-23}
    \end{subfigure}
    \caption{Degenerations of the type $a_3$ surface of \cref{fig:a3_14-12} into types $aa_3$, $a_2a_3$, and $aa_2a_3$
    surfaces, for weights $1/2 < c \leq 2/3$ (cf. \cref{fig:a3_14-12_degens}). Each degeneration is isomorphic to the original $a_3$ surface of
    \cref{fig:a3_12-23} away from the type 3 components shown, which degenerate into one $Bl_4\bP^2$ component and three
    $\bF_0$ components as pictured (cf. also \cref{fig:aa2_12-23_2} for another perspective on these components).}%
    \label{fig:a3_12-23_degens}
\end{figure}

\subsection{Type $a_4$}

\begin{proposition} \label{prop:a4}
    For type $a_4$ weighted stable marked cubic surfaces $(S,cB)$, there are four walls.
    \begin{enumerate}
        \item For weights $1/9 < c \leq 1/6$, the type $a_4$ surfaces are described in \cref{fig:a4_19-16}.
        \item For weights $1/6 < c \leq 1/4$, the type $a_4$ surfaces are described in \cref{fig:a4_16-14}.
        \item For weights $1/4 < c \leq 1/2$, the type $a_4$ surfaces are described in \cref{fig:a4_14-12}.
            Additionally, crossing the wall $c=1/4$ introduces type $aa_4$, $a_2a_4$, $a_3a_4$, $aa_2a_4$, $aa_3a_4$,
            $a_2a_3a_4$, and $aa_2a_3a_4$ surfaces as degenerations of type $a_4$ surfaces, described in
            \cref{fig:a4_14-12_degens}.
        \item For weights $1/2 < c \leq 2/3$, the type $a_4$ surfaces are described in \cref{fig:a4_12-23}. The type
            $aa_4$, $a_2a_4$, $a_3a_4$, $aa_2a_4$, $aa_3a_4$, $a_2a_3a_4$, and $aa_2a_3a_4$ surfaces are described in
            \cref{fig:a4_12-23_degens}.
        \item For weights $2/3 < c \leq 1$, the type $a_4$ surfaces are obtained from the weight $1/2 < c \leq 2/3$ type
            $a_4$ surfaces by resolving Eckardt points as described in \cref{prop:resolve_eckardt}. The possible
            configurations of Eckardt points are summarized in \cref{tab:a4_23-1}. The type
            $aa_4$, $a_2a_4$, $a_3a_4$, $aa_2a_4$, $aa_3a_4$, $a_2a_3a_4$, and $aa_2a_3a_4$ surfaces are described
            similarly.
    \end{enumerate}
\end{proposition}

%\subsubsection{Weights $1/9 < c \leq 1/6$}

\begin{figure}[!htpb]
    \centering
    \includegraphics[width=0.3\linewidth]{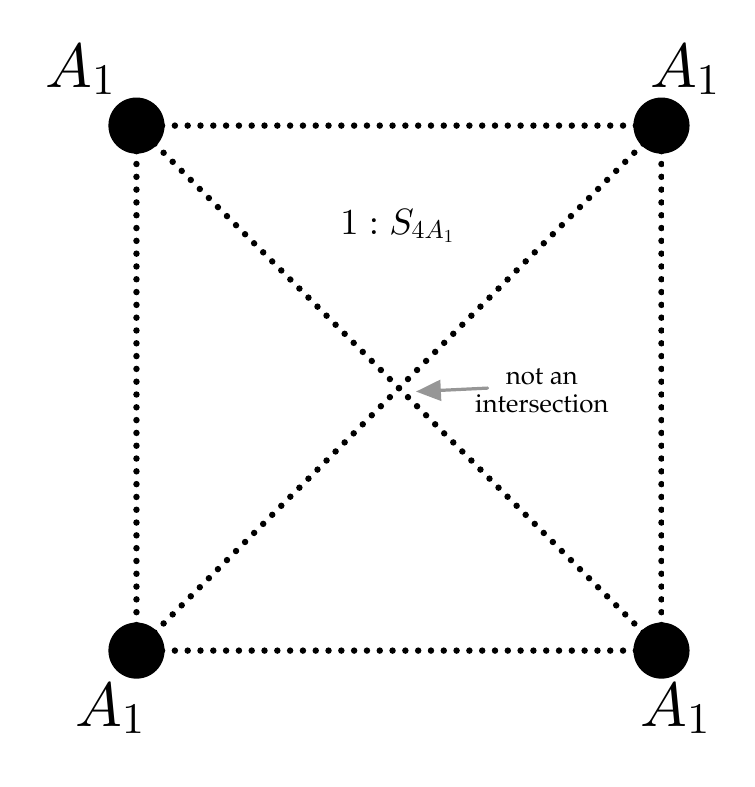}
    \caption{A type $a_4$ surface $(S_0,cB_0)$ for weights $1/9 < c \leq 1/6$. This is a cubic surface with four $A_1$
    singularities. Recall that dotted lines have multiplicity 4. In addition to the lines shown, there are 3 lines of
    multiplicity 1, not shown.}%
    \label{fig:a4_19-16}
\end{figure}

%\subsubsection{Weights $1/6 < c \leq 1/4$}

\begin{figure}[!htpb]
    \centering
    \includegraphics[width=0.5\linewidth]{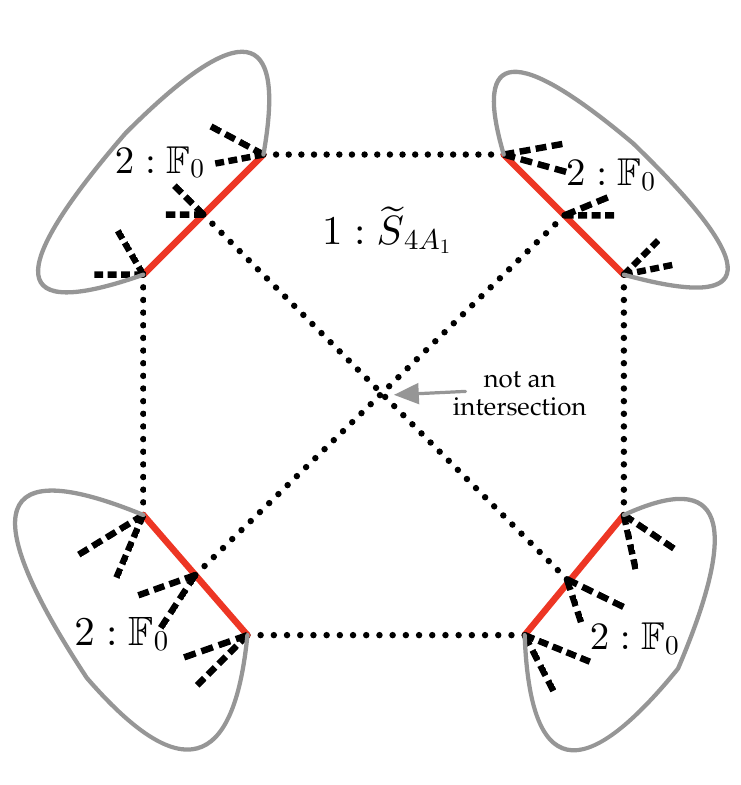}
    \caption{A type $a_4$ surface $(S_1,cB_1)$ for weights $1/6 < c \leq 1/4$, obtained as the stable replacement for
        these weights of the surface $(S_0,cB_0)$ of \cref{fig:a4_19-16}. This is obtained by resolving the four $A_1$
        singularities of $(S_0,cB_0)$, and attaching to each exceptional fiber a type 2 component isomorphic to $\bF_0
        \cong \bP^1 \times \bP^1$, glued along its diagonal. The lines of multiplicity 4 on the type 1 component split
        into two lines of multiplicity 2 on each corresponding $\bF_0$ component, one in each ruling.}%
    \label{fig:a4_16-14}
\end{figure}

%\subsubsection{Weights $1/4 < c \leq 1/2$}

\begin{figure}[!htpb]
    \centering
    \begin{subfigure}[t]{0.55\textwidth}
        \centering
        \includegraphics[width=0.9\linewidth]{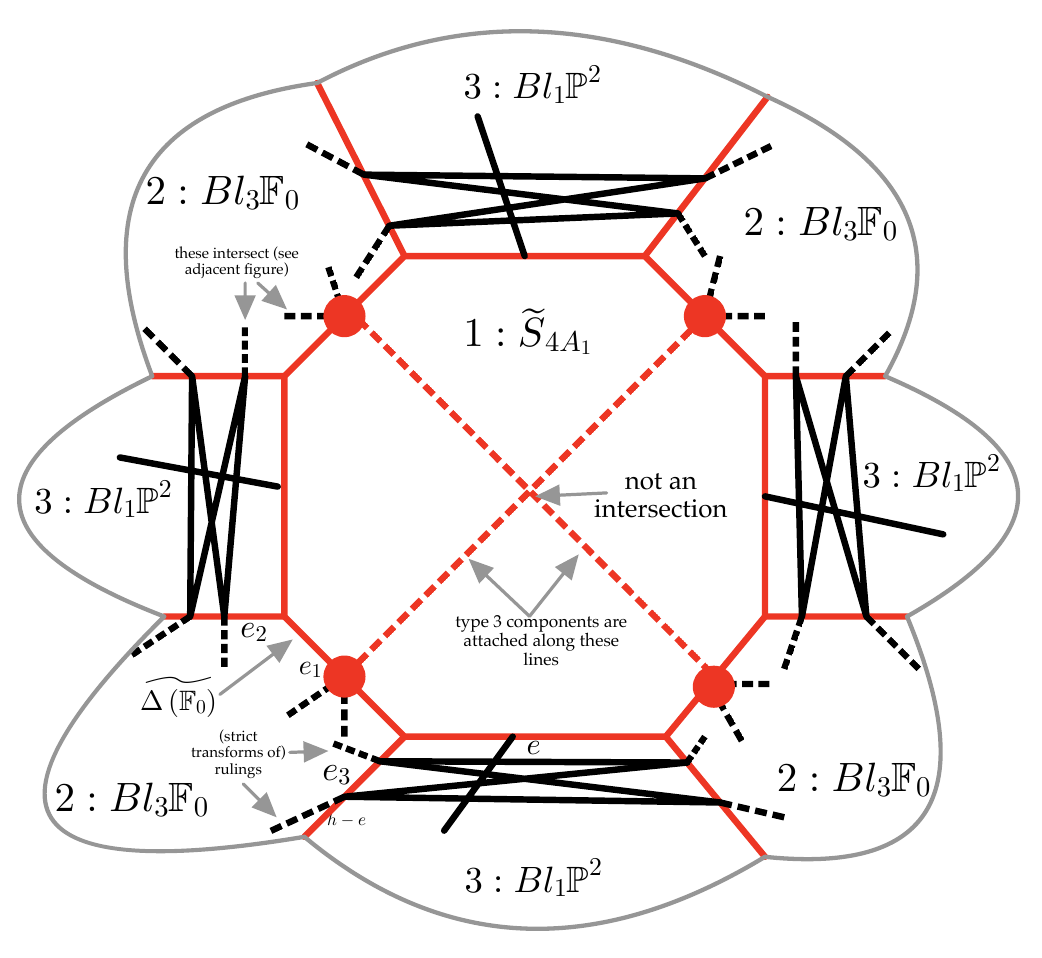}
        \caption{The type $a_4$ surface $(S_2,cB_2)$ for weights $1/4 < c \leq 1/2$.}
        \label{fig:a4_14-12_1}
    \end{subfigure}
    \begin{subfigure}[t]{0.4\textwidth}
        \centering
        \includegraphics[width=0.9\linewidth]{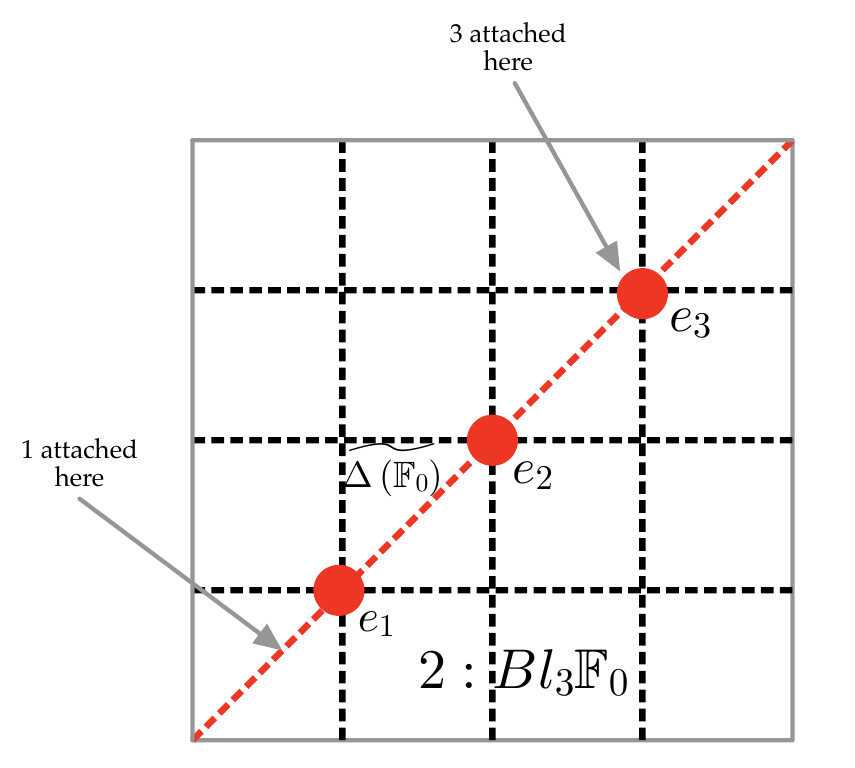}
        \caption{Another view of a type 2 component isomorphic to $Bl_3\bF_0$, showing more clearly the lines appearing
        on this component.}
        \label{fig:a4_14-12_2}
    \end{subfigure}
    \caption{A type $a_4$ surface $(S_2,cB_2)$ for weights $1/4 < c \leq 1/2$, obtained as the stable replacement for
        these weights of the surface $(S_1,cB_1)$ of \cref{fig:a4_16-14}. This is obtained by blowing up the lines of
        multiplicity 4 on $(S_1,cB_1)$, and attaching to the exceptional divisors type 3 components isomorphic to $\bF_1
        \cong Bl_1\bP^2$ (cf. \cref{fig:a2_14-12,fig:a3_14-12}). As in \cref{fig:a2_14-12,fig:a3_14-12}, the five lines
        on a given type 3 component come from the corresponding line of multiplicity 4, plus a line of multiplicity 1
        intersecting this line of multiplicity 4. There are three lines of multiplicity 1 on $\wS_{4A_1}$, and each
        intersects precisely two lines of multiplicity four. For instance, in \cref{fig:a4_14-12_1}, the fifth line on
        the top type 3 component comes from the same line on $\wS_{4A_1}$ as the fifth line on the bottom type 3
    component.}%
    \label{fig:a4_14-12}
\end{figure}

\begin{figure}[!htpb]
    \centering
    \begin{subfigure}[t]{0.45\textwidth}
        \centering
        \includegraphics[width=0.9\linewidth]{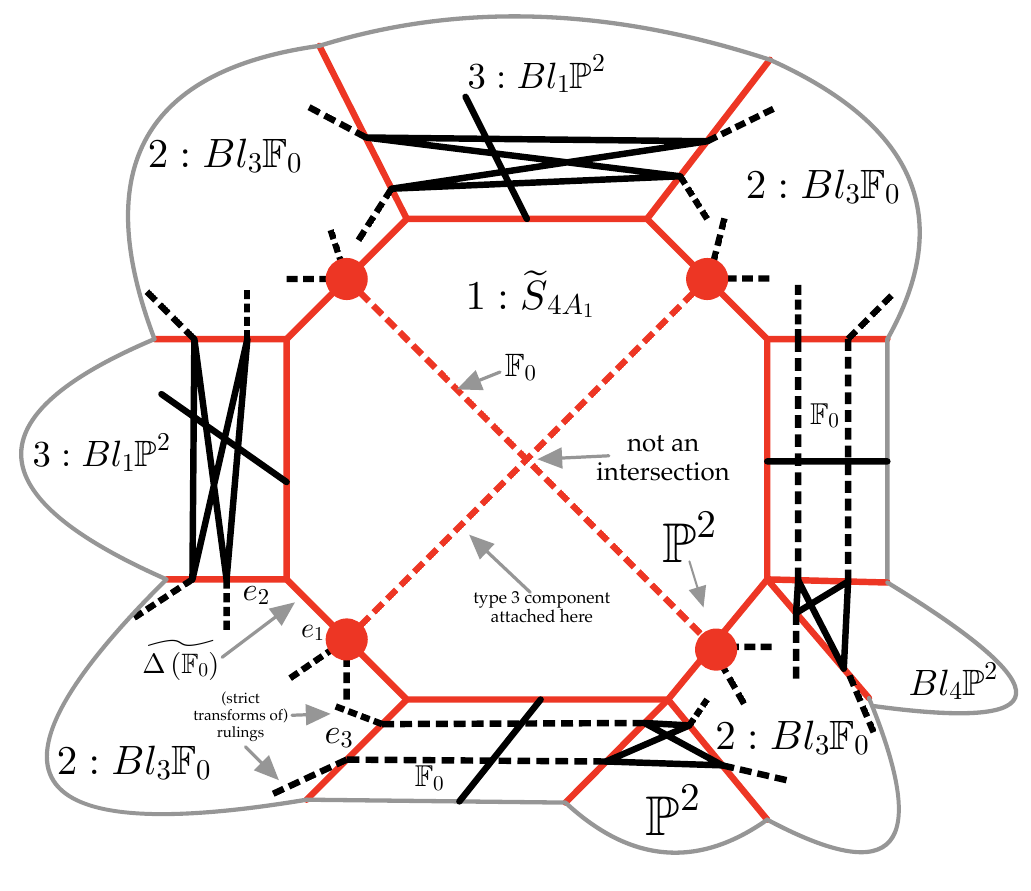}
        \caption{A type $aa_4$ surface for weights $1/4 < c \leq 1/2$.}
        \label{fig:aa4_14-12}
    \end{subfigure}
    \begin{subfigure}[t]{0.45\textwidth}
        \centering
        \includegraphics[width=0.9\linewidth]{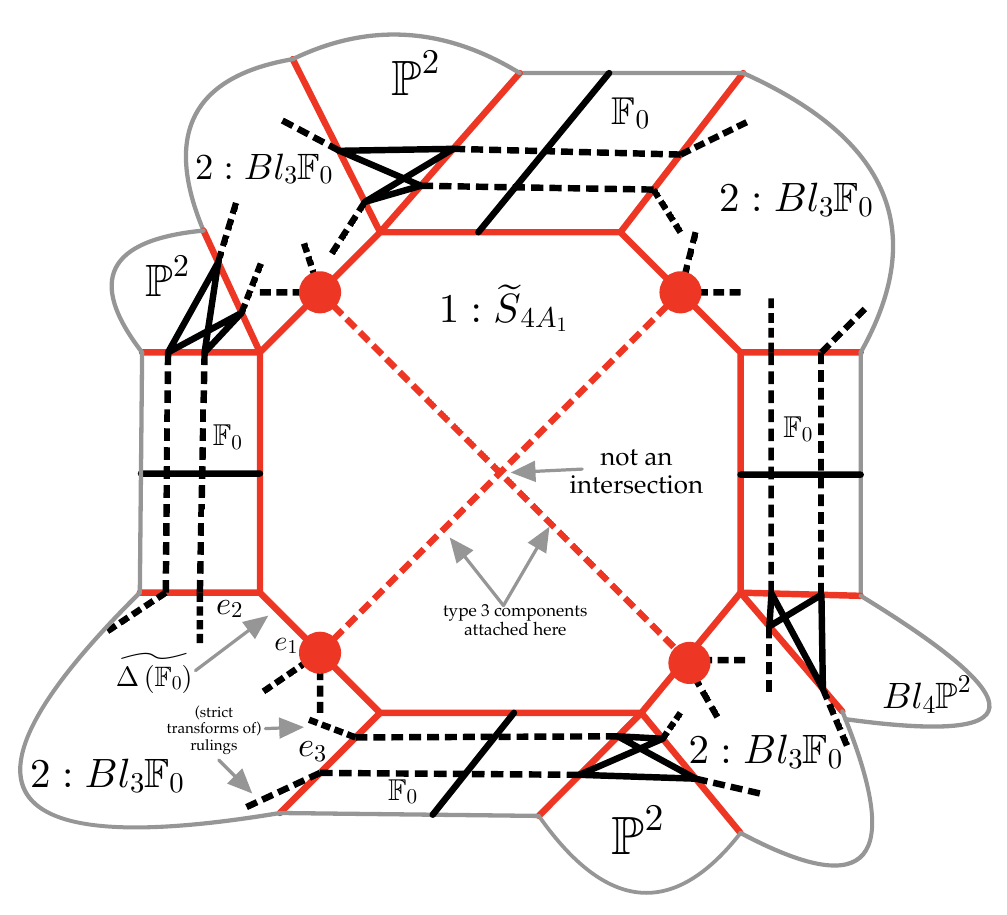}
        \caption{A type $a_2a_4$ surface for weights $1/4 < c \leq 1/2$.}
        \label{fig:a2a4_14-12}
    \end{subfigure}
    \begin{subfigure}[t]{0.45\textwidth}
        \centering
        \includegraphics[width=0.9\linewidth]{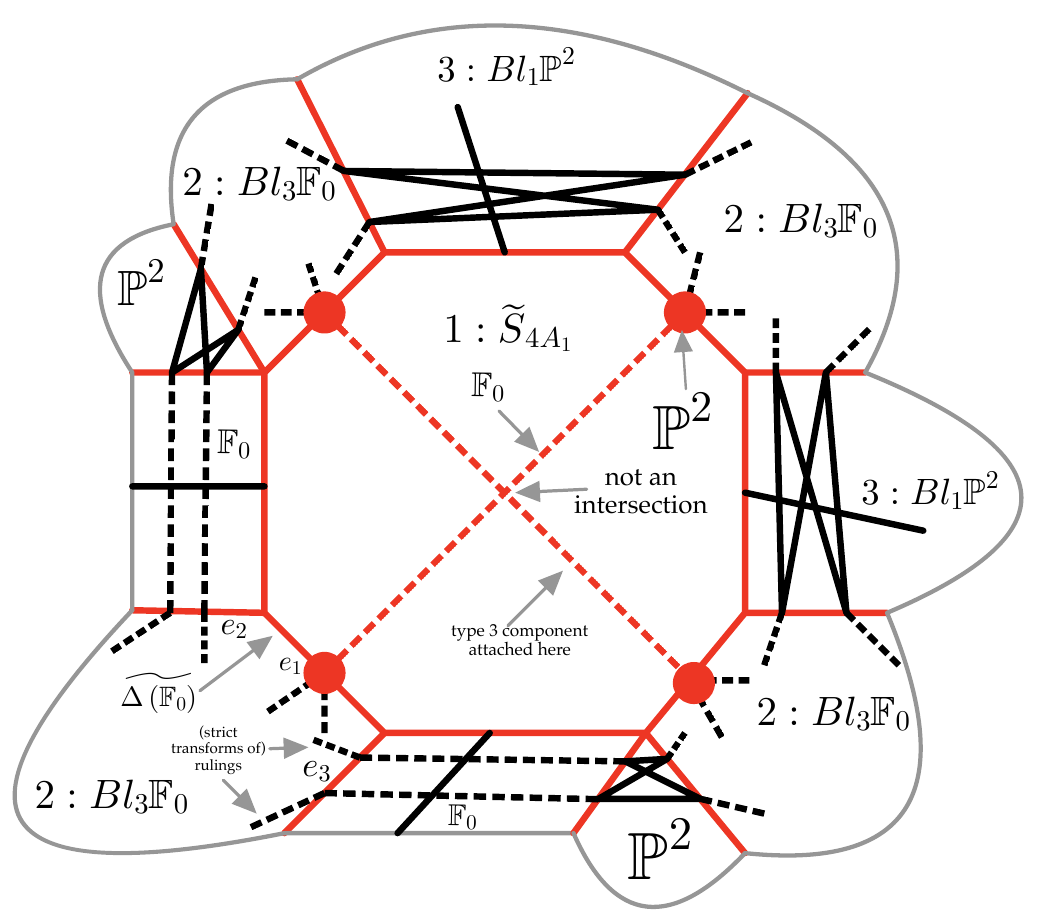}
        \caption{A type $a_3a_4$ surface for weights $1/4 < c \leq 1/2$.}
        \label{fig:a3a4_14-12}
    \end{subfigure}
    \begin{subfigure}[t]{0.45\textwidth}
        \centering
        \includegraphics[width=0.9\linewidth]{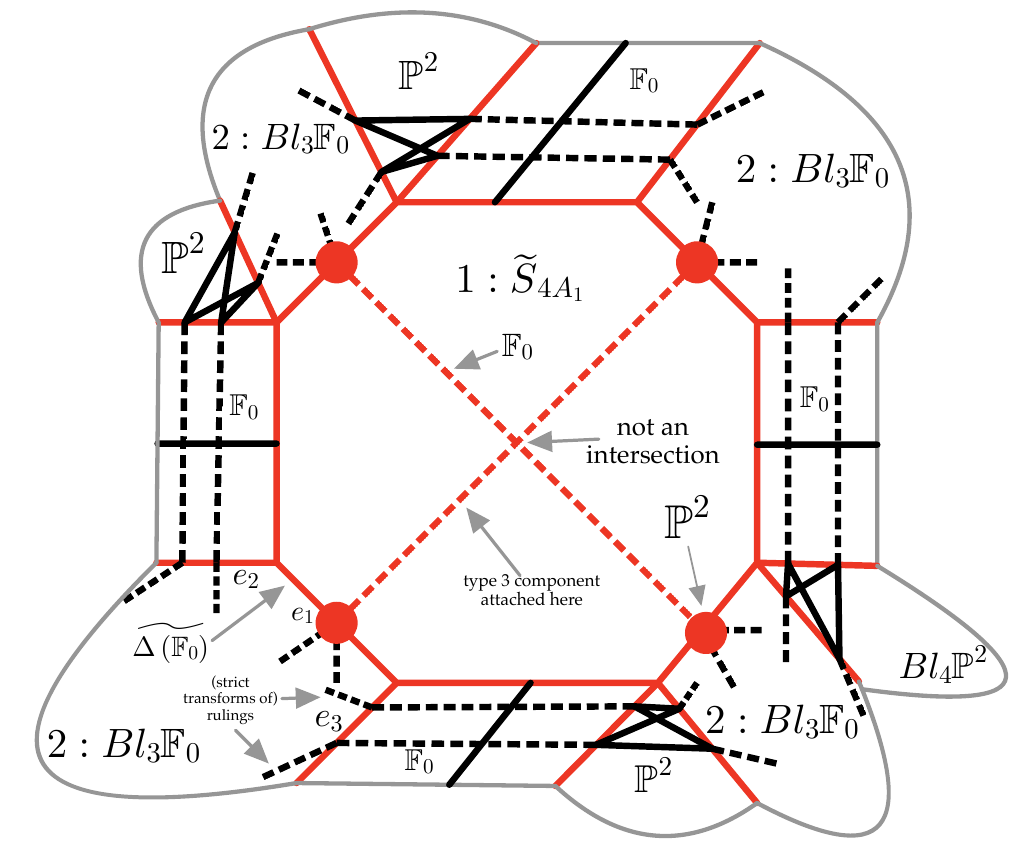}
        \caption{A type $aa_2a_4$ surface for weights $1/4 < c \leq 1/2$.}
        \label{fig:aa2a4_14-12}
    \end{subfigure}
    \caption{Degenerations of the type $a_4$ surface of \cref{fig:a4_14-12} into types $aa_4$, $a_2a_4$, $a_3a_4$,
        $aa_2a_4$, $aa_3a_4$, $a_2a_3a_4$, and $aa_2a_3a_4$ surfaces, for weights $1/4 < c \leq 1/2$. Each degeneration
        is isomorphic to the original $a_4$ surface of \cref{fig:a3_14-12}, away from the type 3 components which
        degenerate into $\bP^2$ and $\bF_0$ as indicated (cf. \cref{fig:aa2_14-12,fig:a3_14-12_degens}).}%
    \label{fig:a4_14-12_degens}
\end{figure}
\begin{figure}[!htpb]\ContinuedFloat
    \centering
    \begin{subfigure}[t]{0.45\textwidth}
        \centering
        \includegraphics[width=0.9\linewidth]{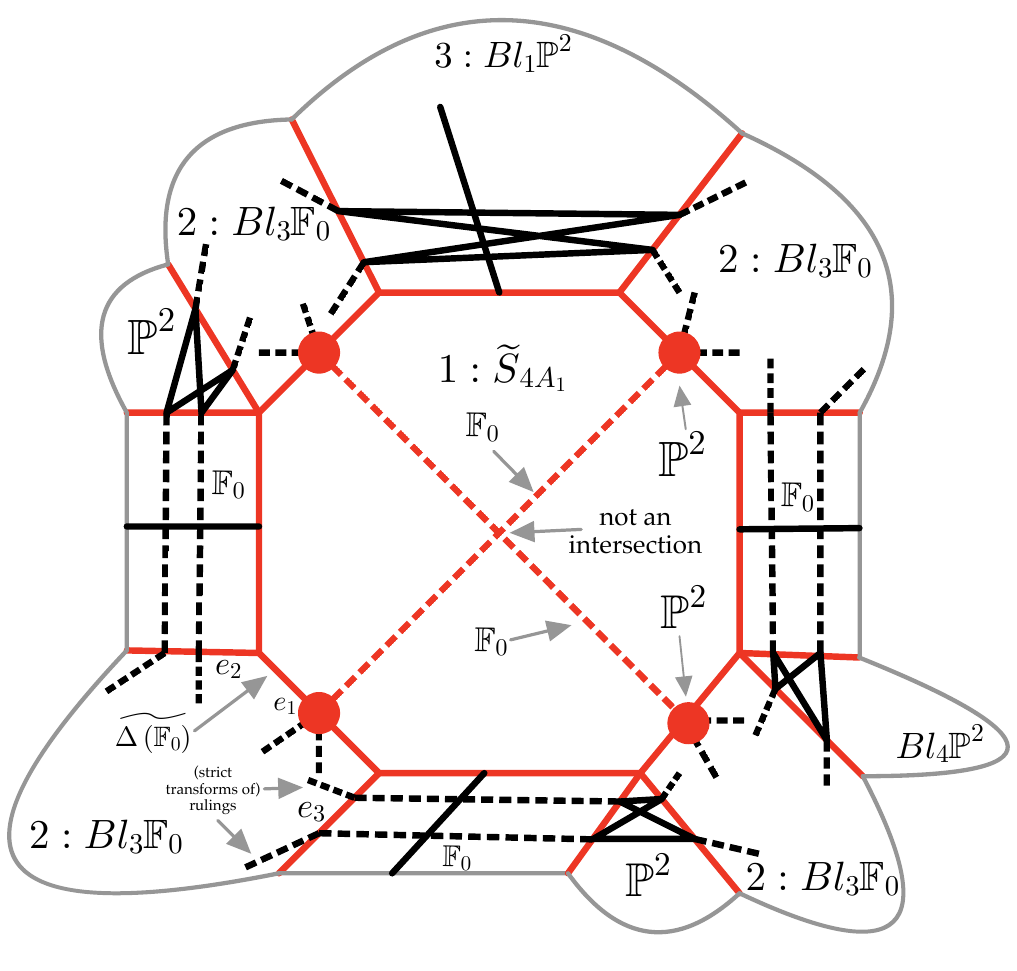}
        \caption{A type $aa_3a_4$ surface for weights $1/4 < c \leq 1/2$.}
        \label{fig:aa3a4_14-12}
    \end{subfigure}
    \begin{subfigure}[t]{0.45\textwidth}
        \centering
        \includegraphics[width=0.9\linewidth]{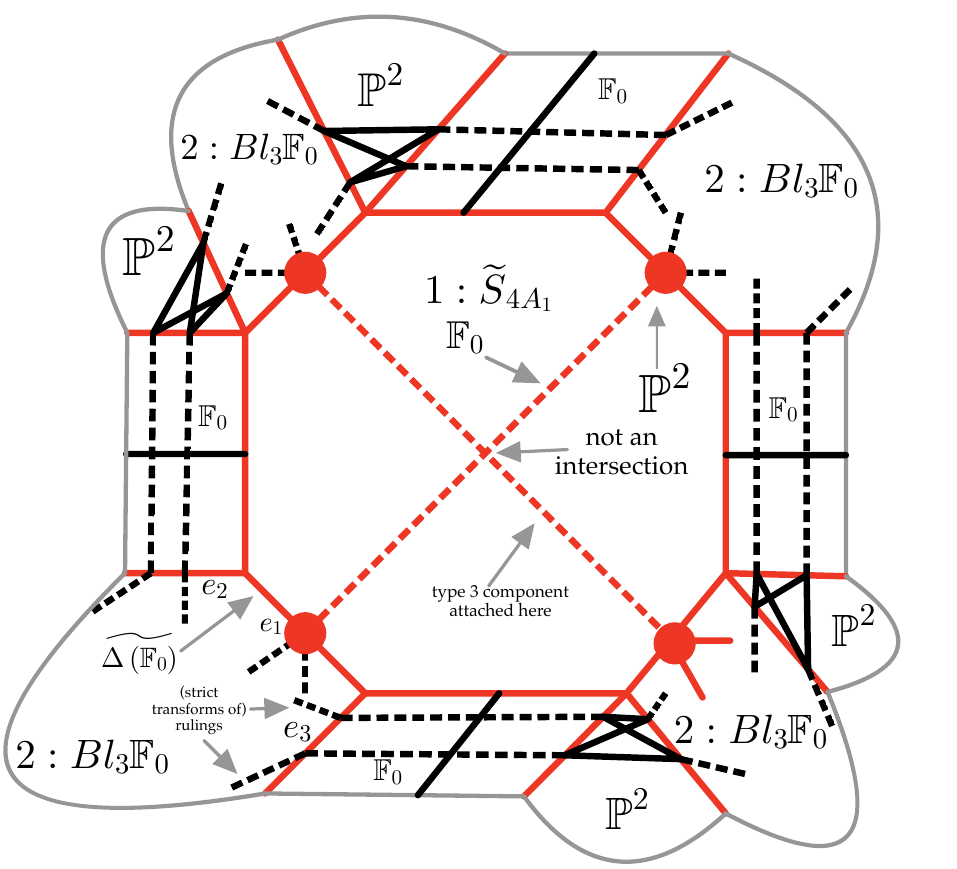}
        \caption{A type $a_2a_3a_4$ surface for weights $1/4 < c \leq 1/2$.}
        \label{fig:a2a3a4_14-12}
    \end{subfigure}
    \begin{subfigure}[t]{0.45\textwidth}
        \centering
        \includegraphics[width=0.9\linewidth]{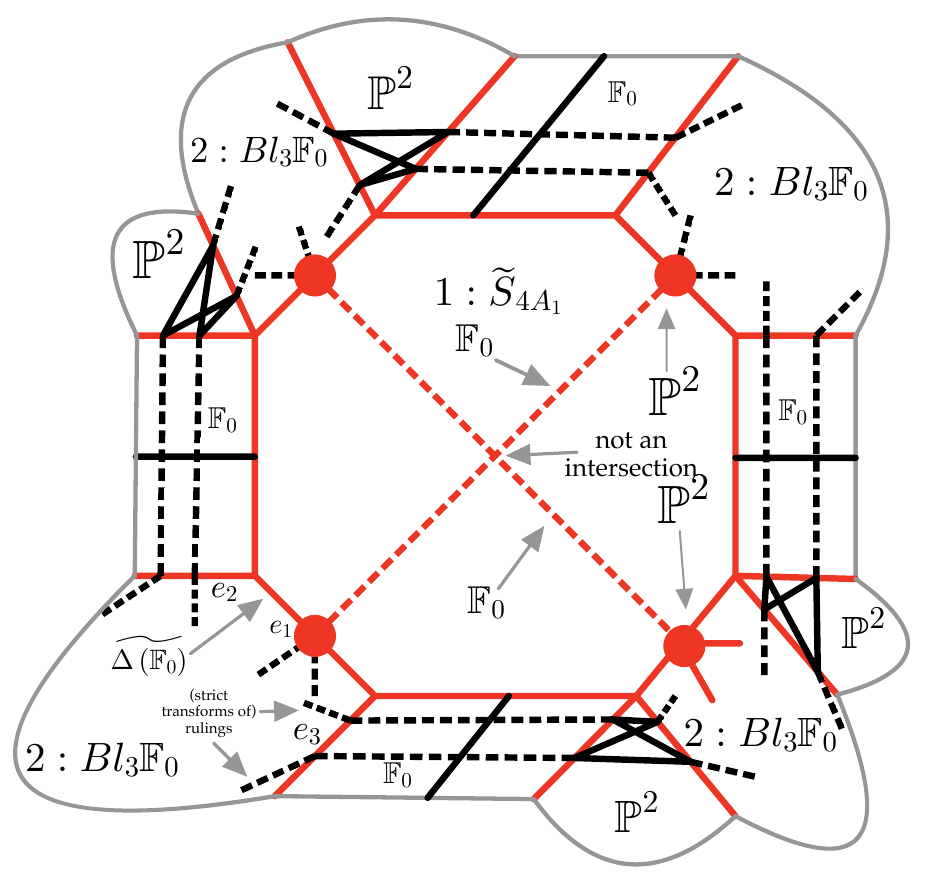}
        \caption{A type $aa_2a_3a_4$ surface for weights $1/4 < c \leq 1/2$.}
        \label{fig:aa2a3a4_14-12}
    \end{subfigure}
    \caption{(Continued) Degenerations of the type $a_4$ surface of \cref{fig:a4_14-12} into types $aa_4$, $a_2a_4$, $a_3a_4$,
        $aa_2a_4$, $aa_3a_4$, $a_2a_3a_4$, and $aa_2a_3a_4$ surfaces, for weights $1/4 < c \leq 1/2$. Each degeneration
        is isomorphic to the original $a_4$ surface of \cref{fig:a3_14-12}, away from the type 3 components which
        degenerate into $\bP^2$ and $\bF_0$ as indicated (cf. \cref{fig:aa2_14-12,fig:a3_14-12_degens}).}%
    \label{fig:a4_14-12_degens_cont}
\end{figure}

%\subsubsection{Weights $1/2 < c \leq 2/3$}

\begin{figure}[!htpb]
    \centering
    \begin{subfigure}[t]{0.7\textwidth}
        \centering
        \includegraphics[width=0.8\linewidth]{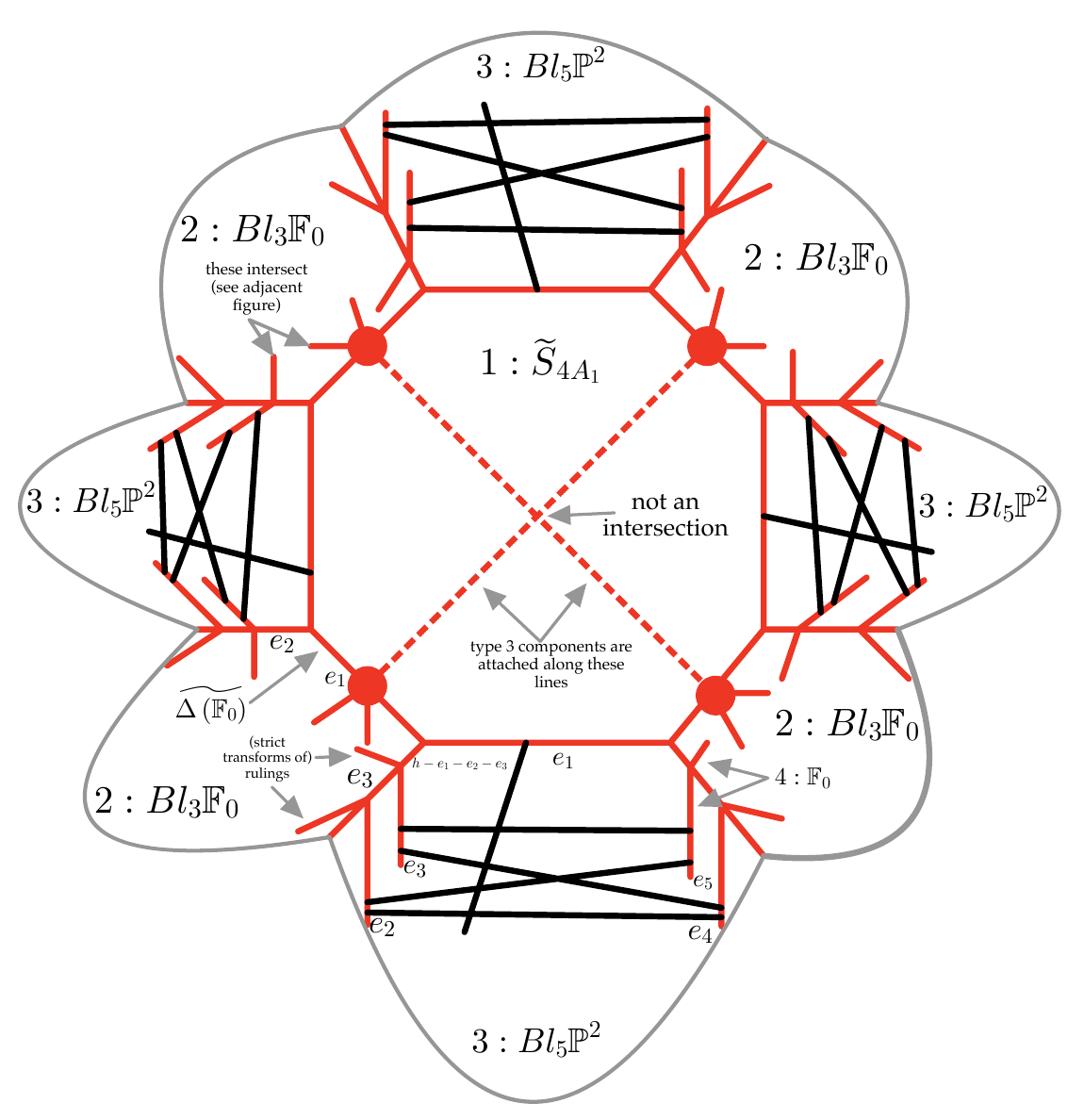}
        \caption{The type $a_4$ surface $(S_2,cB_2)$ for weights $1/2 < c \leq 2/3$.}
        \label{fig:a4_12-23_1}
    \end{subfigure}
    \begin{subfigure}[t]{0.4\textwidth}
        \centering
        \includegraphics[width=0.9\linewidth]{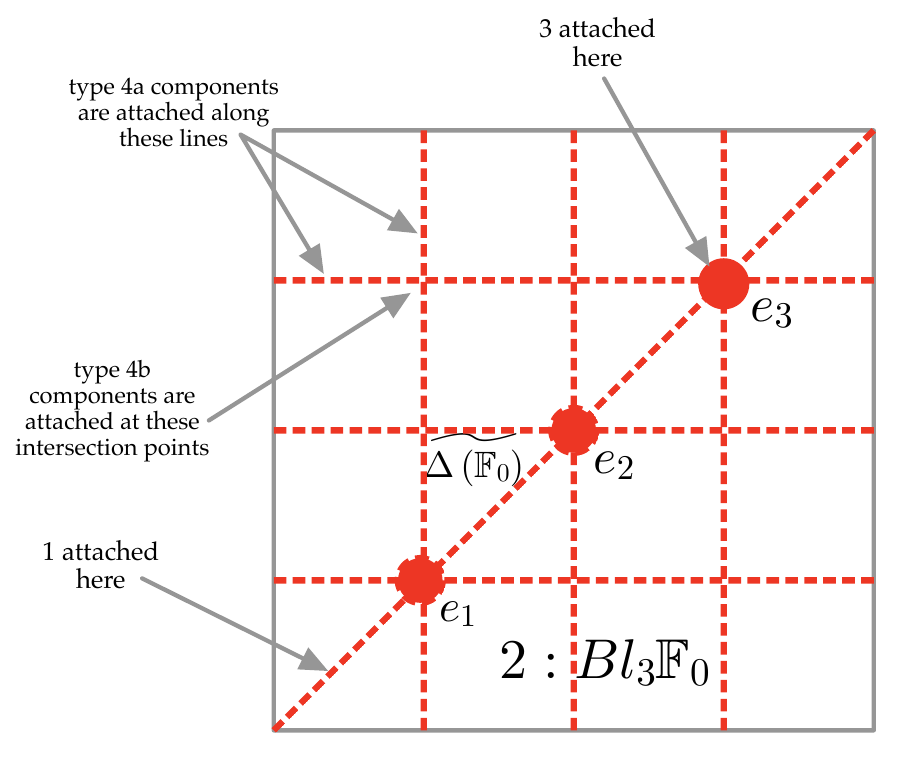}
        \caption{Another view of a type 2 component isomorphic to $Bl_3\bF_0$, showing more clearly the rulings on this
        component.}
        \label{fig:a4_12-23_2}
    \end{subfigure}
    \begin{subfigure}[t]{0.55\textwidth}
        \centering
        \includegraphics[width=0.9\linewidth]{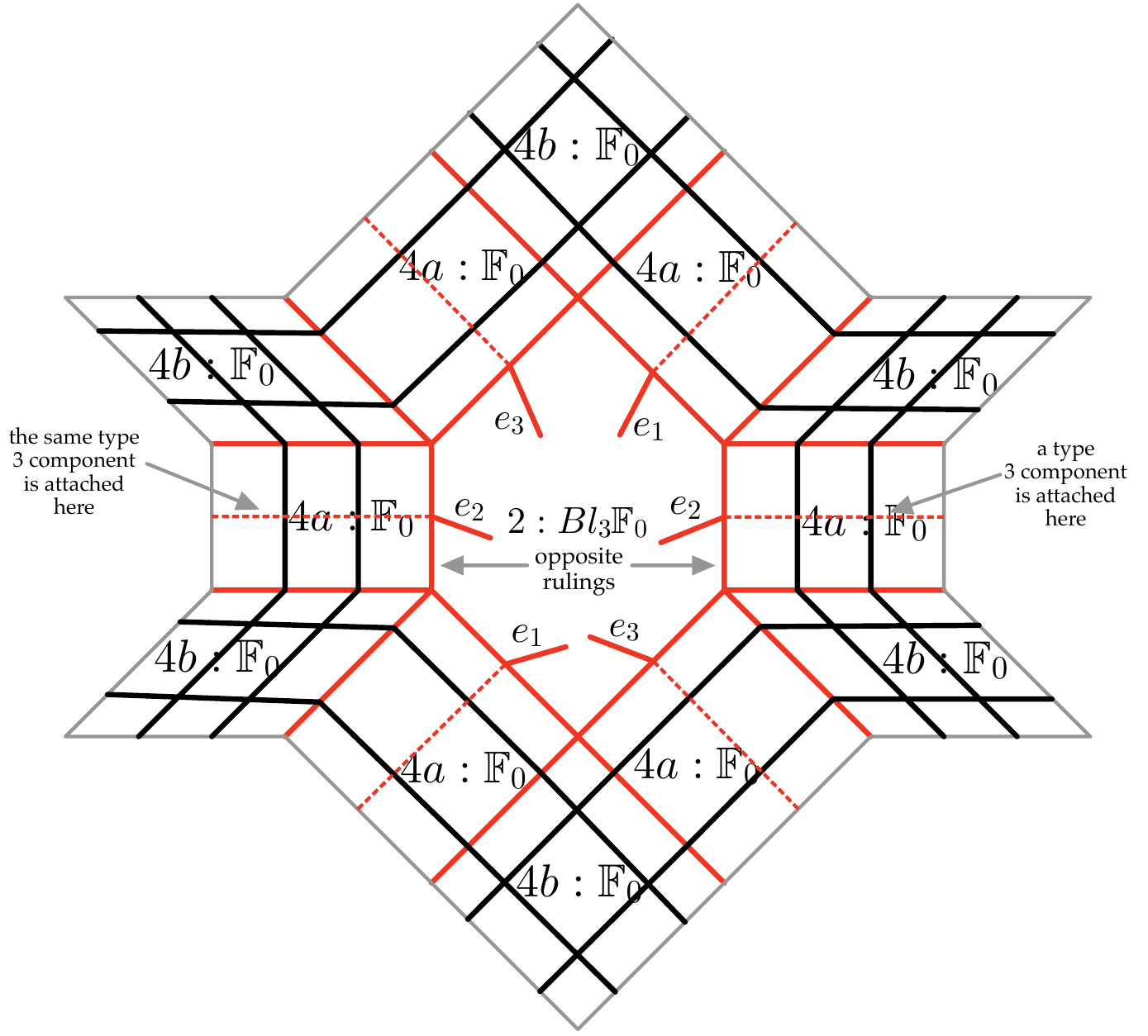}
        \caption{A view of the type 4a and 4b components intersecting a given type 2 component. Note the similarities to
        \cref{fig:a3_12-23_2}. This picture occurs 4 times on the $a_4$ surface $(S_2,cB_2)$, once for each type 2
        component.}
        \label{fig:a4_12-23_3}
    \end{subfigure}
    \caption{A type $a_4$ surface $(S_3,cB_3)$ for weights $1/2 < c \leq 2/3$, obtained as the stable replacement for
    these weights of the surface $(S_2,cB_2)$ of \cref{fig:a4_14-12}. This is obtained by blowing up the lines of
multiplicity 2 on $(S_2,cB_2)$, and attaching to the exceptional divisors copies of $\bF_0$. Some of these lines
intersect, leading to additional components being attached, as pictured in \cref{fig:a4_12-23_1}.}%
    \label{fig:a4_12-23}
\end{figure}

\begin{figure}[!htpb]
    \centering
    \begin{subfigure}[t]{0.45\textwidth}
        \centering
        \includegraphics[width=0.9\linewidth]{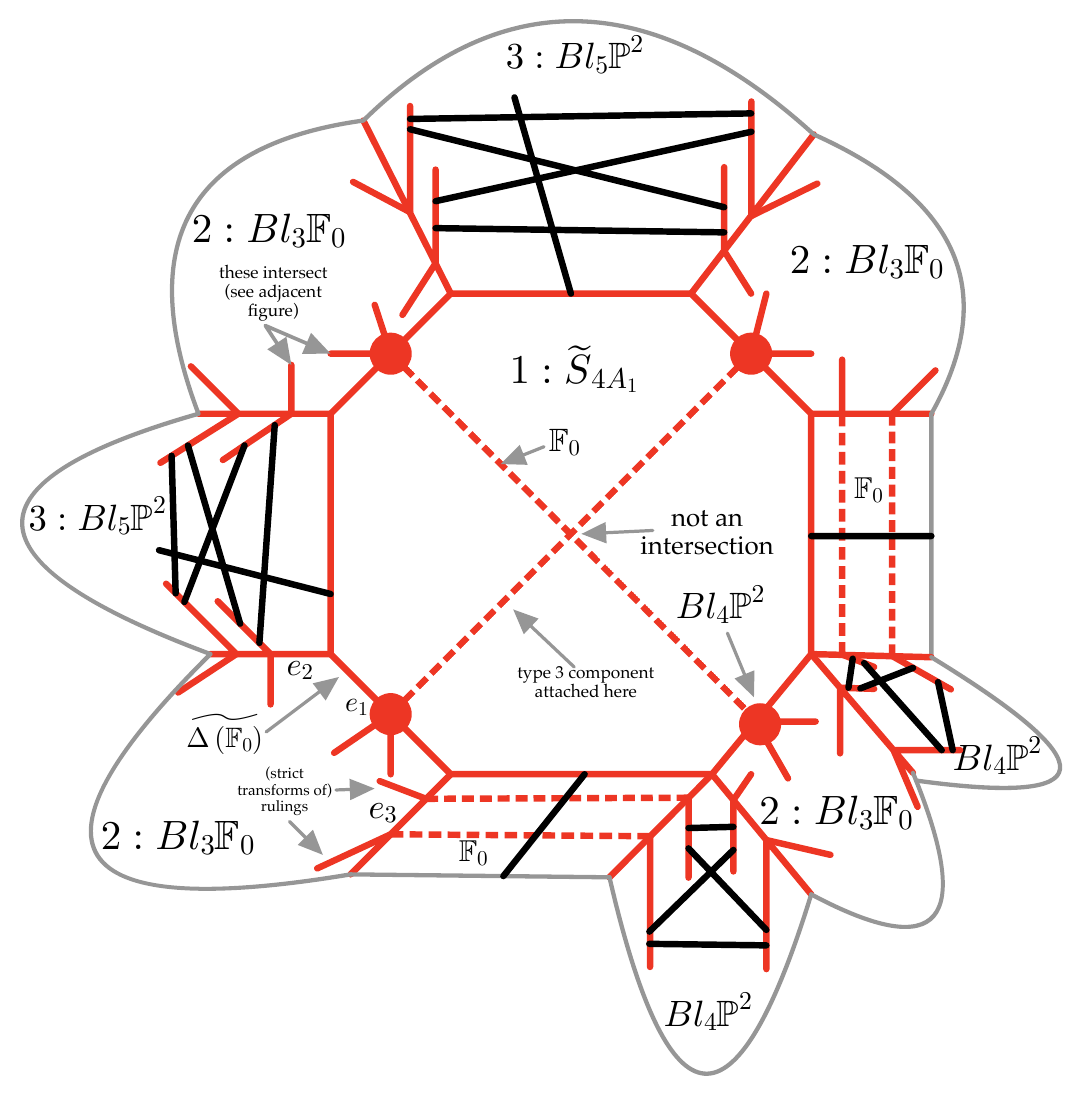}
        \caption{A type $aa_4$ surface for weights $1/2 < c \leq 2/3$.}
        \label{fig:aa4_12-23}
    \end{subfigure}
    \begin{subfigure}[t]{0.45\textwidth}
        \centering
        \includegraphics[width=0.9\linewidth]{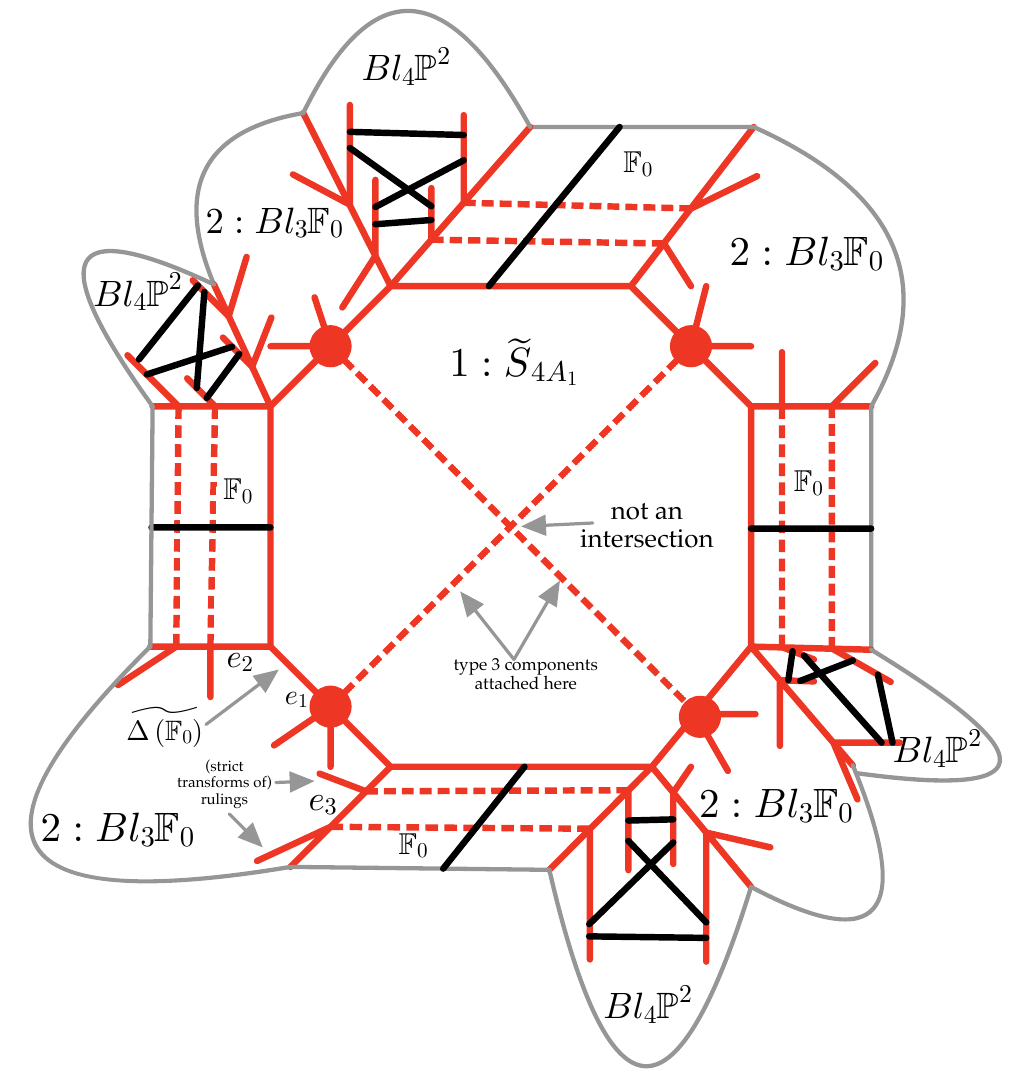}
        \caption{A type $a_2a_4$ surface for weights $1/2 < c \leq 2/3$.}
        \label{fig:a2a4_12-23}
    \end{subfigure}
    \begin{subfigure}[t]{0.45\textwidth}
        \centering
        \includegraphics[width=0.9\linewidth]{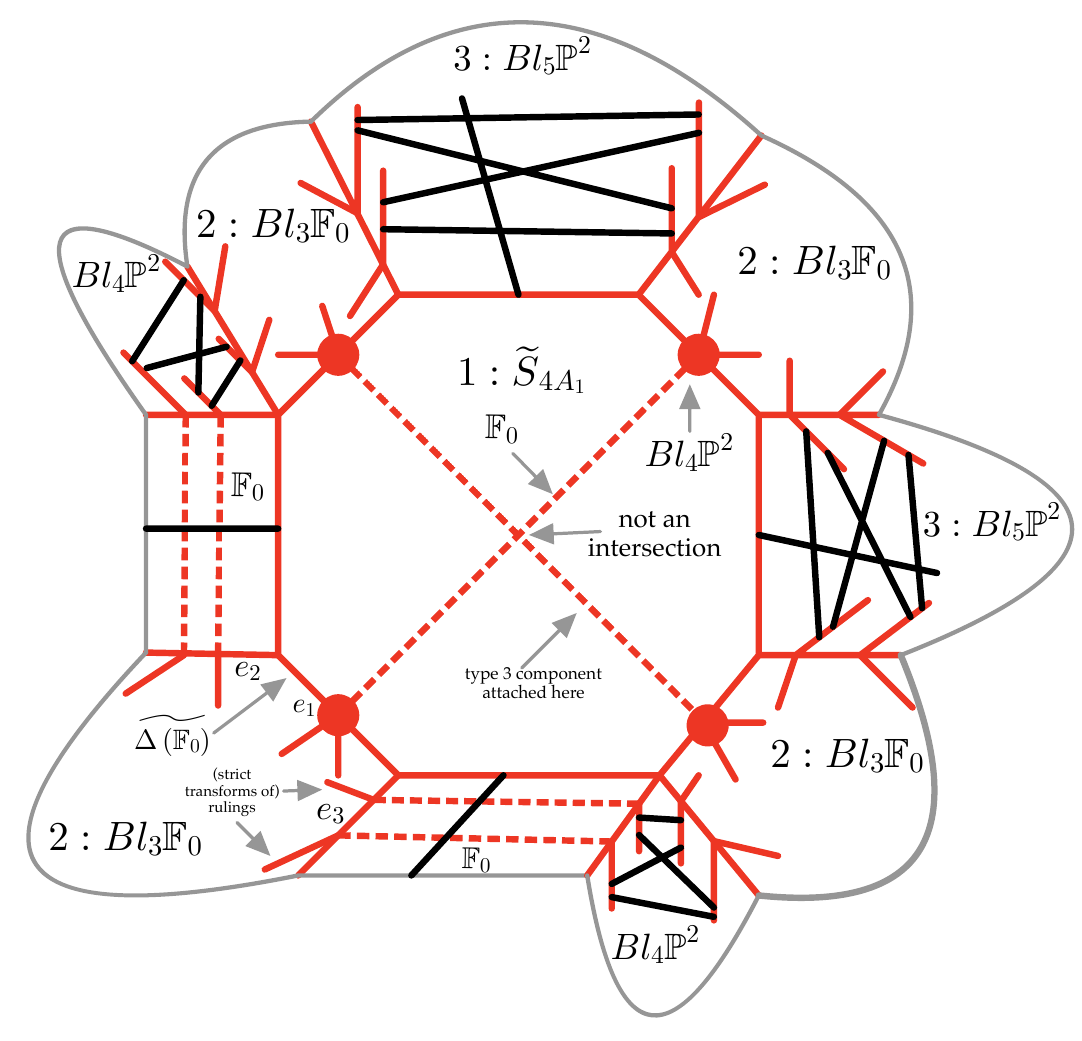}
        \caption{A type $a_3a_4$ surface for weights $1/2 < c \leq 2/3$.}
        \label{fig:a3a4_12-23}
    \end{subfigure}
    \begin{subfigure}[t]{0.45\textwidth}
        \centering
        \includegraphics[width=0.9\linewidth]{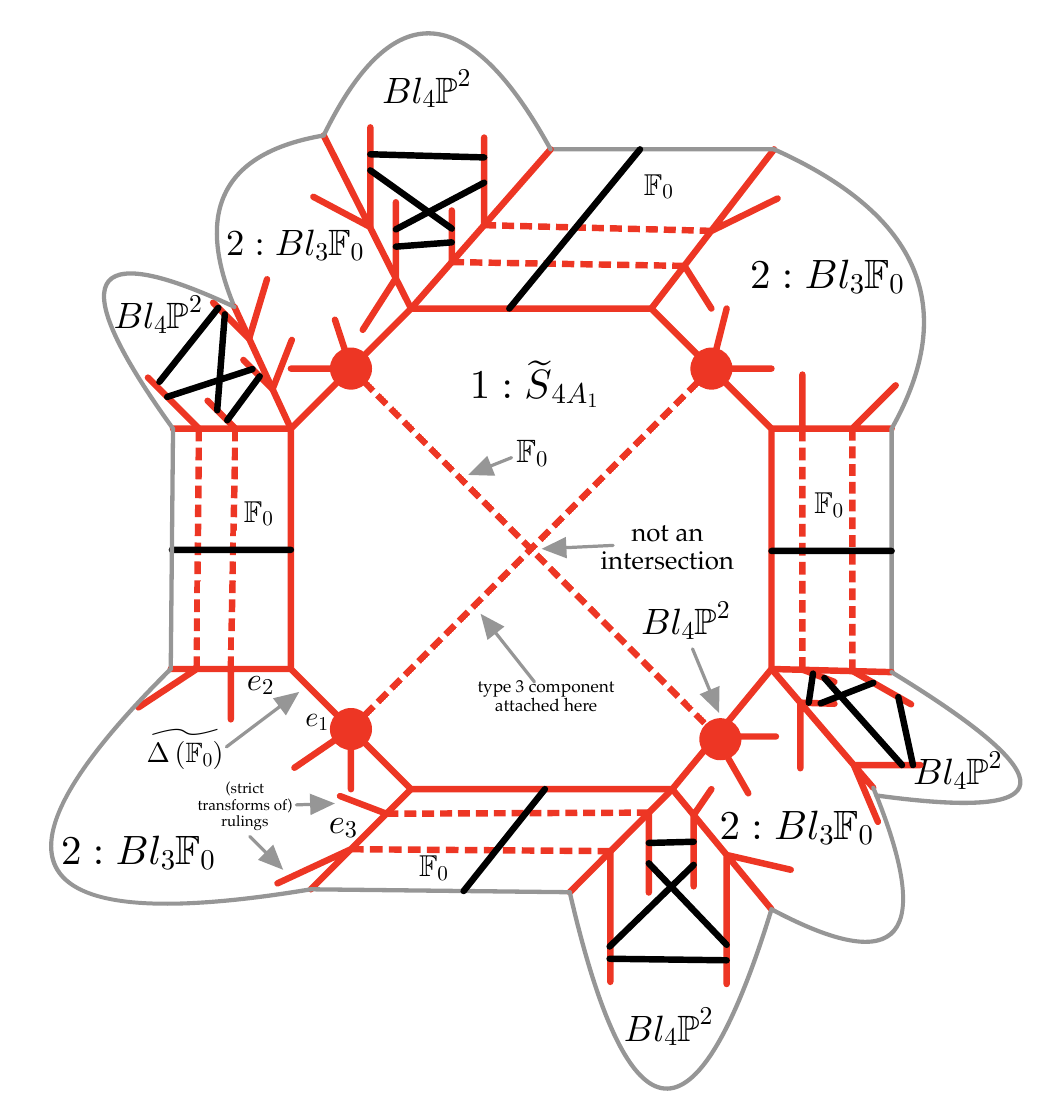}
        \caption{A type $aa_2a_4$ surface for weights $1/2 < c \leq 2/3$.}
        \label{fig:aa2a4_12-23}
    \end{subfigure}
    \caption{Degenerations of the type $a_4$ surface of \cref{fig:a4_12-23} into types $aa_4$, $a_2a_4$, $a_3a_4$,
    $aa_2a_4$, $aa_3a_4$, $a_2a_3a_4$, and $aa_2a_3a_4$ surfaces, for weights $1/2 < c \leq 2/3$. Each degeneration is
    isomorphic to the original $a_4$ surface of \cref{fig:a3_12-23}, away from the type 3 components which degenerate into
    one copy of $Bl_4\bP^2$ and three copies of $\bF_0$ as indicated (cf. \cref{fig:aa2_12-23,fig:a3_12-23_degens}).}%
    \label{fig:a4_12-23_degens}
\end{figure}
\begin{figure}[!htpb]\ContinuedFloat
    \centering
    \begin{subfigure}[t]{0.45\textwidth}
        \centering
        \includegraphics[width=0.9\linewidth]{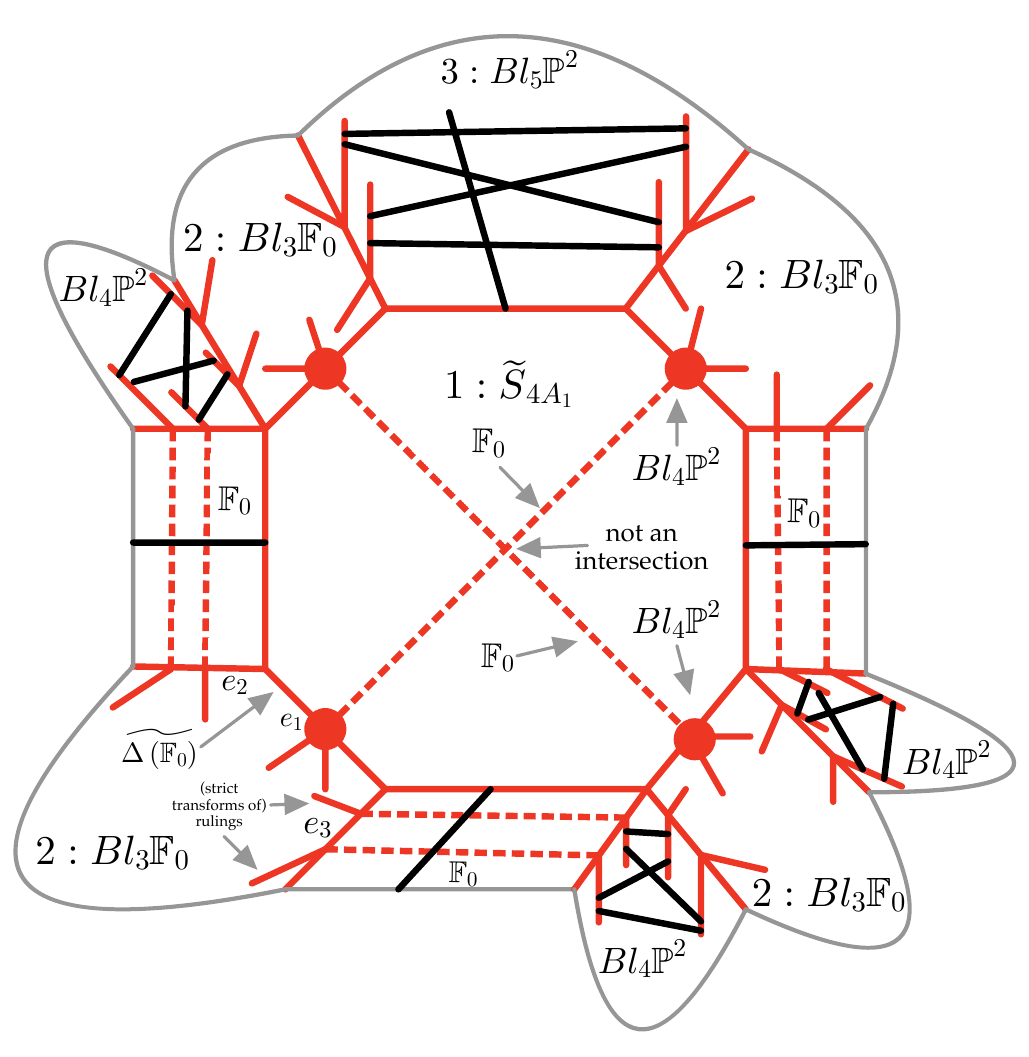}
        \caption{A type $aa_3a_4$ surface for weights $1/2 < c \leq 2/3$.}
        \label{fig:aa3a4_12-23}
    \end{subfigure}
    \begin{subfigure}[t]{0.45\textwidth}
        \centering
        \includegraphics[width=0.9\linewidth]{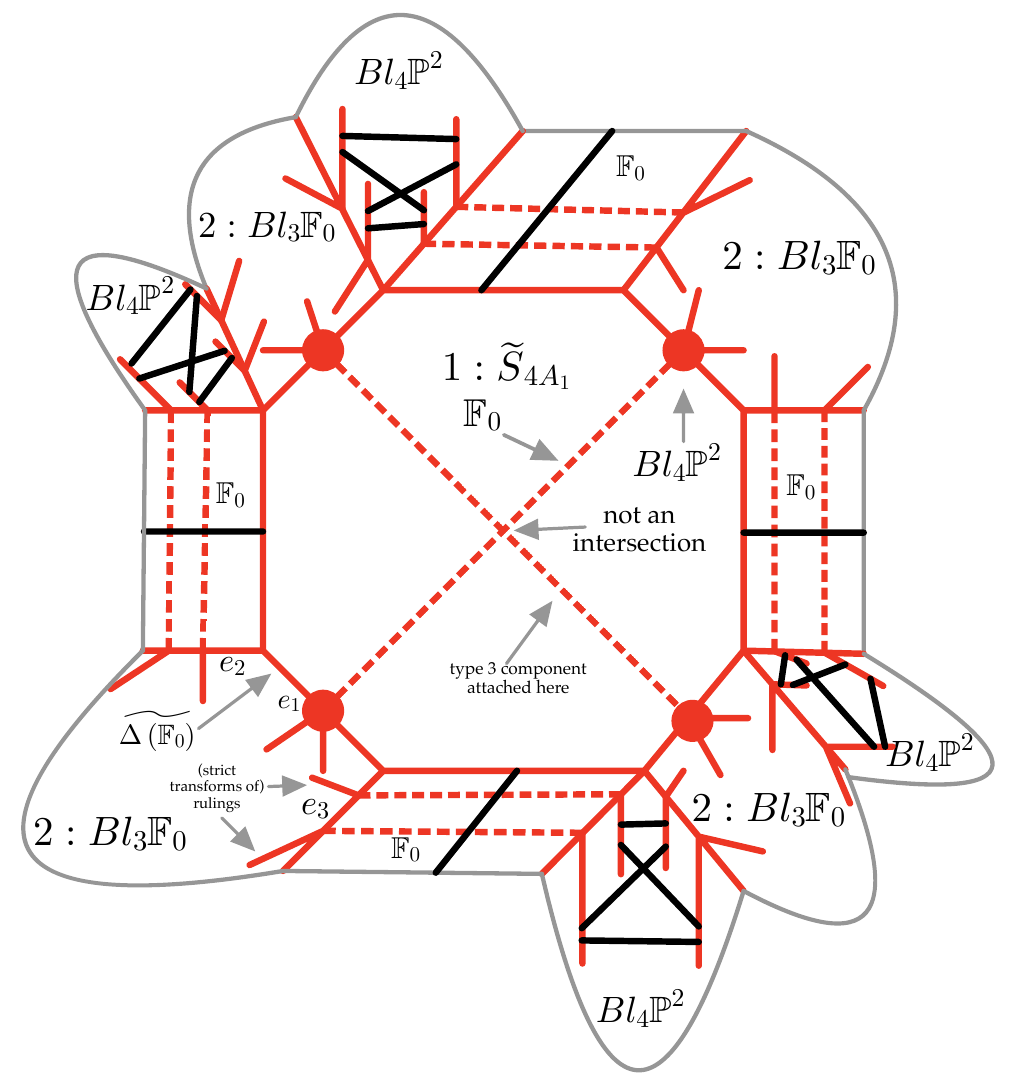}
        \caption{A type $a_2a_3a_4$ surface for weights $1/2 < c \leq 2/3$.}
        \label{fig:a2a3a4_12-23}
    \end{subfigure}
    \begin{subfigure}[t]{0.45\textwidth}
        \centering
        \includegraphics[width=0.9\linewidth]{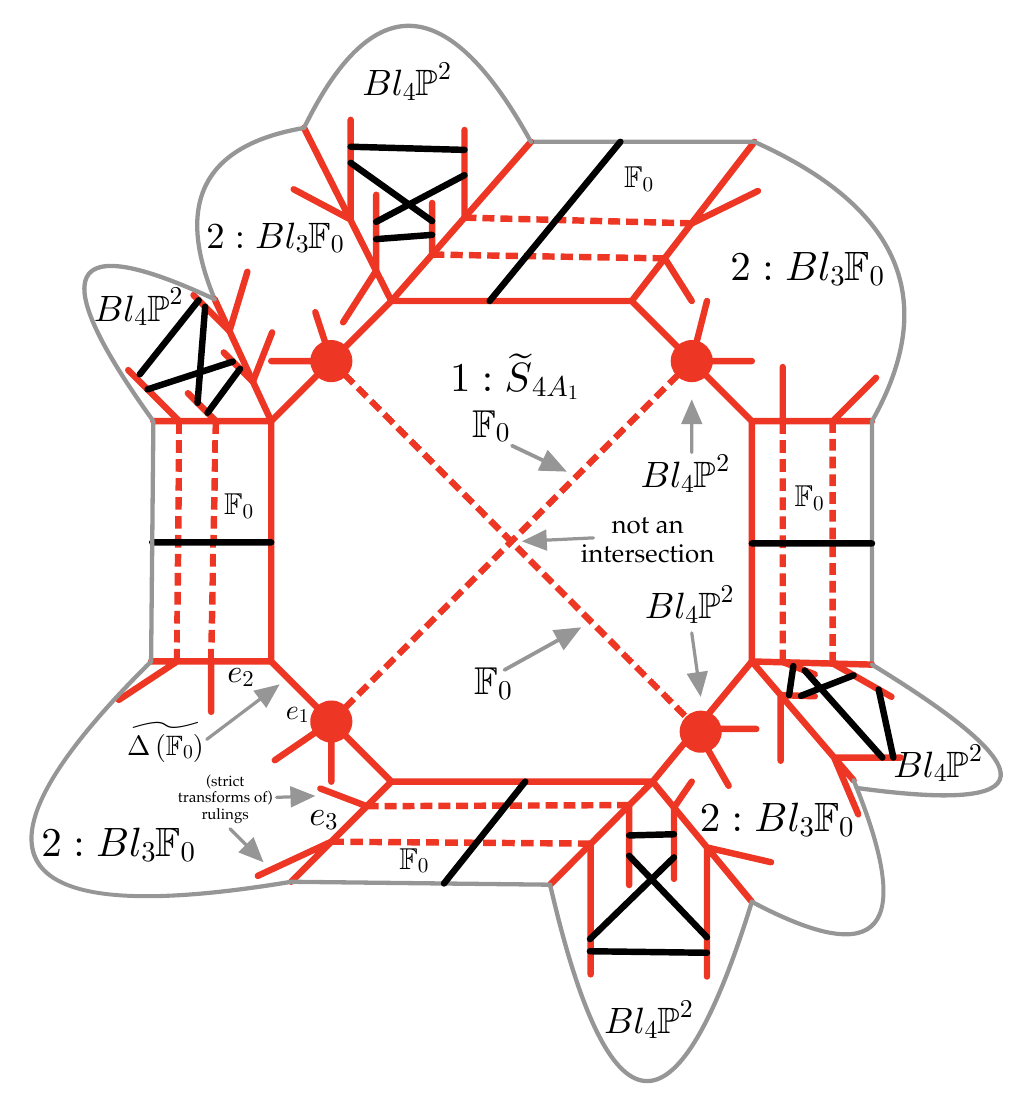}
        \caption{A type $aa_2a_3a_4$ surface for weights $1/2 < c \leq 2/3$.}
        \label{fig:aa2a3a4_12-23}
    \end{subfigure}
    \caption{(Continued) Degenerations of the type $a_4$ surface of \cref{fig:a4_12-23} into types $aa_4$, $a_2a_4$, $a_3a_4$,
    $aa_2a_4$, $aa_3a_4$, $a_2a_3a_4$, and $aa_2a_3a_4$ surfaces, for weights $1/2 < c \leq 2/3$. Each degeneration is
    isomorphic to the original $a_4$ surface of \cref{fig:a3_12-23}, away from the type 3 components which degenerate into
    one copy of $Bl_4\bP^2$ and three copies of $\bF_0$ as indicated (cf. \cref{fig:aa2_12-23,fig:a3_12-23_degens}).}%
    \label{fig:a4_12-23_degens_cont}
\end{figure}

\begin{table}[htpb]
    \centering
    \caption{The types of irreducible components of the weight $1/2 < c \leq 2/3$ surface of type $a_4$ pictured in
        \cref{fig:a4_12-23}, and the possible numbers of Eckardt points on each component. For the component of type 1,
        $\wS_{4A_1}$ refers to the minimal resolution of a cubic surface with four $A_1$ singularities. For the components of type
        2, $Bl_3\bF_0$ refers to the blowup of $\bF_0 \cong \bP^1 \times \bP^1$ at three points on the diagonal. For the
        components of type 3, $Bl_5\bP^2$ refers to the special blowup of $\bP^2$ at 5 points as in
    \cref{lem:nonflat_A32}.}
    \label{tab:a4_23-1}
    \begin{tabular}{| c | c | c | c |}
        \hline
        Label & Surface & \# & Eckardt points \\
        \hline
        \hline
        1 & $\wS_{4A_1}$ & 1 & 0 \\
        2 & $Bl_3\bF_0$ & 4 & 0 \\
        3 & $Bl_5\bP^2$ & 6 & 0 or 1 \\
        4a & $\bF_0$ & 24 & 0 \\
        4b & $\bF_0$ & 24 & 0 \\
        \hline
    \end{tabular}
\end{table}

\subsection{Type $b$}

\begin{proposition} \label{prop:b}
    For type $b$ weighted stable marked cubic surfaces $(S,cB)$, there are two walls.
    \begin{enumerate}
        \item For weights $1/9 < c \leq 1/3$, the type $b$ surfaces are described in \cref{fig:b_19-13}.
        \item For weights $1/3 < c \leq 2/3$, the type $b$ surfaces are described in \cref{fig:b_13-23}.
        \item For weights $2/3 < c \leq 1$, the type $b$ surfaces are obtained from the weight $1/3 < c \leq 2/3$ type
            $b$ surfaces by resolving Eckardt points as described in \cref{prop:resolve_eckardt}. The possible
            configurations of Eckardt points are summarized in \cref{tab:b_23-1}.
    \end{enumerate}
\end{proposition}

%\subsubsection{Weights $1/9 < c \leq 1/3$}

\begin{figure}[!htpb]
    \centering
    \includegraphics[width=0.4\linewidth]{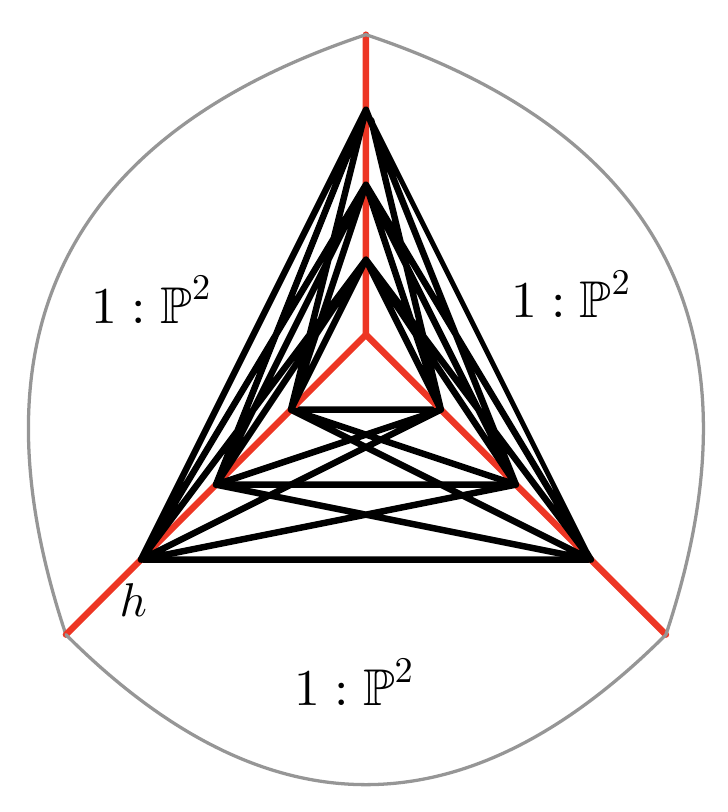}
    \caption{A type $b$ surface $(S_0,cB_0)$ for $1/9 < c \leq 1/3$, given as the union of three planes meeting
    transversally. Note that any two lines on a given component $\cong \bP^2$ intersect (possibly at infinity, not
    shown).}%
    \label{fig:b_19-13}
\end{figure}

%\subsubsection{Weights $1/3 < c \leq 2/3$}

\begin{figure}[!htpb]
    \centering
    \begin{subfigure}[t]{0.65\textwidth}
        \centering
        \includegraphics[width=0.8\linewidth]{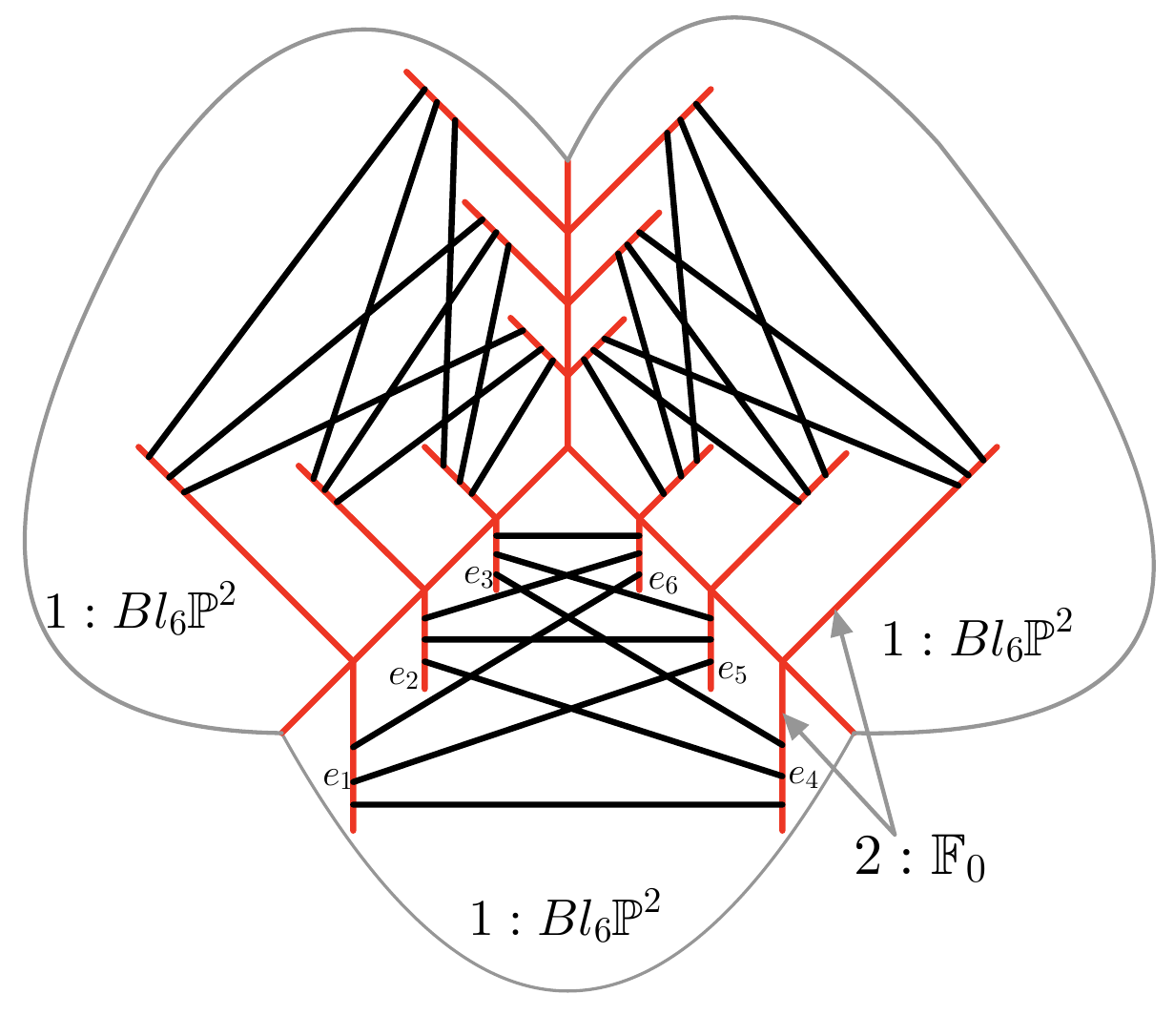}
        \caption{The type $b$ surface $(S_1,cB_1)$, obtained by blowing up 9 points on the surface $(S_0,cB_0)$ of
        \cref{fig:b_19-13}. At each exceptional divisor we attach a copy of $\bF_0$, as pictured in
    \cref{fig:b_13-23_2}.}
        \label{fig:b_13-23_1}
    \end{subfigure}
    \begin{subfigure}[t]{0.34\textwidth}
        \centering
        \includegraphics[width=0.9\linewidth]{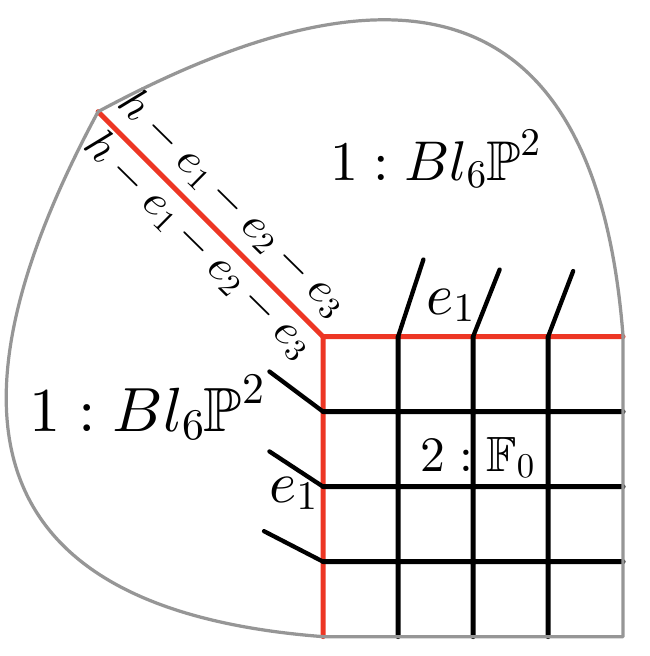}
        \caption{A view of one of the attached $\bF_0$ components in $(S_1,cB_1)$. It is glued to two different
        $Bl_6\bP^2$ components along its two rulings.}
        \label{fig:b_13-23_2}
    \end{subfigure}
    \caption{A type $b$ surface $(S_1,cB_1)$ for $1/3 < c \leq 2/3$. The surface $(S_1,cB_1)$ is obtained from the
    weight $1/9 < c \leq 1/3$ surface $(S_0,cB_0)$ of \cref{fig:b_19-13} by a sequence of transformations described in
\cref{ex:bottom_up_b}.}%
    \label{fig:b_13-23}
\end{figure}

\begin{table}[htpb]
    \centering
    \caption{The types of irreducible components of the weight $1/3 < c \leq 2/3$ surface of type $b$ pictured in
    \cref{fig:b_13-23}, and the possible numbers of Eckardt points on each component.
    For the components of type 1, $Bl_6\bP^2$ refers to the blowup of $\bP^2$ at 3 points on one line and 3 points on
    another line, as pictured in \cref{fig:b_19-13,fig:b_13-23}. Note this is the minimal resolution of a cubic surface
    with an $A_2$ singularity.}
    \label{tab:b_23-1}
    \begin{tabular}{| c | c | c | c |}
        \hline
        Label & Surface & \# & Eckardt points \\
        \hline
        \hline
        1 & $Bl_6\bP^2$ & 3 & 0, 1, 2, or 3 \\
        2 & $\bF_0$ & 9 & 0 \\
        \hline
    \end{tabular}
\end{table}

\subsection{Type $ab$}

\begin{proposition} \label{prop:ab}
    For type $ab$ weighted stable marked cubic surfaces $(S,cB)$, there are four walls.
    \begin{enumerate}
        \item For weights $1/9 < c \leq 1/6$, the type $ab$ surfaces are described in \cref{fig:ab_19-16}.
        \item For weights $1/6 < c \leq 1/3$, the type $ab$ surfaces are described in \cref{fig:ab_16-13}.
        \item For weights $1/3 < c \leq 1/2$, the type $ab$ surfaces are described in \cref{fig:ab_13-12}.
        \item For weights $1/2 < c \leq 2/3$, the type $ab$ surfaces are described in \cref{fig:ab_12-23}.
        \item For weights $2/3 < c \leq 1$, the type $ab$ surfaces are obtained from the weight $1/2 < c \leq 2/3$ type
            $ab$ surfaces by resolving Eckardt points as described in \cref{prop:resolve_eckardt}. The possible
            configurations of Eckardt points are summarized in \cref{tab:ab_23-1}.
    \end{enumerate}
\end{proposition}

%\subsubsection{Weights $1/9 < c \leq 1/6$}

\begin{figure}[!htpb]
    \centering
    \includegraphics[width=0.5\linewidth]{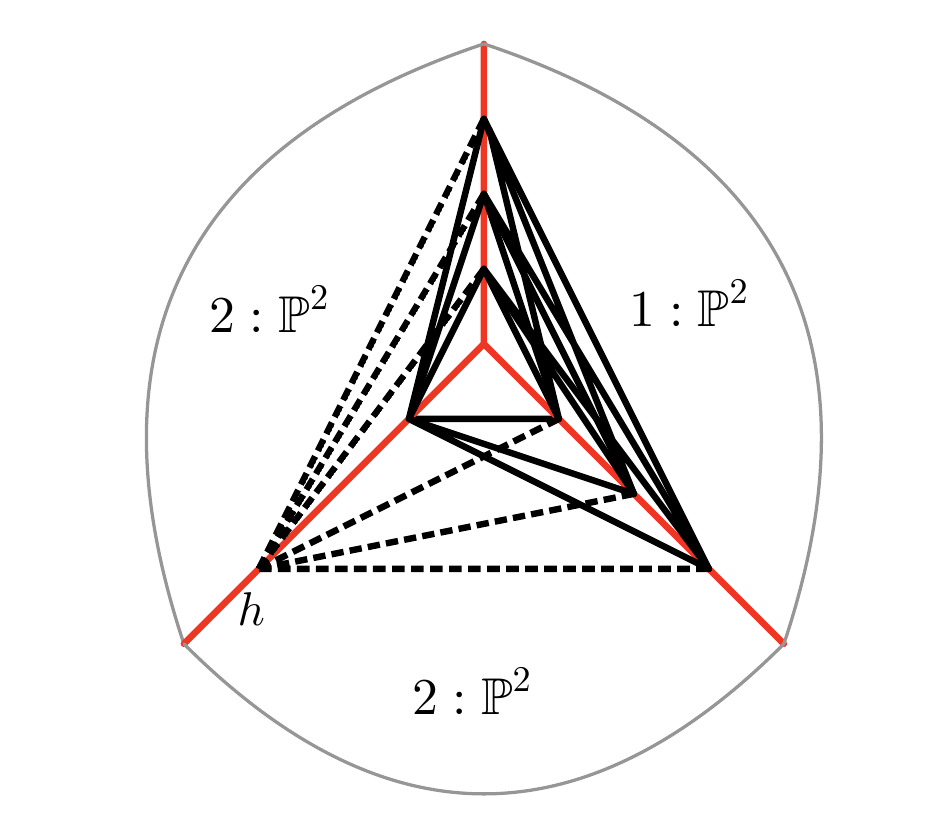}
    \caption{A type $ab$ surface $(S_0,cB_0)$ for $1/9 < c \leq 1/6$, given as the union of three planes meeting
    transversally. Note that any two lines on a given component $\cong \bP^2$ intersect (possibly at infinity, not
    shown).}%
    \label{fig:ab_19-16}
\end{figure}

%\subsubsection{Weights $1/6 < c \leq 1/3$}

\begin{figure}[!htpb]
    \centering
    \includegraphics[width=0.5\linewidth]{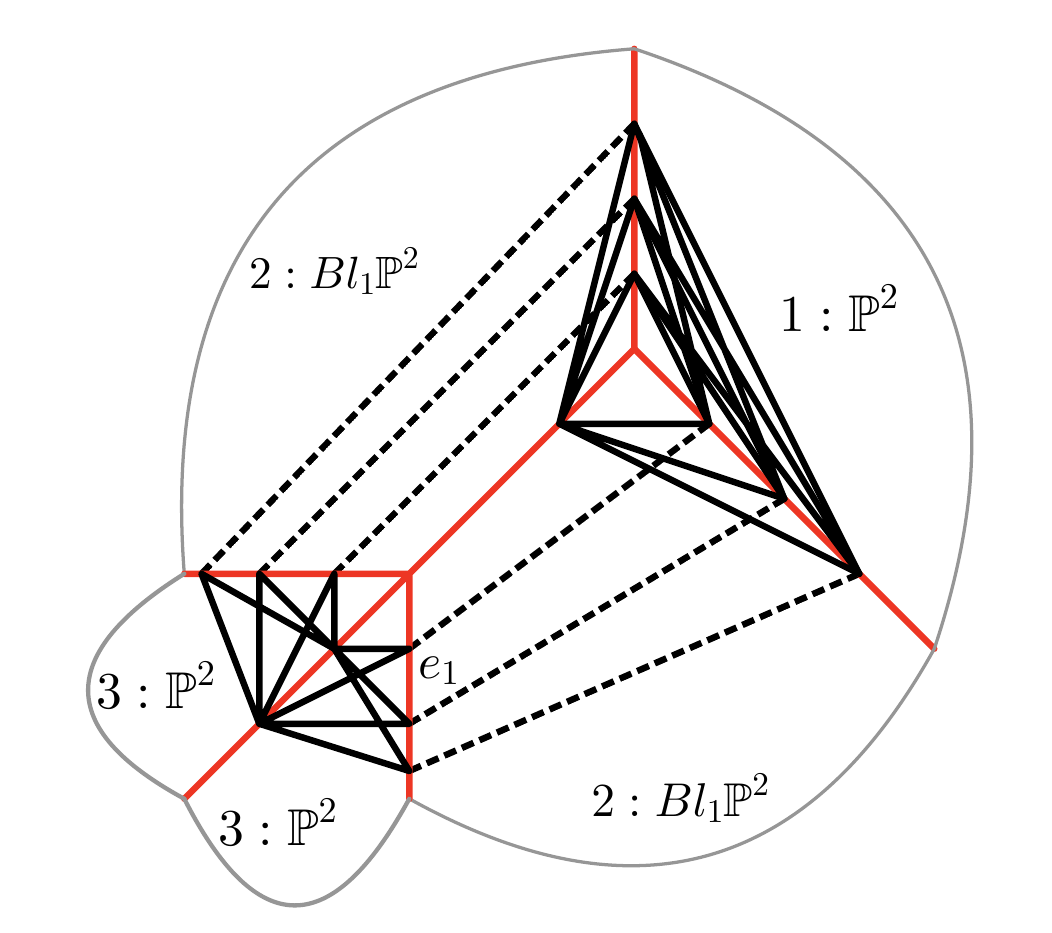}
    \caption{A type $ab$ surface $(S_1,cB_1)$ for $1/6 < c \leq 1/3$, obtained from the weight $1/9 < c \leq 1/6$
    surface $(S_0,cB_0)$ of \cref{fig:ab_19-16} by blowing up the point in the gluing locus where the lines of
    multiplicity 2 intersect, and attaching to the exceptional divisors 2 copies of $\bP^2$.}%
    \label{fig:ab_16-13}
\end{figure}

%\subsubsection{Weights $1/3 < c \leq 1/2$}

\begin{figure}[!htpb]
    \centering
    \begin{subfigure}[t]{0.8\textwidth}
        \centering
        \includegraphics[width=0.9\linewidth]{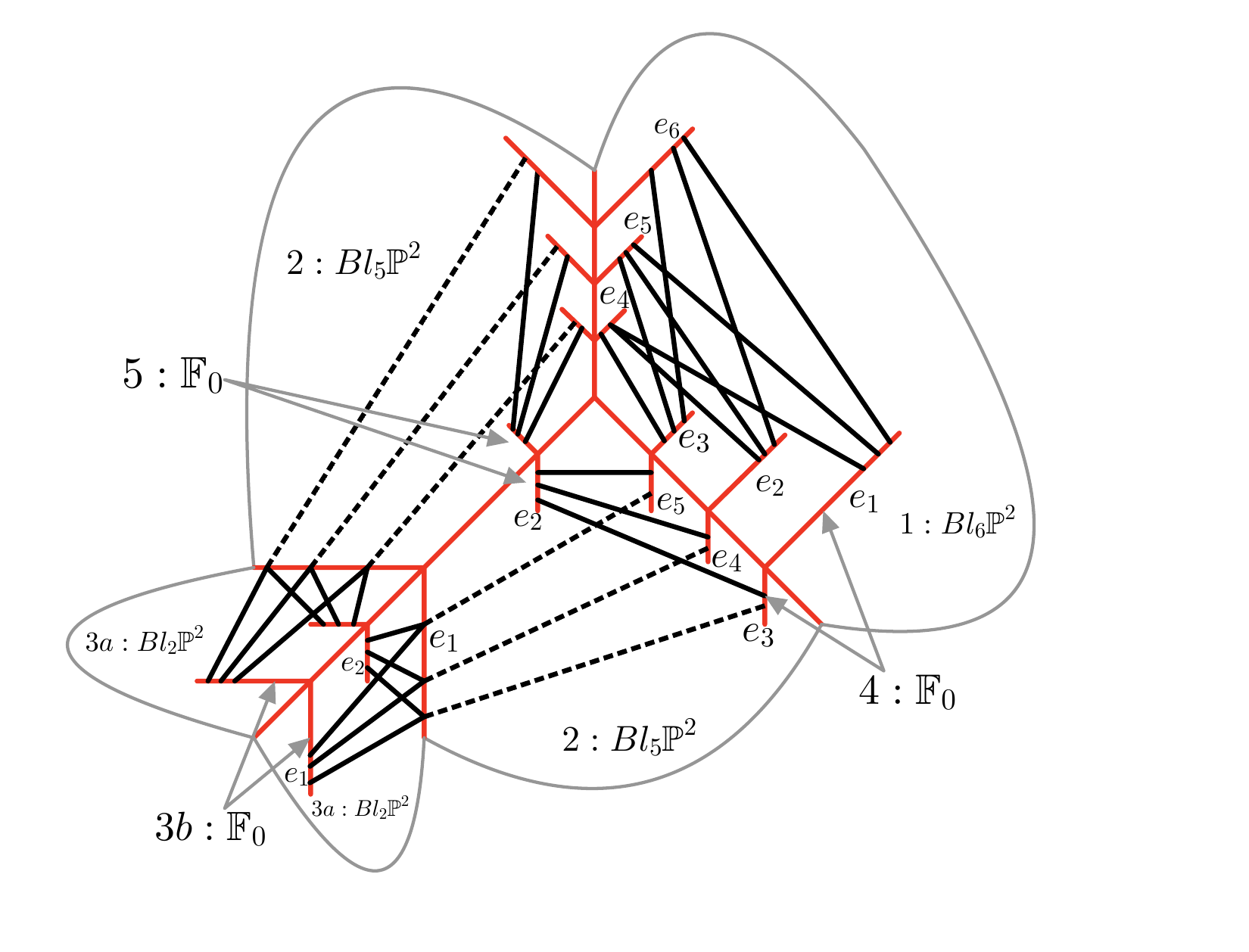}
        \caption{The type $ab$ surface $(S_2,cB_2)$ obtained as the stable replacement of \cref{fig:ab_16-13} for
        weights $1/3 < c \leq 1/2$.}
        \label{fig:ab_13-12_1}
    \end{subfigure}
    \begin{subfigure}[t]{0.3\textwidth}
        \centering
        \includegraphics[width=0.9\linewidth]{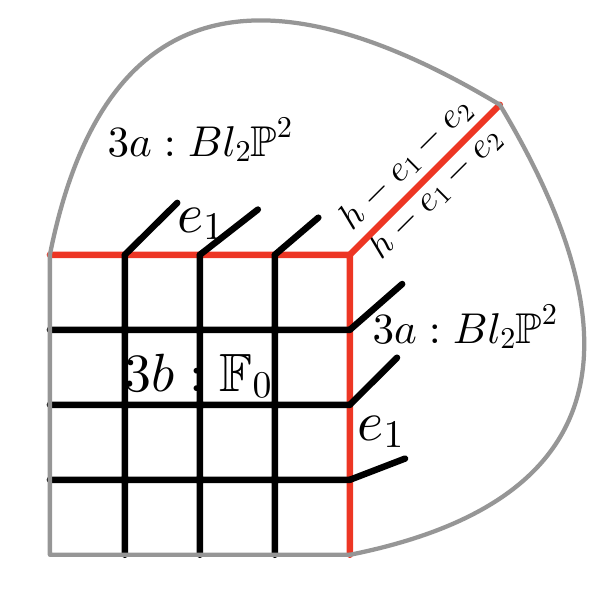}
        \caption{A view of an $\bF_0$ component of type 3b.}
        \label{fig:ab_13-12_2}
    \end{subfigure}
    \begin{subfigure}[t]{0.3\textwidth}
        \centering
        \includegraphics[width=0.9\linewidth]{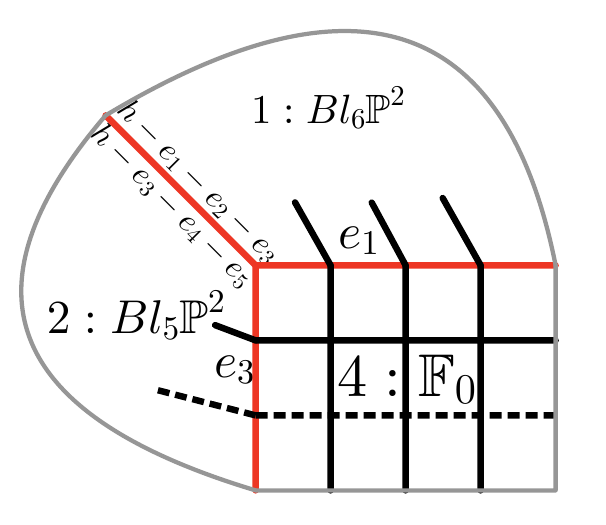}
        \caption{A view of an $\bF_0$ component of type 4.}
        \label{fig:ab_13-12_3}
    \end{subfigure}
    \begin{subfigure}[t]{0.3\textwidth}
        \centering
        \includegraphics[width=0.9\linewidth]{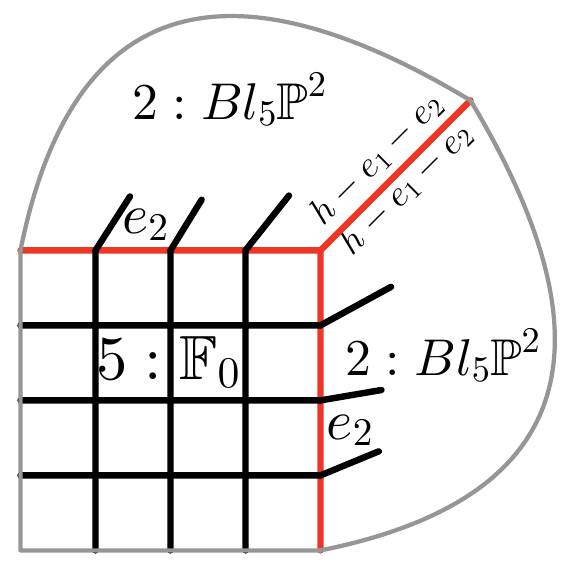}
        \caption{A view of an $\bF_0$ component of type 5.}
        \label{fig:ab_13-12_4}
    \end{subfigure}
    \caption{A type $ab$ surface $(S_2,cB_2)$ for $1/3 < c \leq 1/2$, obtained from the surface $(S_1,cB_1)$ of
    \cref{fig:ab_19-16} by blowing up a total of 9 points, and attaching to the exceptional divisors copies of $\bF_0$.}%
    \label{fig:ab_13-12}
\end{figure}

%\subsubsection{Weights $1/2 < c \leq 2/3$}

\begin{figure}[!htpb]
    \centering
    \begin{subfigure}[t]{0.8\textwidth}
        \centering
        \includegraphics[width=0.9\linewidth]{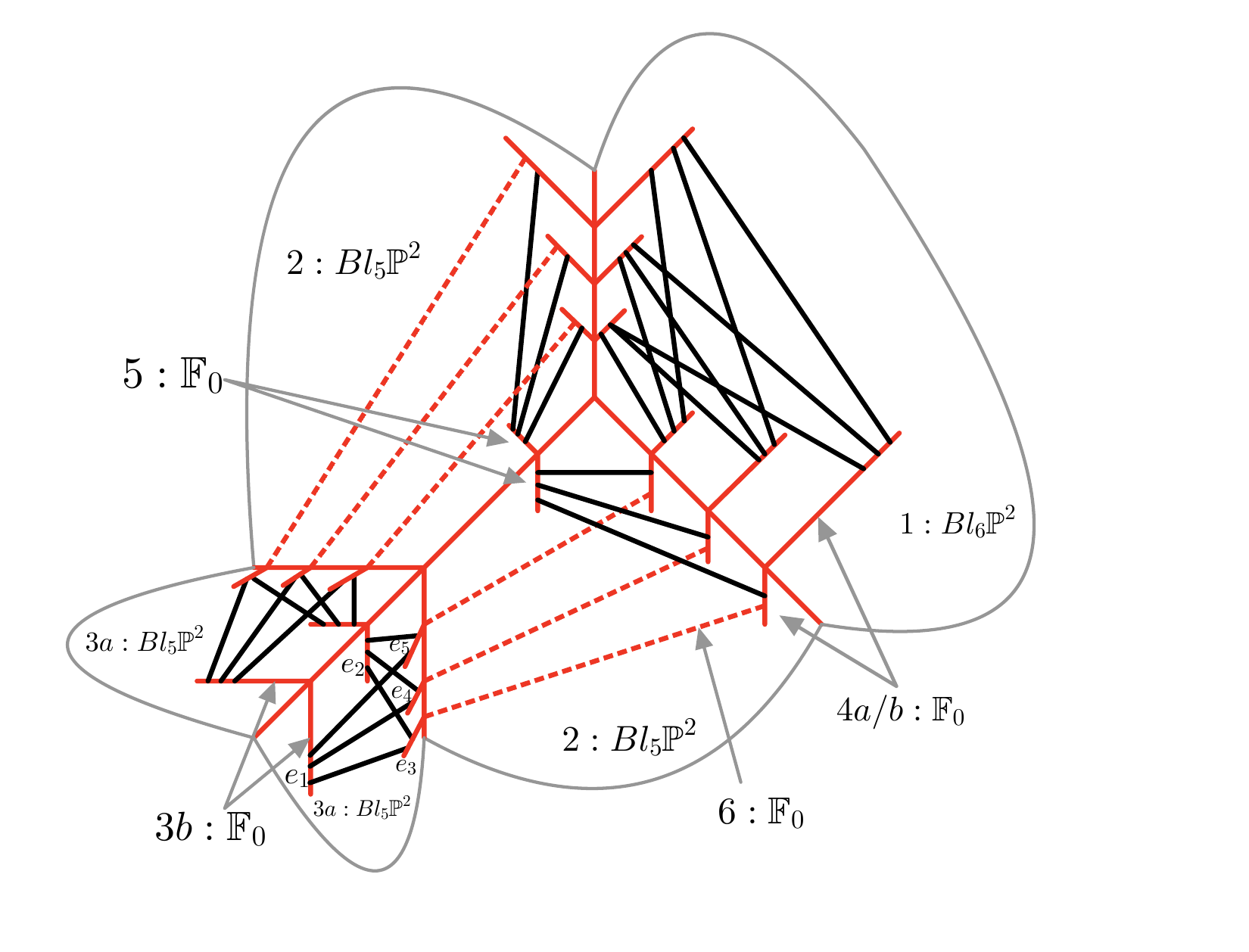}
        \caption{The type $ab$ surface $(S_3,cB_3)$ obtained as the stable replacement of \cref{fig:ab_13-12} for
        weights $1/2 < c \leq 2/3$.}
        \label{fig:ab_12-23_1}
    \end{subfigure}
    \begin{subfigure}[t]{0.25\textwidth}
        \centering
        \includegraphics[width=0.9\linewidth]{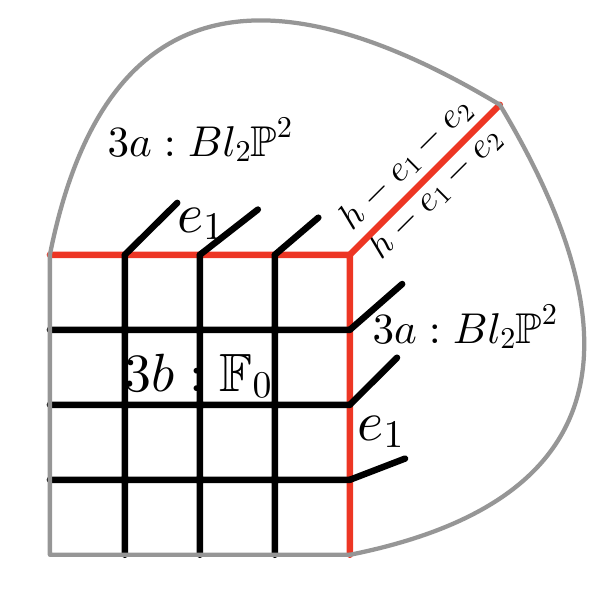}
        \caption{A view of an $\bF_0$ component of type 3b.}
        \label{fig:ab_12-23_2}
    \end{subfigure}
    \begin{subfigure}[t]{0.4\textwidth}
        \centering
        \includegraphics[width=0.9\linewidth]{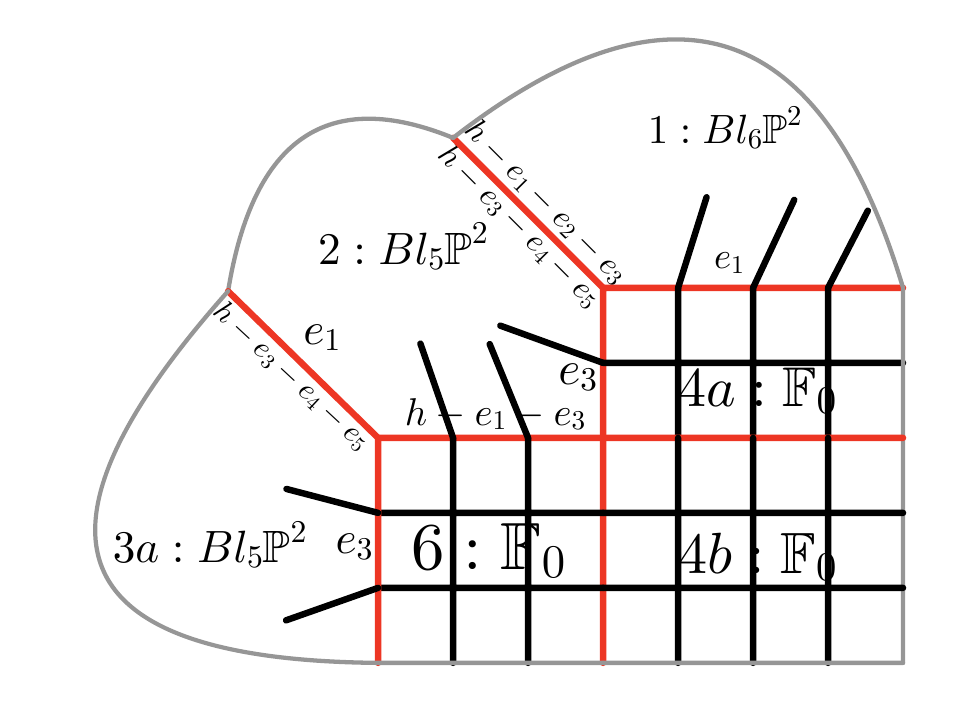}
        \caption{A view of an $\bF_0$ component of type 4, as well as an $\bF_0$ component of type 6. Note that the
        corresponding type 4 component from \cref{fig:ab_13-12_3} has broken into two copies of $\bF_0$, now labeled by
        type 4a and 4b.}
        \label{fig:ab_12-23_3}
    \end{subfigure}
    \begin{subfigure}[t]{0.25\textwidth}
        \centering
        \includegraphics[width=0.9\linewidth]{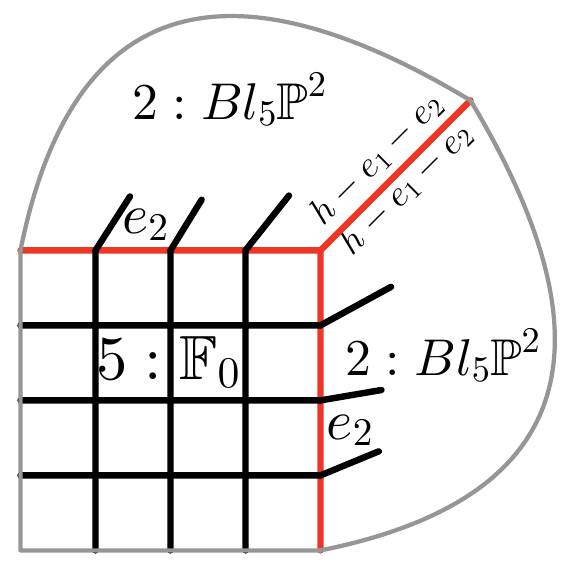}
        \caption{A view of an $\bF_0$ component of type 5.}
        \label{fig:ab_12-23_4}
    \end{subfigure}
    \caption{A type $ab$ surface $(S_3,cB_3)$ for $1/2 < c \leq 2/3$, obtained as the stable replacement of the surface
    $(S_2,cB_2)$ of \cref{fig:ab_13-12}. This is obtained by blowing up all lines of multiplicity 2 on $(S_2,cB_2)$, and
    attaching to the exceptional divisors copies of $\bF_0$.}%
    \label{fig:ab_12-23}
\end{figure}

\begin{table}[htpb]
    \centering
    \caption{The types of irreducible components of the weight $1/2 < c \leq 2/3$ surface of type $ab$ pictured in
    \cref{fig:ab_12-23}, and the possible numbers of Eckardt points on each component.
    For the component of type 1, $Bl_6\bP^2$ refers to the blowup of $\bP^2$ at 3 points on one line and 3 points on
    another line, cf. \cref{tab:b_23-1}. For the components of types 2 and 3a, $Bl_5\bP^2$ refers to the blowup of $\bP^2$
    at 3 points on one line and 2 points on another line, as pictured in \cref{fig:ab_12-23}.}
    \label{tab:ab_23-1}
    \begin{tabular}{| c | c | c | c |}
        \hline
        Label & Surface & \# & Eckardt points \\
        \hline
        \hline
        1 & $Bl_6\bP^2$ & 1 & 0, 1, 2, or 3 \\
        2 & $Bl_5\bP^2$ & 2 & 0 \\
        3a & $Bl_5\bP^2$ & 2 & 0 \\
        3b & $\bF_0$ & 2 & 0 \\
        4a & $\bF_0$ & 6 & 0 \\
        4b & $\bF_0$ & 6 & 0 \\
        5 & $\bF_0$ & 1 & 0 \\
        6 & $\bF_0$ & 6 & 0 \\
        \hline
    \end{tabular}
\end{table}

\subsection{Type $a_2b$}
\begin{proposition} \label{prop:a2b}
    For type $a_2b$ weighted stable marked cubic surfaces $(S,cB)$, there are five walls.
    \begin{enumerate}
        \item For weights $1/9 < c \leq 1/6$, the type $a_2b$ surfaces are described in \cref{fig:a2b_19-16}.
        \item For weights $1/6 < c \leq 1/4$, the type $a_2b$ surfaces are described in \cref{fig:a2b_16-14}.
        \item For weights $1/4 < c \leq 1/3$, the type $a_2b$ surfaces are described in \cref{fig:a2b_14-13}.
            Additionally, crossing the wall $c=1/4$ introduces type $aa_2b$ surfaces as degenerations of $a_2b$
            surfaces, described in \cref{fig:aa2b_14-13}.
        \item For weights $1/3 < c \leq 1/2$, the type $a_2b$ surfaces are described in \cref{fig:a2b_13-12}. The type
            $aa_2b$ surfaces are described in \cref{fig:aa2b_13-12}.
        \item For weights $1/2 < c \leq 2/3$, the type $a_2b$ surfaces are described in \cref{fig:a2b_12-23}. The type
            $aa_2b$ surfaces are described in \cref{fig:aa2b_12-23}.
        \item For weights $2/3 < c \leq 1$, the type $a_2b$ surfaces are obtained from the weight $1/2 < c \leq 2/3$ type
            $a_2b$ surfaces by resolving Eckardt points as described in \cref{prop:resolve_eckardt}. The possible
            configurations of Eckardt points are summarized in \cref{tab:a2b_23-1}. The type $aa_2b$ surfaces are
            described similarly.
    \end{enumerate}
\end{proposition}

%\subsubsection{Weights $1/9 < c \leq 1/6$}

\begin{figure}[!htpb]
    \centering
    \includegraphics[width=0.45\linewidth]{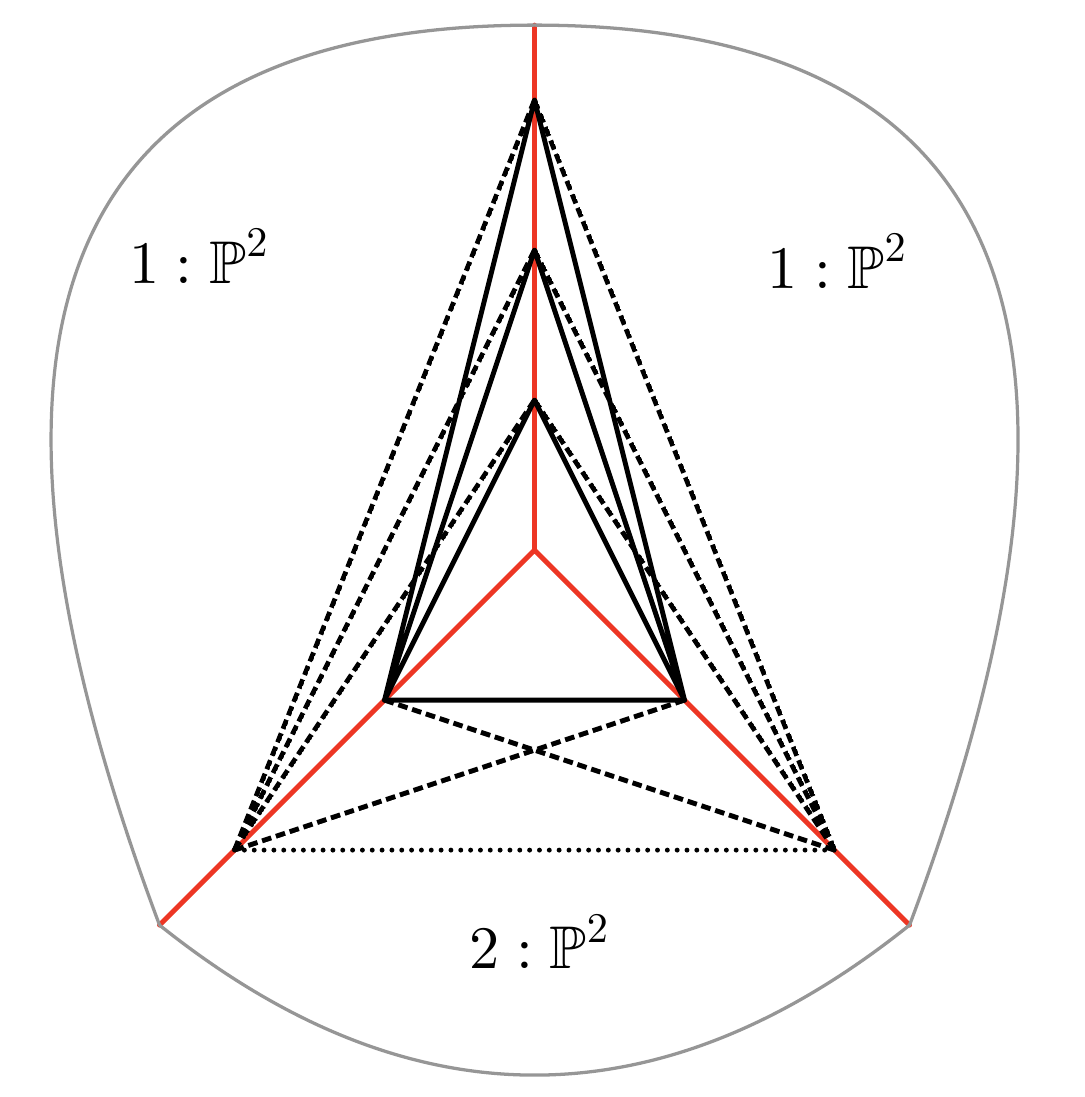}
    \caption{A type $a_2b$ surface $(S_0,cB_0)$ for $1/9 < c \leq 1/6$. Recall dotted lines have multiplicity 4 and
    dashed lines have multiplicity 2. Any two lines on a given $\bP^2$ component intersect (possibly at infinity, not
    shown).}%
    \label{fig:a2b_19-16}
\end{figure}

%\subsubsection{Weights $1/6 < c \leq 1/4$}

\begin{figure}[!htpb]
    \centering
    \includegraphics[width=0.6\linewidth]{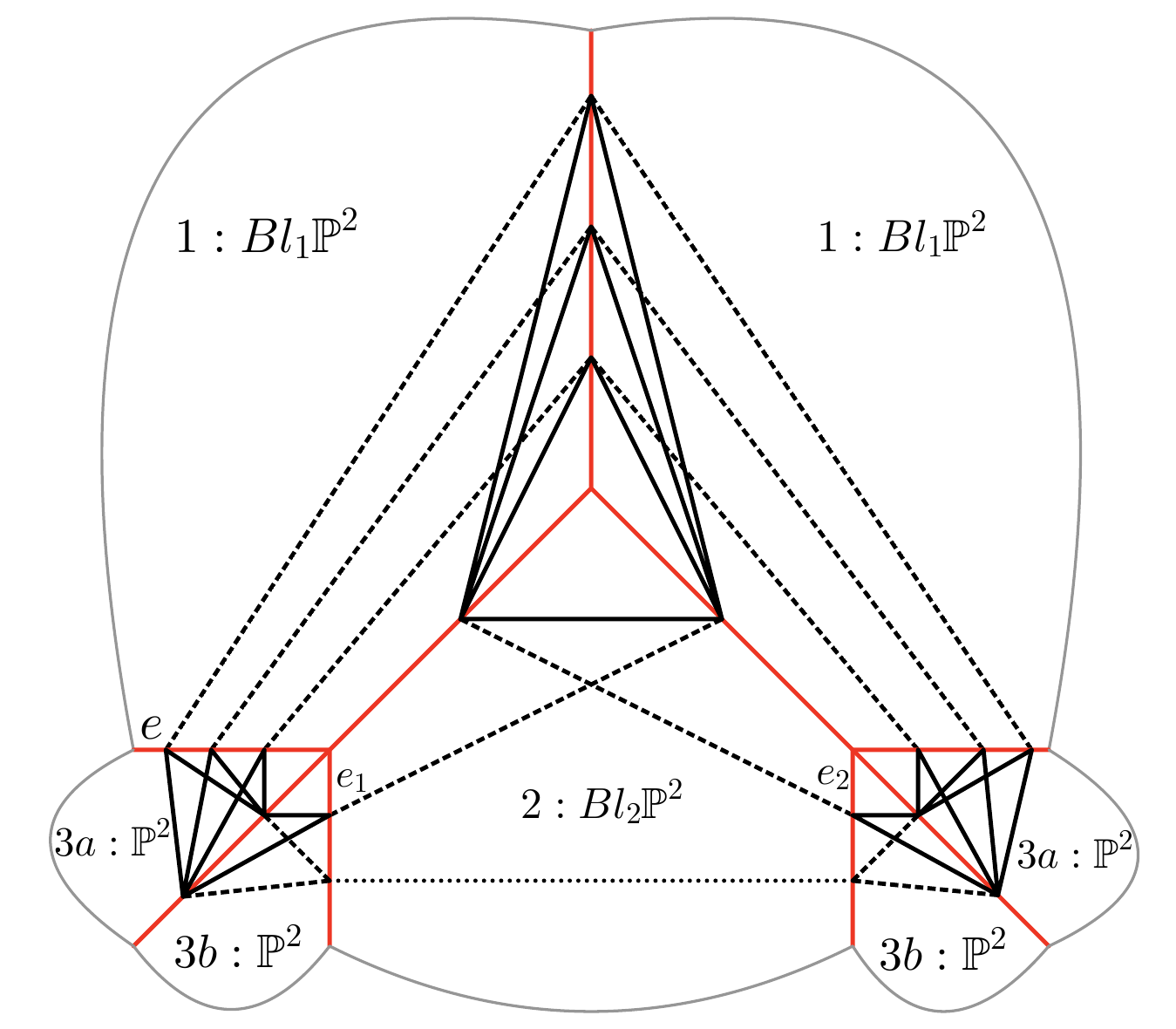}
    \caption{A type $a_2b$ surface $(S_1,cB_1)$ for $1/6 < c \leq 1/4$, obtained as the stable replacement for these
    weights of the surface of \cref{fig:a2b_19-16}. The surface $(S_1,cB_1)$ has four more components than $(S_0,cB_0)$,
    two each of types 3a and 3b.}%
    \label{fig:a2b_16-14}
\end{figure}

%\subsubsection{Weights $1/4 < c \leq 1/3$}

\begin{figure}[!htpb]
    \centering
    \includegraphics[width=0.6\linewidth]{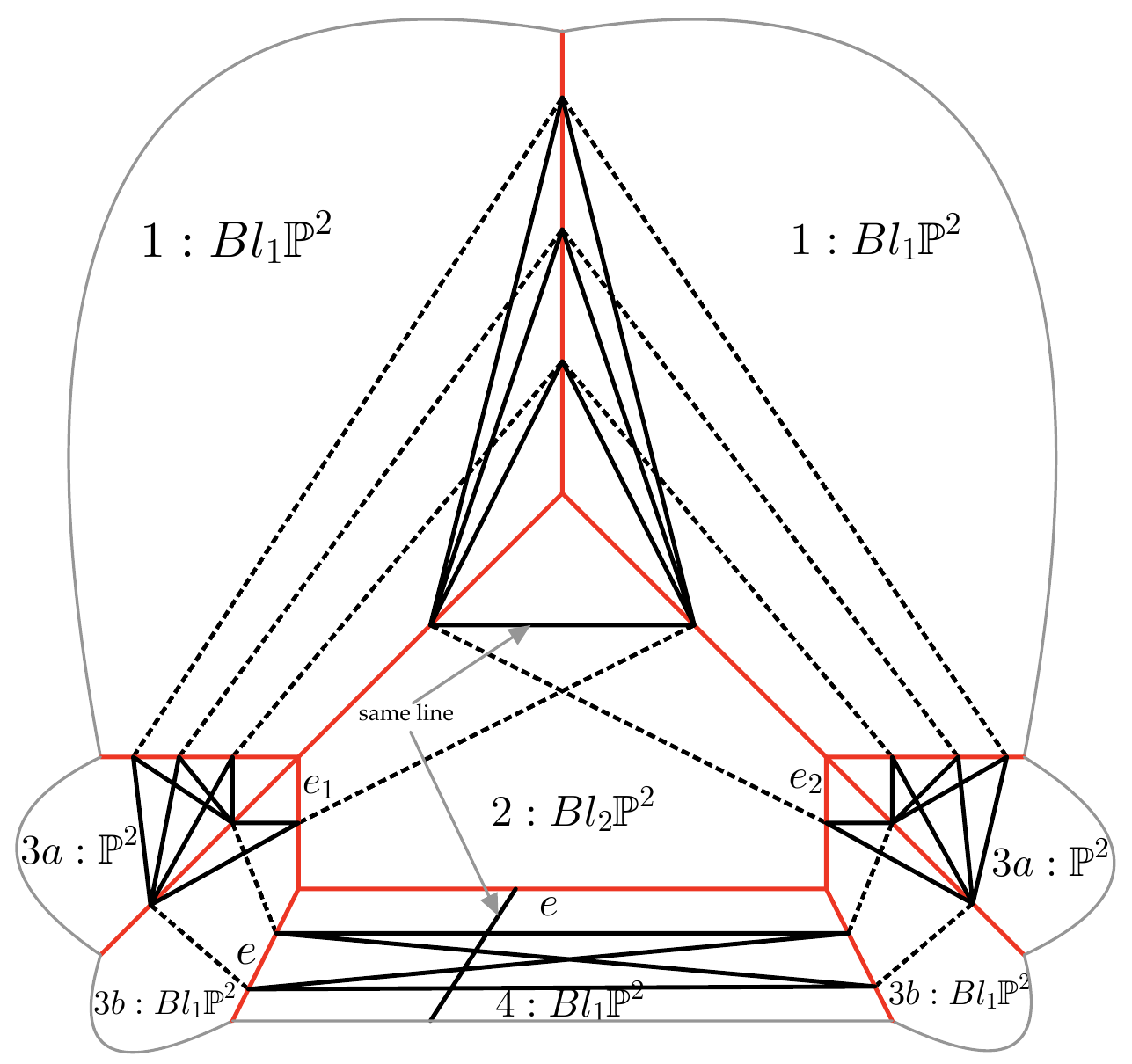}
    \caption{A type $a_2b$ surface $(S_2,cB_2)$ for $1/4 < c \leq 1/3$, obtained as the stable replacement for these
    weights of the surface of \cref{fig:a2b_16-14}. The surface $(S_2,cB_2)$ has one more component than $(S_1,cB_1)$,
    a copy of $Bl_1\bP^2$ labeled by type 4.}%
    \label{fig:a2b_14-13}
\end{figure}

\begin{figure}[!htpb]
    \centering
    \includegraphics[width=0.6\linewidth]{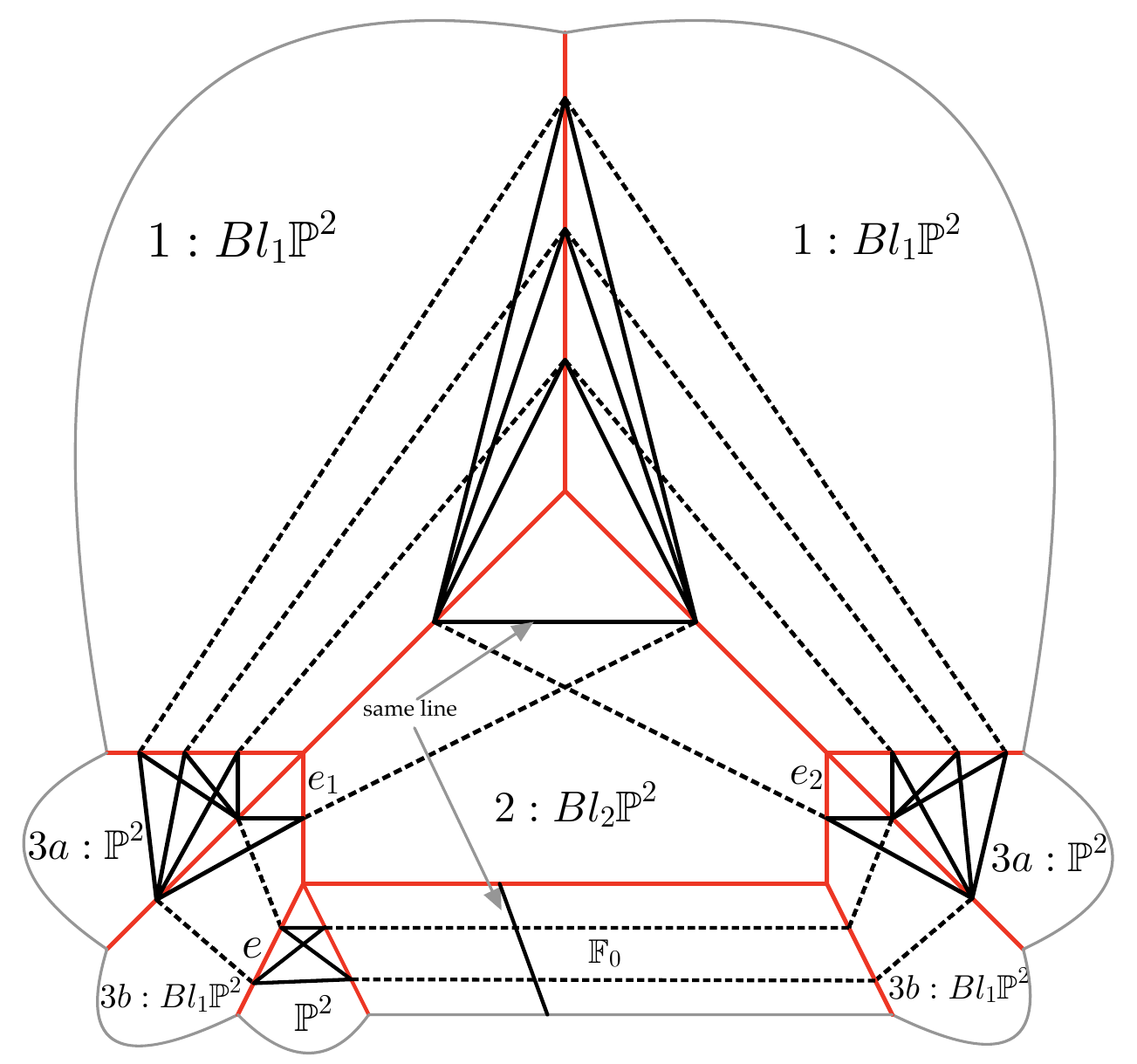}
    \caption{A type $aa_2b$ surface $(S_2',cB_2')$ for $1/4 < c \leq 1/3$, obtained as a degeneration of the type $a_2b$
    surface $(S_2,cB_2)$ of \cref{fig:a2b_14-13}, where the type 4 component $\cong Bl_1\bP^2$ splits into a $\bP^2$
    component and an $\bF_0$ component (cf. \cref{fig:aa2_14-12,fig:a3_14-12_degens,fig:a4_14-12_degens}).}%
    \label{fig:aa2b_14-13}
\end{figure}

%\subsubsection{Weights $1/3 < c \leq 1/2$}

\begin{figure}[!htpb]
    \centering
    \begin{subfigure}[t]{0.75\textwidth}
        \centering
        \includegraphics[width=0.9\linewidth]{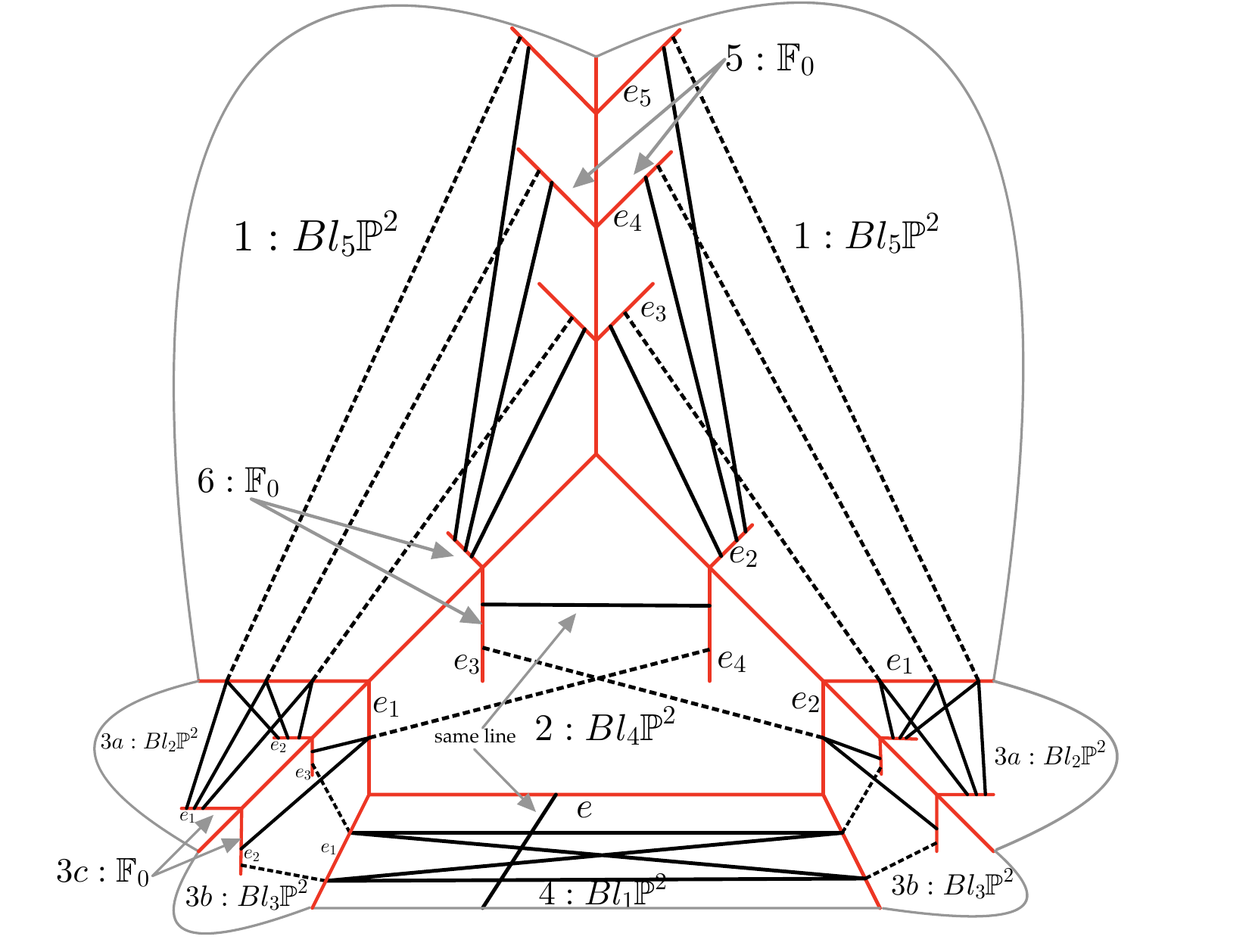}
        \caption{The stable replacement $(S_3,cB_3)$ of the surface of \cref{fig:a2b_14-13}, for weights $1/3 < c \leq
        1/2$. See the adjacent subfigures for local views of the new components of types 3c, 5, and 6.}
        \label{fig:a2b_13-12_1}
    \end{subfigure}
    \begin{subfigure}[t]{0.33\textwidth}
        \centering
        \includegraphics[width=0.9\linewidth]{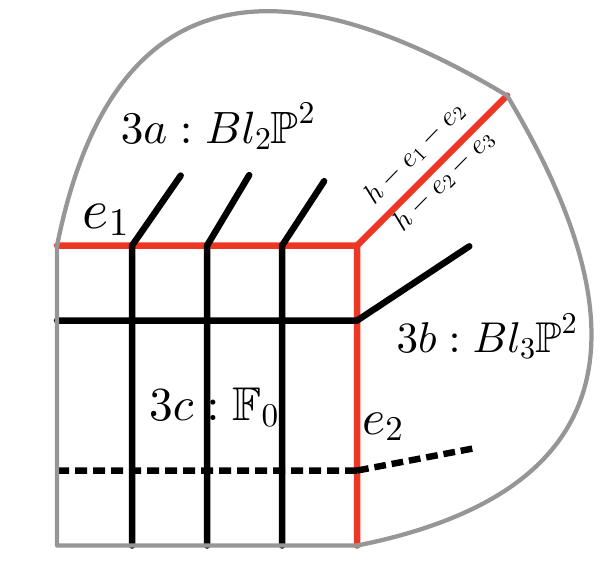}
        \caption{A view of an $\bF_0$ component of type 3c.}
        \label{fig:a2b_13-12_2}
    \end{subfigure}
    \begin{subfigure}[t]{0.33\textwidth}
        \centering
        \includegraphics[width=0.9\linewidth]{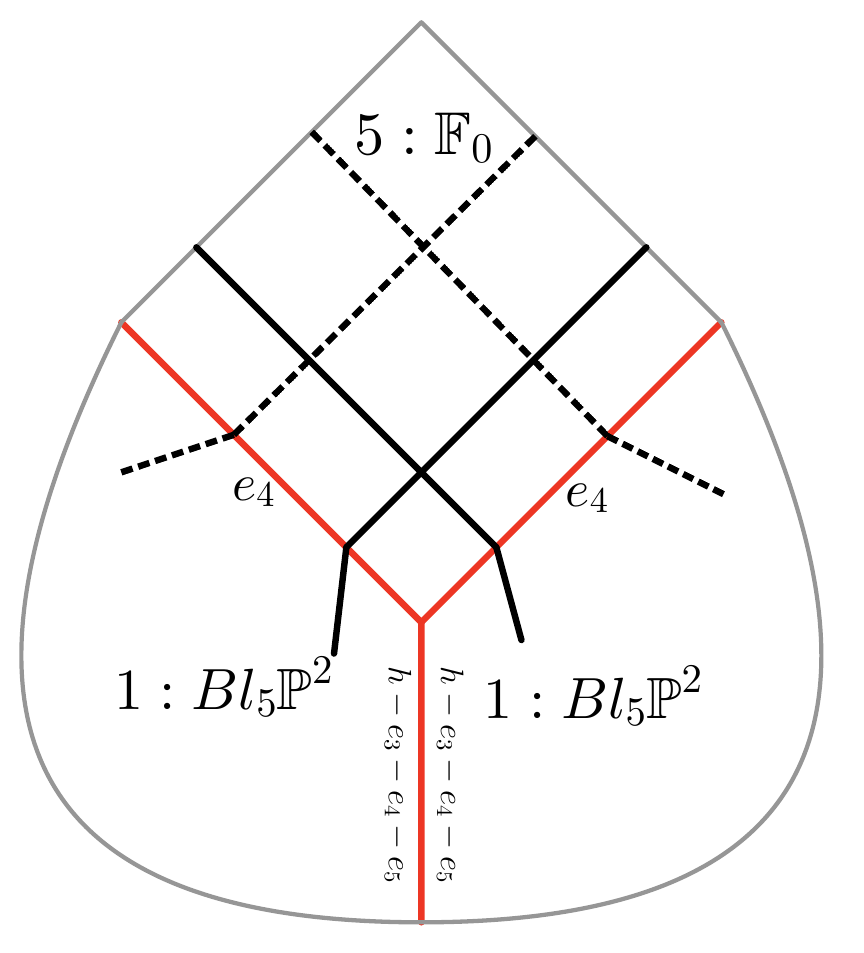}
        \caption{A view of an $\bF_0$ component of type 5.}
        \label{fig:a2b_13-12_3}
    \end{subfigure}
    \begin{subfigure}[t]{0.33\textwidth}
        \centering
        \includegraphics[width=0.9\linewidth]{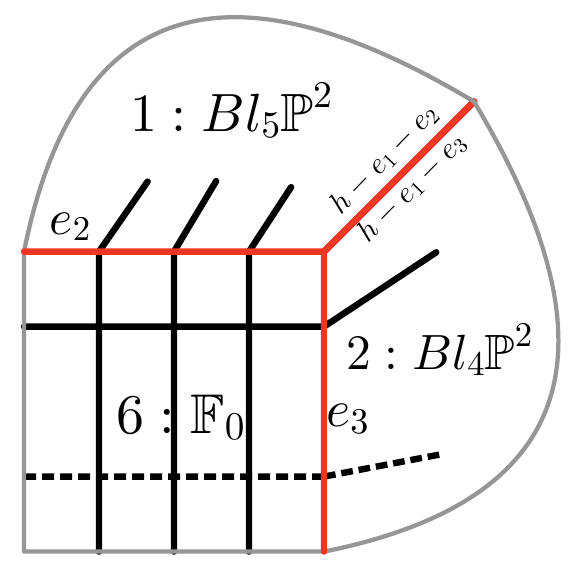}
        \caption{A view of an $\bF_0$ component of type 6.}
        \label{fig:a2b_13-12_4}
    \end{subfigure}
    \caption{A type $a_2b$ surface $(S_3,cB_3)$ for $1/3 < c \leq 1/2$, obtained as the stable replacement for these
    weights of the surface of \cref{fig:a2b_14-13}. The surface $(S_3,cB_3)$ has nine more components than $(S_2,cB_2)$:
    four of type 3c, three of type 5, and two of type 6.}%
    \label{fig:a2b_13-12}
\end{figure}

\begin{figure}[!htpb]
    \centering
    \includegraphics[width=0.6\linewidth]{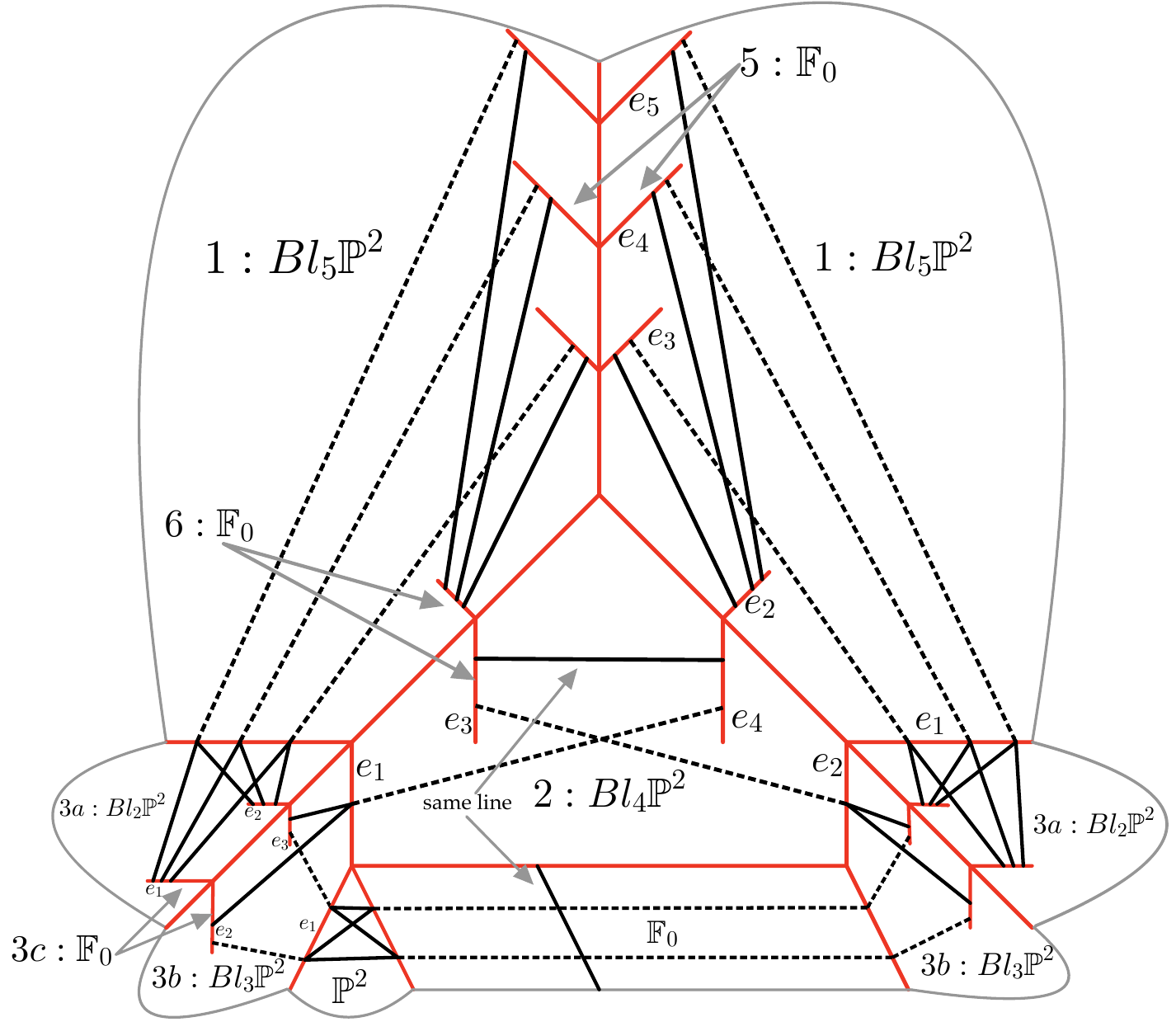}
    \caption{A type $aa_2b$ surface $(S_3',cB_3')$ for $1/3 < c \leq 1/2$, obtained as the stable replacement for these
    weights of the surface of \cref{fig:aa2b_14-13}, and as a degeneration of the type $a_2b$ surface $(S_3,cB_3)$ of
    \cref{fig:a2b_13-12}. Away from the $Bl_1\bP^2$ component of type 4 splitting into a $\bP^2$ and $\bF_0$ component as
    shown, the surface $(S_3',cB_3')$ is isomorphic to $(S_3,cB_3)$.}%
    \label{fig:aa2b_13-12}
\end{figure}

%\subsubsection{Weights $1/2 < c \leq 2/3$}

\begin{figure}[!htpb]
    \centering
    \begin{subfigure}[t]{0.75\textwidth}
        \centering
        \includegraphics[width=0.9\linewidth]{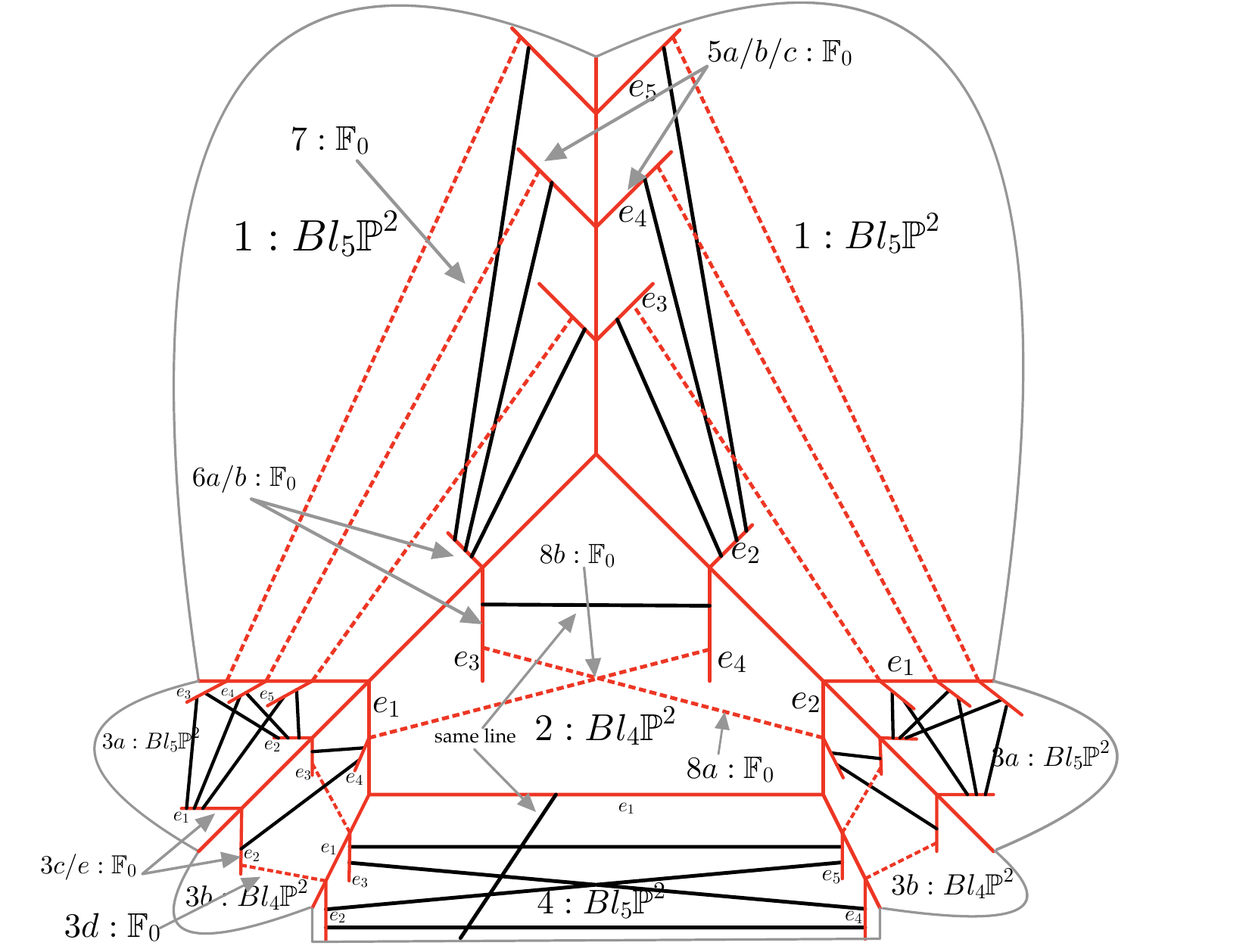}
        \caption{The stable replacement $(S_3,cB_3)$ of the surface of \cref{fig:a2b_14-13}, for weights $1/2 < c \leq
        2/3$. See the adjacent subfigures for local views of the new components of types 3c, 5, and 6.}
        \label{fig:a2b_12-23_1}
    \end{subfigure}
    \begin{subfigure}[t]{0.4\textwidth}
        \centering
        \includegraphics[width=0.9\linewidth]{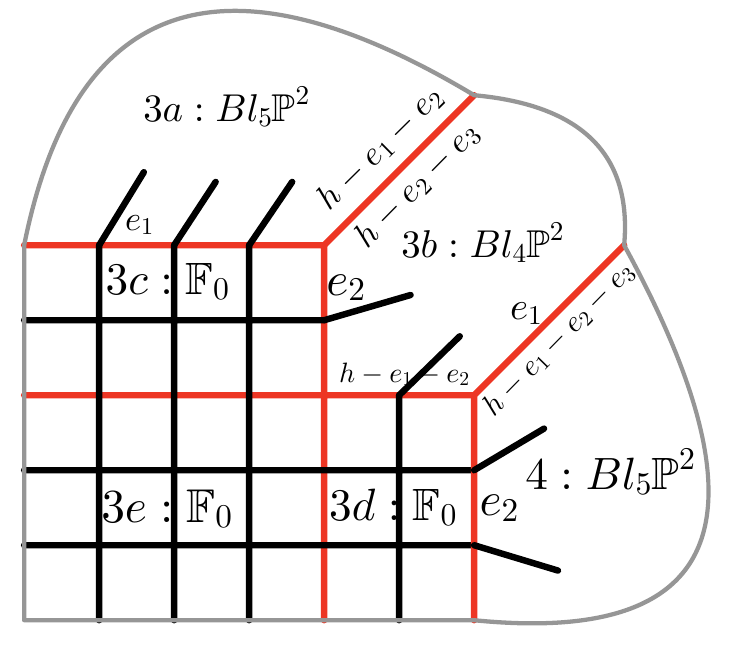}
        \caption{In $(S_4,cB_4)$, one attaches to each line of multiplicity 2 in a type 3b component a copy of $\bF_0$,
        labeled by type 3d. Likewise, one attaches to the corresponding line of multiplicity 2 in a type 3c component
        another copy of $\bF_0$, labeled by type 3e. The configuration of all 5 different types of type 3 components is
        shown.}
        \label{fig:a2b_12-23_2}
    \end{subfigure}
    \begin{subfigure}[t]{0.5\textwidth}
        \centering
        \includegraphics[width=0.9\linewidth]{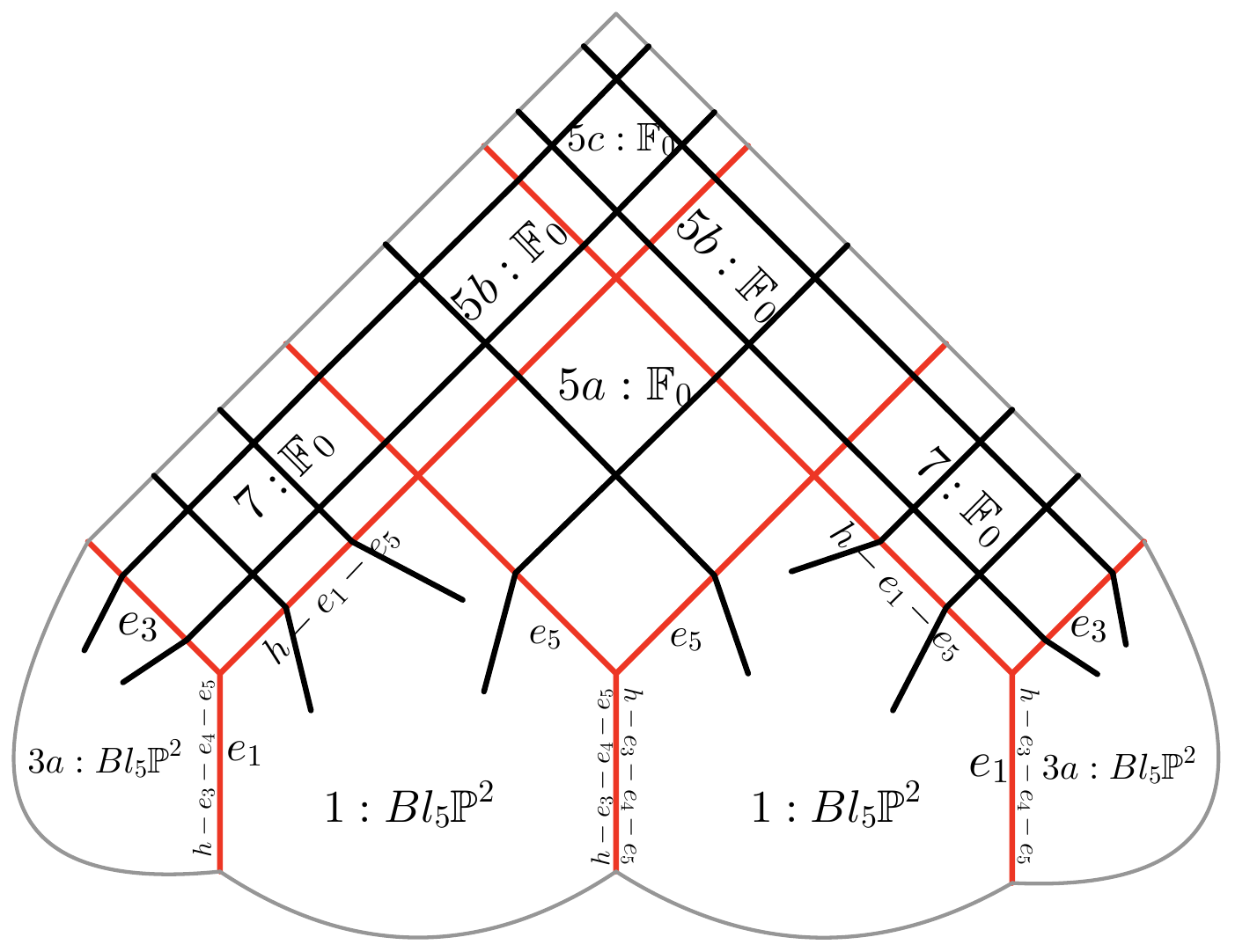}
        \caption{Each type 5 component $\bF_0$ from the surface $(S_3,cB_3)$ of \cref{fig:a2b_13-12} splits into four
        components, one of type 5a, two of type 5b, and one of type 5c. The type 5b components are also attached to the
        type 7 components, which arise from the lines of multiplicity 2 on the type 1 components of $(S_3,cB_3)$. These
        type 7 components also intersect the type 3a components as shown.}
        \label{fig:a2b_12-23_3}
    \end{subfigure}
    \caption{A type $a_2b$ surface $(S_4,cB_4)$ for $1/2 < c \leq 2/3$, obtained as the stable replacement for these
    weights of the surface of \cref{fig:a2b_13-12}. The surface $(S_3,cB_3)$ has nine more components than $(S_2,cB_2)$:
    four of type 3c, three of type 5, and two of type 6.}%
    \label{fig:a2b_12-23}
\end{figure}
\begin{figure}[!htpb]\ContinuedFloat
    \centering
    \begin{subfigure}[t]{0.4\textwidth}
        \centering
        \includegraphics[width=0.9\linewidth]{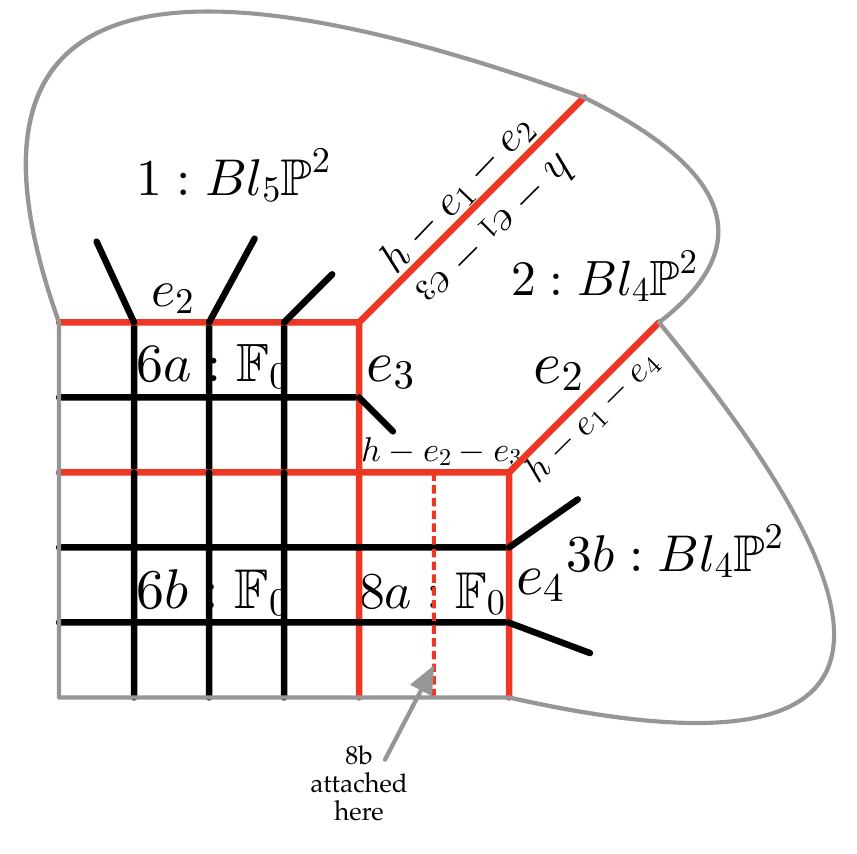}
        \caption{Each type 6 component $\bF_0$ from the surface $(S_3,cB_3)$ of \cref{fig:a2b_13-12} splits into two
        components, one of type 6a and one of type 6b. Attached to the type 6b component is a type 8a component as
        described in \cref{fig:a2b_12-23_5}.}
        \label{fig:a2b_12-23_4}
    \end{subfigure}
    \begin{subfigure}[t]{0.4\textwidth}
        \centering
        \includegraphics[width=0.9\linewidth]{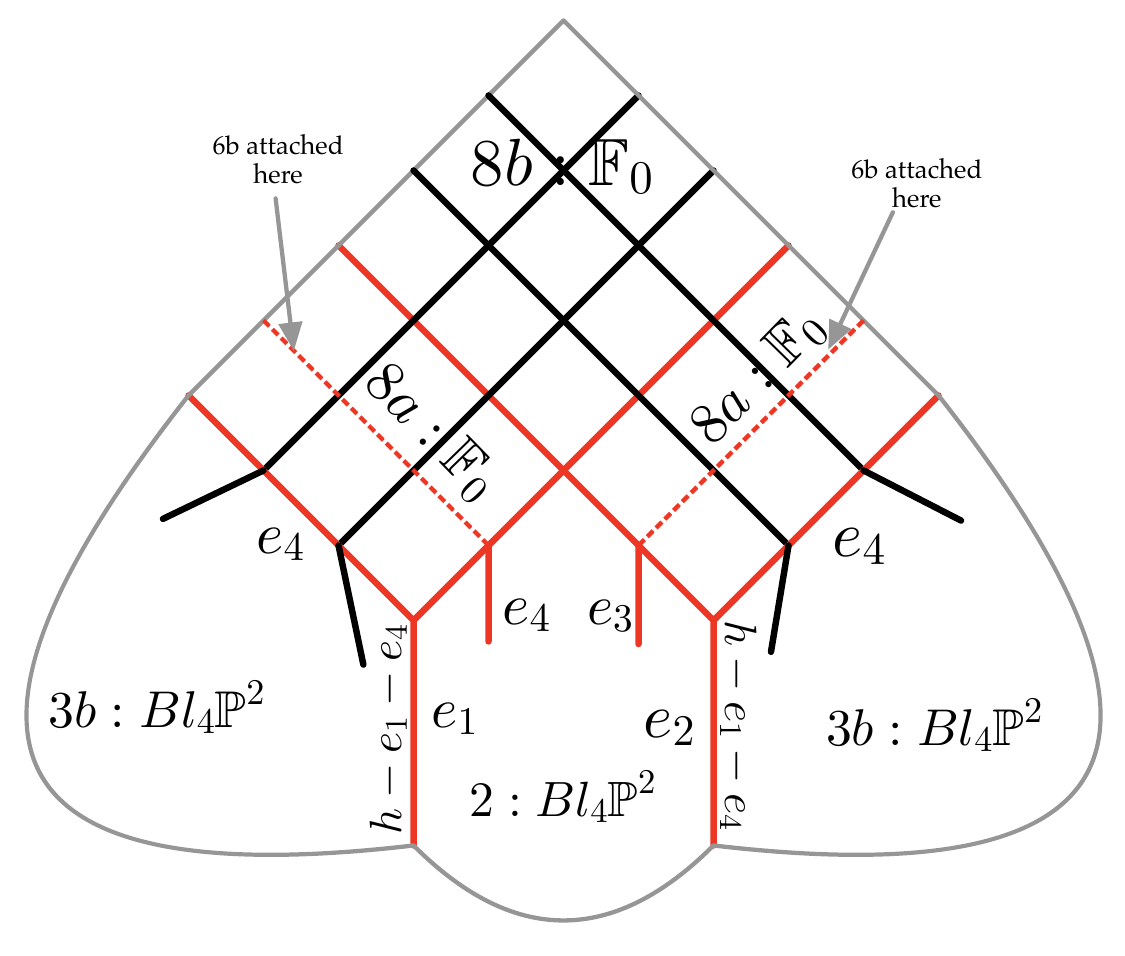}
        \caption{In $(S_4,cB_4)$, one attaches to each line of multiplicity 2 in the type 2 component of $(S_3,cB_3)$ a
        copy of $\bF_0$, labeled by type 8a. Another copy of $\bF_0$, labeled by type 8b, is attached to the
        intersection of the two 8a components.}
        \label{fig:a2b_12-23_5}
    \end{subfigure}
    \caption{(Continued) A type $a_2b$ surface $(S_4,cB_4)$ for $1/2 < c \leq 2/3$, obtained as the stable replacement for these
    weights of the surface of \cref{fig:a2b_13-12}.}%
    \label{fig:a2b_12-23_cont}
\end{figure}

\begin{table}[!htpb]
    \centering
    \caption{The types of irreducible components of the weight $1/2 < c \leq 2/3$ surface of type $a_2b$ pictured in
    \cref{fig:a2b_12-23}, and the possible numbers of Eckardt points on each component.
    For the components of types 1 and 3a, $Bl_5\bP^2$ refers to the blowup of $\bP^2$ at 3 points on one line and 2
    points on another line. For the components of types 2 and 3b, $Bl_4\bP^2$ refers to the blowup of $\bP^2$ at 2
    points on one line and 2 points on another line (i.e., 4 points in general position). For the component of type 4,
    $Bl_5\bP^2$ refers to the blowup of $\bP^2$ at 2 points on one line, 2 points on another line, and at the
    intersection point of these two lines.}
    \label{tab:a2b_23-1}
    \begin{tabular}{| c | c | c | c |}
        \hline
        Label & Surface & \# & Eckardt points \\
        \hline
        \hline
        1 & $Bl_5\bP^2$ & 2 & 0 \\
        2 & $Bl_4\bP^2$ & 1 & 0 \\
        3a & $Bl_5\bP^2$ & 2 & 0 \\
        3b & $Bl_4\bP^2$ & 2 & 0 \\
        3c & $\bF_0$ & 4 & 0 \\
        3d & $\bF_0$ & 4 & 0 \\
        3e & $\bF_0$ & 4 & 0 \\
        4 & $Bl_5\bP^2$ & 1 & 0 or 1 \\
        5a & $\bF_0$ & 3 & 0 \\
        5b & $\bF_0$ & 6 & 0 \\
        5c & $\bF_0$ & 3 & 0 \\
        6a & $\bF_0$ & 2 & 0 \\
        6b & $\bF_0$ & 2 & 0 \\
        7 & $\bF_0$ & 6 & 0 \\
        8a & $\bF_0$ & 2 & 0 \\
        8b & $\bF_0$ & 1 & 0 \\
        \hline
    \end{tabular}
\end{table}

\begin{figure}[!htpb]
    \centering
    \includegraphics[width=0.7\linewidth]{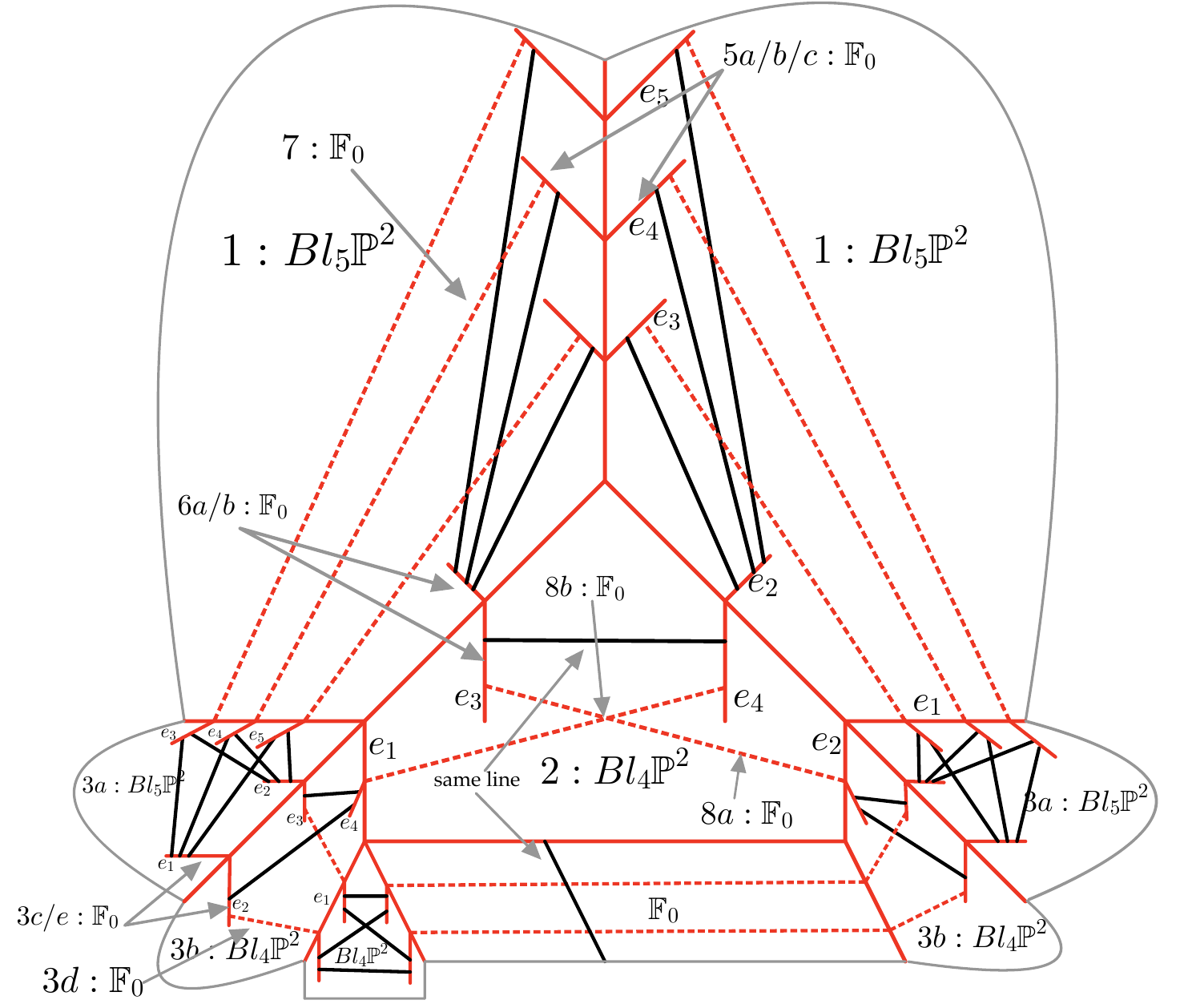}
    \caption{A type $aa_2b$ surface $(S_4',cB_4')$ for $1/2 < c \leq 2/3$, obtained as the stable replacement for these
    weights of the surface of \cref{fig:aa2b_13-12}, and as a degeneration of the type $a_2b$ surface $(S_4,cB_4)$ of
    \cref{fig:a2b_12-23}. The type 4 component $\cong Bl_5\bP^2$ of \cref{fig:a2b_12-23} splits into one $Bl_4\bP^2$
    component and three $\bF_0$ components as shown (cf. \cref{fig:aa2_12-23_2}), and otherwise
    $(S_4',cB_4')$ is isomorphic to $(S_4,cB_4)$.}
    \label{fig:aa2b_12-23}
\end{figure}

\subsection{Type $a_3b$}
\begin{proposition} \label{prop:a3b}
    For type $a_3b$ weighted stable marked cubic surfaces $(S,cB)$, there are five walls.
    \begin{enumerate}
        \item For weights $1/9 < c \leq 1/6$, the type $a_3b$ surfaces are described in \cref{fig:a3b_19-16}.
        \item For weights $1/6 < c \leq 1/4$, the type $a_3b$ surfaces are described in \cref{fig:a3b_16-14}.
        \item For weights $1/4 < c \leq 1/3$, the type $a_3b$ surfaces are described in \cref{fig:a3b_14-13}.
            Additionally, crossing the wall $c=1/4$ introduces type $aa_3b$, $a_2a_3b$, and $aa_2a_3b$ surfaces as
            degenerations of $a_3b$ surfaces, described in \cref{fig:a3b_degens_14-13}.
        \item For weights $1/3 < c \leq 1/2$, the type $a_3b$ surfaces are described in \cref{fig:a3b_13-12}. The type
            $aa_3b$, $a_2a_3b$, and $aa_2a_3b$ surfaces are described in \cref{fig:a3b_degens_13-12}.
        \item For weights $1/2 < c \leq 2/3$, the type $a_3b$ surfaces are described in \cref{fig:a3b_12-23}. The type
            $aa_3b$, $a_2a_3b$, and $aa_2a_3b$ surfaces are described in \cref{fig:a3b_degens_12-23}.
        \item For weights $2/3 < c \leq 1$, the type $a_3b$ surfaces are obtained from the weight $1/2 < c \leq 2/3$ type
            $a_3b$ surfaces by resolving Eckardt points as described in \cref{prop:resolve_eckardt}. The possible
            configurations of Eckardt points are summarized in \cref{tab:a3b_23-1}. The type $aa_3b$, $a_2a_3b$, and
            $aa_2a_3b$ surfaces are described similarly.
    \end{enumerate}
\end{proposition}

%\subsubsection{Weights $1/9 < c \leq 1/6$}

\begin{figure}[!htpb]
    \centering
    \includegraphics[width=0.4\linewidth]{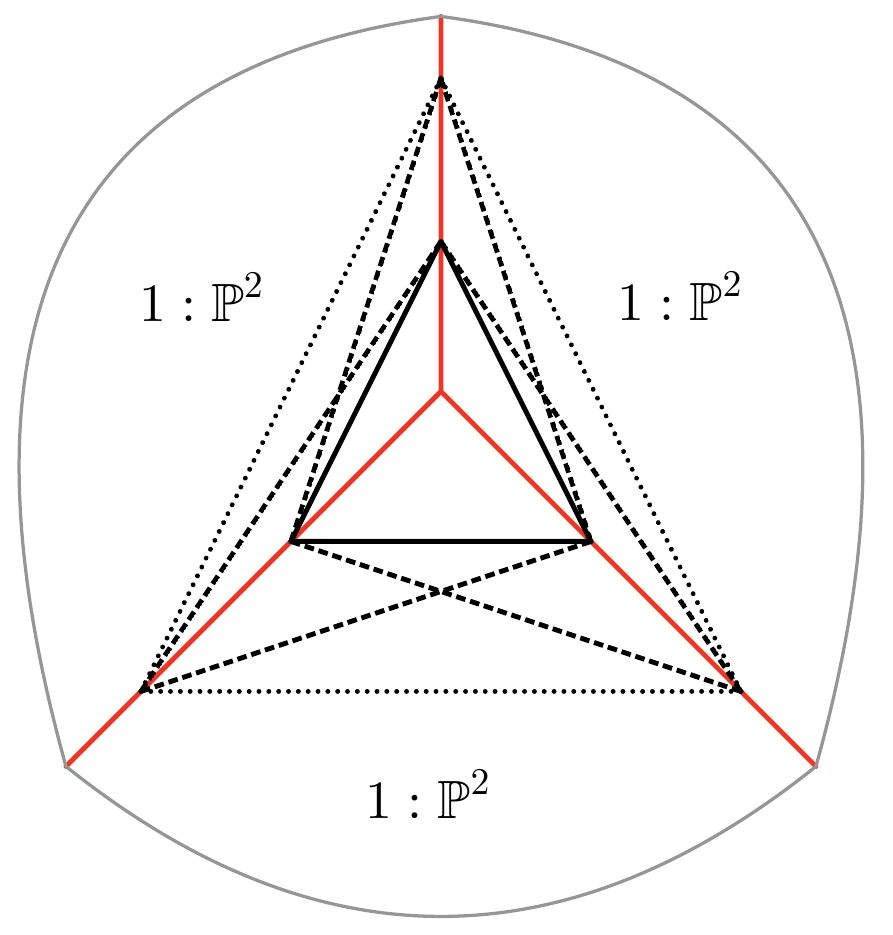}
    \caption{A type $a_3b$ surface $(S_0,cB_0)$ for $1/9 < c \leq 1/6$. Recall dotted lines have multiplicity 4 and
    dashed lines have multiplicity 2. Any two lines on a given $\bP^2$ component intersect (possibly at infinity, not
    shown).}%
    \label{fig:a3b_19-16}
\end{figure}

%\subsubsection{Weights $1/6 < c \leq 1/4$}

\begin{figure}[!htpb]
    \centering
    \includegraphics[width=0.5\linewidth]{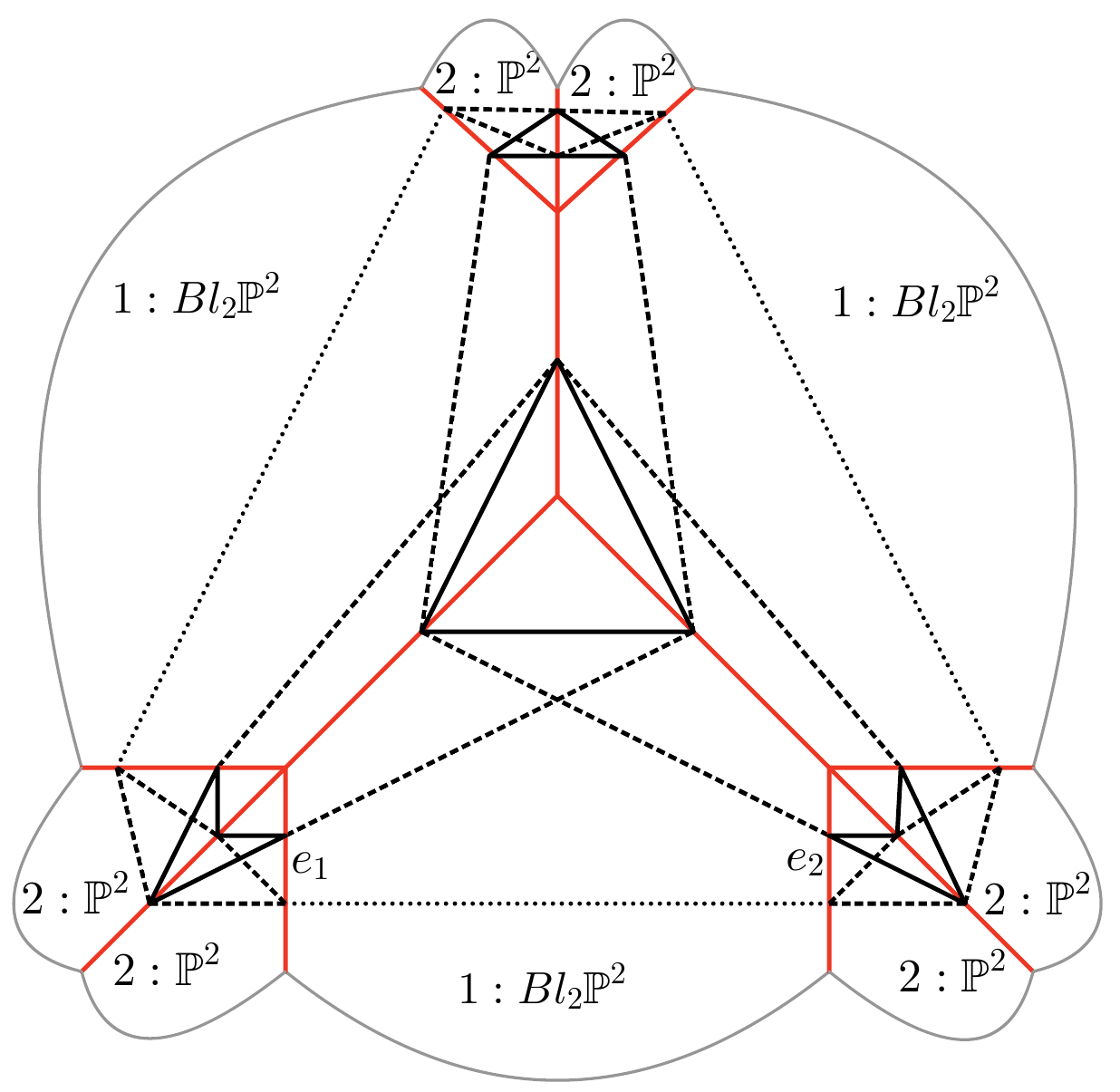}
    \caption{A type $a_3b$ surface $(S_1,cB_1)$ for $1/6 < c \leq 1/4$, obtained as the stable replacement for these
    weights of the surface of \cref{fig:a3b_19-16}. The surface $(S_1,cB_1)$ has six more components than $(S_0,cB_0)$,
    each a copy of $\bP^2$ labeled by type 2.}%
    \label{fig:a3b_16-14}
\end{figure}

%\subsubsection{Weights $1/4 < c \leq 1/3$}

\begin{figure}[!htpb]
    \centering
    \includegraphics[width=0.5\linewidth]{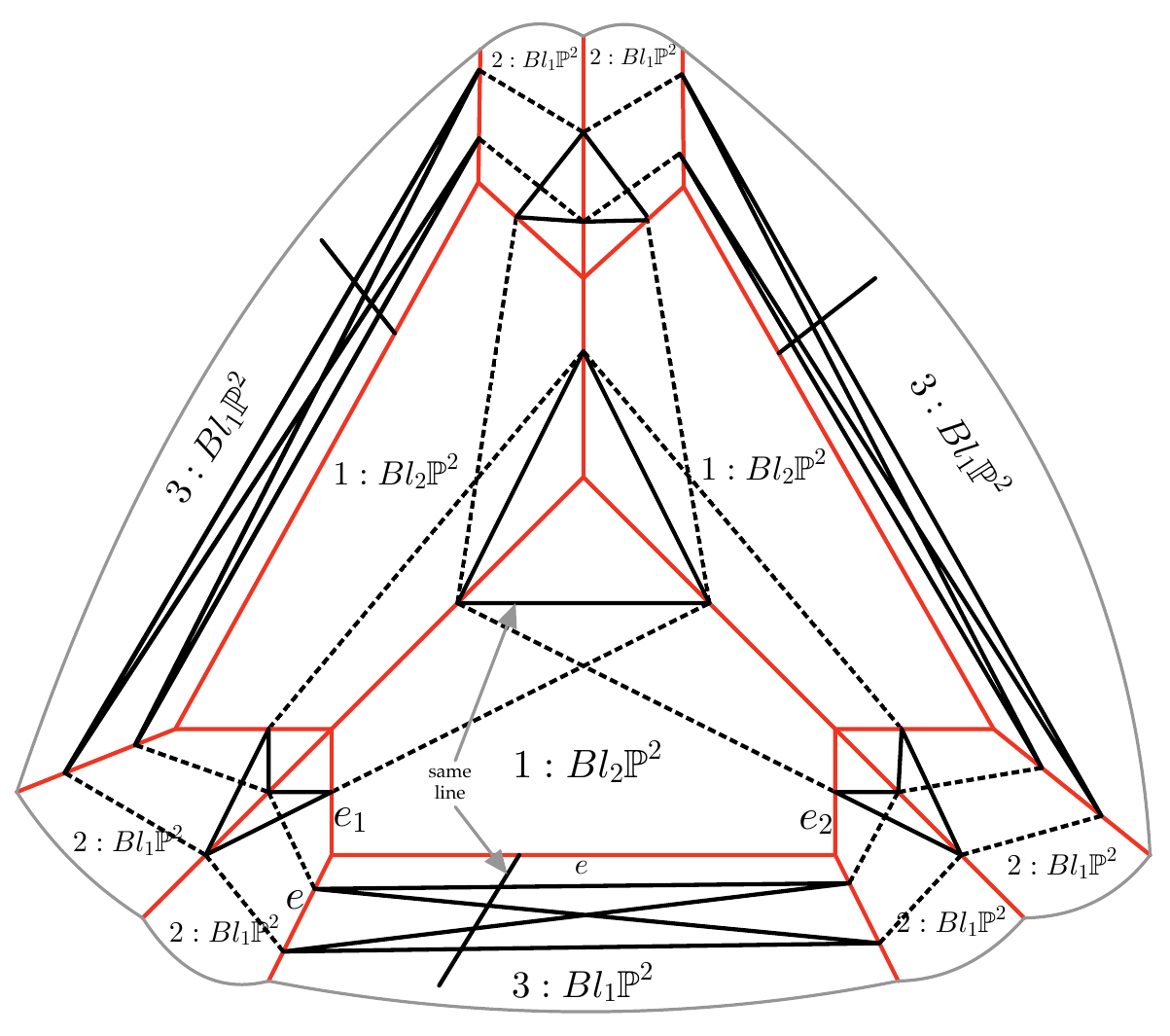}
    \caption{A type $a_3b$ surface $(S_2,cB_2)$ for $1/4 < c \leq 1/3$, obtained as the stable replacement for these
    weights of the surface of \cref{fig:a3b_16-14}. The surface $(S_2,cB_2)$ has three more components than $(S_1,cB_1)$,
    each a copy of $Bl_1\bP^2$ labeled by type 3.}%
    \label{fig:a3b_14-13}
\end{figure}

\begin{figure}[!htpb]
    \centering
    \begin{subfigure}[t]{0.61\textwidth}
        \centering
        \includegraphics[width=0.8\linewidth]{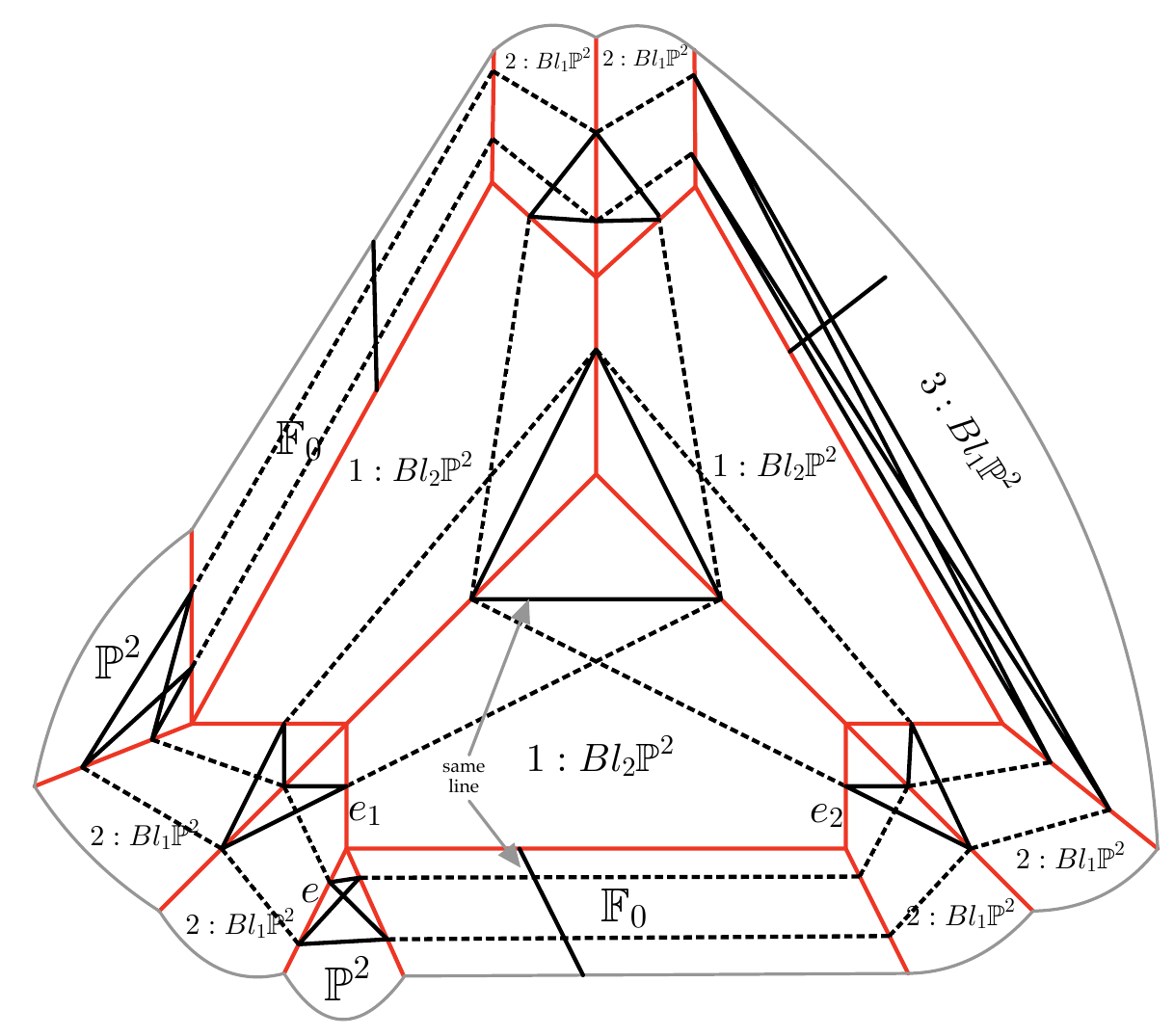}
        \caption{A type $aa_3b$ surface for $1/4 < c \leq 1/3$.}%
        \label{fig:aa3b_14-13}
    \end{subfigure}
    \begin{subfigure}[t]{0.61\textwidth}
        \centering
        \includegraphics[width=0.8\linewidth]{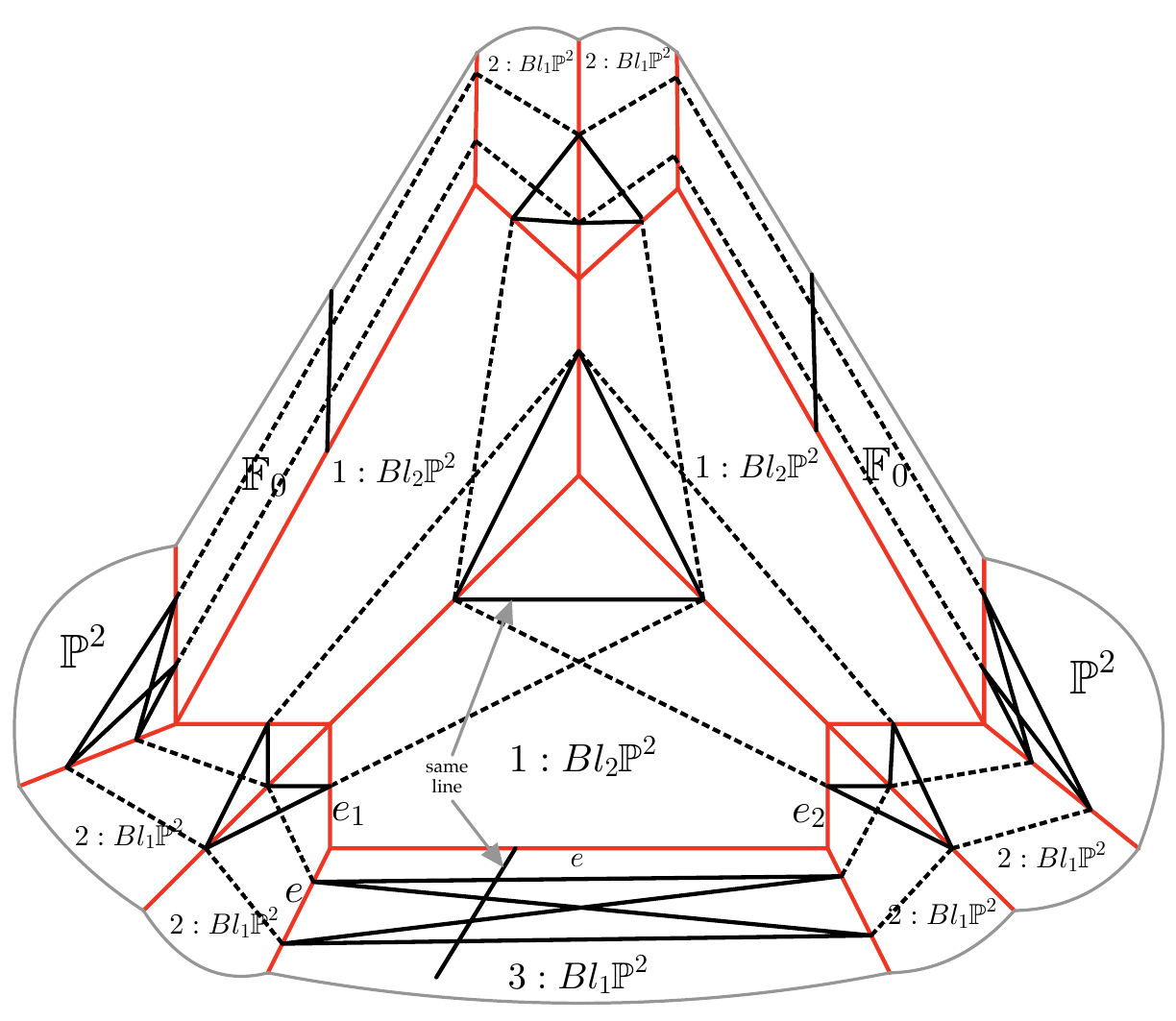}
        \caption{A type $a_2a_3b$ surface for $1/4 < c \leq 1/3$.}%
        \label{fig:a2a3b_14-13}
    \end{subfigure}
    \begin{subfigure}[t]{0.61\textwidth}
        \centering
        \includegraphics[width=0.8\linewidth]{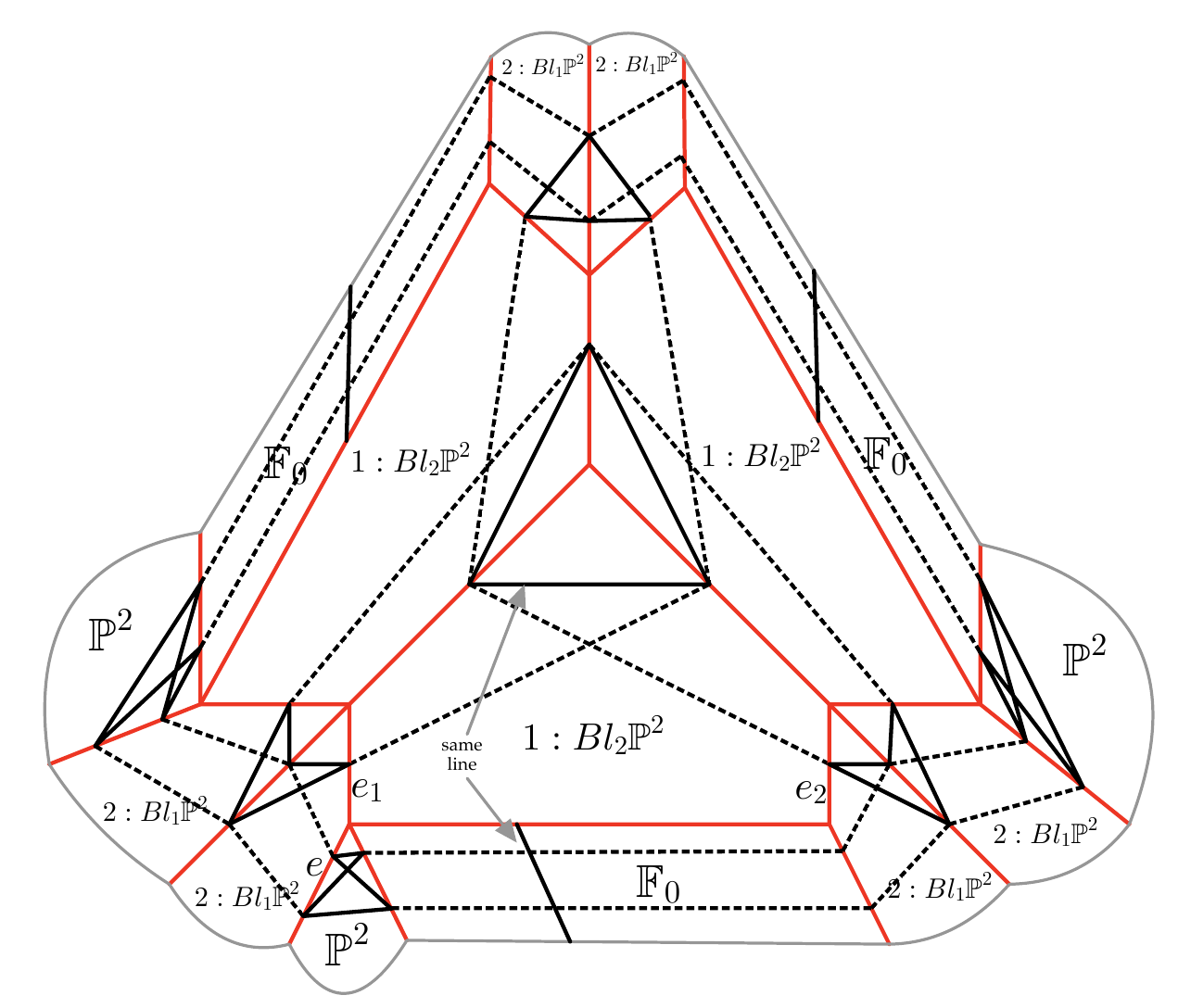}
        \caption{A type $aa_2a_3b$ surface for $1/4 < c \leq 1/3$}%
        \label{fig:aa2a3b_14-13}
    \end{subfigure}
    \caption{Degenerations of the type $a_3b$ surface of \cref{fig:a3b_14-13}, for weights $1/4 < c \leq 1/3$. These
    surfaces are isomorphic to the $a_3b$ surface away from the type 3 components which degenerate into a copy of
    $\bP^2$ and a copy of $\bF_0$ as shown.}%
    \label{fig:a3b_degens_14-13}
\end{figure}

%\subsubsection{Weights $1/3 < c \leq 1/2$}

\begin{figure}[!htpb]
    \centering
    \begin{subfigure}[t]{0.7\textwidth}
        \centering
        \includegraphics[width=0.9\linewidth]{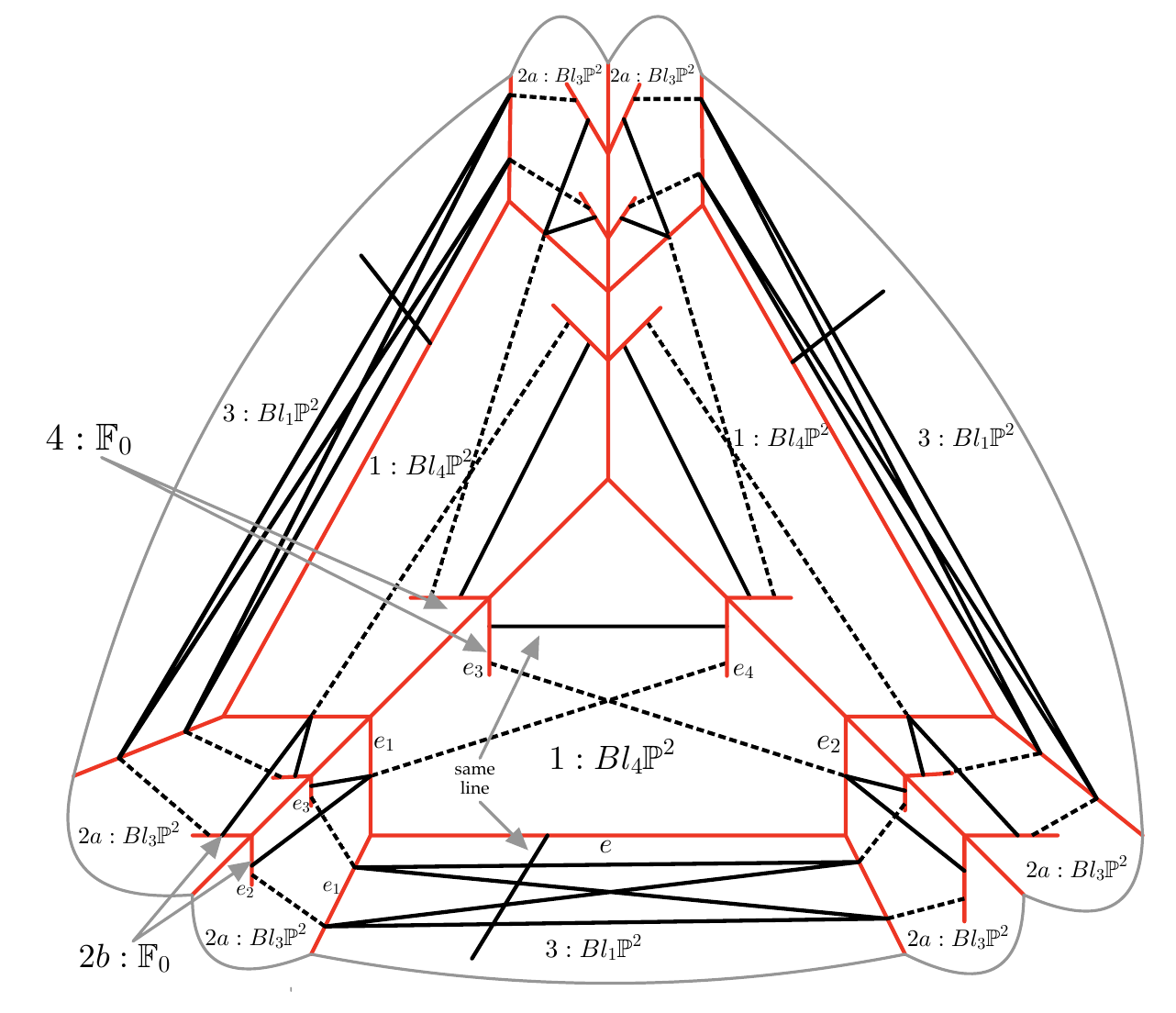}
        \caption{The stable replacement $(S_3,cB_3)$ of the surface of \cref{fig:a3b_14-13}, for weights $1/3 < c \leq
        1/2$. See the adjacent subfigures for local views of the new components of types 2b and 4.}
        \label{fig:a3b_13-12_1}
    \end{subfigure}
    \begin{subfigure}[t]{0.3\textwidth}
        \centering
        \includegraphics[width=0.9\linewidth]{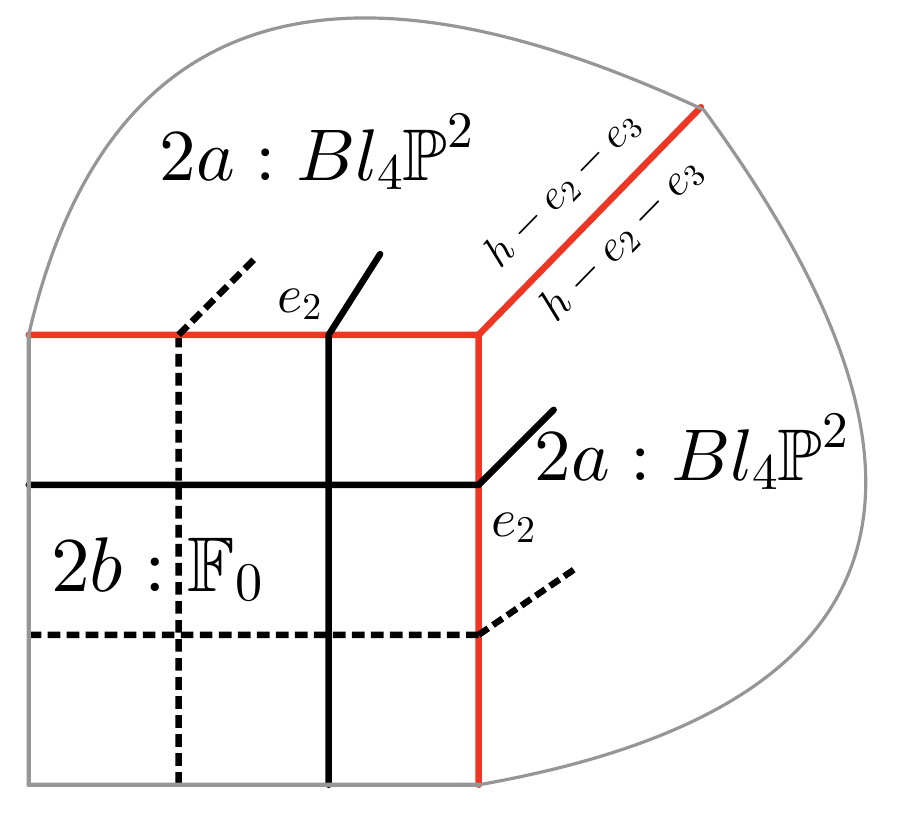}
        \caption{A view of an $\bF_0$ component of type 2b.}
        \label{fig:a3b_13-12_2}
    \end{subfigure}
    \begin{subfigure}[t]{0.3\textwidth}
        \centering
        \includegraphics[width=0.9\linewidth]{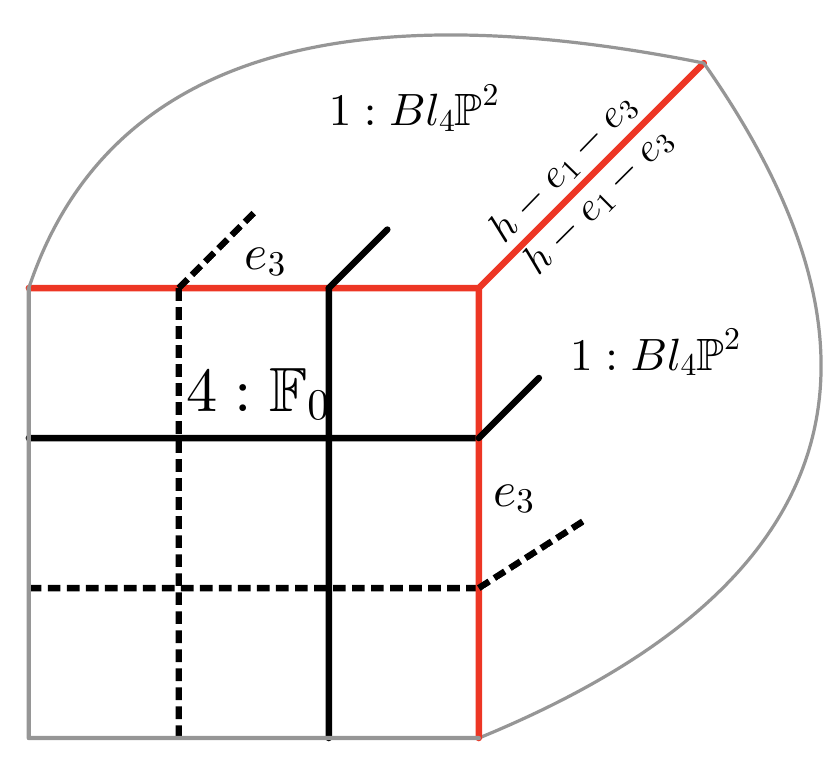}
        \caption{A view of an $\bF_0$ component of type 4.}
        \label{fig:a3b_13-12_3}
    \end{subfigure}
    \caption{A type $a_3b$ surface $(S_3,cB_3)$ for $1/3 < c \leq 1/2$, obtained as the stable replacement for these
    weights of the surface of \cref{fig:a3b_14-13}. The surface $(S_3,cB_3)$ has nine more components than $(S_2,cB_2)$:
    six of type 2b and three of type 4.}%
    \label{fig:a3b_13-12}
\end{figure}

\begin{figure}[!htpb]
    \centering
    \begin{subfigure}[t]{0.61\textwidth}
        \centering
        \includegraphics[width=0.8\linewidth]{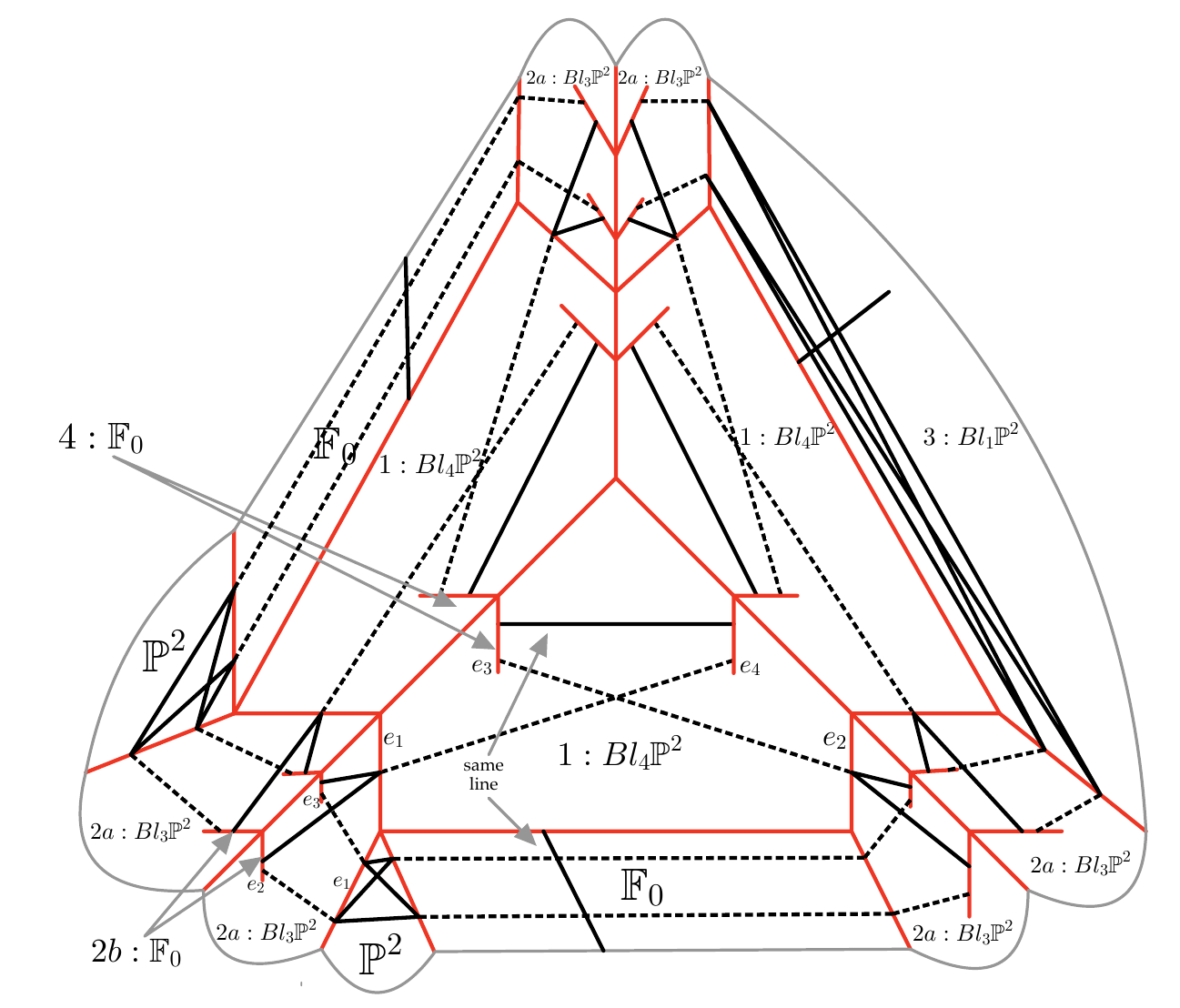}
        \caption{A type $aa_3b$ surface for $1/3 < c \leq 1/2$.}%
        \label{fig:aa3b_13-12}
    \end{subfigure}
    \begin{subfigure}[t]{0.61\textwidth}
        \centering
        \includegraphics[width=0.8\linewidth]{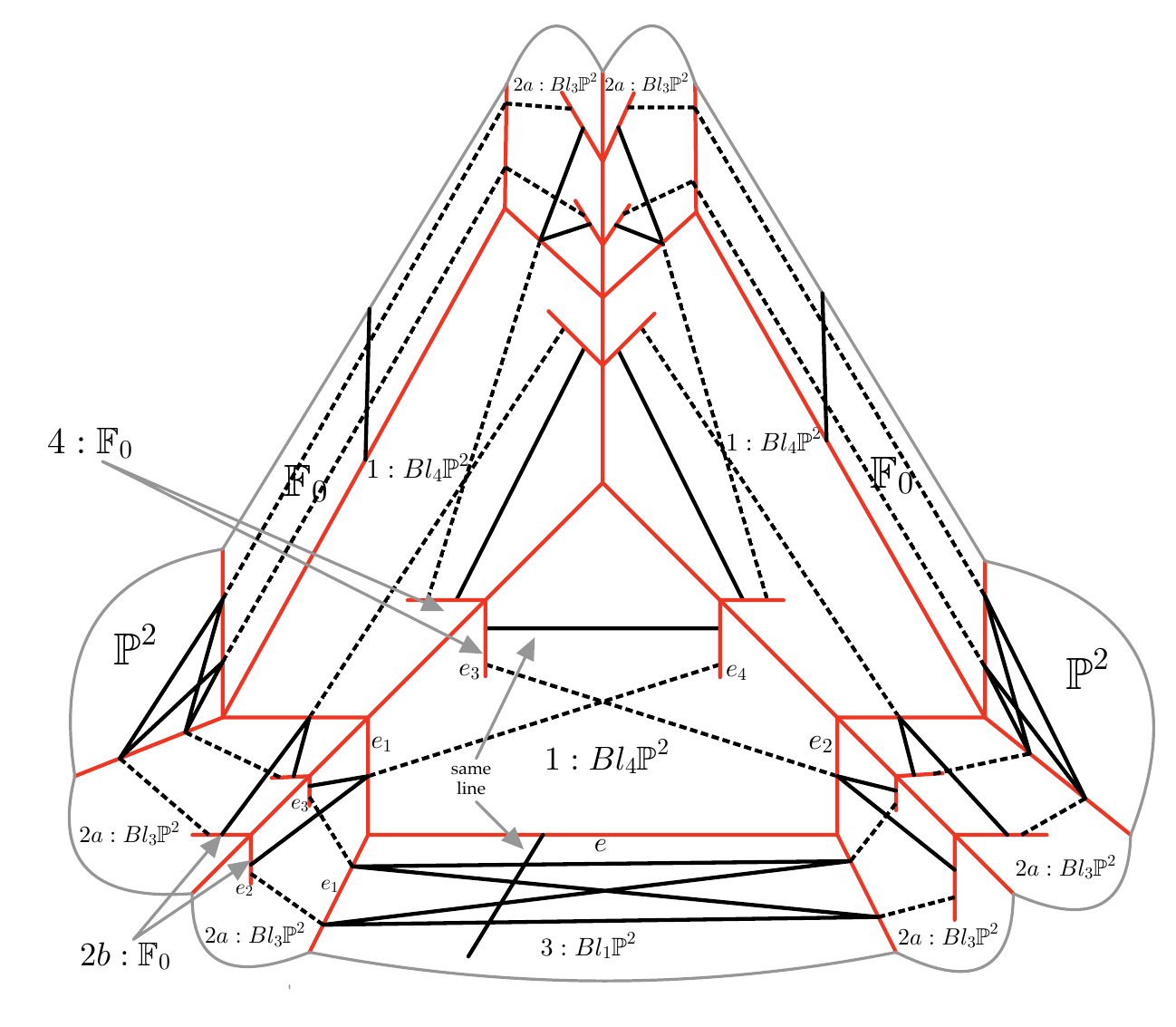}
        \caption{A type $a_2a_3b$ surface for $1/3 < c \leq 1/2$.}%
        \label{fig:a2a3b_13-12}
    \end{subfigure}
    \begin{subfigure}[t]{0.61\textwidth}
        \centering
        \includegraphics[width=0.8\linewidth]{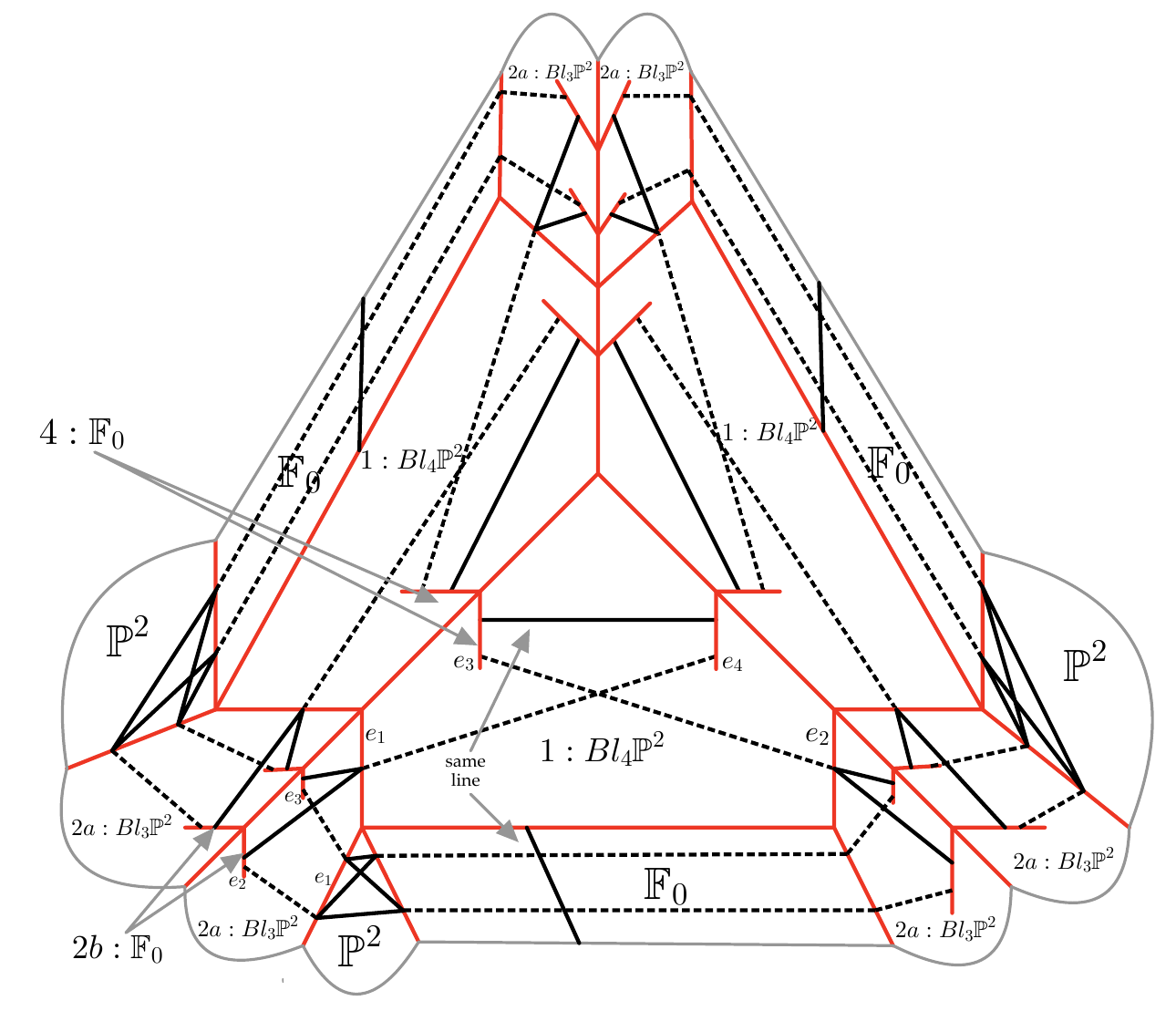}
        \caption{A type $aa_2a_3b$ surface for $1/3 < c \leq 1/2$}%
        \label{fig:aa2a3b_13-12}
    \end{subfigure}
    \caption{Degenerations of the type $a_3b$ surface of \cref{fig:a3b_13-12}, for weights $1/3 < c \leq 1/2$. These
    surfaces are isomorphic to the $a_3b$ surface away from the type 3 components which degenerate into a copy of
    $\bP^2$ and a copy of $\bF_0$ as shown.}%
    \label{fig:a3b_degens_13-12}
\end{figure}

%\subsubsection{Weights $1/2 < c \leq 2/3$}

\begin{figure}[!htpb]
    \centering
    \begin{subfigure}[t]{0.7\textwidth}
        \centering
        \includegraphics[width=0.9\linewidth]{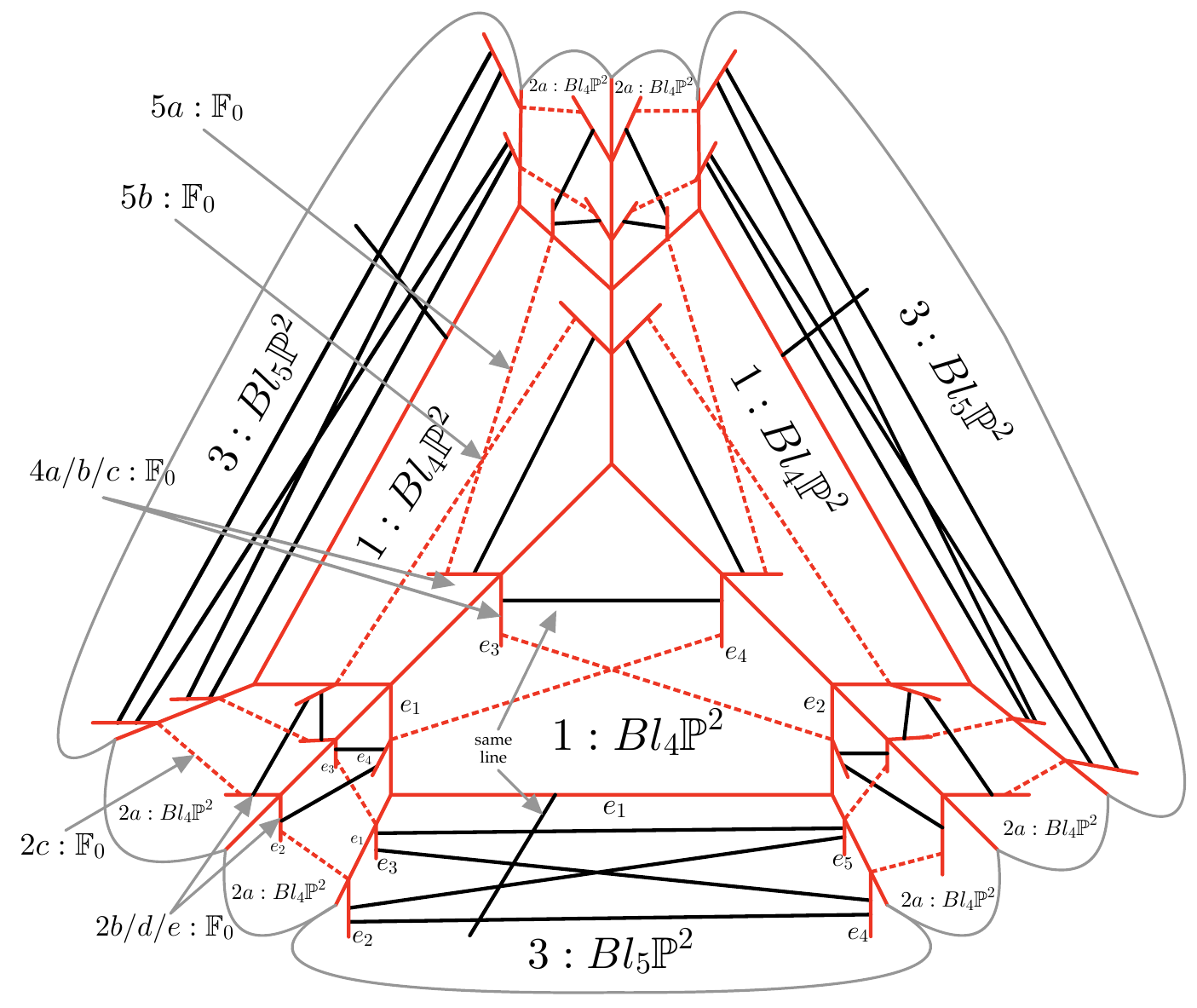}
        \caption{The stable replacement $(S_4,cB_4)$ of the surface of \cref{fig:a3b_14-13}, for weights $1/2 < c \leq
        2/3$. See the adjacent subfigures for local views of the new components of types 2c, 2d, 2e, 4a, 4b, 4c, and 5a,
        5b.}
        \label{fig:a3b_12-23_1}
    \end{subfigure}
    \begin{subfigure}[t]{0.46\textwidth}
        \centering
        \includegraphics[width=0.9\linewidth]{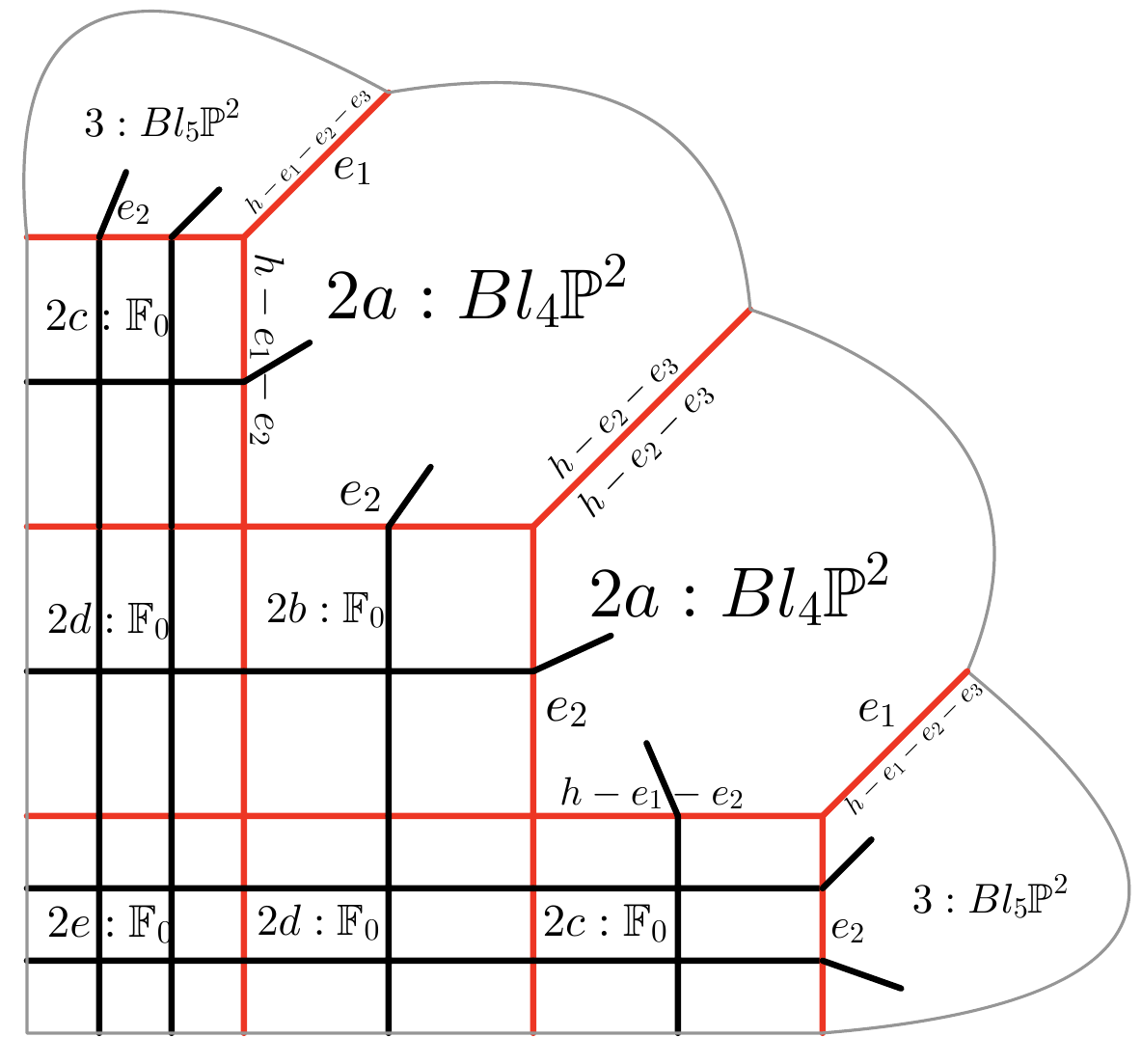}
        \caption{A view of a type 2b component $\cong \bF_0$, and the resulting new components of types 2d, 2c, and 2e,
        also isomorphic to $\bF_0$ (cf. \cref{fig:a3b_13-12_2}).}
        \label{fig:a3b_12-23_2}
    \end{subfigure}
    \begin{subfigure}[t]{0.52\textwidth}
        \centering
        \includegraphics[width=0.9\linewidth]{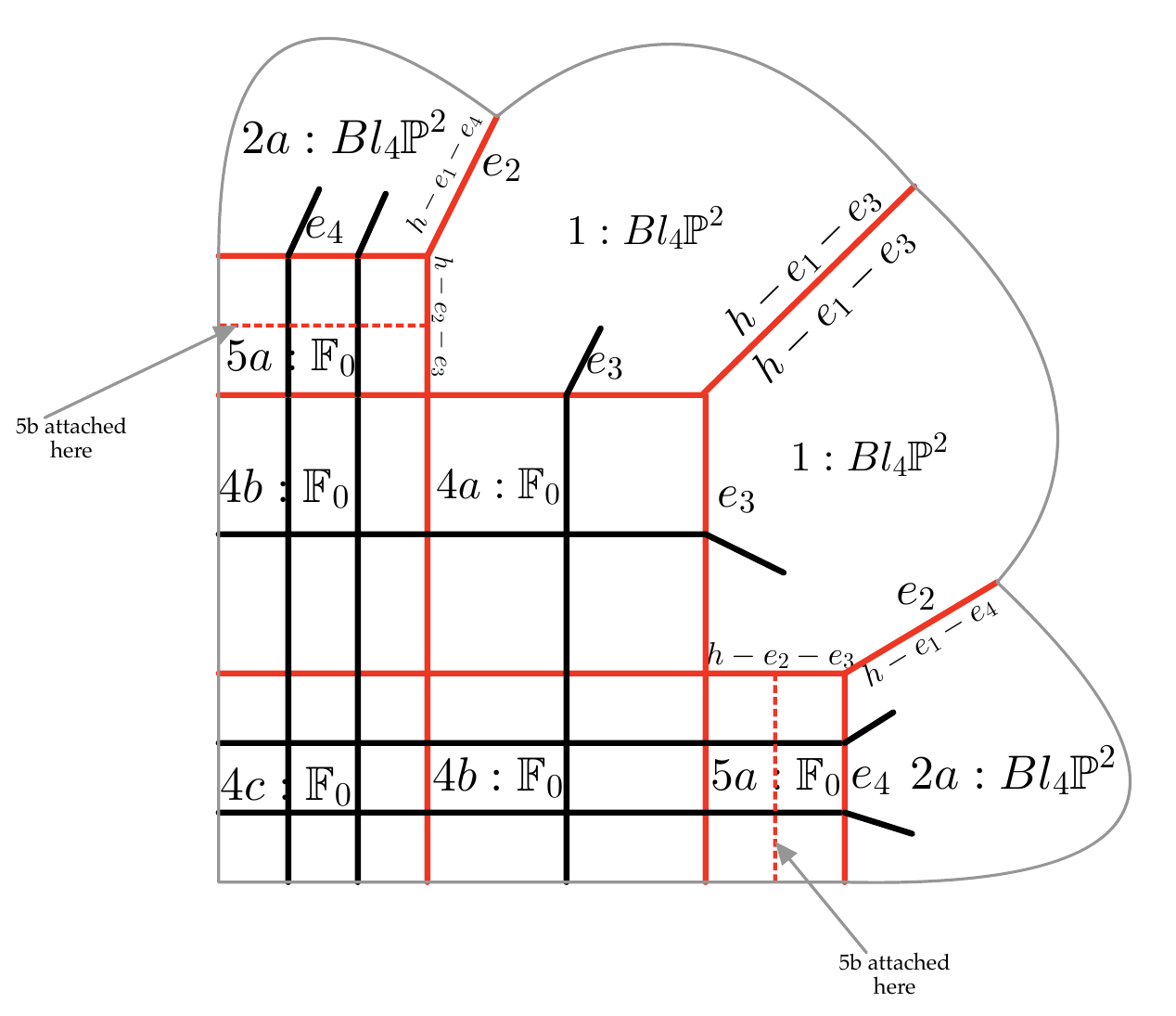}
        \caption{A view of a type 4a component $\cong \bF_0$, arising from a type 4 component in \cref{fig:a3b_13-12_3},
        and the surrounding new components of types 4b, 4c, and 5a, also isomorphic to $\bF_0$.}
        \label{fig:a3b_12-23_3}
    \end{subfigure}
    \caption{A type $a_3b$ surface $(S_4,cB_4)$ for $1/2 < c \leq 2/3$, obtained as the stable replacement for these
    weights of the surface $(S_3,cB_3)$ of \cref{fig:a3b_13-12}.}%
    \label{fig:a3b_12-23}
\end{figure}

\begin{figure}[!htpb]\ContinuedFloat
    \centering
    \begin{subfigure}[t]{0.45\textwidth}
        \centering
        \includegraphics[width=0.9\textwidth]{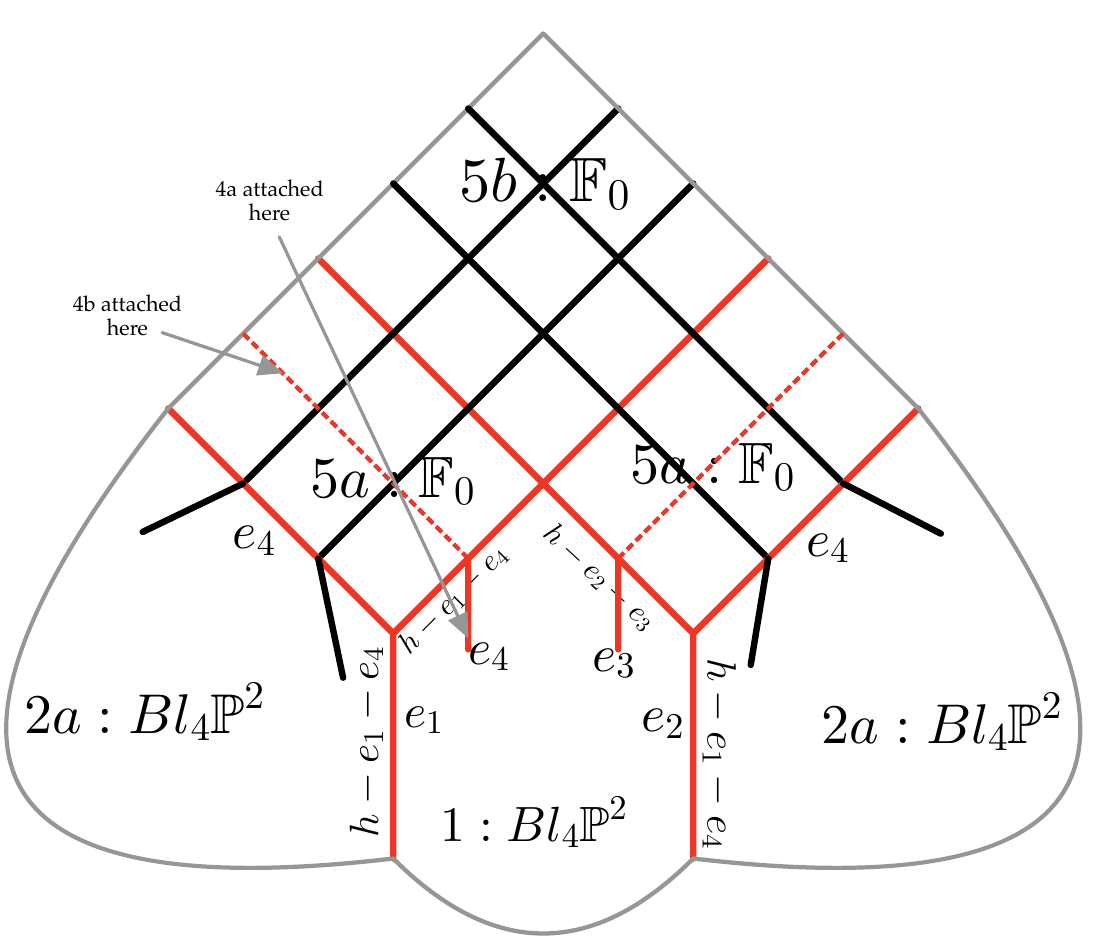}
        \caption{A view of the two type 5a components $\cong \bF_0$ intersecting a given type 1 component. At the
            intersection point of the two type 5a components, one attaches another copy of $\bF_0$, labeled by type 5b.
            The gluings of the types 4a and 4b components indicated are described in \cref{fig:a3b_12-23_3}.}
        \label{fig:a3b_12-23_4}
    \end{subfigure}
    \caption{(Continued) A type $a_3b$ surface $(S_4,cB_4)$ for $1/2 < c \leq 2/3$, obtained as the stable replacement
    for these weights of the surface $(S_3,cB_3)$ of \cref{fig:a3b_13-12}.}%
    \label{fig:a3b_12-23_cont}
\end{figure}

\begin{table}[!htpb]
    \centering
    \caption{The types of irreducible components of the weight $1/2 < c \leq 2/3$ surface of type $a_3b$ pictured in
    \cref{fig:a3b_12-23}, and the possible numbers of Eckardt points on each component.
    For the components of types 1 and 2a, $Bl_4\bP^2$ refers to the blowup of $\bP^2$ at 2 points on one line and 2
    points on another line (i.e., 4 points in general position). For the components of type 3, $Bl_5\bP^2$ refers to the
    blowup of $\bP^2$ at 2 points on one line, 2 points on another line, and at the intersection point of these two lines.}
    \label{tab:a3b_23-1}
    \begin{tabular}{| c | c | c | c |}
        \hline
        Label & Surface & \# & Eckardt points \\
        \hline
        \hline
        1 & $Bl_4\bP^2$ & 3 & 0 \\
        2a & $Bl_4\bP^2$ & 6 & 0 \\
        2b & $\bF_0$ & 6 & 0 \\
        2c & $\bF_0$ & 12 & 0 \\
        2d & $\bF_0$ & 12 & 0 \\
        2e & $\bF_0$ & 6 & 0 \\
        3 & $Bl_5\bP^2$ & 3 & 0 or 1 \\
        4a & $\bF_0$ & 3 & 0 \\
        4b & $\bF_0$ & 6 & 0 \\
        4c & $\bF_0$ & 3 & 0 \\
        5a & $\bF_0$ & 6 & 0 \\
        5b & $\bF_0$ & 3 & 0 \\
        \hline
    \end{tabular}
\end{table}

\begin{figure}[!htpb]
    \centering
    \begin{subfigure}[t]{0.65\textwidth}
        \centering
        \includegraphics[width=0.8\linewidth]{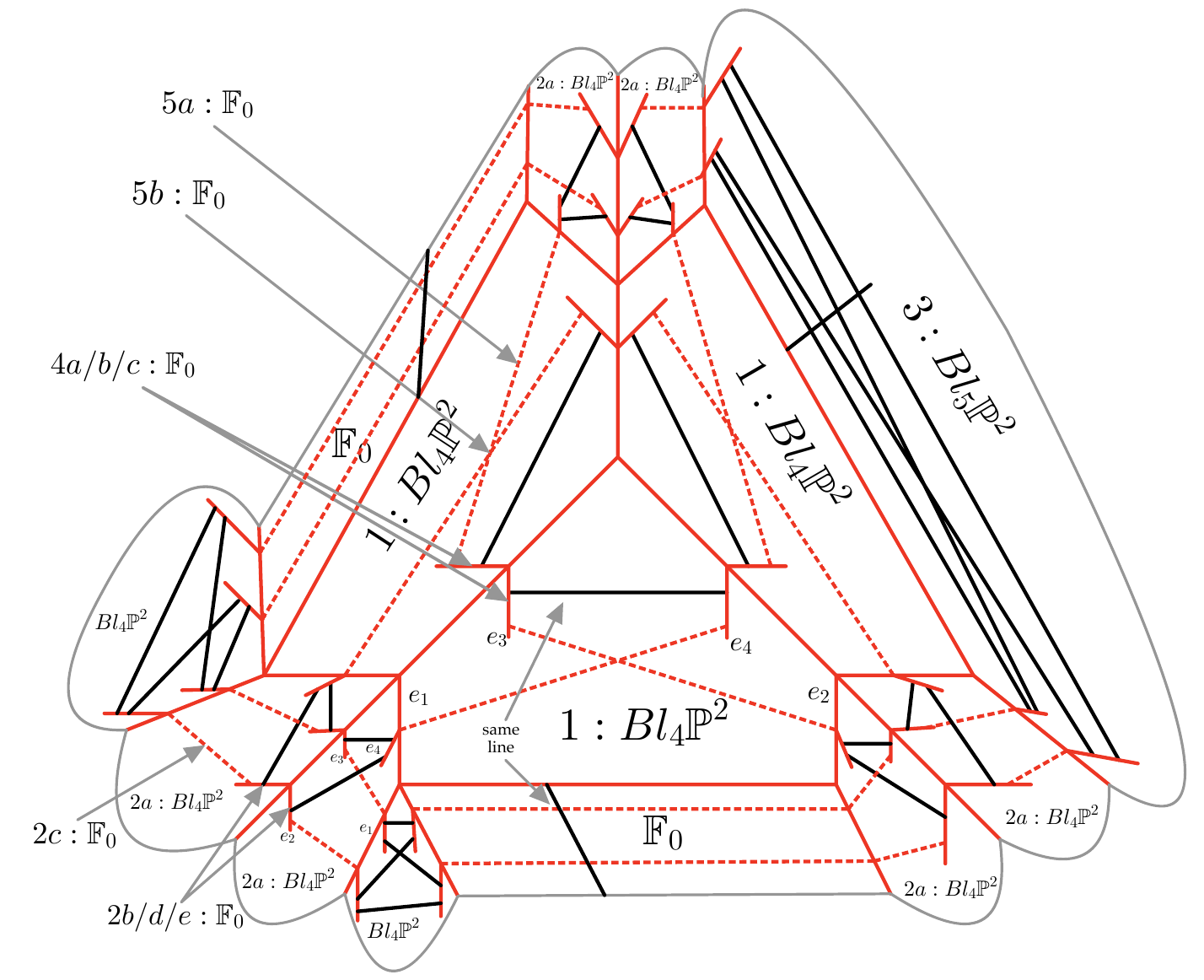}
        \caption{A type $aa_3b$ surface for $1/2 < c \leq 2/3$.}%
        \label{fig:aa3b_12-23}
    \end{subfigure}
    \begin{subfigure}[t]{0.65\textwidth}
        \centering
        \includegraphics[width=0.8\linewidth]{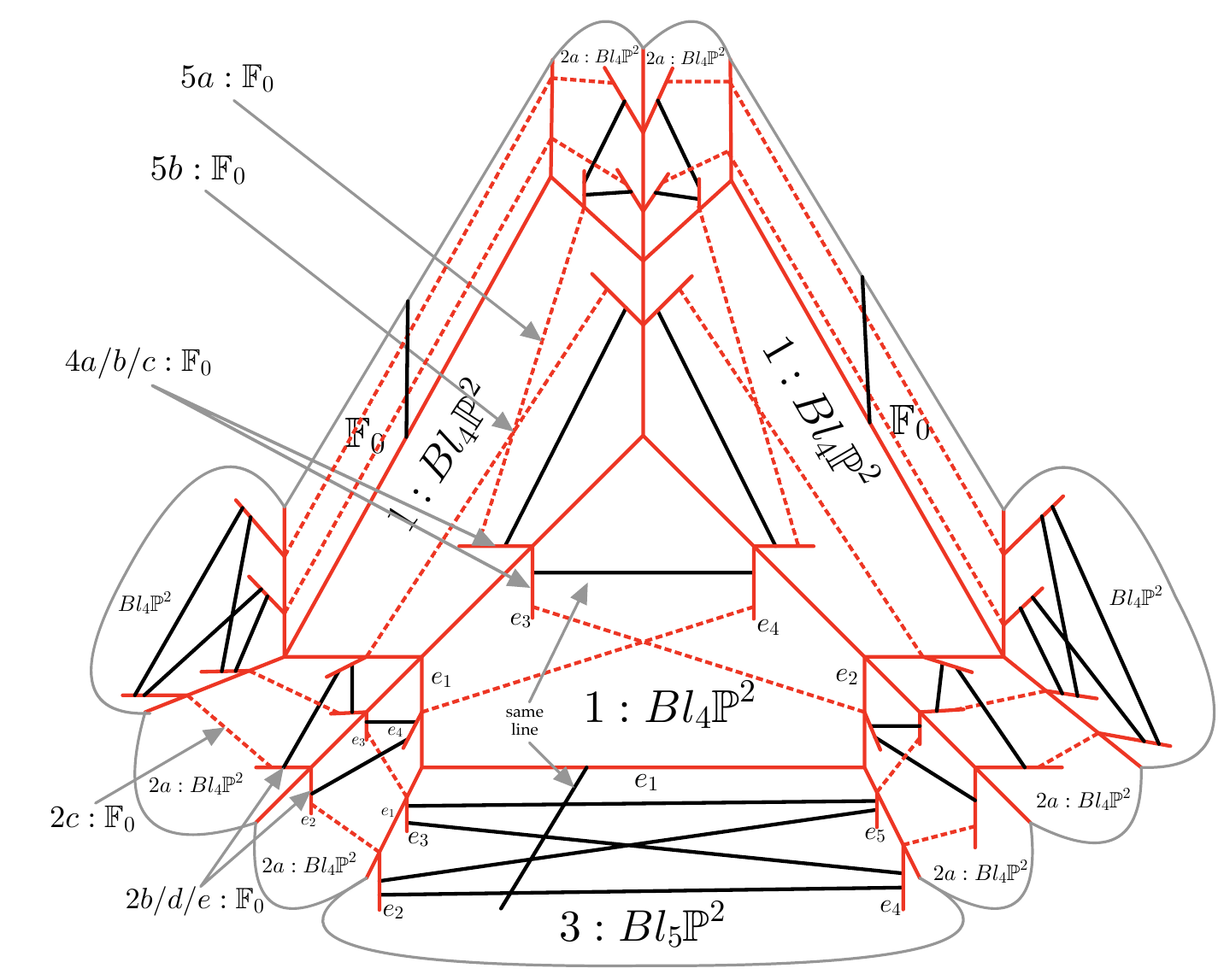}
        \caption{A type $a_2a_3b$ surface for $1/2 < c \leq 2/3$.}%
        \label{fig:a2a3b_12-23}
    \end{subfigure}
    \begin{subfigure}[t]{0.65\textwidth}
        \centering
        \includegraphics[width=0.8\linewidth]{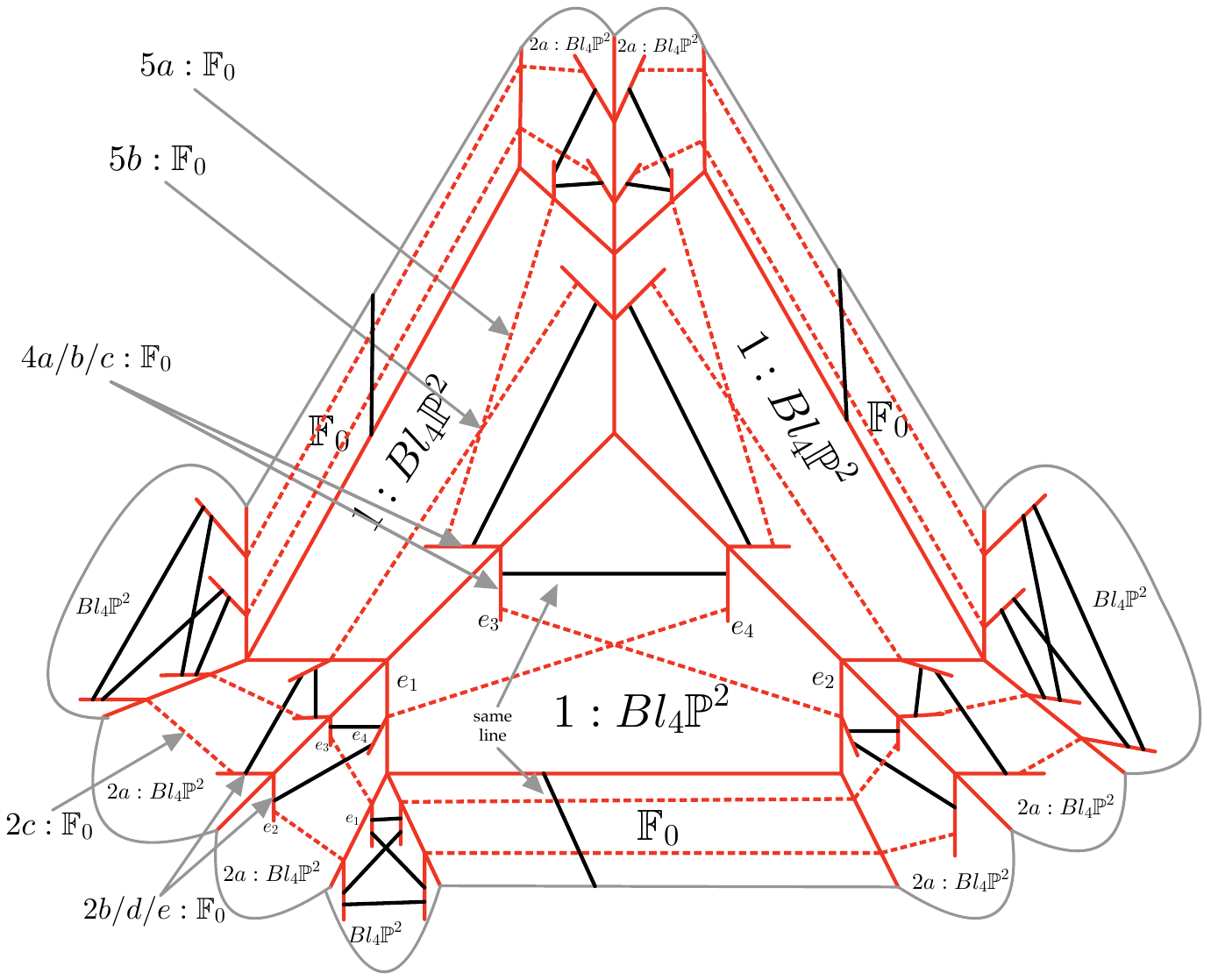}
        \caption{A type $aa_2a_3b$ surface for $1/2 < c \leq 2/3$.}%
        \label{fig:aa2a3b_12-23}
    \end{subfigure}
    \caption{Degenerations of the type $a_3b$ surface of \cref{fig:a3b_12-23}, for weights $1/2 < c \leq 2/3$. These
    surfaces are isomorphic to the $a_3b$ surface away from the type 3 components which degenerate into a copy of
    $\bP^2$ and three copies of $\bF_0$ as shown (cf. \cref{fig:aa2_12-23_2}).}%
    \label{fig:a3b_degens_12-23}
\end{figure}

\section{Proof of \cref{thm:cubics_main}} \label{sec:proof_main}

In this section we prove \cref{thm:cubics_main}. Using \cref{thm:gen_walls}, we see that the proof reduces to the
following steps.
\begin{enumerate}
    \item Compute the fibers of the family $\ddot{\pi} : (\ddot{Y}(E_7),B_{\ddot{\pi}}) \to \wY(E_6)$, which by
        \cite[Theorem 10.31]{hackingStablePairTropical2009} gives the universal family of weight 1 stable marked cubic
        surfaces.
    \item Fix a fiber $(S,B)$ of $\ddot{\pi} : (\ddot{Y}(E_7),B_{\ddot{\pi}}) \to \wY(E_6)$, and compute the sequential
        stable models of the pair $(S,cB)$ as one decreases the weight $c$ in the interval $(1/9,1]$.
\end{enumerate}
We have already completed step (1) in \cref{sec:fibers,sec:eckardts}. Namely, in \cref{sec:fibers}, we computed the
fibers of the family $\wt\pi : (\wY(E_7),B_{\wt\pi}) \to \wY(E_6)$, and saw that they give the weighted stable marked
cubic surfaces $(S,cB)$ for weight $c=2/3$ described in the previous section. Then in \cref{sec:eckardts} we explained
how to obtain the fibers of $\ddot{\pi} : (\ddot{Y}(E_7),B_{\ddot{\pi}}) \to \wY(E_6)$ by resolving Eckardt points on
the fibers of $\wt\pi$. It follows directly from our descriptions that each fiber of $\ddot{\pi}$ is slc with ample log
canonical class, hence is a stable pair. (This re-proves in a more explicit manner the first part of \cite[Theorem
10.31]{hackingStablePairTropical2009}.) Thus we reduce to carrying out step (2) beginning with the weight $c=2/3$, and
decreasing $c$ towards $1/9$.

In order to carry out step (2), we use the following strategy. Let $(S,cB)$ be a weighted stable marked cubic surface
for some fixed weight $c \in (1/9,1]$. Thus $(S,cB)$ is slc and $K_S+cB$ is ample. The restriction of $K_S+cB$ to an
irreducible component $S'$ of $S$ is given by the standard formula
\begin{equation} \label{eq:restrict_component}
    (K_S+cB)\vert_{S'} = K_{S'} + \Delta + cB',
\end{equation}
where $\Delta$ is the gluing locus of $S'$ to the other irreducible components of $S$, and $B'$ is the restriction of
$B$ to $S'$. Walls occur at values $t \in (1/9,1]$ for which $K_{S'} + \Delta + tB'$ is nef but not ample for some
irreducible components $S'$ of $S$; when this occurs one contracts the curves of $S'$ along which $K_{S'}+\Delta+tB'$ is
zero, in order to obtain the log canonical model of $(S',\Delta+tB')$. In general, one must then show that the remaining
components can still be glued to obtain the stable model for the new weight. In our case, the gluing is obvious, and we
directly verify that the resulting surfaces $(S_t,tB_t)$ are stable. We continue in this fashion, decreasing $c$ towards
$1/9$, until all walls have been crossed and $c=1/9 + \epsilon$.

We now carry out the strategy of the above paragraph explicitly for weighted stable marked cubic surfaces. We break this
into the cases described in \cref{prop:a,prop:a2,prop:a3,prop:a4,prop:b,prop:ab,prop:a2b,prop:a3b}, and in each case
begin with the weight $c=2/3$ surface constructed as a fiber of $\wt\pi : (\wY(E_7),B_{\wt\pi}) \to \wY(E_6)$ as
described above.

\begin{proof}[Proof of \cref{prop:a}]
    Let $(S,B)$ be the surface of \cref{fig:a_12-23}. Fix $c \leq 2/3$. From the explicit description of $(S,B)$, we see
    that the restriction of $K_S+cB$ to an irreducible component of $(S,B)$ of type 3, isomorphic to $\bF_0 \cong \bP^1
    \times \bP^1$, is
    \[
        (-1+5c)h_1 + (-1+2c)h_2,
    \]
    where $h_1,h_2$ are the classes of the two rulings. We find that $K_S+cB$ is ample on this component for $1/2 < c
    \leq 2/3$, and when $c=1/2$, the type 3 components are contracted by the projection $\bP^1 \times \bP^1 \to \bP^1$.
    This is compatible with the analogous contractions for the type 1 and 2 components, and we obtain for $c=1/2$ the
    surface of \cref{fig:a_16-12}. By abuse of notation we now refer to this surface as $(S,B)$, and consider $c \leq
    1/2$. The restriction of $K_S+cB$ to the irreducible component of $(S,B)$ of type 2, isomorphic to $\bF_0 \cong
    \bP^1 \times \bP^1$, is
    \[
        (-1+6c)h_1 + (-1+6c)h_2.
    \]
    We see that $K_S+cB$ is ample on this component for $1/6 < c \leq 2/3$. When $c=1/6$, we contract this entire
    component to a point. The analogous contraction on the type 1 component is the map $\wS_{A_1} \to S_{A_1}$. Thus we
    obtain the surface of \cref{fig:a_19-16}, which is stable for $1/9 < c \leq 1/6$.
\end{proof}

\begin{proof}[Proof of \cref{prop:a2}]
    Let $(S,B)$ be the surface of \cref{fig:a2_12-23}. Fix $c \leq 2/3$. There are three types of components of $(S,B)$
    that are isomorphic to $\bF_0 \cong \bP^1 \times \bP^1$, namely, the components of types 4, 5a, and 5b. We compute
    that the restrictions of $K_S+cB$ to these components are as follows.
    \begin{align*}
        \text{Type 4} &: (-1 + 4c)h_1 + (-1+2c)h_2, \\
        \text{Type 5a} &: 3ch_1 + (-1+2c)h_2, \\
        \text{Type 5b} &: (-1+2c)h_1 + (-1+2c)h_2.
    \end{align*}
    We see that $K_S+cB$ is ample on these components for $1/2 < c \leq 2/3$, and when $c=1/2$, we contract the
    components of types 4 and 5a via the projection $\bP^1 \times \bP^1 \to \bP^1$, and contract the component of type
    5b to a point. This is compatible with the analogous contractions for the remaining components of $(S,B)$, and as a
    result we obtain for $c=1/2$ the surface of \cref{fig:a2_14-12}. By abuse of notation we now refer to this surface
    as $(S,B)$ and consider $c \leq 1/2$. The restrictions of $K_S+cB$ to the components of type 2, isomorphic to
    $Bl_1\bF_0$, and type 3, isomorphic to $Bl_1\bP^2$, are as follows.
    \begin{align*}
        \text{Type 2} &: (-1+6c)h_1 + (-1+6c)h_2 + (1-4c)e, \\
        \text{Type 3} &: (-1+4c)h + c(h-e).
    \end{align*}
    Here $e$ refers to the exceptional divisor in either case, for type 2, $h_1,h_2$ are the pullbacks of the classes of
    the two rulings on $\bP^1 \times \bP^1$, and for type 3, $h$ is the class of the pullback of a line. We find that
    $K_S+cB$ is ample on these components for $1/4 < c \leq 1/2$, and when $c=1/4$, we contract the component of type 3
    vertically onto its exceptional divisor, and at the same time we contract the exceptional divisors of the type 2
    components. As a result, we obtain the surface pictured in \cref{fig:a2_16-14}. By abuse of notation, we now refer
    to this surface as $(S,B)$, and consider $c \leq 1/4$. Now we compute that the restriction of $K_S+cB$ to a
    component of type 2, isomorphic to $\bF_0 \cong \bP^1 \times \bP^1$, is
    \[
        (-1+6c)h_1 + (-1+6c)h_2.
    \]
    Thus when $c=1/6$, we contract the type 2 components to points. The analogous contraction on the type 1 component is
    the map $\wS_{2A_1} \to S_{2A_1}$. Thus we obtain the surface of \cref{fig:a2_19-16}, which is stable for $1/9 < c
    \leq 1/6$.

    We must also describe the wall crossings for the surface of type $aa_2$ pictured in \cref{fig:aa2_12-23}. So let
    $(S,B)$ be this surface, and fix $c \leq 2/3$. The surface $(S,B)$ is isomorphic to the type $a_2$ surface of
    \cref{fig:a2_12-23} away from type 3 component of the $a_2$ surface, which degenerates into four irreducible
    components as pictured in \cref{fig:aa2_12-23}. Thus we can focus our attention on just these four components. One
    component is isomorphic to $Bl_4\bP^2$, and the restriction of $K_S+cB$ to this component is
    \[
        (-1+5c)h + (1-2c)(e_1+e_2+e_3+e_4),
    \]
    where $h$ is the class of the pullback of a line, and $e_1,\ldots,e_4$ are the classes of the exceptional divisors.
    The remaining three components are isomorphic to $\bF_0$, but the lines and gluings differ. For the central $\bF_0$
    in \cref{fig:aa2_12-23_2}, the restriction of $K_S+cB$ is
    \[
        h_1 + ch_2.
    \]
    For the other two $\bF_0$'s, pictured in the top and bottom in \cref{fig:aa2_12-23_2}, the restriction of $K_S+cB$
    is
    \[
        (-1+2c)h_1 + ch_2.
    \]
    Thus we see that when $c=1/2$, we contract these last two $\bF_0$'s via the projection $\bF_0 \cong \bP^1 \times
    \bP^1 \to \bP^1$, and at the same time we contract the four exceptional divisors on the $Bl_4\bP^2$ component. We
    obtain the picture of \cref{fig:aa2_14-12}, where now we have just two components of interest, one isomorphic to
    $\bP^2$, and one isomorphic to $\bF_0$. Again we let $(S,B)$ denote the surface of \cref{fig:aa2_14-12}, and fix $c
    \leq 1/2$. We compute that the restriction of $K_S+cB$ to the $\bP^2$ component is
    \[
        (-1+4c)h,
    \]
    and the restriction of $K_S+cB$ to the $\bF_0$ component is
    \[
        (-1 + 4c)h_1 + ch_2.
    \]
    We find that when $c=1/4$, we contract the $\bP^2$ component to a point, and contract the $\bF_0$ component via the
    projection $\bF_0 \cong \bP^1 \times \bP^1 \to \bP^1$. As a result, we obtain the weight $1/6 < c \leq 1/4$ type
    $a_2$ surface of \cref{fig:a2_16-14}.
\end{proof}

\begin{remark}
    The above proof shows that at the wall $c=1/4$, each $aa_2$ surface collapses to the same surface as the
    corresponding $a_2$ surface. Similar applies for all other surfaces involving $a_i$ for $i=2,3,4$ by the same
    calculation. It follows that at the wall $c=1/4$ we obtain the blowdown of moduli spaces $\wY(E_6) \to \oY(E_6)$.
    This explains the additional moduli appearing in \cref{thm:cubics_main} for weights greater than $1/4$.
\end{remark}

\begin{proof}[Proofs of \cref{prop:a3,prop:a4}]
    In order to save space, we omit the details of the wall crossings for surfaces of types $a_3$ and $a_4$ and their
    degenerations as described in \cref{prop:a3,prop:a4}. The computations are exactly analogous to the type $a$ and
    $a_2$ cases, and the reader should have no trouble adapting these computations to prove \cref{prop:a3,prop:a4}.
\end{proof}

\begin{proof}[Proof of \cref{prop:b}]
    Let $(S,B)$ be the surface of \cref{fig:b_13-23}, and fix $c \leq 2/3$. The restriction of $K_S+cB$ to a component
    of type 2, isomorphic to $\bF_0$, is
    \[
        (-1+3c)h_1 + (-1+3c)h_2.
    \]
    The restriction of $K_S+cB$ to a component of type 1, isomorphic to the blowup of $\bP^2$ at 6 points, 3 on one line
    and 3 on another, is
    \[
        (-1+9c)h + (1-3c)(e_1+ \cdots + e_6),
    \]
    where $h$ is the class of the pullback of a line, and $e_1,\ldots,e_6$ are the classes of the exceptional divisors.
    Thus when $c=1/3$, we contract the type 2 components to points, and obtain the surface of \cref{fig:b_19-13}, which
    is stable for all $1/9 < c \leq 1/3$.
\end{proof}

\begin{proof}[Proof of \cref{prop:ab}]
    Let $(S,B)$ be the surface of \cref{fig:ab_12-23}, and fix $c \leq 2/3$. It is straightforward to observe, using the
    same sorts of computations as in the previous proofs, that $(S,cB)$ is stable for $1/2 < c \leq 2/3$, and when
    $c=1/2$ one contracts the components of types 4b and 6 via the projections $\bP^1 \times \bP^1 \to \bP^1$, yielding
    the surface of \cref{fig:ab_13-12}. Let $(S,B)$ now denote this surface and fix $c \leq 1/2$. We compute from
    \cref{fig:ab_13-12} that the restriction of $K_S+cB$ to any of the components of types 3b, 4, or 5, isomorphic to
    $\bF_0$, is
    \[
        (-1+3c)h_1 + (-1+3c)h_2.
    \]
    Thus when $c=1/3$ we contract these components to points. This is compatible with the contractions on the remaining
    components of $(S,B)$, and as a result we obtain the surface of \cref{fig:ab_16-13}. Again we let $(S,B)$ now denote
    this surface, and fix $c \leq 1/3$. Then the restriction of $K_S+cB$ to a component of type 3, isomorphic to
    $\bP^2$, is
    \[
        (-1+6c)h.
    \]
    Thus when $c=1/6$, we contract the type 3 components to points, and obtain the surface of \cref{fig:ab_19-16}, which
    is stable for weights $1/9 < c \leq 1/6$.
\end{proof}

\begin{proof}[Proof of \cref{prop:a2b,prop:a3b}]
    In order to save space, we also omit the details of the wall crossings for the surfaces of types $a_2b$ and $a_3b$
    and their degenerations as described in \cref{prop:a2b,prop:a3b}. The computations are exactly analogous to those
    carried out in the previous proofs, and indeed many components of these surfaces are isomorphic to components of
    previously described surfaces, with the same lines and gluing loci.
\end{proof}

\begin{proof}[Proof of \cref{thm:cubics_main}]
    By the discussion above, we have now fully described the walls for the moduli space $\oY^6_c$ of $c$-weighted stable
    marked cubic surfaces, as well as all of the surfaces parameterized by the moduli spaces in each chamber. All that
    remains is to prove the descriptions of the moduli spaces. This is straightforward---we already knew (by
    construction) that $\wY(E_6)$ was the moduli space of weighted stable marked cubic surfaces for weight $1$, and our
    descriptions of the stable surfaces show that it contains to be the moduli space until weight $c=1/4$, at which
    point one contracts components of weighted stable surfaces resulting in identifications of surfaces of different
    types. Namely, for any surface whose type involves $a_i$, $i=2,3,4$, we contract the components arising from the
    non-flat $A_3^2$ of $\pi : \oY(E_7) \to \oY(E_6)$ as in \cref{sec:nonflat_A32}. This results in the blowdown map
    $\oY^6_{1/4+\epsilon} \cong \wY(E_6) \to \oY(E_6) \cong \oY^6_{1/4}$, which geometrically is explained by the stable
    models of the surfaces parameterized higher codimension strata involving $a_i$ becoming isomorphic to the stable
    models of the surfaces parameterized by the $a_i$ divisors themselves.
\end{proof}

\begin{remark}
    In particular, our results recover an alternative proof of \cite[Theorem
    1.5]{gallardoGeometricInterpretationToroidal2021}, which asserts that $\oY(E_6)$ is the moduli space of weighted
    stable marked cubic surfaces for weight $c=1/9+\epsilon$.
\end{remark}

\appendix
\section{Bottom-up wall crossing} \label{sec:bottom_up}

As our proof of \cref{thm:cubics_main} in \cref{sec:proof_main} hopefully indicates, it is easy to compute the stable
replacements and wall crossings for weighted stable marked cubic surfaces if one starts with the weight $c=1$ and
decreases $c$. It is much more difficult to work in the opposite direction, beginning with the smallest weight
$c=1/9+\epsilon$ and moving upwards. On the other hand, the latter approach better clarifies the structure of weighted
stable marked cubic surfaces, showing how they appear as stable replacements of simpler surfaces. In this section, we
describe some examples of computing the stable replacements in this ``bottom-up'' fashion, explaining how one can more
directly obtain the pictures of \cref{sec:weighted_cubics}.

\begin{example} \label{ex:bottom_up_a}
    In this example we describe how to obtain the wall crossings for type $a$ surfaces as described in \cref{prop:a}.

    Begin with a trivial one-parameter family $\bP^2 \times T \to T$ with six sections $B_1,\ldots,B_6$ cutting out six
    points in general position on a general fiber, but such that $B_5$ and $B_6$ intersect at a point in the special
    fiber. The blowup of the scheme-theoretic union of $B_1,\ldots,B_6$ in $\bP^2 \times T$ results in a family $\cS$ of
    surfaces over $T$ whose general fiber is a smooth marked cubic surface, but whose special fiber is a marked cubic
    surface with an $A_1$ singularity, as in \cref{fig:a_19-16}. Working locally near the intersection point of $B_5$
    and $B_6$, we can write $B_5 = V(x,y)$, $B_6 = V(x,y-t)$ in $\bA^2_{x,y} \times \bA^1_t$. Then a direct computation
    shows that the $A_1$ singularity of the special fiber of $\cS \to T$ is also an $A_1$ singularity of the ambient
    family $\cS$. Then another local computation shows that blowing up this point in $\cS$ results in a family $\cwS$ of
    surfaces over $T$ whose general fiber is a smooth marked cubic surface, but whose special fiber is a surface with
    two irreducible components---one isomorphic to the minimal resolution $\wS_{A_1}$ of a cubic surface with an $A_1$
    singularity, and the other isomorphic to $\bF_0 \cong \bP^1 \times \bP^1$, glued along its diagonal to the
    exceptional fiber of $\wS_{A_1}$ over the $A_1$ singularity. In other words, the central fiber of $\cwS$ is the
    weight $1/6 < c \leq 1/2$ stable marked surface $(S,cB)$ of \cref{fig:a_16-12}. When $c > 1/2$, this surface is no
    longer stable, because it is no longer slc along the 6 lines of multiplicity 2. We therefore blowup the reduced
    subscheme supported on these 6 lines, and attach to each of the 6 exceptional divisors a copy of $\bF_0$, obtaining
    the surface of \cref{fig:a_12-23} as desired.

    We explain in some more detail the computation of the last step above. Write $X_1 \cong \wS_{A_1}$ and $X_2 \cong
    \bF_0$ for the two components of the special fiber of $\cwS \to T$. Then since $X_1 \cup X_2$ is a fiber of $\cwS
    \to T$, it has trivial normal bundle, so $(X_1 + X_2)\vert_{X_1} = 0$. On the other hand, since $X_2$ is attached to
    $X_1$ along the exceptional fiber $e$ of $X_1 \cong \wS_{A_1} \to S_{A_1}$, we have $X_2\vert_{X_1} = e$. We
    conclude that $N_{X_1/\cwS} = \cO(-e)$. Now, if $\ell$ is the support of a multiplicity 2 line in $X_1$, then
    $N_{\ell/X_1} = \cO(-1)$. The exact sequence
    \[
        0 \to N_{\ell/X_1} \to N_{\ell/\cwS} \to N_{X_1/\cwS}\vert_{\ell} \to 0
    \]
    splits, and we find that $N_{\ell/\cwS} \cong \cO(-1) \oplus \cO(-1)$. Thus the blowup of $\cwS$ along $\ell$ has
    exceptional divisor $\bP(N_{\ell/\cwS}) \cong \bP^1 \times \bP^1$, as desired.
\end{example}

\begin{example} \label{ex:bottom_up_a2}
    In this example we describe how to obtain the wall crossings for type $a_2$ surfaces as described in
    \cref{prop:a2}.

    Begin with a trivial one-parameter family $\bP^2 \times T \to T$ with six sections $B_1,\ldots,B_6$ cutting out six
    points in general position on a general fiber, but such that $B_5$ and $B_6$ intersect at a point in the special
    fiber, and $B_3$ and $B_4$ intersect at another point in the special fiber. The same local computation as in
    \cref{ex:bottom_up_a} shows that the blowup of $\bP^2 \times T$ along the scheme-theoretic union of $B_1,\ldots,B_6$
    results in a family $\cS$ of surfaces over $T$ whose general fiber is a smooth marked cubic surface, but whose
    special fiber is a marked cubic surface with two $A_1$ singularities, as in \cref{fig:a2_19-16}; furthermore, the two
    $A_1$ singularities of the special fiber are also $A_1$ singularities of the ambient family $\cS$. Again a local
    computation as in \cref{ex:bottom_up_a} shows that blowing up these two singular points results in a family $\cwS
    \to T$ with central fiber consisting of three irreducible components---one isomorphic to the minimal resolution
    $\wS_{2A_1}$ of a cubic surface with two $A_1$ singularities, and two components isomorphic to $\bF_0$, each glued
    along its diagonal to one of the exceptional fibers of $\wS_{2A_1}$, giving the picture of \cref{fig:a2_16-14}.

    Now the central fiber of $\cwS \to T$ is stable for weights $1/6 < c \leq 1/4$, but is no longer stable for $c >
    1/4$, because it is no longer slc along the line of multiplicity 4, appearing on $\wS_{2A_1}$ as the strict
    transform of the line connecting the two $A_1$ singularities of $S_{2A_1}$. A similar computation to the last
    paragraph of \cref{ex:bottom_up_a} shows that $N_{\wS_{2A_1}/\cwS} = \cO(-e_1-e_2)$, where $e_1,e_2$ are the two
    exceptional fibers of $\wS_{2A_1} \to S_{2A_1}$, and if $\ell$ is the support of the line of multiplicity 4, then
    $N_{\ell/\wS_{2A_1}} = \cO(-1)$, and the exact sequence
    \[
        0 \to N_{\ell/\cwS_{2A_1}} \to N_{\ell/\cwS} \to N_{\wS_{2A_1}/\cwS} \to 0
    \]
    splits, showing that
    \[
        N_{\ell/\cwS} \cong \cO(-1) \oplus \cO(-2).
    \]
    Thus the blowup of $\cwS$ along $\ell$ has exceptional divisor $\bP(N_{\ell/\cwS}) \cong \bF_1 \cong Bl_1\bP^2$, and
    we obtain as the central fiber the surface of \cref{fig:a2_14-12}, stable for weights $1/4 < c \leq 1/2$.

    As a final step, when $c > 1/2$, we must resolve the lines of multiplicity 2 in the special fiber, as done in
    \cref{ex:bottom_up_a}. This step is somewhat more involved, because some of the multiplicity 2 lines intersect, and
    over the intersection points we should obtain additional components (see \cref{fig:a2_12-23}). One way to perform
    this step is as follows. Let $L$ and $L'$ be two lines of multiplicity two meeting in a point in the central fiber.
    First blowup the intersection point. This is a smooth point in the family, and the exceptional divisor is a copy of
    $\bP^2$, glued along a general line to the exceptional line in the central fiber. Now the lines $L$ and $L'$ split
    into four lines on $\bP^2$, with two of the lines intersecting the gluing locus at one point, and the other two
    lines intersecting the gluing locus at another point, as pictured in \cref{fig:bottom_up_a2_1}.

    Next we blowup the reduced support of the strict transforms of $L$ and $L'$; this is just the reduced support of the
    intersections of $L$ and $L'$ with the strict transform of the $\wS_{2A_1}$ component. We obtain two copies of
    $\bF_0$ over the supports of these two lines, and the strict transform of the $\bP^2$ component is isomorphic to
    $Bl_2\bP^2$, the blowup at the intersection points of $L$ and $L'$ with $\bP^2$, as pictured in
    \cref{fig:bottom_up_a2_2}.  This component is not stable---the restriction of its log canonical class to the strict
    transform of the line through the two blown up points is zero---and we contract the gluing line to a point, yielding
    the contraction $Bl_2\bP^2 \to \bP^1 \times \bP^1$. This results in the desired picture of \cref{fig:a2_12-23_3}.
    \begin{figure}[!htpb]
        \centering
        \begin{subfigure}[t]{0.3\textwidth}
            \centering
            \includegraphics[width=0.9\linewidth]{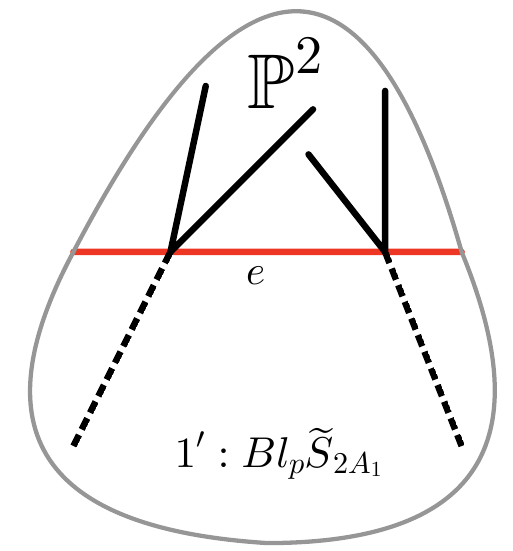}
            \caption{A local picture of the blowup of the intersection point of two lines of multiplicity 2 on
            $\wS_{2A_1}$. We attach a copy of $\bP^2$ glued along a general line to the exceptional divisor of the
            blowup, and the two lines of multiplicity 2 are distributed on this $\bP^2$ as shown.}
            \label{fig:bottom_up_a2_1}
        \end{subfigure}
        \begin{subfigure}[t]{0.6\textwidth}
            \centering
            \includegraphics[width=0.9\linewidth]{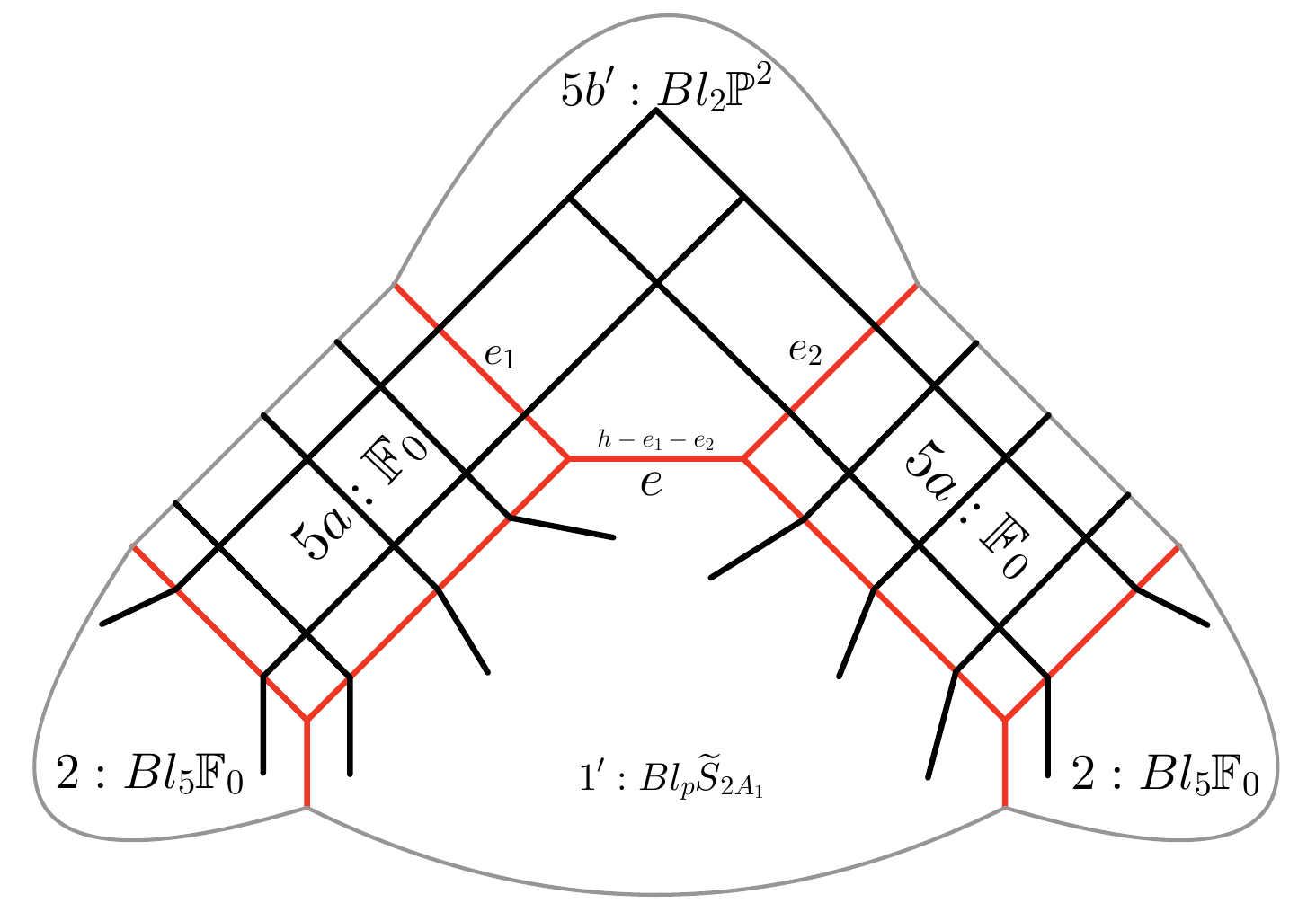}
            \caption{The blowup of the two lines of multiplicity 2 in \cref{fig:bottom_up_a2_1}. We draw more details in
                this picture outside of the more zoomed in picture of \cref{fig:bottom_up_a2_1}, in order to draw the
                connection to \cref{fig:a2_12-23_3}.  The $\bP^2$ component in \cref{fig:bottom_up_a2_1} becomes the
                component isomorphic to $Bl_2\bP^2$, labeled 5b'. Contracting the line connecting the two exceptional
            divisors on this component, we obtain the picture of \cref{fig:a2_12-23_3}.}
            \label{fig:bottom_up_a2_2}
        \end{subfigure}
        \caption{The steps to resolve two lines of multiplicity 2 intersecting in a point.}%
        \label{fig:bottom_up_a2}
    \end{figure}
\end{example}

\begin{example} \label{ex:aa_2-extra_moduli}
    In this example we explain the degenerations of $a_2$ surfaces that occur starting at the wall $c=1/4$, and as a
    result obtain a geometric explanation for the extra moduli appearing in the blowup $\wY(E_6) \to \oY(E_6)$ at
    $c=1/4+\epsilon$.

    As discussed in \cref{ex:bottom_up_a2}, increasing the weight $c$ beginning at $c=1/9+\epsilon$, we obtain at
    $c=1/4+\epsilon$ the surface of \cref{fig:a2_14-12}. The new component of this surface, labeled $3: \bF_1 =
    Bl_1\bP^2$ in \cref{fig:a2_14-12} introduced at the weight $c=1/4$ has nontrivial degenerations as shown in
    \cref{fig:aa2_14-12}.  These degenerations can be obtained as follows. Label the five lines on $3 : \bF_1$ as
    $1,\ldots,5$, and the intersections of two such lines $i,j$ by $ij$, as shown in \cref{fig:14_labels}. There are two
    degenerations of this configuration of five lines: one where the points $12$ and $34$ coincide, as shown in
    \cref{fig:12=34}, and symmetrically where the points $13$ and $24$ coincide. When either of these degenerations
    occur, we obtain a point on this surface where four lines of weight $c=1/4+\epsilon$ meet a boundary line (of weight
    1). The sum of the weights at this intersection point is therefore greater than $2$, so the surface is not slc at
    this point. Blowing up this point, we see that the surface $3 : \bF_1 = Bl_1\bP^2$ splits into two surfaces, one
    isomorphic to $\bP^2$, and the other isomorphic to $Bl_2\bP^2$, as shown in \cref{fig:12=34_blowup}. The restriction
    of the log canonical class to the component isomorphic to $Bl_2\bP^2$ is
    \[
        (-1 + 5c)h - ce_1 - (4c-1)e_2.
    \]
    A direct calculation shows this is ample except along the line $h-e_1-e_2$, where the intersection product is $0$.
    Therefore, we contract this line, yielding the contraction $Bl_2\bP^2 \to \bF_0$ and the final picture of
    \cref{fig:aa2_14-12}.

    \begin{figure}[!htpb]
        \centering
        \begin{subfigure}[t]{0.6\textwidth}
            \centering
            \includegraphics[width=0.9\linewidth]{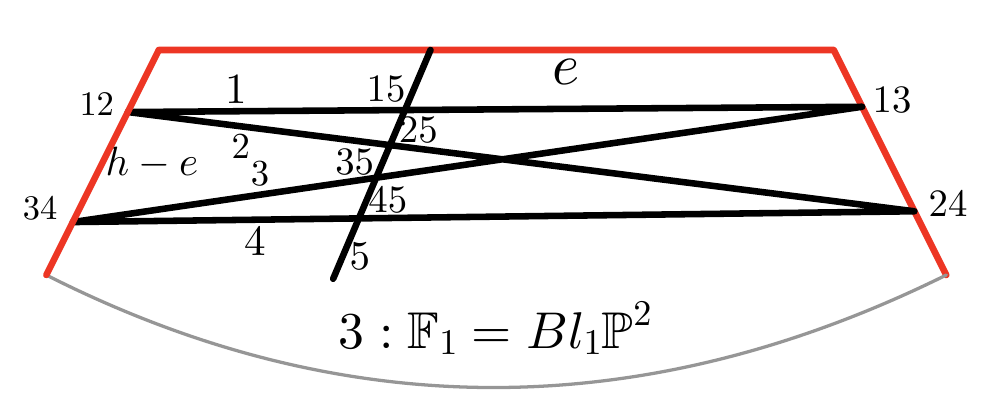}
            \caption{A labeling of the lines and points on the $3 : \bF_1 = Bl_1\bP^2$ component of a type $a_2$ surface
            for weights $1/4 < c \leq 1/2$.}
            \label{fig:14_labels}
        \end{subfigure}
        \begin{subfigure}[t]{0.6\textwidth}
            \centering
            \includegraphics[width=0.9\linewidth]{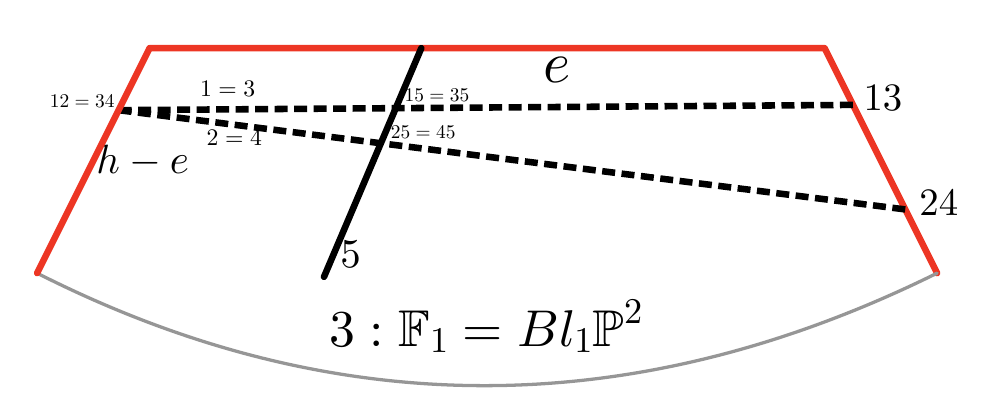}
            \caption{A degeneration of the lines on the $3 : \bF_1 = Bl_1\bP^2$ component pictured in
            \cref{fig:14_labels} so that lines $1$ and $3$ coincide, and lines $2$ and $4$ coincide. Notice there is a
        symmetric degeneration (not pictured) where lines $1$ and $2$ coincide, and lines $3$ and $4$ coincide.}
            \label{fig:12=34}
        \end{subfigure}
        \begin{subfigure}[t]{0.6\textwidth}
            \centering
            \includegraphics[width=0.9\linewidth]{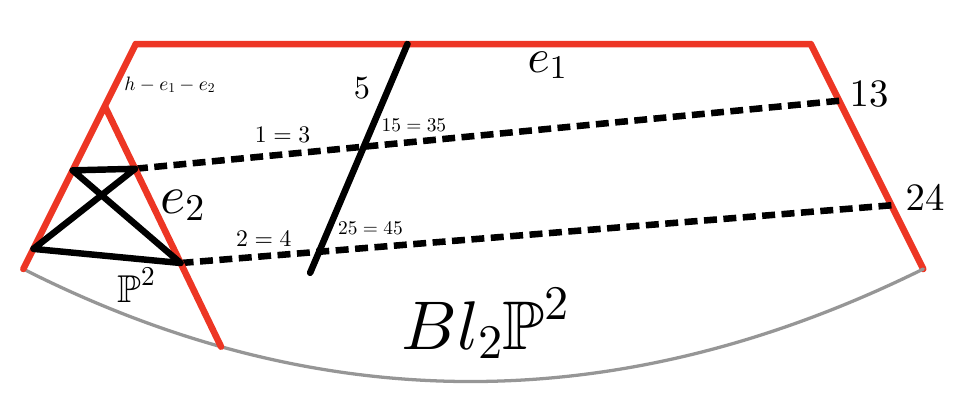}
            \caption{The blowup of the degeneration pictured in \cref{fig:12=34} at the intersection point labeled
            $12=34$. The resulting surface is not stable along the line $h - e_1 - e_2$; contracting this line yields
        the stable surface pictured in \cref{fig:aa2_14-12}.}
            \label{fig:12=34_blowup}
        \end{subfigure}
        \caption{An isolated view of a degeneration of the type $3 : \bF_1 = Bl_1\bP^2$ component of a type $a_2$
        surface for weights $1/4 < c \leq 1/2$, and the steps required to obtain its stable replacement as pictured in
    \cref{fig:aa2_14-12}.}%
        \label{fig:14_degens}
    \end{figure}
\end{example}

\begin{remark}
    It follows from \cref{lem:nonflat_A32} that the extra component $3 : \bF_1$ in the $a_2$ surface, and its two
    degenerations as described in the above example, are all parameterized by $\bP^1$, where the two degenerations
    correspond to the points $0$ and $1$ with an appropriate choice of coordinates. This can be identified with the
    weighted moduli space of four points on $\bP^1$, with weights $(1/4, 3/4, 3/4, 1/4)$, by taking the four
    intersection points $15,25,35,45$ on the line $5$ in \cref{fig:14_labels}.

    Indeed, in this labeling, note that the degeneration of $3 : \bF_1$ where points $12$ and $34$ coincide yields the
    degeneration $15=35, 25=45$ on the line $5$ as shown in \cref{fig:12=34}, and symmetrically the degeneration where
    points $13$ and $24$ coincide yields the degeneration $15=25, 35=45$ on the line $5$.

    Note that the degeneration $15=45, 25=35$ on the line $5$ is not stable with these weights, since the sum of the
    weights on points $25$ and $35$ is greater than one. This degeneration occurs when the lines $2,3,5$ all intersect
    at an Eckardt point. The blowup of this Eckardt point yields the corresponding stable $4$-pointed curve with two
    copies of $\bP^1$ meeting at a point; the points $25$ and $35$ on one component, and the points $15$ and $45$ on
    another component, cf. \cref{fig:eckardt}, where the three lines in that figure are labeled $2,3,5$ here.
\end{remark}

\begin{example} \label{ex:bottom_up_b}
    In this example we describe how to obtain the wall-crossings for type $b$ surfaces as described in \cref{prop:b}.

    In \cite{narukiCrossRatioVariety1982}, Naruki gives explicit equations for the universal family over $\oY(E_6)$,
    including a description of the lines on this family and their degenerations to type $b$ fibers, see also
    \cite[Section 5.2]{gallardoGeometricInterpretationToroidal2021}. In particular, Naruki's family directly describes
    the type $b$ surfaces pictured in \cref{fig:b_19-13}, and a direct computation shows that these surfaces are stable
    for $1/9 < c \leq 1/3$, while for $c > 1/3$, these surfaces are no longer slc at the 9 points where 3 lines on a
    plane intersect. Using Naruki's explicit equations, we compute that at each of these 9 points, the exceptional
    divisor of the blowup is a union of two $\bP^2$'s as pictured in \cref{fig:b_13_step_1}. Note that these new components
    are still not slc because of the intersection point of the three lines on each component. So we blowup this point as
    well; another local computation shows that the exceptional divisor here is isomorphic to $\bF_0$, with lines as
    shown in \cref{fig:b_13_step_2}. On the resulting blowup of each $\bP^2$ component at a point, we compute that the
    restriction of the log canonical class is
    \[
        (-1 + 3c)h + (1 - 3c)e.
    \]
    The intersection product of this with $h-e$ is $0$, while the intersection product with $e$ is $3c - 1$, which is
    greater than $0$ since $c > 1/3$. It follows that each $Bl_1\bP^2$ component contracts onto the exceptional divisor,
    and we obtain the surface of \cref{fig:b_13-23}.

    \begin{figure}[!htpb]
        \centering
        \begin{subfigure}[t]{0.6\textwidth}
            \centering
            \includegraphics[width=0.9\linewidth]{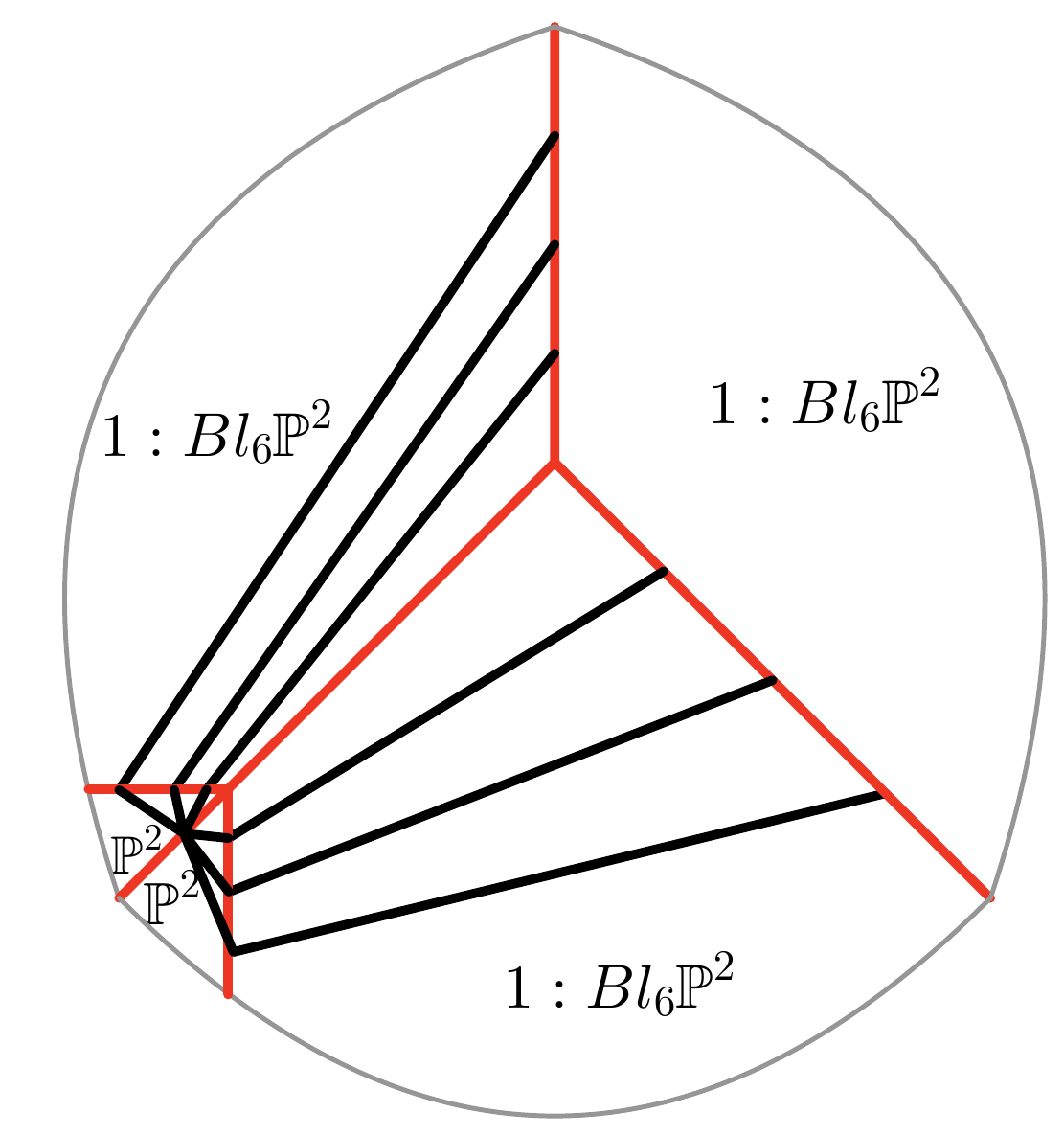}
            \caption{The blowup of a weight $1/9 < c \leq 1/3$ type $b$ surface at one of the 9 points where the surface
            is no longer stable for $c > 1/3$. The exceptional divisor consists of two copies of $\bP^2$ as shown. For
            cleanliness we show only one of these blowups and the relevant lines; the full picture includes the same
        blowup at all of the 9 intersection points shown in \cref{fig:b_19-13}.}
            \label{fig:b_13_step_1}
        \end{subfigure}
        \begin{subfigure}[t]{0.6\textwidth}
            \centering
            \includegraphics[width=0.9\linewidth]{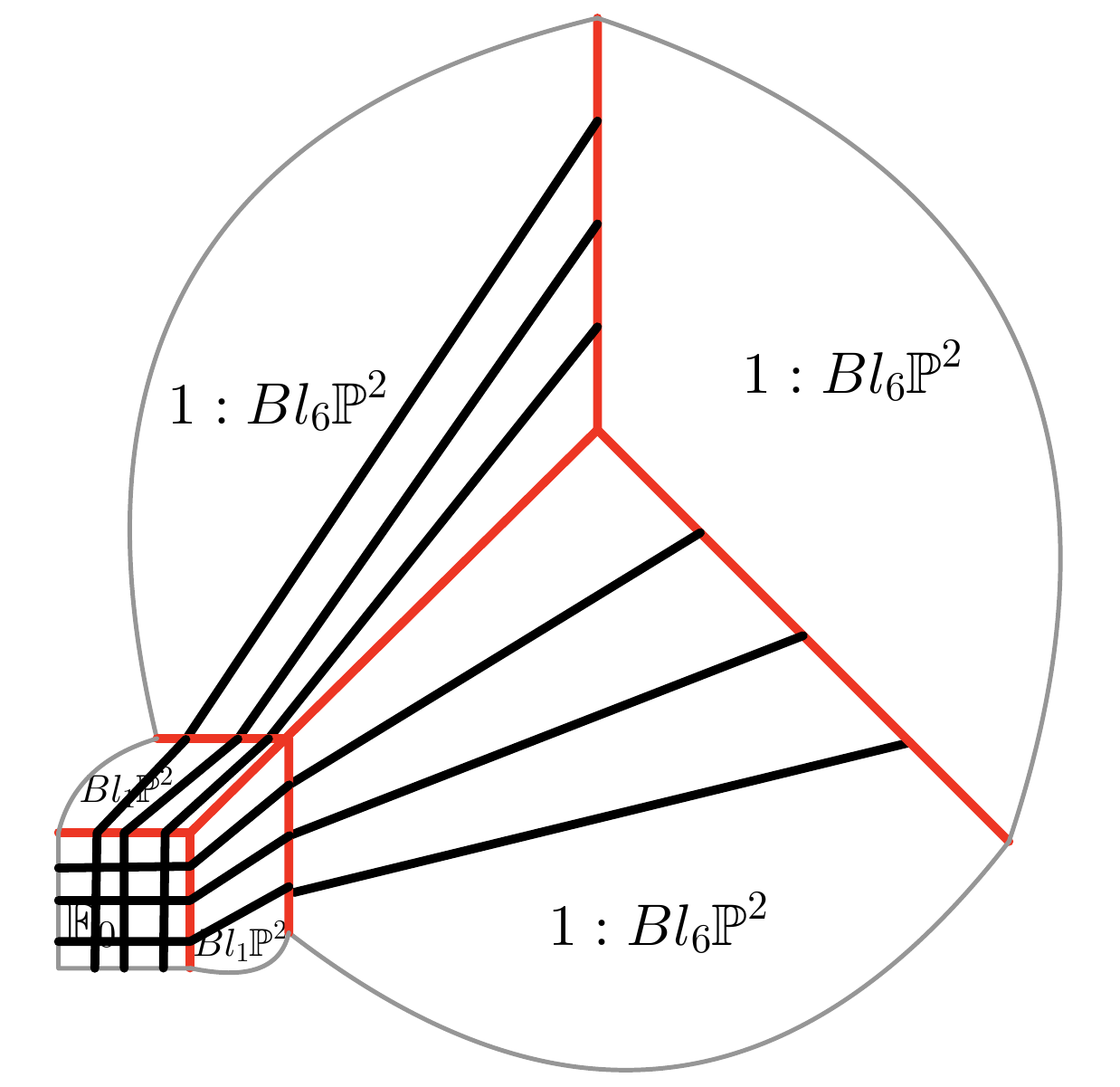}
            \caption{The next blowup of the surface of \cref{fig:b_13_step_1}, at the resulting point where the surface
            is still not stable for weights $c > 1/3$. The exceptional divisor of this blowup is a copy of $\bF_0$.
        Contracting the $Bl_1\bP^2$ components connected to this $\bF_0$ we obtain the picture of \cref{fig:b_13-23}.
    (As in \cref{fig:b_13_step_1}, we show only one of the 9 identical blowups to obtain the full \cref{fig:b_13-23}.)}
            \label{fig:b_13_step_2}
        \end{subfigure}
        \caption{The intermediate blowups required to obtain the stable replacement of a weight $1/9 < c \leq 1/3$
        type $b$ surface for weights $c > 1/3$.}%
        \label{fig:bottom_up_b}
    \end{figure}
\end{example}

\begin{remark}
    Note that a similar sequence of transformations to the above examples yields the descriptions of the surfaces of
    types $a_2b$, $a_3b$, etc.
\end{remark}

In \cite[Theorem 1.5]{gallardoGeometricInterpretationToroidal2021}, Gallardo, Kerr, and Schaffler describe all of the
weighted stable marked cubic surfaces for the minimal weight $c=1/9+\epsilon$. Carrying out the analogous computations
to the above examples for the remaining surfaces in this moduli space, one may give an alternative proof of
\cref{thm:cubics_main} in a bottom-up fashion. In particular, this approach also partially recovers the result of
\cite[Theorem 10.31]{hackingStablePairTropical2009}, describing the moduli space of weight 1 stable marked cubic
surfaces, while at the same time going a step further by explicitly constructing such surfaces via explicit wall
crossings.

\begin{remark} \label{rmk:Kmoduli}
    For weights $0 < c < 1/9$, one may study the K-moduli space of \emph{K-semistable} weighted marked cubic surfaces
    $(S,cB)$. In \cite{zhaoCompactificationsModuliPezzo2023} shows that there are no walls on the K-moduli side. The
    K-moduli space $\oY_{(0,1/9)}^6$ of weighted marked cubic surfaces is the natural marked cover of the usual GIT
    moduli space of cubic surfaces, cf. \cite{odakaCompactModuliSpaces2016}; said differently, it is obtained from the
    moduli space $\oY_{1/9+\epsilon}^6 \cong \oY(E_6)$ by contracting the 40 $A_2^3$ divisors $\cong (\bP^1)^3$ to 40
    singular points, each locally isomorphic to the cone point of the cone over the Segre embedding of $(\bP^1)^3$ (see
    \cite{narukiCrossRatioVariety1982}, where Naruki explicitly constructs this contraction). These singular points
    parameterize the different markings of the unique cubic surface $\{w^3=xyz\} \subset \bP^3$ with three $A_2$
    singularities. In principle it should be possible to run a log minimal model program as in the above examples to
    replace this cubic surface with the union of three planes, as parameterized by the $A_2^3$ divisors in
    $\oY_{1/9+\epsilon}^6$; however, carrying this out in detail seems to be challenging. It is also interesting to
    consider the pairs for the log Calabi-Yau boundary $c=1/9$---steps towards studying moduli of log Calabi-Yau pairs
    have been taken in \cite{ascherModuliBoundaryPolarized2023,blumGoodModuliSpaces2024}, and this intermediate moduli
    space should capture common degenerations of surfaces appearing on both the weight $0 < c < 1/9$ and the weight $1/9
    < c \leq 1$ sides.
\end{remark}

\section{Weighted stable marked del Pezzo surfaces of degree 4} \label{sec:deg4}

In this appendix we describe the wall crossings for moduli of weighted stable marked del Pezzo surfaces of degree 4.
This is a much simpler version of the case of weighted stable marked cubic surfaces described in the main body of this
article.

Recall from \cref{sec:comps,thm:comps} that $Y(E_5)$ denotes the moduli space of marked del Pezzo surfaces of degree 4,
and $\oY(E_5) \cong \oM_{0,5}$ is the log canonical compactification of $Y(E_5)$. (Indeed, $Y(E_5)$ is isomorphic to the
moduli space of 5 points in general position in $\bP^2$, which by Gale duality is isomorphic to the moduli space
$M_{0,5}$ of 5 points in general position in $\bP^1$, cf. \cite[Section 2.3]{kapranovChowQuotientsGrassmannians1993},
\cite[Section 4.9]{alexeevModuliWeightedHyperplane2015}.) There are 10 boundary divisors on $\oY(E_5)$,
identified with $D_2 = A_1 \times A_1$ root subsystems of $E_5$; each boundary divisor is isomorphic to $\oM_{0,4} \cong
\bP^1$, and two boundary divisors intersect (in a point) if the corresponding $D_2$ subsystems of $E_5$ are orthogonal.

\subsection{The morphism $\pi : \oY(E_6) \to \oY(E_5)$}

The following theorem due to Hacking, Keel, and Tevelev \cite[Proposition 10.9]{hackingStablePairTropical2009} shows
that the natural morphism $Y(E_6) \to Y(E_5)$ obtained by contracting a $(-1)$-curve extends to an explicitly described
morphism of compactifications $\oY(E_6) \to \oY(E_5)$. This is the analogue for degree 4 del Pezzo surfaces of
\cref{thm:pi}, cf. \cref{rmk:deg4_map}. In contrast to the case of cubic surfaces, the morphism $\pi : \oY(E_6) \to
\oY(E_5)$ is already flat, and gives the universal family of (unweighted) stable marked del Pezzo surfaces of degree 4
\cite[Theorem 10.19]{hackingStablePairTropical2009}.

\begin{theorem}[{\cite[Proposition 10.9]{hackingStablePairTropical2009}}] \label{thm:pi_deg4}
    The morphism $Y(E_6) \to Y(E_5)$ obtained by contracting a $(-1)$-curve on a del Pezzo surface of degree 3 extends
    to a morphism of compactifications $\pi : \oY(E_6) \to \oY(E_5)$, which is flat with reduced fibers. The restriction
    of $\pi$ to each boundary divisor of $\oY(E_6)$ is described in \cref{tab:pi_deg4}, where each morphism $\pi_I :
    \oM_{0,n} \to \oM_{0,m}$ denotes the canonical fibration dropping points not labeled by $I$.
    \begin{table}[htpb]
        \centering
        \caption{The morphism $\pi : \oY(E_6) \to \oY(E_5)$}
        \label{tab:pi_deg4}
        \begin{tabular}{| c | c | c | c |}
            \hline
            $D(\Theta) \subset \oY(E_6)$ & $\pi(D(\Theta)) \subset \oY(E_5)$ & Condition & $\pi\vert_{D(\Theta)}$ \\
            \hline
            \hline
            \multirow{2}{*}{$D(A_1)$} & \multirow{2}{*}{$D(D_2)$} & $A_1 \subset E_5$ & \multirow{2}{*}{$\oM_{0,6} \to \oM_{0,4}$} \\
                     & & $A_1 \times A_1^{\perp} = D_2 \times A_3 \subset E_5$ & \\
            $D(A_1)$ & $\oY(E_5)$ & $A_1 \not\subset E_5$ & $\oM_{0,6} \to \oY(E_5)$ \\
            \hline
            $D(A_2^3)$ & $D(D_2)$ & $A_2^3 \cap E_5 = A_2 \times D_2$ & $(\oM_{0,4})^3 \to \oM_{0,4}$ \\
            \hline
        \end{tabular}
    \end{table}
\end{theorem}

\subsection{Fibers of $\pi : \oY(E_6) \to \oY(E_5)$} \label{sec:fibers_deg4}

Let $B_{\pi} \subset \oY(E_6)$ be the sum of the horizontal $A_1$ divisors of $\pi : \oY(E_6) \to \oY(E_5)$, i.e., the
$A_1$ divisors of $\oY(E_6)$ that surject onto $\oY(E_5)$ (cf. \cref{not:hor_divs}). As mentioned above and shown in
\cite[Theorem 10.19]{hackingStablePairTropical2009}, the morphism $\pi : (\oY(E_6),B_{\pi}) \to \oY(E_5)$ gives the
universal family of (unweighted) stable marked del Pezzo surfaces of degree 4.  In the following theorem, we explicitly
compute the fibers of this family. This is a much simpler version of \cref{sec:fibers}. We remark that the fibers of
$\pi : (\oY(E_6),B_{\pi}) \to \oY(E_5)$ were also previously computed by Hassett, Kresch, and Tschinkel, who describe
them (without proof) in \cite[Section 7]{hassettModuliDegreePezzo2014}.

In the pictures that follow, we use red lines to indicate the gluing loci between different irreducible components, and
black lines to indicate the marked lines on the irreducible components.

\begin{theorem} \label{thm:fibers_deg4}
    \begin{enumerate}
        \item The fiber of $\oY(E_6) \to \oY(E_5)$ over a general point in the interior $Y(E_5)$ is a smooth marked del
            Pezzo surface $(S,B)$ of degree $4$.
        \item The fiber of $\oY(E_6) \to \oY(E_5)$ over a general point of a boundary divisor is a surface $(S,B)$ with
            6 irreducible components as pictured in \cref{fig:deg4_1}.
            \begin{enumerate}
                \item 2 components isomorphic to $Bl_4\bF_0$, the blowup of $\bF_0$ along 4 points on its diagonal.
                \item 4 components isomorphic to $\bF_0$.
            \end{enumerate}
        \item The fiber of $\oY(E_6) \to \oY(E_5)$ over a codimension 2 boundary stratum is a surface $(S,B)$ with 12
            irreducible components as pictured in \cref{fig:deg4_2}.
            \begin{enumerate}
                \item 4 components isomorphic to $\oM_{0,5} \cong Bl_4\bP^2$, the blowup of $\bP^2$ at 4 points in
                    general position.
                \item 8 components isomorphic to $\bF_0$.
            \end{enumerate}
    \end{enumerate}
\end{theorem}

\begin{figure}[htpb]
    \centering
    \includegraphics[width=0.55\linewidth]{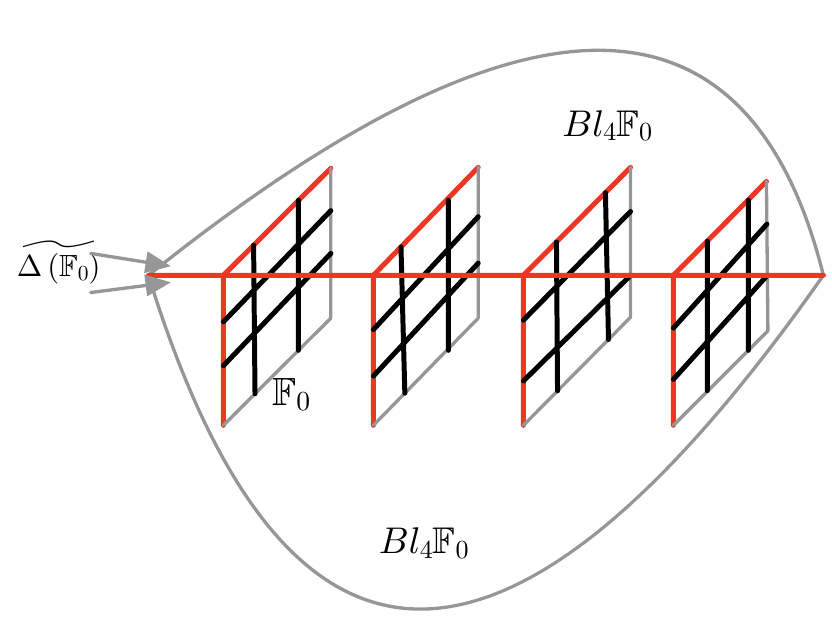}
    \caption{A weighted stable marked del Pezzo surface of degree 4 for weights $1/2 < c \leq 1$, parameterized by a
        general point of a boundary divisor of $\oY^5_{(1/4,1/2]}$. This surface has 6 irreducible components, 2
        isomorphic to $Bl_4\bF_0$ and 4 isomorphic to $\bF_0$ (compare with \cref{fig:a_12-23}). Though not drawn, the
        lines on the $\bF_0$ components split into the strict transforms of lines in the two rulings of $\bF_0$ on each
    $Bl_4\bF_0$ component.}
    \label{fig:deg4_1}
\end{figure}

\begin{figure}[!htpb]
    \centering
    \begin{subfigure}[t]{0.5\textwidth}
        \centering
        \includegraphics[width=0.9\linewidth]{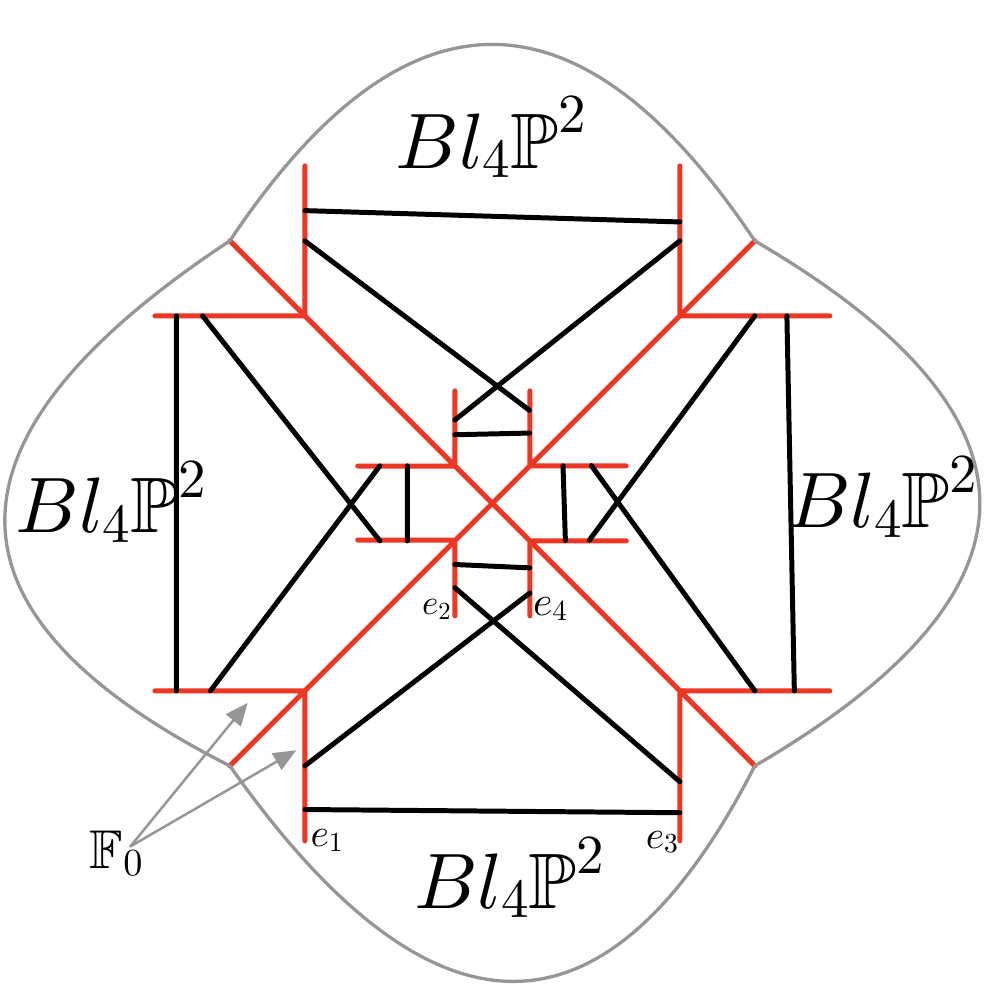}
        \caption{A surface parameterized by a codimension 2 boundary stratum of $\oY^5_{(1/4,1/2]}$.}
        \label{fig:deg4_2_1}
    \end{subfigure}
    \begin{subfigure}[t]{0.3\textwidth}
        \centering
        \includegraphics[width=0.9\linewidth]{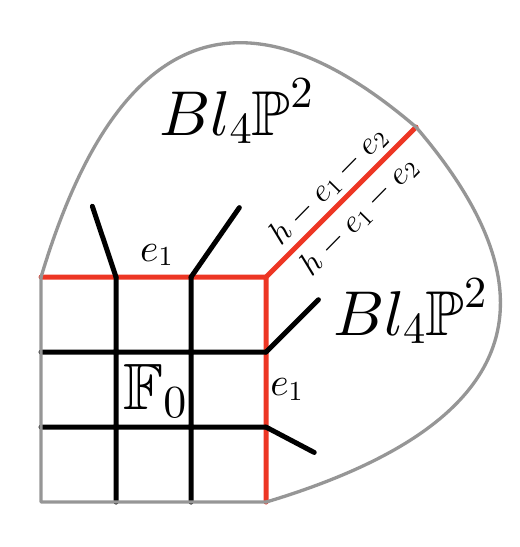}
        \caption{A closer view of an $\bF_0$ component.}
        \label{fig:deg4_2_2}
    \end{subfigure}
    \caption{A weighted stable marked del Pezzo surface of degree 4 for weights $1/4 < c \leq 1/2$, parameterized by a
    codimension 2 boundary stratum of $\oY^5_{(1/4,1/2]}$. This surface has 12 irreducible components, 4 isomorphic to
    $Bl_4\bP^2$ and 8 isomorphic to $\bF_0$.}%
    \label{fig:deg4_2}
\end{figure}

\begin{proof}
    Part (1) is immediate.

    All boundary divisors are permuted equivariantly by $W(E_5)$, likewise for codimension 2 boundary strata, so for
    parts (2) and (3) it suffices to take a particular example.

    Consider the boundary divisor $D=D(45,123)$ of $\oY(E_5)$.  The divisors of $\oY(E_6)$ mapping onto $D$ are
    $D(45)$, $D(123)$, and $D(A_2^3)$ for the following four choices of $A_2^3$.
    \begin{align*}
        (123,456,7) \times (12,13,23) \times (45,46,56), \\
        (45,146,156) \times (16,123,236) \times (23,245,345) \\
        (45,246,256) \times (26,123,136) \times (13,145,345) \\
        (45,346,356) \times (36,123,126) \times (12,145,245).
    \end{align*}
    The restriction of $\pi$ to a $D(A_1)$ divisor induces a forgetful map $\oM_{0,6} \to \oM_{0,4}$ by
    \cref{tab:pi_deg4}, whose general fiber is the blowup of $\bF_0$ at 4 points on the diagonal by
    \cref{lem:fiber_m0n}. The intersection of this with the other $D(A_1)$ fiber is the strict transform of the diagonal
    of $\bF_0$, and the intersections with the $A_2^3$ divisors are the 4 exceptional divisors. This describes the
    double locus on this component. The lines on this component are given by the intersections with the horizontal $A_1$
    divisors. For instance, the $D(45)$ component intersects the horizontal $A_1$ divisors $D(16), D(26), D(36), D(7)$
    and $D(126), D(136), D(236), D(456)$ in the strict transforms of 4 horizontal and 4 vertical rulings on $\bF_0$.

    Similarly, the restriction of $\pi$ to a $D(A_2^3)$ divisor induces a map $(\oM_{0,4})^3 \to \oM_{0,4}$ whose
    general fiber is $\bF_0 \cong \oM_{0,4} \times \oM_{0,4}$. Any two $A_2^3$ divisors are disjoint, while each $A_2^3$
    divisor meets the two $A_1$ divisors, in the two rulings on $\bF_0$. Furthermore, $D(A_2^3)$ for, e.g., the first
    $A_2^3$ above intersects the horizontal $A_1$ divisors $D(456)$, $D(7)$ and $D(46), D(56)$ in 2 vertical and 2
    horizontal rulings. The description of the $D(A_2^3)$ components follows, and we get the complete description of the
    fiber over a general point of $D(45,123)$ by gluing the above descriptions.

    Now consider the codimension 2 boundary stratum $z = D(45,123) \cap D(12,345) \subset \oY(E_5)$. There are 12
    codimension 2 boundary strata of $\oY(E_6)$ which map onto $z$, 4 of the form $D(A_1) \cap D(A_1)$, and 8 of the
    form $D(A_1) \cap D(A_2^3)$. The induced maps on these boundary strata are $\oM_{0,5} \to z$ and $\oM_{0,4} \times
    \oM_{0,4} \to z$, respectively. A similar calculation to the above gives the description of the lines and double
    loci of these components.
\end{proof}

\subsection{Moduli of weighted stable marked del Pezzo surfaces of degree 4}

For $1/4 < c \leq 1$, let $\oY^5_c$ denote the moduli space of weighted stable marked del Pezzo surfaces of degree 4,
with constant weight $c$, as in \cref{sec:weighted_cubics,thm:gen_walls}. In this section we describe the wall crossings
for $\oY^5_c$ as one varies $c$ in the interval $(1/4,1]$. As in \cref{sec:weighted_cubics}, if $(t_{i-1},t_i]$ is a
chamber, we write $\oY_{(t_{i-1},t_i]}^5 \cong \oY^5_c$ for any $c \in (t_{i-1},t_i]$. The following is the analogue for
degree 4 del Pezzo surfaces of \cref{thm:cubics_main}.

\begin{theorem} \label{thm:deg4_main}
    The weight domain $\left(\frac{1}{4},1\right]$ for the moduli space weighted stable marked del Pezzo surfaces of
    degree 4, $(S,cB)$, admits precisely one wall at $c=1/2$, inducing an isomorphism of moduli spaces
    \[
        \oY_{(1/4,1/2]}^5 \xleftarrow{\sim} \oY_{(1/2,1]}^5 \cong \oY(E_5) \cong \oM_{0,5}.
    \]
    For weights $1/4 < c \leq 1/2$, the stable model of the surface parameterized by a general point of a boundary
    divisor of $\oY_{(1/2,1]}^5$ as in \cref{fig:deg4_1} is a surface with two irreducible components, both isomorphic
    to $\bF_0$, glued along their diagonals, as pictured in \cref{fig:deg4_1_14-12}. Likewise, the stable model for
    weights $1/4 < c \leq 1/2$ of the surface parameterized by a codimension two boundary stratum of $\oY_{(1/2,1]}^5$
    as in \cref{fig:deg4_2} is a surface with four irreducible components, each isomorphic to $\bP^2$, as pictured in
    \cref{fig:deg4_2_14-12}.
\end{theorem}

\begin{figure}[htpb]
    \centering
    \includegraphics[width=0.5\linewidth]{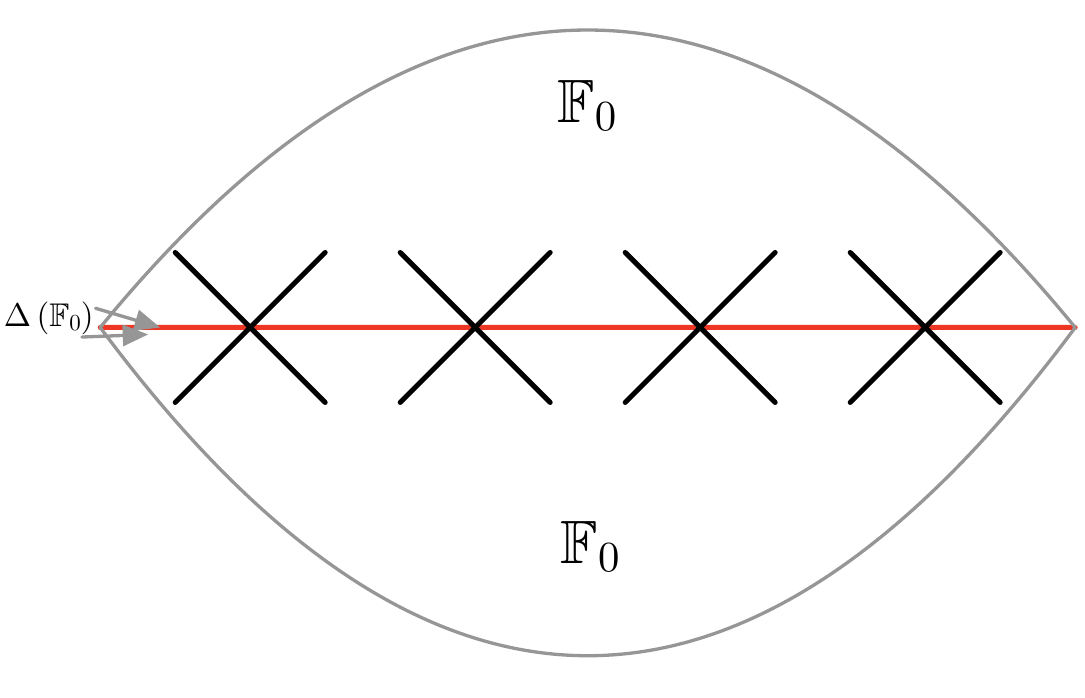}
    \caption{The stable model for weights $1/4 < c \leq 1/2$ of the surface of \cref{fig:deg4_1}. This is obtained from
    the latter surface by contracting the 4 components isomorphic to $\bF_0$. We obtain 2 irreducible components, each
    isomorphic to $\bF_0$, glued to each other along their diagonals. There are 8 lines on each component, 4 in each
    ruling.}
    \label{fig:deg4_1_14-12}
\end{figure}

\begin{figure}[htpb]
    \centering
    \includegraphics[width=0.4\linewidth]{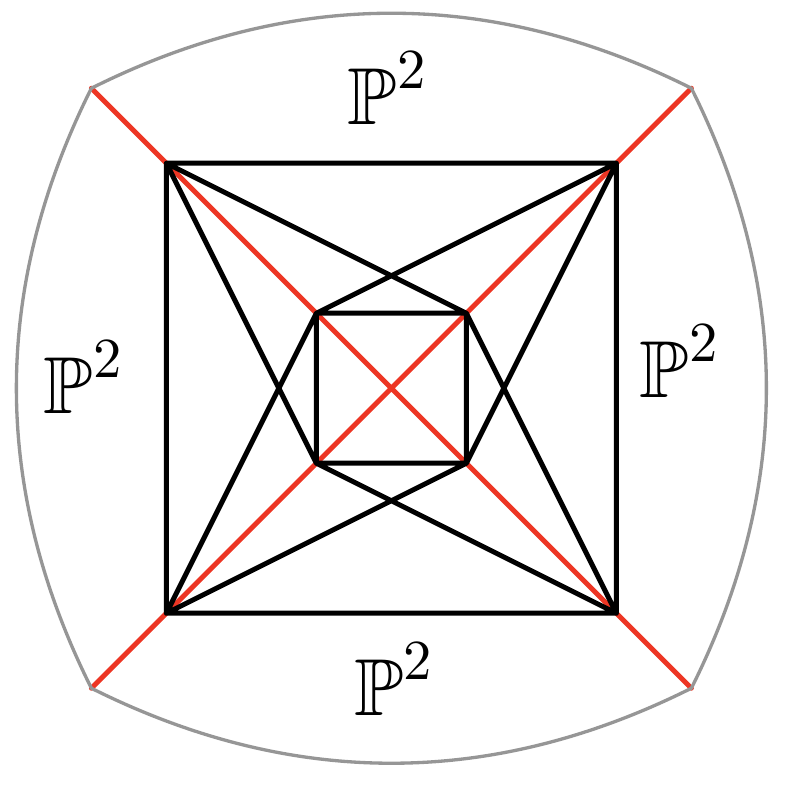}
    \caption{The stable model for weights $1/4 < c \leq 1/2$ of the surface of \cref{fig:deg4_2}. This is obtained from
    the latter surface by contracting the 8 components isomorphic to $\bF_0$. We obtain 4 irreducible components, each
    isomorphic to $\bP^2$, and each glued to two other components along a general line in each component. All four
    components intersect in a point, the intersection point of the two gluing lines on a given component.}
    \label{fig:deg4_2_14-12}
\end{figure}

\begin{proof}
    As described previously, we know that $\pi : (\oY(E_6),B_{\pi}) \to \oY(E_5)$ is the universal family of weight 1
    stable marked del Pezzo surfaces of degree 4. Since we described the fibers of this family in
    \cref{thm:fibers_deg4}, it is straightforward to now describe the wall crossings by calculating the stable models as
    one decreases the weight $c$ from $1$ towards $1/4$, as in \cref{sec:proof_main}.

    Namely, let $(S,B)$ be the fiber of $\pi$ over a general point of a boundary divisor of $\oY(E_5)$. Then $(S,B)$ is
    a reducible surface with six components as pictured in \cref{fig:deg4_1}. Fix $c \leq 1$. There are two types of
    irreducible components of $(S,B)$. The first type is isomorphic to the blowup $Bl_4\bF_0$ of $\bF_0 \cong \bP^1
    \times \bP^1$ at 4 points on its diagonal. Using Equation \eqref{eq:restrict_component}, we compute that the
    restriction of $K_S+cB$ to this component is
    \[
        (-1+4c)(h_1+h_2) + (1-2c)(e_1 + e_2 + e_3 + e_4),
    \]
    where $h_1,h_2$ are the pullbacks of the classes of the two rulings, and $e_1,\ldots,e_4$ are the classes of the
    four exceptional divisors. The second type of irreducible component of $(S,B)$ is isomorphic to $\bF_0$, and again
    using Equation \eqref{eq:restrict_component}, we compute that the restriction of $K_S+cB$ to this component is
    \[
        (-1+2c)h_1 + (-1+2c)h_2.
    \]
    We find that $(S,cB)$ is stable for $1/2 < c \leq 1$, but when $c=1/2$ we contract the four components isomorphic to
    $\bF_0$ to a point each, and correspondingly contract the four exceptional divisors on each of the two $Bl_4\bF_0$
    components. We obtain a reducible surface $(S',cB')$ with two irreducible components, each isomorphic to $\bF_0$,
    glued to each other along their diagonals, as pictured in \cref{fig:deg4_1_14-12}. We directly see that the
    resulting surface $(S',cB')$ is slc for all $1/4 < c \leq 1/2$, and the restriction of $K_{S'}+cB'$ to an
    irreducible component is
    \[
        (-1+4c)h_1 + (-1+4c)h_2,
    \]
    where $h_1,h_2$ are the classes of the two rulings on $\bF_0 \cong \bP^1 \times \bP^1$. Thus $(S',cB')$ is stable
    for all $1/4 < c \leq 1/2$.

    Now let $(S,B)$ be a fiber of $\pi$ over a codimension two boundary stratum of $\oY(E_5)$. Then $(S,B)$ is a
    reducible surface with 12 components as pictured in \cref{fig:deg4_2}. Fix $c \leq 1$ There are two types of
    irreducible components of $(S,B)$. The first type is isomorphic to $Bl_4\bP^2$, the blowup of $\bP^2$ at four points
    in general position, and using Equation \eqref{eq:restrict_component} we compute that the restriction of $K_S+cB$ to
    this component is
    \[
        (-1+4c)h + (1-2c)(e_1+e_2+e_3+e_4),
    \]
    where $h$ is the pullback of the class of a line, and $e_1,\ldots,e_4$ are the classes of the four exceptional
    divisors. The second type of irreducible component of $(S,B)$ is isomorphic to $\bF_0$, and exactly as above we
    compute that the restriction of $K_S+cB$ to this component is
    \[
        (-1+2c)h_1 + (-1+2c)h_2.
    \]
    Thus we find that $(S,cB)$ is stable for $1/2 < c \leq 1$, and when $c=1/2$ we contract the eight irreducible
    components isomorphic to $\bF_0$ to a point each, and correspondingly contract the four exceptional divisors on each
    of the four copies of $Bl_4\bP^2$. We obtain a reducible surface $(S',cB')$ with four irreducible components, each
    isomorphic to $\bP^2$, glued to each other along general lines as pictured in \cref{fig:deg4_2_14-12}. We directly
    see that the resulting surface $(S',cB')$ is slc for all $1/4 < c \leq 1/2$, and the restriction of $K_{S'}+cB'$ to
    an irreducible component is
    \[
        (-1+4c)h,
    \]
    hence $(S',cB')$ is stable for all $1/4 < c \leq 1/2$.
\end{proof}

\begin{remark}
    As in \cref{sec:bottom_up}, one could also prove \cref{thm:deg4_main} using a bottom up approach. In contrast to the
    case of cubic surfaces, it is less clear where the minimal weight surfaces come from for degree 4 del Pezzo
    surfaces. This can be answered by comparing with the surfaces on the K-moduli side, i.e., with weights $(0,1/4)$, as
    in \cref{rmk:Kmoduli}. Namely, the K-moduli space of degree 4 weighted marked del Pezzo surfaces is isomorphic to
    the marked $W(E_5)$ cover of the GIT moduli space of degree 4 del Pezzo surfaces \cite{odakaCompactModuliSpaces2016,
    zhaoCompactificationsModuliPezzo2023}. The surface of \cref{fig:deg4_1_14-12} arises from a del Pezzo surfaces of
    degree 4 with two $A_1$ singularities by resolving (in a one-parameter family) the two singularities, attaching to
    each exceptional line a copy of $\bF_0$ glued along its diagonal (by the same arguments as in
    \cref{ex:bottom_up_a,ex:bottom_up_a2}), and then contracting the central component isomorphic to the minimal
    resolution of a del Pezzo surface of degree 4 with two $A_1$ singularities.
\end{remark}

\begin{remark}
    As mentioned previously, Gale duality induces an isomorphism $\oY(E_5) \cong \oM_{0,5}$, cf.
    \cite{hackingStablePairTropical2009}. It is interesting to compare the wall crossings for $\oY(E_5)$ with the wall
    crossings for $\oM_{0,5}$.

    The space $\oM_{0,5}$ has 10 boundary divisors, $D_{ij}$ for $i,j \subset \{1,\ldots,5\}$ distinct, which correspond
    in the obvious way to the boundary divisors of $\oY(E_5)$, i.e., $D_{ij}$ corresponds to the boundary divisor
    $D(ij)$ defined by the root subsystem $ij \subset E_5$. The boundary divisor $D_{ij}$ on $\oM_{0,5}$ parameterizes
    the stable $5$-pointed curve consisting of two copies of $\bP^1$ meeting at a point, with one $\bP^1$ containing the
    points $i,j$, and the other $\bP^1$ containing the remaining three points $\{k,l,m\} = \{1,\ldots,5\} \setminus
    \{i,j\}$.

    Consider a stable $5$-pointed curve $(C, B = \sum_{i=1}^5p_i)$ as above. For the weighted pair $(C, cB =
    c\sum_{i=1}^5p_i)$, we have that the restriction of $K_C + cB$ to the component containing $i,j$ is $(-1 + 2c)h$,
    and the restriction of $K_C + cB$ to the component containing $k,l,m$ is $(-1 + 3c)h$. We see that at the weight
    $c=1/2$, the component containing $i,j$ is contracted to a point, and we are left with a copy of $\bP^1$ containing
    five points, where the points $i,j$ coincide, and the remaining three points are distinct. The log canonical class
    here is $(-2 + 5c)h$, and we see that the resulting pair is stable for $c > 2/5$.

    To summarize, we see that the weight domain for the pairs $(C, cB = c\sum_{i=1}^5p_i)$ is $(2/5, 1]$, and there is
    exactly one wall in this weight domain, at $c=1/2$. At this wall the moduli space does not change, but the universal
    family does change by contracting the stable curves as described above. Thus, although the weight domain for these
    pairs is different from the weight domain for the dual moduli space of weighted stable marked del Pezzo surfaces of
    degree 4, the wall crossings still agree.
\end{remark}

\begin{remark}
    Although the correspondence of boundary divisors between $\oY(E_5)$ and $\oM_{0,5}$ is obvious combinatorially and
    in terms of root systems (cf. \cite{hackingStablePairTropical2009}), it is less clear how to geometrically obtain
    the stable curves parameterized by the boundary $\oM_{0,5}$ from the stable pairs parameterized by the boundary of
    $\oY(E_5)$, and vice-versa. In this direction, note that each of the 16 (broken) lines on the stable marked del
    Pezzo surface of degree 4 pictured in \cref{fig:deg4_1} consists of two copies of $\bP^1$ meeting at a point. The
    intersections of this line with the other lines on the surfaces cut out two points on one of these $\bP^1$s and
    three points on the other. Therefore, each of the 16 lines on this surface is a stable $5$-pointed curve. Similarly,
    in \cref{fig:deg4_1_14-12}, each one of these broken lines contracts to a stable weight $1/2$ $5$-pointed curve.
\end{remark}

\section{Root combinatorics of the universal family $\wY(E_7) \to \wY(E_6)$} \label{sec:root_combs}

In this appendix we summarize the combinatorics of root subsystems of $E_6$ and $E_7$ yielding the fibers of the
universal family $\wY(E_7) \to \wY(E_6)$ of weight $1/2 < c \leq 2/3$ stable marked cubic surfaces. We choose a
$W(E_6)$-representative $Z$ of each boundary stratum of $\wY(E_6)$, and describe in terms of root subsystems of $E_7$
all of the boundary strata of $\wY(E_7)$ mapping onto $Z$. These correspond the irreducible components of the corresponding
stable surfaces described in \cref{sec:weighted_cubics}, as explained in \cref{sec:fibers}. The pictures of weight $1/2
< c \leq 2/3$ stable marked cubic surfaces given in \cref{sec:weighted_cubics} are constructed from the tables of the
present section using the methodology of \cref{sec:fibers}, cf. \cref{ex:a_fiber,ex:a2_fiber,ex:aa2_fiber,ex:b_fiber}.

The reader is advised to view this appendix at their own peril.

\begin{table}[htpb]
    \centering
    \caption{Non-flat $A_3^2$s which appear in our chosen representatives of boundary strata.}
    \label{tab:nonflat_A32}
    \begin{tabular}{| c | c |}
        \hline
        Label & Roots \\
        \hline
        \hline
        $A_3^2(7,56)$ &  $(5,6,7,56,57,67) \times (12,13,14,23,24,34)$ \\
        $A_3^2(7,34)$ &  $(3,4,7,34,37,47) \times (12,15,16,25,26,56)$ \\
        $A_3^2(7,12)$ &  $(1,2,7,12,17,27) \times (34,35,36,45,46,56)$ \\
        $A_3^2(56,34)$ &  $(34,56,357,367,457,467) \times (7,12,156,256,134,234)$ \\
        $A_3^2(56,12)$ &  $(12,56,157,167,257,267) \times (7,34,356,456,125,126)$ \\
        $A_3^2(34,12)$ &  $(12,34,137,147,237,247) \times (7,56,125,126,345,346)$ \\
        \hline
    \end{tabular}
\end{table}

\begin{table}[htpb]
    \centering
    \caption{Flat $A_3^2$s yielding components of the type $b$ surface of \cref{tab:b}, and its degenerations.}
    \label{tab:flat_A32}
    \begin{tabular}{| c | c |}
        \hline
        Label & Roots \\
        \hline
        \hline
         $W_1$ & $(1,7,17,123,456,237) \times (45,46,56,147,157,167)$ \\
         $W_2$ & $(2,7,27,123,456,137) \times (45,46,56,247,257,267)$ \\
         $W_3$ & $(3,7,37,123,456,127) \times (45,46,56,347,357,367)$ \\
         $W_4$ & $(4,7,47,123,456,567) \times (12,13,23,147,247,347)$ \\
         $W_5$ & $(5,7,57,123,456,467) \times (12,13,23,157,257,357)$ \\
         $W_6$ & $(6,7,67,123,456,457) \times (12,13,23,167,267,367)$ \\
         $W_7$ & $(12,13,23,1,2,3) \times (45,46,47,56,57,67)$ \\
         $W_8$ & $(45,46,56,4,5,6) \times (12,13,17,23,27,37)$ \\
         $W_9$ & $(12,13,23,127,137,237) \times (45,46,56,457,467,567)$ \\
        \hline
    \end{tabular}
\end{table}

\begin{table}[htpb]
    \centering
    \caption{$A_7$s which appear in our chosen representatives of boundary strata.}
    \label{tab:A7s}
    \begin{tabular}{| c | c |}
        \hline
        Label & Roots \\
        \hline
        \hline
         $X_0$ & $12,13,14,15,16,17,23,24,25,26,27,34,35,$ \\
               & $36,37,45,46,47,56,57,67,1,2,3,4,5,6,7$ \\
        \hline
         $X_{567}$ & $56,7,12,13,14,23,24,34,123,124,134,234,156,256,$ \\
                   & $356,456, 157,257,357,457, 167,267,367,467, 5,6,57,67$ \\
        \hline
         $X_{347}$ & $34,7,12,15,16,25,26,56, 125,126,156,256,134,234,$\\
                   & $345,346, 137,237,357,367, 147,247,457,467, 3,4,37,47$ \\
        \hline
         $X_{127}$ & $12,7,34,35,36,45,46,56, 345,346,356,456,123,124,$ \\
                   & $125,126, 137,147,157,167, 237,247,257,267, 1,2,17,27$ \\
        \hline
    \end{tabular}
\end{table}

\begin{table}[htpb]
    \centering
    \caption{The type $a$ surface parameterized by a general point of the boundary divisor $D(7) \subset \wY(E_6)$.
    The surface has a total of 8 irreducible components (see \cref{fig:a_12-23}).}
    \label{tab:a}
    \begin{tabular}{| c | c | c | c |}
        \hline
        Label & Surface & Count & Root Systems \\
        \hline\hline
        1 & $\wS_{A_1}$ & 1 & $A_1$ \\
        \hline
          & & & $(7)$ \\
        \hline
        \hline
        2 & $Bl_6\bF_0$ & 1 & $A_7$ \\
        \hline
          & & & $X_0$ \\
        \hline
        \hline
        3 & $\bF_0$ & 6 & $A_2$ \\
        \hline
          & & & $(i,7,i7)$, $i=1,\ldots,6$ \\
        \hline
    \end{tabular}
\end{table}

\begin{table}[htpb]
    \centering
    \caption{The type $a_2$ surface parameterized by a general point of the boundary divisor $D(7 \perp 56) \subset
    \wY(E_6)$. The surface has a total of 20 irreducible components (see \cref{fig:a2_12-23}).}
    \label{tab:a2}
    \begin{tabular}{| c | c | c | c |}
        \hline
        Label & Surface & Count & Root Systems  \\
        \hline
        \hline
        1 & $\wS_{2A_1}$ & 1 & $A_1 \perp A_1$ \\
        \hline
         &  &  & $(7) \perp (56)$ \\
        \hline
        \hline
        2 & $Bl_5\bF_0$ & 2 & $A_1 \subset A_7$ \\
        \hline
         &  &  & $(7) \subset X_{567}$ \\
         &  &  & $(56) \subset X_{0}$  \\
        \hline
        \hline
        3 & $Bl_5\bP^2$ & 1 & $A_3^2$ \\
        \hline
         &  &  & $A_3^2(7,56)$  \\
        \hline
        \hline
        4 & $\bF_0$ & 4 & $A_2 \subset A_7$ \\
        \hline
         &  &  & $(5,7,57) \subset X_{567}$ \\
         &  &  & $(6,7,67) \subset X_{567}$ \\
         &  &  & $(5,6,56) \subset X_0$ \\
         &  &  & $(56,57,67) \subset X_0$ \\
        \hline
        \hline
        5a & $\bF_0$ & 8 & $A_1 \perp A_2$ \\
        \hline
         &  &  & $(7) \perp (56,i57,i67)$, $i=1,\ldots,4$ \\
         &  &  & $(56) \perp (i,7,i7)$, $i=1,\ldots,4$ \\
        \hline
        \hline
        5b & $\bF_0$ & 4 & $A_2 \perp A_2$ \\
        \hline
         &  &  & $(i,7,i7) \perp (56,i57,i67)$, $i=1,\ldots,4$ \\
         \hline
    \end{tabular}
\end{table}

\begin{table}[htpb]
    \centering
    \caption{Degenerations of the type 3 component of the $a_2$ surface of \cref{tab:a2} into four irreducible
        components, yielding a type $aa_2$ surface parameterized by a general point of the codimension two boundary
        stratum $D(7) \cap D(7 \perp 56)$ (see \cref{fig:aa2_12-23}). (The plus in the table separates the root systems
        coming from the type $a$ combinatorics and the type $a_2$ combinatorics.) Outside of the degeneration of this
        type 3 component, the $aa_2$ surface is isomorphic to the $a_2$ surface, and the root combinatorics of these
    other irreducible components stay the same (see \cref{ex:aa2_fiber}).}
    \label{tab:aa2}
\begin{tabular}{| c | c | c |}
        \hline
         Surface & Count & Root Systems \\
         \hline
         \hline
                       $Bl_4\bP^2$ & 1 & $A_7 + A_3^2$ \\
         \hline
                       & & $X_0 + A_3^2(7,56)$ \\
         \hline
         \hline
                       $\bF_0$ & 1 & $A_1 + A_3^2$ \\
         \hline
                       & & $(7) + A_3^2(7,56)$ \\
         \hline
         \hline
                       $\bF_0$ & 2 & $A_2 + A_3^2$ \\
         \hline
                       & & $(i,7,i7) + A_3^2(7,56)$, $i=5,6$ \\
         \hline
    \end{tabular}
\end{table}

\begin{table}[htpb]
    \centering
    \caption{The type $a_3$ surface parameterized by a general point of the boundary divisor $D(7 \perp 56 \perp 34)
    \subset \wY(E_6)$. The surface has a total of 37 irreducible components (see \cref{fig:a3_12-23}).}
    \label{tab:a3}
    \begin{tabular}{| c | c | c | c |}
        \hline
        Label & Surface & Count & Root Systems \\
        \hline
        \hline
        1 & $\wS_{3A_1}$ & 1 & $A_1 \perp A_1 \perp A_1$ \\
        \hline
         &  &  & $(7) \perp (56) \perp (34)$  \\
        \hline
        \hline
        2 & $Bl_4\bF_0$ & 3 & $(A_1 \perp A_1) \subset A_7$ \\
        \hline
         &  &  & $((7) \perp (56)) \subset X_{347}$ \\
         &  &  & $((7) \perp (34)) \subset X_{567}$ \\
         &  &  & $((56) \perp (34)) \subset X_{0}$ \\
        \hline
        \hline
        3 & $Bl_5\bP^2$ & 3 & $A_1 \subset A_3^2$ \\
\hline
    \end{tabular}
\end{table}

\begin{table}[htpb]
    \centering
    \caption{Degenerations of two out of three of the type 3 components of the $a_3$ surface of \cref{tab:a3} into four
    irreducible components each, yielding a type $aa_3$ surface parameterized by a general point of the codimension two
boundary stratum $D(7) \cap D(7 \perp 56 \perp 34)$ (see \cref{fig:aa3_12-23}).}
    \label{tab:aa3}
    \begin{tabular}{| c | c | c |}
        \hline
         Surface & Count & Root Systems \\
         \hline
         \hline
                       $Bl_4\bP^2$ & 2 & $A_7 + (A_1 \subset A_3^2)$ \\
         \hline
                       & & $X_0 + ((34) \subset A_3^2(7,56)$ \\
                       & & $X_0 + ((56) \subset A_3^2(7,34)$ \\
         \hline
         \hline
                       $\bF_0$ & 2 & $A_1 + (A_1 \subset A_3^2)$ \\
         \hline
                       & & $(7) + ((34) \subset A_3^2(7,56)$ \\
                       & & $(7) + ((56) \subset A_3^2(7,34)$ \\
         \hline
         \hline
                       $\bF_0$ & 4 & $A_2 + (A_1 \subset A_3^2)$ \\
         \hline
                       & & $(i,7,i7) + ((34) \subset A_3^2(7,56))$, $i=5,6$ \\
                       & & $(i,7,i7) + ((56) \subset A_3^2(7,34))$, $i=3,4$ \\
         \hline
    \end{tabular}
\end{table}

\begin{table}[htpb]
    \centering
    \caption{Degenerations of two out of three of the type 3 components of the $a_3$ surface of \cref{tab:a3} into four
    irreducible components, yielding a type $a_2a_3$ surface parameterized by a general point of the codimension two
    boundary stratum $D(7 \perp 56) \cap D(7 \perp 56 \perp 34)$ (see \cref{fig:a2a3_12-23}).}
    \label{tab:a2a3}
    \begin{tabular}{| c | c | c |}
        \hline
         Surface & Count & Root Systems \\
         \hline
         \hline
                       $Bl_4\bP^2$ & 2 & $(A_1 \subset A_7) + (A_1 \subset A_3^2)$ \\
         \hline
                       & & $((56) \subset X_0) + ((56) \subset A_3^2(7,34))$ \\
                       & & $((7) \subset X_{567}) + ((7) \subset A_3^2(56,34))$ \\
         \hline
         \hline
                       $\bF_0$ & 2 & $(A_1 \perp A_1) + (A_1 \subset A_3^2)$ \\
         \hline
                       & & $((7) \perp (56)) +  ((7) \subset A_3^2(56,34))$ \\
                       & & $((7) \perp (56)) +  ((56) \subset A_3^2(7,34))$ \\
         \hline
         \hline
                       $\bF_0$ & 4 & $(A_1 \perp A_2) + (A_1 \subset A_3^2)$ \\
         \hline
                       & & $((7) \perp (56,i57,i67)) +  ((7) \subset A_3^2(56,34))$, $i=3,4$ \\
                       & & $((56) \perp (i,7,i7)) +  ((56) \subset A_3^2(7,34))$, $i=3,4$ \\
         \hline
    \end{tabular}
\end{table}

\begin{table}[htpb]
    \centering
    \caption{Degenerations of the three type 3 components of the $a_3$ surface of \cref{tab:a3} into four
    irreducible components each, yielding a type $aa_2a_3$ surface parameterized by a general point of the codimension
    three boundary stratum $D(7) \cap D(7 \perp 56) \cap D(7 \perp 56 \perp 34)$ (see \cref{fig:aa2a3_12-23}).}
    \label{tab:aa2a3}
    \begin{tabular}{| c | c | c |}
        \hline
         Surface & Count & Root Systems \\
         \hline
         \hline
                       $Bl_4\bP^2$ & 2 & $A_7 + (A_1 \subset A_7) + (A_1 \subset A_3^2)$ \\
         \hline
                       & & $X_0 + ((56) \subset X_0) + ((56) \subset A_3^2(7,34))$ \\
                       & & $X_{567} + ((7) \subset X_{567}) + ((7) \subset A_3^2(56,34))$ \\
         \hline
                       $Bl_4\bP^2$ & 1 & $A_7 + A_3^2 + (A_1 \subset A_3^2)$ \\
         \hline
                       & & $X_0 + A_3^2(7,56) + ((34) \subset A_3^2(7,56))$ \\
         \hline
         \hline
                       $\bF_0$ & 2 & $A_1 + (A_1 \perp A_1) + (A_1 \subset A_3^2)$ \\
         \hline
                       & & $((7) \perp (56)) +  ((7) \subset A_3^2(56,34))$ \\
                       & & $((7) \perp (56)) +  ((56) \subset A_3^2(7,34))$ \\
         \hline
                       $\bF_0$ & 1 & $A_1 + A_3^2 + (A_1 \subset A_3^2)$ \\
         \hline
                       & & $(7) + A_3^2(7,56) + ((34) \subset A_3^2(7,56))$ \\
         \hline
         \hline
                       $\bF_0$ & 2 & $A_1 + (A_1 \perp A_2) + (A_1 \subset A_3^2)$ \\
         \hline
                       & & $(7) + ((7) \perp (56,i57,i67)) +  ((7) \subset A_3^2(56,34))$, $i=3,4$ \\
         \hline
                       $\bF_0$ & 2 & $A_2 + (A_1 \perp A_2) + (A_1 \subset A_3^2)$ \\
         \hline
                       & & $(i,7,i7) + ((56) \perp (i,7,i7)) +  ((56) \subset A_3^2(7,34))$, $i=3,4$ \\
         \hline
                       $\bF_0$ & 2 & $A_2 + A_3^2 + (A_1 \subset A_3^2)$ \\
         \hline
                       & & $(i,7,i7) + A_3^2(7,56) + ((34) \subset A_3^2(7,56))$, $i=5,6$ \\
         \hline
    \end{tabular}
\end{table}

\begin{table}[htpb]
    \centering
    \caption{The type $a_4$ surface parameterized by a general point of the boundary divisor $D(7 \perp 56 \perp 34
        \perp 12) \subset \wY(E_6)$. The surface has a total of 59 irreducible components (see \cref{fig:a4_12-23}).}
    \label{tab:a4}
    \begin{tabular}{| c | c | c | c |}
        \hline
        Label & Surface & Count & Root Systems \\
        \hline
        \hline
        1 & $\wS_{4A_1}$ & 1 & $A_1 \perp A_1 \perp A_1 \perp A_1$ \\
        \hline
         &  &  & $(7) \perp (56) \perp (34) \perp (12)$  \\
        \hline
        \hline
        2 & $Bl_3\bF_0$ & 4 & $(A_1 \perp A_1 \perp A_1) \subset A_7$ \\
        \hline
         &  &  & $((7) \perp (56) \perp (34)) \subset X_{127}$ \\
         &  &  & $((7) \perp (56) \perp (12)) \subset X_{347}$ \\
         &  &  & $((7) \perp (34) \perp (12)) \subset X_{567}$ \\
         &  &  & $((56) \perp (34) \perp (12)) \subset X_{0}$ \\
\hline
    \end{tabular}
\end{table}

\begin{table}[htpb]
    \centering
    \caption{Degenerations of three out of six of the type 3 components of the $a_4$ surface of \cref{tab:a4} into four
    irreducible components each, yielding a type $aa_4$ surface parameterized by a general point of the codimension two
boundary stratum $D(7) \cap D(7 \perp 56 \perp 34 \perp 12)$ (see \cref{fig:aa4_12-23}).}
    \label{tab:aa4}
    \begin{tabular}{| c | c | c |}
        \hline
         Surface & Count & Root Systems \\
         \hline
         \hline
                       $Bl_4\bP^2$ & 3 & $A_7 + ((A_1 \perp A_1) \subset A_3^2)$ \\
         \hline
                       & & $X_0 + (((34) \perp (12)) \subset A_3^2(7,56)$ \\
                       & & $X_0 + (((56) \perp (12)) \subset A_3^2(7,34)$ \\
                       & & $X_0 + (((56) \perp (34)) \subset A_3^2(7,12)$ \\
         \hline
         \hline
                       $\bF_0$ & 3 & $A_1 + ((A_1 \perp A_1) \subset A_3^2)$ \\
         \hline
                       & & $(7) + (((34) \perp (12)) \subset A_3^2(7,56)$ \\
                       & & $(7) + (((56) \perp (12)) \subset A_3^2(7,34)$ \\
                       & & $(7) + (((56) \perp (34)) \subset A_3^2(7,12)$ \\
         \hline
         \hline
                       $\bF_0$ & 6 & $A_2 + ((A_1 \perp A_1) \subset A_3^2)$ \\
         \hline
                       & & $(i,7,i7) + (((34) \perp (12)) \subset A_3^2(7,56)$, $i=5,6$ \\
                       & & $(i,7,i7) + (((56) \perp (12)) \subset A_3^2(7,34)$, $i=3,4$ \\
                       & & $(i,7,i7) + (((56) \perp (34)) \subset A_3^2(7,12)$, $i=1,2$ \\
         \hline
    \end{tabular}
\end{table}

\begin{table}[htpb]
    \centering
    \caption{Degenerations of four out of six of the type 3 components of the $a_4$ surface of \cref{tab:a4} into four
    irreducible components each, yielding a type $a_2a_4$ surface parameterized by a general point of the codimension two
boundary stratum $D(7 \perp 56) \cap D(7 \perp 56 \perp 34 \perp 12)$ (see \cref{fig:a2a4_12-23}).}
    \label{tab:a2a4}
    \begin{tabular}{| c | c | c |}
        \hline
         Surface & Count & Root Systems \\
         \hline
         \hline
                       $Bl_4\bP^2$ & 4 & $(A_1 \subset A_7) + ((A_1 \perp A_1) \subset A_3^2)$ \\
         \hline
                       & & $((7) \subset X_{567}) + (((7) \perp (12)) \subset A_3^2(56,34)$ \\
                       & & $((7) \subset X_{567}) + (((7) \perp (34)) \subset A_3^2(56,12)$ \\
                       & & $((56) \subset X_0) + (((56) \perp (12)) \subset A_3^2(7,34)$ \\
                       & & $((56) \subset X_0) + (((56) \perp (34)) \subset A_3^2(7,12)$ \\
         \hline
         \hline
                       $\bF_0$ & 4 & $(A_1 \perp A_1) + ((A_1 \perp A_1) \subset A_3^2)$ \\
         \hline
                       & & $((7) \perp (56)) + (((7) \perp (12)) \subset A_3^2(56,34)$ \\
                       & & $((7) \perp (56)) + (((7) \perp (34)) \subset A_3^2(56,12)$ \\
                       & & $((7) \perp (56)) + (((56) \perp (12)) \subset A_3^2(7,34)$ \\
                       & & $((7) \perp (56)) + (((56) \perp (34)) \subset A_3^2(7,12)$ \\
         \hline
         \hline
                       $\bF_0$ & 8 & $(A_1 \perp A_2) + ((A_1 \perp A_1) \subset A_3^2)$ \\
         \hline
                       & & $((7) \perp (56,i57,i67)) + (((7) \perp (12)) \subset A_3^2(56,34)$, $i=3,4$ \\
                       & & $((7) \perp (56,i57,i67)) + (((7) \perp (34)) \subset A_3^2(56,12)$, $i=1,2$ \\
                       & & $((56) \perp (i,7,i7)) + (((56) \perp (12)) \subset A_3^2(7,34)$, $i=3,4$ \\
                       & & $((56) \perp (i,7,i7)) + (((56) \perp (34)) \subset A_3^2(7,12)$, $i=1,2$ \\
         \hline
    \end{tabular}
\end{table}

\begin{table}[htpb]
    \centering
    \caption{Degenerations of three out of six of the type 3 components of the $a_4$ surface of \cref{tab:a4} into four
    irreducible components each, yielding a type $a_3a_4$ surface parameterized by a general point of the codimension two
    boundary stratum $D(7 \perp 56 \perp 34) \cap D(7 \perp 56 \perp 34 \perp 12)$ (see \cref{fig:a3a4_12-23}).}
    \label{tab:a3a4}
    \begin{tabular}{| c | c | c |}
        \hline
         Surface & Count & Root Systems \\
         \hline
         \hline
                       $Bl_4\bP^2$ & 3 & $((A_1 \perp A_1) \subset A_7) + ((A_1 \perp A_1) \subset A_3^2)$ \\
         \hline
                       & & $(((56) \perp (34)) \subset X_0) + (((56) \perp (34)) \subset A_3^2(7,12))$ \\
                       & & $(((7) \perp (34)) \subset X_{567}) + (((7) \perp (34)) \subset A_3^2(56,12))$ \\
                       & & $(((7) \perp (56)) \subset X_{347}) + (((7) \perp (56)) \subset A_3^2(34,12))$ \\
         \hline
         \hline
                       $\bF_0$ & 3 & $(A_1 \perp A_1 \perp A_1) + ((A_1 \perp A_1) \subset A_3^2)$ \\
         \hline
                       & & $((7) \perp (56) \perp (34)) + (((56) \perp (34)) \subset A_3^2(7,12))$ \\
                       & & $((7) \perp (56) \perp (34)) + (((7) \perp (34)) \subset A_3^2(56,12))$ \\
                       & & $((7) \perp (56) \perp (34)) + (((7) \perp (56)) \subset A_3^2(34,12))$ \\
         \hline
         \hline
                       $\bF_0$ & 6 & $(A_1 \perp A_1 \perp A_2) + ((A_1 \perp A_1) \subset A_3^2)$ \\
         \hline
                       & & $(56) \perp (34) \perp (i,7,i7)) + (((56) \perp (34)) \subset A_3^2(7,12))$, $i=1,2$ \\
                       & & $((7) \perp (34) \perp (56,i57,i67)) + (((7) \perp (34)) \subset A_3^2(56,12))$, $i=1,2$ \\
                       & & $((7) \perp (56) \perp (56,i57,i67)) + (((7) \perp (56)) \subset A_3^2(34,12))$, $i=3,4$ \\
         \hline
    \end{tabular}
\end{table}

\begin{table}[htpb]
    \centering
    \caption{Degenerations of five out of six of the type 3 components of the $a_4$ surface of \cref{tab:a4} into four
    irreducible components each, yielding a type $aa_2a_4$ surface parameterized by a general point of the codimension
    three boundary stratum $D(7) \cap D(7 \perp 56) \cap D(7 \perp 56 \perp 34 \perp 12)$ (see \cref{fig:a2a4_12-23}).}
    \label{tab:aa2a4}
    \begin{tabular}{| c | c | c |}
        \hline
         Surface & Count & Root Systems \\
         \hline
         \hline
                       $Bl_4\bP^2$ & 2 & $A_1 + (A_1 \subset A_7) + ((A_1 \perp A_1) \subset A_3^2)$ \\
         \hline
                       & & $(7) + ((7) \subset X_{567}) + (((7) \perp (12)) \subset A_3^2(56,34))$ \\
                       & & $(7) + ((7) \subset X_{567}) + (((7) \perp (34)) \subset A_3^2(56,12))$ \\
         \hline
                       $Bl_4\bP^2$ & 2 & $A_7 + (A_1 \subset A_7) + ((A_1 \perp A_1) \subset A_3^2)$ \\
         \hline
                       & & $X_0 + ((56) \subset X_0) + (((56) \perp (12)) \subset A_3^2(7,34))$ \\
                       & & $X_0 + ((56) \subset X_0) + (((56) \perp (34)) \subset A_3^2(7,12))$ \\
         \hline
                       $Bl_4\bP^2$ & 1 & $A_7 + A_3^2 + ((A_1 \perp A_1) \subset A_3^2)$ \\
\hline
    \end{tabular}
\end{table}

\begin{table}[htpb]
    \centering
    \caption{Degenerations of five out of six of the type 3 components of the $a_4$ surface of \cref{tab:a4} into four
    irreducible components each, yielding a type $aa_3a_4$ surface parameterized by a general point of the codimension two
    boundary stratum $D(7) \cap D(7 \perp 56 \perp 34) \cap D(7 \perp 56 \perp 34 \perp 12)$ (see \cref{fig:aa3a4_12-23}).}
    \label{tab:aa3a4}
    \begin{tabular}{| c | c | c |}
        \hline
         Surface & Count & Root Systems \\
         \hline
         \hline
                       $Bl_4\bP^2$ & 1 & $A_7 + ((A_1 \perp A_1) \subset A_7) + ((A_1 \perp A_1) \subset A_3^2)$ \\
         \hline
                       & & $X_0 + (((56) \perp (34)) \subset X_0) + (((56) \perp (34)) \subset A_3^2(7,12))$ \\
         \hline
                       $Bl_4\bP^2$ & 2 & $A_1 + ((A_1 \perp A_1) \subset A_7) + ((A_1 \perp A_1) \subset A_3^2)$ \\
         \hline
                       & & $(7) + (((7) \perp (34)) \subset X_{567}) + (((7) \perp (34)) \subset A_3^2(56,12))$ \\
                       & & $(7) + (((7) \perp (56)) \subset X_{347}) + (((7) \perp (56)) \subset A_3^2(34,12))$ \\
         \hline
                       $Bl_4\bP^2$ & 2 & $A_7 + (A_1 \subset A_3^2) + ((A_1 \perp A_1) \subset A_3^2)$ \\
         \hline
                       & & $X_0 + ((56) \subset A_3^2(7,34)) + (((56) \perp (12)) \subset A_3^2(7,34))$ \\
\hline
    \end{tabular}
\end{table}

\begin{table}[htpb]
    \centering
    \caption{Degenerations of five out of six of the type 3 components of the $a_4$ surface of \cref{tab:a4} into four
    irreducible components each, yielding a type $a_2a_3a_4$ surface parameterized by a general point of the codimension
    three boundary stratum $D(7 \perp 56) \cap D(7 \perp 56 \perp 34) \cap D(7 \perp 56 \perp 34 \perp 12)$ (see \cref{fig:a2a3a4_12-23}).}
    \label{tab:a2a3a4}
    \begin{tabular}{| c | c | c |}
        \hline
         Surface & Count & Root Systems \\
         \hline
         \hline
                       $Bl_4\bP^2$ & 2 & $(A_1 \subset A_7) + ((A_1 \perp A_1) \subset A_7) + ((A_1 \perp A_1) \subset A_3^2)$ \\
         \hline
                       & & $((56) \subset X_0) + (((56) \perp (34)) \subset X_0) + (((56) \perp (34)) \subset A_3^2(7,12))$ \\
                       & & $((7) \subset X_{567}) + (((7) \perp (34)) \subset X_{567}) + (((7) \perp (34)) \subset A_3^2(56,12))$ \\
         \hline
                       $Bl_4\bP^2$ & 1 & $(A_1 \perp A_1) + ((A_1 \perp A_1) \subset A_7) + ((A_1 \perp A_1) \subset A_3^2)$ \\
         \hline
                       & & $((7) \perp (56)) + (((7) \perp (56)) \subset X_{347}) + (((7) \perp (56)) \subset A_3^2(34,12))$ \\
         \hline
                       $Bl_4\bP^2$ & 2 & $(A_1 \subset A_7) + (A_1 \subset A_3^2) + ((A_1 \perp A_1) \subset A_3^2)$ \\
         \hline
                       & & $((56) \subset X_0) + ((56) \subset A_3^2(7,34)) + (((56) \perp (12)) \subset A_3^2(7,34))$ \\
\hline
    \end{tabular}
\end{table}

\begin{table}[htpb]
    \centering
    \caption{Degenerations of the six type 3 components of the $a_4$ surface of \cref{tab:a4} into four
    irreducible components each, yielding a type $aa_2a_3a_4$ surface parameterized by the
    codimension four boundary stratum $D(7) \cap D(7 \perp 56) \cap D(7 \perp 56 \perp 34) \cap D(7 \perp 56 \perp 34
\perp 12)$ (see \cref{fig:aa2a3a4_12-23}).}
    \label{tab:aa2a3a4}
    \begin{tabular}{| c | c | c |}
        \hline
         Surface & Count & Root Systems \\
         \hline
         \hline
                       $Bl_4\bP^2$ & 1 & $A_1 + (A_1 \subset A_7) + ((A_1 \perp A_1) \subset A_7) + ((A_1 \perp A_1) \subset A_3^2)$ \\
         \hline
                       & & $(7) + ((7) \subset X_{567}) + (((7) \perp (34)) \subset X_{567}) + (((7) \perp (34)) \subset A_3^2(56,12))$ \\
         \hline
                       $Bl_4\bP^2$ & 1 & $A_7 + (A_1 \subset A_7) + ((A_1 \perp A_1) \subset A_7) + ((A_1 \perp A_1) \subset A_3^2)$ \\
         \hline
                       & & $X_0 + ((56) \subset X_0) + (((56) \perp (34)) \subset X_0) + (((56) \perp (34)) \subset A_3^2(7,12))$ \\
         \hline
                       $Bl_4\bP^2$ & 1 & $A_1 + (A_1 \perp A_1) + ((A_1 \perp A_1) \subset A_7) + ((A_1 \perp A_1) \subset A_3^2)$ \\
         \hline
                       & & $(7) + ((7) \perp (56)) + (((7) \perp (56)) \subset X_{347}) + (((7) \perp (56)) \subset A_3^2(34,12))$ \\
         \hline
                       $Bl_4\bP^2$ & 1 & $A_1 + (A_1 \subset A_7) + (A_1 \subset A_3^2) + ((A_1 \perp A_1) \subset A_3^2)$ \\
\hline
    \end{tabular}
\end{table}

\begin{table}[htpb]
    \centering
    \caption{The type $b$ surface parameterized by a general point of the boundary divisor $D((123,456,7) \times
    (12,13,23) \times (45,46,56)) \subset \wY(E_6)$. The surface has a total of 12 irreducible components (see \cref{fig:b_13-23}).}
    \label{tab:b}
    \begin{tabular}{| c | c | c | c |}
        \hline
        Label & Surface & Count & Root Systems \\
        \hline
        \hline
        1 & $Bl_6\bP^2$ & 3 & $A_2$ \\
        \hline
         &  &  & $(123,456,7)$  \\
         &  &  & $(12,13,23)$  \\
         &  &  & $(45,46,56)$  \\
        \hline
        \hline
        2 & $\bF_0$ & 9 & $A_3^2$ \\
        \hline
         &  &  & $W_1 = (1,7,17,123,456,237) \times (45,46,56,147,157,167)$ \\
         &  &  & $W_2 = (2,7,27,123,456,137) \times (45,46,56,247,257,267)$ \\
         &  &  & $W_3 = (3,7,37,123,456,127) \times (45,46,56,347,357,367)$ \\
         &  &  & $W_4 = (4,7,47,123,456,567) \times (12,13,23,147,247,347)$ \\
         &  &  & $W_5 = (5,7,57,123,456,467) \times (12,13,23,157,257,357)$ \\
         &  &  & $W_6 = (6,7,67,123,456,457) \times (12,13,23,167,267,367)$ \\
         &  &  & $W_7 = (12,13,23,1,2,3) \times (45,46,47,56,57,67)$ \\
         &  &  & $W_8 = (45,46,56,4,5,6) \times (12,13,17,23,27,37)$ \\
         &  &  & $W_9 = (12,13,23,127,137,237) \times (45,46,56,457,467,567)$ \\
        \hline
    \end{tabular}
\end{table}

\begin{table}[htpb]
    \centering
    \caption{The type $ab$ surface parameterized by a general point of the codimension two $D(7) \cap D((123,456,7) \times
    (12,13,23) \times (45,46,56)) \subset \wY(E_6)$. The surface has a total of 26 irreducible components (see
\cref{fig:ab_12-23}).}
    \label{tab:ab}
    \begin{tabular}{| c | c | c | c |}
        \hline
        Label & Surface & Count & Root Systems \\
        \hline
        \hline
        1 & $Bl_6\bP^2$ & 1 & $A_1 \subset A_2$ \\
        \hline
         &  &  & $(7) \subset (123,456,7)$  \\
        \hline
        \hline
        2 & $Bl_5\bP^2$ & 2 & $A_1 \perp A_2$ \\
        \hline
         &  &  & $(7) \perp (12,13,23)$  \\
         &  &  & $(7) \perp (45,46,56)$  \\
        \hline
        \hline
        3a & $Bl_5\bP^2$ & 2 & $A_2 \subset A_7$ \\
        \hline
         &  &  & $(12,13,23) \subset X_0$  \\
         &  &  & $(45,46,56) \subset X_0$  \\
        \hline
        \hline
        3b & $\bF_0$ & 2 & $A_3^2 \subset A_7$ \\
        \hline
         &  &  & $W_i \subset X_0$, $i=7,8$  \\
        \hline
        \hline
        4a & $\bF_0$ & 6 & $A_1 \subset A_3^2$ \\
        \hline
         &  &  & $(7) \subset W_i$, $i=1,\ldots,6$  \\
        \hline
        \hline
        4b & $\bF_0$ & 6 & $A_2 \subset A_3^2$ \\
        \hline
         &  &  & $(i,7,i7) \subset W_i$, $i=1,\ldots,6$  \\
        \hline
        \hline
        5 & $\bF_0$ & 1 & $A_1 \perp A_3^2$ \\
        \hline
         &  &  & $(7) \perp W_9$  \\
        \hline
        \hline
        6 & $\bF_0$ & 6 & $A_2 \perp A_2$ \\
        \hline
         &  &  & $(i,7,i7) \perp (45,46,56)$, $i=1,2,3$  \\
         &  &  & $(i,7,i7) \perp (12,13,23)$, $i=4,5,6$  \\
        \hline
    \end{tabular}
\end{table}

\begin{table}[htpb]
    \centering
    \caption{The type $a_2b$ surface parameterized by a general point of the codimension two $D(7 \perp 56) \cap D((123,456,7) \times
    (12,13,23) \times (45,46,56)) \subset \wY(E_6)$. The surface has a total of 45 irreducible components (see
\cref{fig:a2b_12-23}).}
    \label{tab:a2b}
    \begin{tabular}{| c | c | c | c |}
        \hline
        Label & Surface & Count & Root Systems \\
        \hline
        \hline
        1 & $Bl_5\bP^2$ & 2 & $A_1 \perp (A_1 \subset A_2)$ \\
        \hline
         &  &  & $(56) \perp ((7) \subset (123,456,7))$  \\
         &  &  & $(7) \perp ((56) \subset (45,46,56))$  \\
        \hline
        \hline
        2 & $Bl_4\bP^2$ & 1 & $A_1 \perp A_1 \perp A_2$ \\
        \hline
         &  &  & $(7) \perp (56) \perp (12,13,23)$  \\
        \hline
        \hline
        3a  & $Bl_5\bP^2$ & 2 & $A_1 \subset A_2 \subset A_7$ \\
        \hline
         &  &  & $(56) \subset (45,46,56) \subset X_0$  \\
\hline
    \end{tabular}
\end{table}

\begin{table}[htpb]
    \centering
    \caption{Degenerations of the type 4 component of the $a_2b$ surface of \cref{tab:a2b} into four irreducible
        components, yielding a type $aa_2b$ surface parameterized by a general point of the codimension two boundary
        stratum $D(7) \cap D(7 \perp 56) \cap D((123,456,7) \times (12,13,23) \times (45,46,56))$ (see \cref{fig:aa2b_12-23}).}
    \label{tab:aa2b}
    \begin{tabular}{| c | c | c |}
        \hline
         Surface & Count & Root Systems \\
         \hline
         \hline
                       $Bl_4\bP^2$ & 1 & $A_7 + (A_2 \subset A_3^2)$ \\
         \hline
                       & & $X_0 + ((12,13,23) \subset A_3^2(7,56))$ \\
         \hline
         \hline
                       $\bF_0$ & 1 & $A_1 + (A_2 \subset A_3^2)$ \\
         \hline
                       & & $(7) + ((12,13,23) \subset A_3^2(7,56))$ \\
         \hline
         \hline
                       $\bF_0$ & 2 & $A_2 + (A_2 \subset A_3^2)$ \\
         \hline
                       & & $(i,7,i7) + ((12,13,23) \subset A_3^2(7,56))$, $i=5,6$ \\
         \hline
    \end{tabular}
\end{table}

\begin{table}[htpb]
    \centering
    \caption{The type $a_3b$ surface parameterized by a general point of the codimension two boundary stratum $D(7 \perp
        56 \perp 12) \cap D((123,456,7) \times (12,13,23) \times (45,46,56)) \subset \wY(E_6)$. The surface has a total
        of 69 irreducible components (see \cref{fig:a3b_12-23}).}
    \label{tab:a3b}
    \begin{tabular}{| c | c | c | c |}
        \hline
        Label & Surface & Count & Root Systems \\
        \hline
        \hline
        1 & $Bl_4\bP^2$ & 3 & $A_1 \perp A_1 \perp (A_1 \subset A_2)$ \\
        \hline
         &  &  & $(56) \perp (12) \perp ((7) \subset (123,456,7))$  \\
         &  &  & $(7) \perp (12) \perp ((56) \subset (45,46,56))$  \\
         &  &  & $(7) \perp (56) \perp ((12) \subset (12,13,23))$  \\
        \hline
        \hline
        2a  & $Bl_4\bP^2$ & 6 & $(A_1 \perp (A_1 \subset A_2)) \subset A_7$ \\
        \hline
         &  &  & $((12) \perp ((56) \subset (45,46,56))) \subset X_0$  \\
         &  &  & $((56) \perp ((12) \subset (12,13,23))) \subset X_0$  \\
\hline
    \end{tabular}
\end{table}

\begin{table}[htpb]\ContinuedFloat
    \centering
    \caption{A type $a_3b$ surface (continued).}
    \label{tab:a3b_cont}
    \begin{tabular}{| c | c | c | c |}
        \hline
        Label & Surface & Count & Root Systems \\
        \hline
        \hline
        3 & $Bl_5\bP^2$ & 3 & $A_1 \subset A_2 \subset A_3^2$ \\
        \hline
         &  &  & $(12) \subset (12,13,23) \subset A_3^2(7,56)$  \\
         &  &  & $(56) \subset (45,46,56) \subset A_3^2(7,12)$  \\
         &  &  & $(7) \subset (123,456,7) \subset A_3^2(56,12)$  \\
        \hline
        \hline
        4a & $\bF_0$ & 3 & $A_1 \perp ((A_1 \perp A_1) \subset A_3^2)$ \\
        \hline
         &  &  & $(12) \perp (((7) \perp (56)) \subset W_3)$  \\
         &  &  & $(56) \perp (((7) \perp (12)) \subset W_4)$  \\
\hline
    \end{tabular}
\end{table}

\begin{table}[htpb]
    \centering
    \caption{Degenerations of two of the three type 4 components of the $a_3b$ surface of \cref{tab:a3b} into four irreducible
        components each, yielding a type $aa_3b$ surface parameterized by a general point of the codimension three boundary
        stratum $D(7) \cap D(7 \perp 56 \perp 12) \cap D((123,456,7) \times (12,13,23) \times (45,46,56))$ (see \cref{fig:aa3b_12-23}).}
    \label{tab:aa3b}
    \begin{tabular}{| c | c | c |}
        \hline
         Surface & Count & Root Systems \\
         \hline
         \hline
                       $Bl_4\bP^2$ & 2 & $A_7 + (A_1 \subset A_2 \subset A_3^2)$ \\
         \hline
                       & & $X_0 + ((12) \subset (12,13,23) \subset A_3^2(7,56))$ \\
                       & & $X_0 + ((56) \subset (45,46,56) \subset A_3^2(7,12))$ \\
         \hline
         \hline
                       $\bF_0$ & 2 & $A_1 + (A_1 \subset A_2 \subset A_3^2)$ \\
         \hline
                       & & $(7) + ((12) \subset (12,13,23) \subset A_3^2(7,56))$ \\
                       & & $(7) + ((56) \subset (45,46,56) \subset A_3^2(7,12))$ \\
         \hline
         \hline
                       $\bF_0$ & 4 & $A_2 + (A_1 \subset A_2 \subset A_3^2)$ \\
         \hline
                       & & $(i,7,i7) + ((12) \subset (12,13,23) \subset A_3^2(7,56))$, $i=5,6$ \\
                       & & $(i,7,i7) + ((56) \subset (45,46,56) \subset A_3^2(7,12))$, $i=1,2$ \\
         \hline
    \end{tabular}
\end{table}

\begin{table}[htpb]
    \centering
    \caption{Degenerations of two of the three type 4 components of the $a_3b$ surface of \cref{tab:a3b} into four irreducible
        components each, yielding a type $a_2a_3b$ surface parameterized by a general point of the codimension two boundary
        stratum $D(7 \perp 56) \cap D(7 \perp 56 \perp 12) \cap D((123,456,7) \times (12,13,23) \times (45,46,56))$ (see
    \cref{fig:a2a3b_12-23}).}
    \label{tab:a2a3b}
    \begin{tabular}{| c | c | c |}
        \hline
         Surface & Count & Root Systems \\
         \hline
         \hline
                       $Bl_4\bP^2$ & 2 & $(A_1 \subset A_7) + (A_1 \subset A_2 \subset A_3^2)$ \\
         \hline
                       & & $((56) \subset X_0) + ((56) \subset (45,46,56) \subset A_3^2(7,12))$ \\
                       & & $((7) \subset X_{567}) + ((7) \subset (123,456,7) \subset A_3^2(56,12))$ \\
         \hline
         \hline
                       $\bF_0$ & 2 & $(A_1 \perp A_1) + (A_1 \subset A_2 \subset A_3^2)$ \\
         \hline
                       & & $((7) \perp (56)) + ((56) \subset (45,46,56) \subset A_3^2(7,12))$ \\
                       & & $((7) \perp (56)) + ((7) \subset (123,456,7) \subset A_3^2(56,12))$ \\
         \hline
         \hline
                       $\bF_0$ & 4 & $(A_1 \perp A_2) + (A_1 \subset A_2 \subset A_3^2)$ \\
         \hline
                       & & $((56) \perp (i,7,i7)) + ((56) \subset (45,46,56) \subset A_3^2(7,12))$, $i=1,2$ \\
                       & & $((7) \perp (56,i57,i67)) + ((7) \subset (123,456,7) \subset A_3^2(56,12))$, $i=1,2$ \\
         \hline
    \end{tabular}
\end{table}

\begin{table}[htpb]
    \centering
    \caption{Degenerations of the three type 4 components of the $a_3b$ surface of \cref{tab:a3b} into four irreducible
        components each, yielding a type $aa_2a_3b$ surface parameterized by a general point of the codimension two boundary
        stratum $D(7) \cap D(7 \perp 56) \cap D(7 \perp 56 \perp 12) \cap D((123,456,7) \times (12,13,23) \times (45,46,56))$ (see
    \cref{fig:aa2a3b_12-23}).}
    \label{tab:aa2a3b}
    \begin{tabular}{| c | c | c |}
        \hline
         Surface & Count & Root Systems \\
         \hline
         \hline
                       $Bl_4\bP^2$ & 1 & $A_1 + (A_1 \subset A_7) + (A_1 \subset A_2 \subset A_3^2)$ \\
         \hline
                       & & $(7) + ((7) \subset X_{567}) + ((7) \subset (123,456,7) \subset A_3^2(56,12))$ \\
         \hline
                       $Bl_4\bP^2$ & 1 & $A_7 + (A_1 \subset A_7) + (A_1 \subset A_2 \subset A_3^2)$ \\
         \hline
                       & & $X_0 + ((56) \subset X_0) + ((56) \subset (45,46,56) \subset A_3^2(7,12))$ \\
         \hline
                       $Bl_4\bP^2$ & 1 & $A_7 + A_3^2 + (A_1 \subset A_2 \subset A_3^2)$ \\
         \hline
                       & & $X_0 + A_3^2(7,56) + ((12) \subset (12,13,23) \subset A_3^2(7,56))$ \\
         \hline
         \hline
                       $\bF_0$ & 2 & $A_1 + (A_1 \perp A_1) + (A_1 \subset A_2 \subset A_3^2)$ \\
\hline
    \end{tabular}
\end{table}

\bibliographystyle{amsalpha} \bibliography{cubics}
\end{document}